\documentclass[11pt]{article}
\usepackage{fullpage, amssymb}
\usepackage{caption}
\usepackage{subcaption}
\usepackage{graphicx}
\usepackage{amsmath}
\usepackage{float}
\usepackage{listings}
\usepackage{xcolor}
\usepackage{amsthm}
\usepackage{algorithmic}
\usepackage{algorithm}
\usepackage[colorlinks,
linkcolor=blue,
anchorcolor=blue,
citecolor=blue]{hyperref}
\usepackage{natbib}
\usepackage{comment}
\usepackage{bbm}
\usepackage[shortlabels]{enumitem}
\usepackage{romannum}
\AtBeginDocument{\pagenumbering{arabic}}
\usepackage{multirow}
\usepackage{stmaryrd}

\setcitestyle{authoryear,round}

\newtheorem{theorem}{Theorem}
\newtheorem{lemma}{Lemma}
\newtheorem{definition}{Definition}
\newtheorem{proposition}{Proposition}

\newtheorem{assumption}{Assumption}
\newtheorem{remark}{Remark}
\newtheorem{claim}{Claim}

\def\rank{\textrm{rank}}
\def\wt{\widetilde}

\newcommand\fro[1]{\left\lVert #1 \right\rVert_{\rm{F}}}

\newcommand\frorr[1]{\| #1 \|_{\rm{F},\r}}

\newcommand\op[1]{\left\lVert #1 \right\rVert}

\newcommand\lone[1]{\left\lVert #1 \right\rVert_{1}}
\newcommand\ltwo[1]{\left\lVert #1 \right\rVert_{2}}
\newcommand\linft[1]{\left\lVert#1\right\rVert_{\infty}}
\newcommand\ltinf[1]{\left\lVert#1\right\rVert_{2,\infty}}
\newcommand\lph[1]{\left\lVert#1\right\rVert_{\rm{H}_p}}

\newcommand{\inp}[2]{\left\langle #1,#2\right\rangle}

\def\calD{{\mathcal D}}

\def\calN{{\mathcal N}}

\def\calP{{\mathcal P}}

\def\calT{{\mathcal T}}

\def\bSigma{{\boldsymbol{\Sigma}}}
\def\bXi{{\boldsymbol{\Xi}}}

\def\bcalA{{\boldsymbol{\mathcal A}}}
\def\bcalB{{\boldsymbol{\mathcal B}}}
\def\bcalC{{\boldsymbol{\mathcal C}}}
\def\bcalD{{\boldsymbol{\mathcal D}}}
\def\bcalE{{\boldsymbol{\mathcal E}}}

\def\bcalG{{\boldsymbol{\mathcal G}}}
\def\bcalH{{\boldsymbol{\mathcal H}}}

\def\bcalM{{\boldsymbol{\mathcal M}}}

\def\bcalS{{\boldsymbol{\mathcal S}}}
\def\bcalT{{\boldsymbol{\mathcal T}}}

\def\bcalX{{\boldsymbol{\mathcal X}}}
\def\bcalY{{\boldsymbol{\mathcal Y}}}

\def\EE{{\mathbb E}}
\def\FF{{\mathbb F}}

\def\MM{{\mathbb M}}

\def\OO{{\mathbb O}}
\def\PP{{\mathbb P}}

\def\RR{{\mathbb R}}

\def\TT{{\mathbb T}}

\def\e{{\mathbf e}}

\def\r{{\mathbf r}}

\def\u{{\mathbf u}}
\def\v{{\mathbf v}}
\def\w{{\mathbf w}}
\def\x{{\mathbf x}}

\def\A{{\mathbf A}}
\def\B{{\mathbf B}}

\def\G{{\mathbf G}}
\def\H{{\mathbf H}}
\def\I{{\mathbf I}}

\def\M{{\mathbf M}}

\def\Q{{\mathbf Q}}

\def\U{{\mathbf U}}
\def\V{{\mathbf V}}
\def\W{{\mathbf W}}
\def\X{{\mathbf X}}

\def\Z{{\mathbf Z}}

\def\fraM{{\mathfrak M}}

\def\lmax{{l_{\textsf{max}}}}

\def\eps{\varepsilon}
\def\dkm{d_k^{-}}

\def\mins{\underline{\lambda}}
\def\maxs{\overline{\lambda}}
\def\dmax{\bar{d}}

\def\muT{\mu^*}
\def\LS{\textsf{\tiny LS}}

\begin{document}
	\pagenumbering{arabic}
	
	\title{Quantile and pseudo-Huber Tensor Decomposition}
	
\author{Yinan Shen and Dong Xia \footnote{Dong Xia's research was partially supported by Hong Kong RGC Grant GRF 16300121.}	\\
{\small Department of Mathematics, Hong Kong University of Science and Technology}}

\date{(\today)}

	\maketitle
\begin{abstract}
This paper studies the computational and statistical aspects of quantile and pseudo-Huber tensor decomposition. The integrated investigation of computational and statistical issues of robust tensor decomposition poses challenges due to the non-smooth loss functions.  We propose a projected sub-gradient descent algorithm for tensor decomposition, equipped with either the pseudo-Huber loss or the quantile loss. 
In the presence of both heavy-tailed noise and Huber's contamination error,  we demonstrate that our algorithm exhibits a so-called phenomenon of two-phase convergence with a carefully chosen step size schedule.  The algorithm converges linearly and delivers an estimator that is statistically optimal with respect to both the heavy-tailed noise and arbitrary corruptions.  
Interestingly, our results achieve the first minimax optimal rates under Huber's contamination model for noisy tensor decomposition.  Compared with existing literature,  quantile tensor decomposition removes the requirement of specifying a sparsity level in advance,  making it more flexible for practical use.  
We also demonstrate the effectiveness of our algorithms in the presence of missing values.  Our methods are subsequently applied to the food balance dataset and the international trade flow dataset, both of which yield intriguing findings.
\end{abstract}

\section{Introduction}

Data in the form of multi-dimensional arrays, commonly referred to as tensors, have become increasingly prevalent in the era of big data. For instance, the monthly international trade flow \citep{cai2022generalized} of commodities among  countries is representable by a $47 \textrm{\small(countries)}\times 47 \textrm{\small(countries)} \times 97 \textrm{\small(commodities)} \times 12 \textrm{\small(months)}$ fourth-order tensor; the food balance data\footnote{The data is accessible from \url{https://www.fao.org/faostat/en/\#data/FBS}.}  describing the detailed report on the food supply of countries consist of several third-order tensors; the comprehensive climate dataset (CCDS, \cite{chen2020semiparametric}) – a collection of climate records of North America can be represented as a $125 \textrm{\small(locations)}\times 16 \textrm{\small(variables)}\times 156 \textrm{\small(time points)}$ third-order tensor. Tensor decomposition aims to find a low-rank approximation of tensorial data, which is a powerful tool of extracting hidden signal of low-dimensional structure. A tensor is considered low-rank if it can be expressed as the sum of a few rank-one tensors. A formal definition can be found in Section~\ref{sec:tensor-decomp}. Tensor decomposition has a variety of applications, including tensor denoising and dimension reduction \citep{lu2016tensor,zhang2018tensor}, community detection in hypergraph networks \citep{ke2019community}, node embedding in multi-layer networks \citep{jing2021community,cai2022generalized}, imputing missing data through tensor completion \citep{zhang2019cross,cai2019nonconvex,xia2021statistically}, clustering \citep{sun2019dynamic,wang2020learning}, and link prediction in general higher-order networks \citep{lyu2023latent}, among others.

While a tensor can be viewed as a natural extension of a matrix into a multi-dimensional space, finding a ``good" low-rank approximation of a tensor is fundamentally more challenging than finding the best low-rank approximation of a matrix. For any given matrix, its optimal low-rank approximation can be obtained through a singular value decomposition (SVD, \cite{golub2013matrix}), a process facilitated by highly efficient algorithms. In stark contrast, our understanding of the best low-rank approximation of a tensor is relatively limited \citep{kolda2009tensor}. Furthermore, computing the optimal low-rank approximation of a tensor is generally an NP-hard problem \citep{hillar2013most}. Therefore, computational feasibility becomes a crucial factor when we design statistical methods for tensor data analysis, even including the convex ones. To date, a variety of polynomial-time algorithms have been developed to find a good low-rank approximation of a tensor in Euclidean distance, such as the Frobenius norm. These algorithms can be locally or even globally optimal under certain statistical models, provided they are well-initialized. For example, \cite{de2000multilinear} introduced a higher-order singular value decomposition (HOSVD) method for tensor low-rank approximation which solely relies on multiple SVDs of rectangular matrices. They also found that an iterative refinement algorithm, known as Higher-Order Orthogonal Iterations (HOOI), can often enhance the performance in tensor low-rank approximation when applied after HOSVD. The sub-Gaussian tensor PCA model (also referred to as tensor SVD, as defined in Section~\ref{sec:tensor-decomp}) is a useful tool for studying the theoretical performance of tensor low-rank approximation algorithms. \cite{liu2022characterizing}, \cite{xia2019sup},  \cite{zhang2018tensor} and \cite{xia2021statistically} examined HOSVD and HOOI under sub-Gaussian noise, showing that while HOSVD is generally sub-optimal, HOOI achieves minimax optimality. A Burer-Monteiro type gradient descent algorithm, proposed by \cite{han2022optimal}, also achieves a minimax optimal rate under sub-Gaussian noise for tensor decomposition. \cite{cai2019nonconvex} studied a vanilla gradient descent algorithm and derived sharp error rates not only in Frobenius norm but also in sup-norm. A Riemannian gradient descent algorithm was also shown to be minimax optimal under sub-Gaussian noise by \cite{cai2022generalized}. More recently, \cite{lyu2023latent} investigated the Grassmannian gradient descent algorithm and demonstrated its minimax optimality under sub-Gaussian noise.

The technological revolution of recent decades has enabled the collection of vast amounts of information across a wide range of domains. The inherent heterogeneity of these domains can introduce outliers and heavy-tailed noise \citep{crovella1998heavy,rachev2003handbook,roberts2015heavy,sun2020adaptive} into tensorial datasets. Existing tensor decomposition algorithms typically seek a tensor low-rank approximation in the Frobenius norm, utilizing squared error as the loss function. However, the square loss is sensitive to outliers and heavy-tailed noise, which can render these algorithms unreliable in many real-world applications.
For example, when analyzing international trade flow data, a central objective is to study the economic ties between countries and their respective positions in the global supply chain. This structured and interconnected nature of global industries can often be encapsulated by a handful of multi-way principal components. However, outliers may occur if two countries have a substantial amount of trade flow simply due to geographical proximity or because one country is a primary supplier of a particular natural resource. Although such outliers are relatively rare in tensorial data, they can significantly skew the results of tensor low-rank approximation since they do not accurately reflect the countries' positions in the global supply chain.  Figure~\ref{fig:example} highlights the advantage of using absolute loss in handling outliers.  The figure focuses on the trading flow among approximately 50 countries,  specifically for the product `Petroleum oils and oils obtained from bituminous minerals; crude', from 2018 to 2022.
The top two sub-figures represent the node embedding of countries.  Red triangles represent (net) importers and blue circles represent (net) exporters.  A country is considered a (net) importer if it imports more than it exports, as is the case with the U.S.A.  Countries such as Saudi Arabia, Canada, and the Russian Federation, which export significant amounts, dominate the principal components in tensor decomposition using square loss.  Meanwhile, all other countries cluster together, as shown in the top-left sub-figure. The top-right figure represents the node embedding from tensor decomposition using absolute loss.  This is less sensitive to outlier entries caused by those three countries, leading to a more dispersed but better clustered embedding.  The bottom two sub-figures display the embedding results of months,  i.e.,  the third dimension of the tensor data.  Intuitively,  we would expect similar trading patterns for months within the same year.  This is indeed observed in the bottom-right sub-figure, which is produced by absolute-loss tensor decomposition. In contrast,  clusters are much less clear based on node embedding from the square-loss tensor decomposition, as shown in the bottom-left sub-figure.  It's important to note that the trade amount in the two months 202209 and 202210 is significantly smaller,  likely due to incomplete data, causing outlier slices in the tensor data.  The bottom-right sub-figure illustrates that absolute loss is insensitive to these outlier points.

\begin{figure}
	\centering
	\begin{subfigure}[b]{0.45\textwidth}
		\centering
		\includegraphics[width=\textwidth]{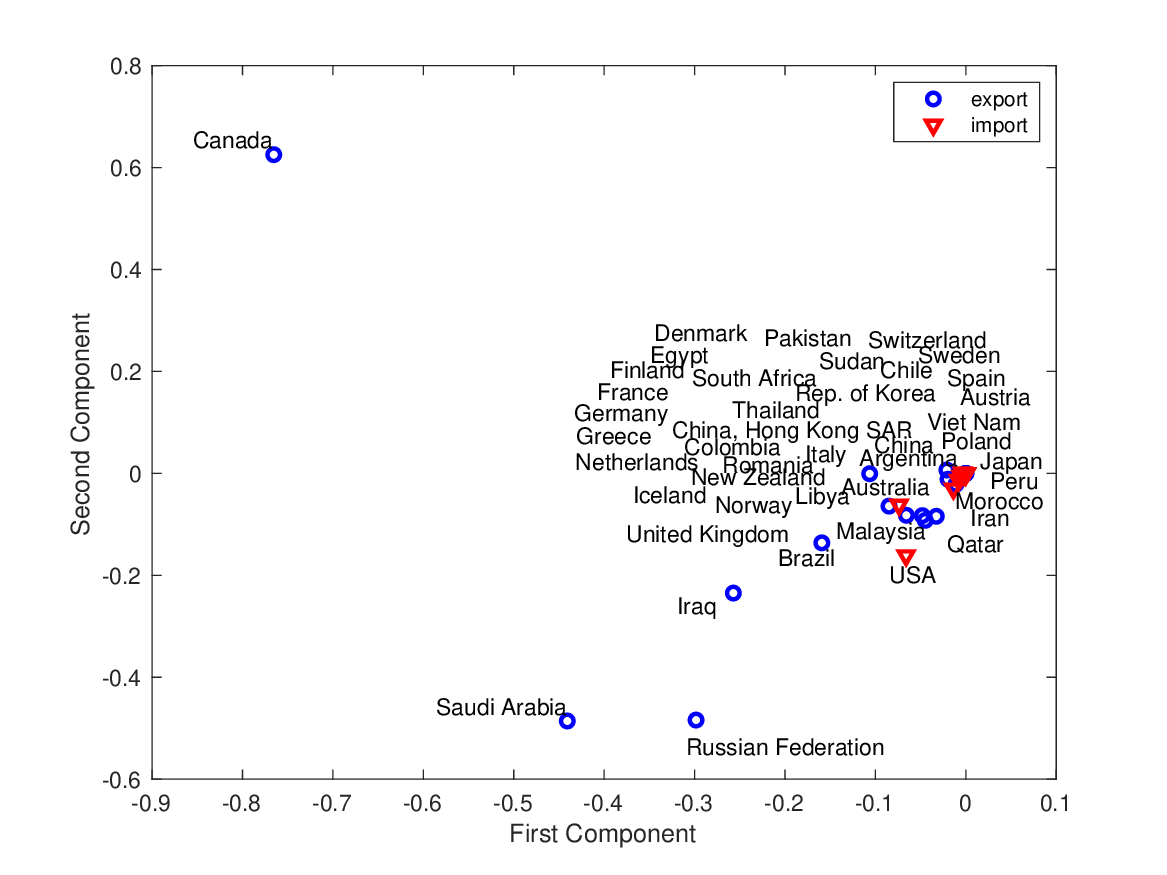}
		\caption{}
		\label{fig:intro1}
	\end{subfigure}
	\hfill
	\begin{subfigure}[b]{0.45\textwidth}
		\centering
		\includegraphics[width=\textwidth]{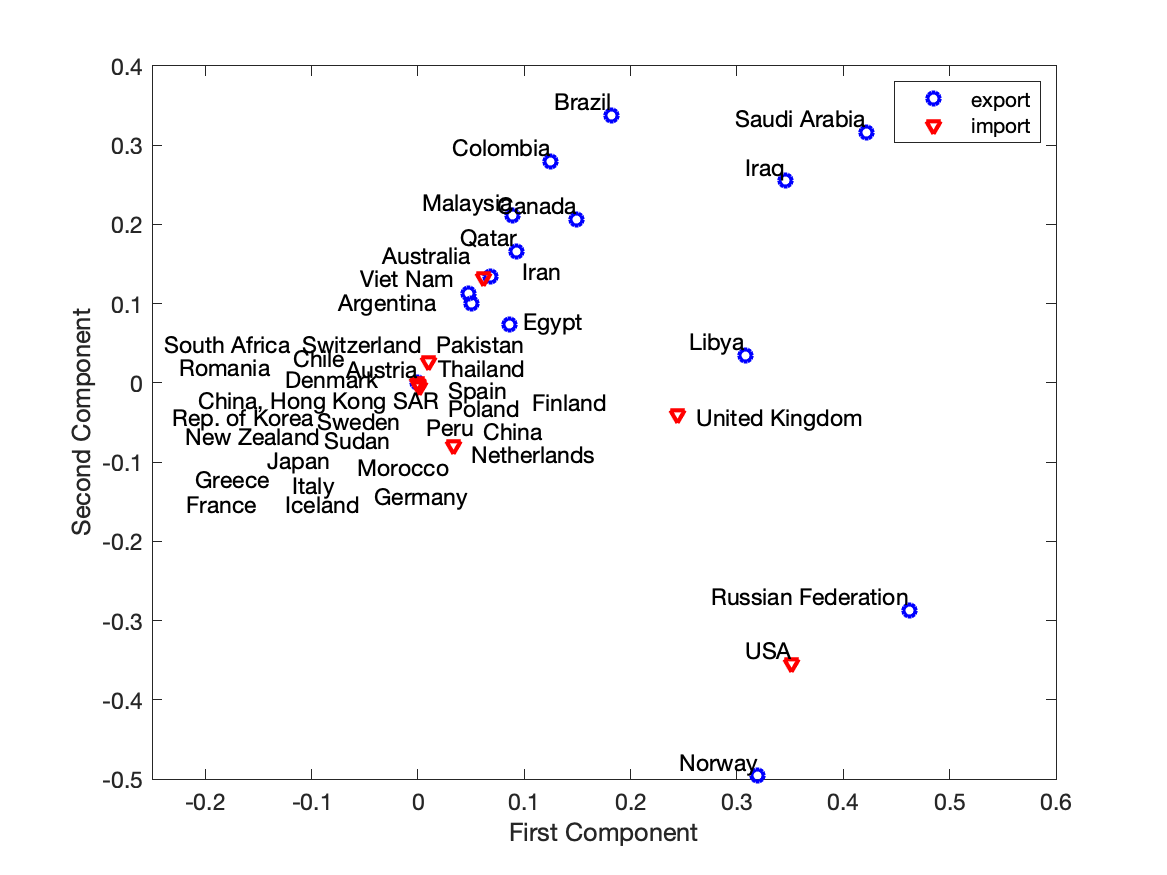}
		\caption{}
		\label{fig:intro2}
	\end{subfigure}
\\
\begin{subfigure}[b]{0.45\textwidth}
	\centering
	\includegraphics[width=\textwidth]{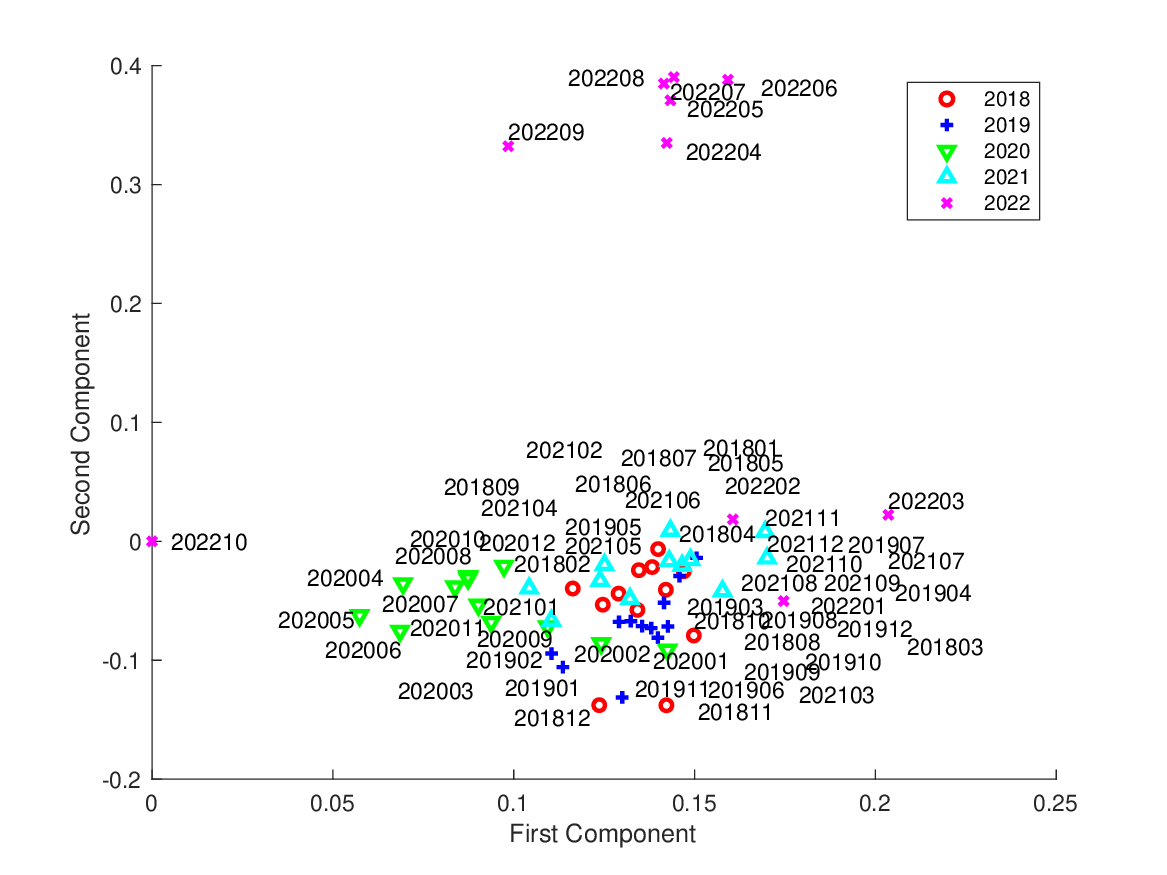}
	\caption{}
	\label{fig:intro3}
\end{subfigure}
\hfill
\begin{subfigure}[b]{0.45\textwidth}
	\centering
	\includegraphics[width=\textwidth]{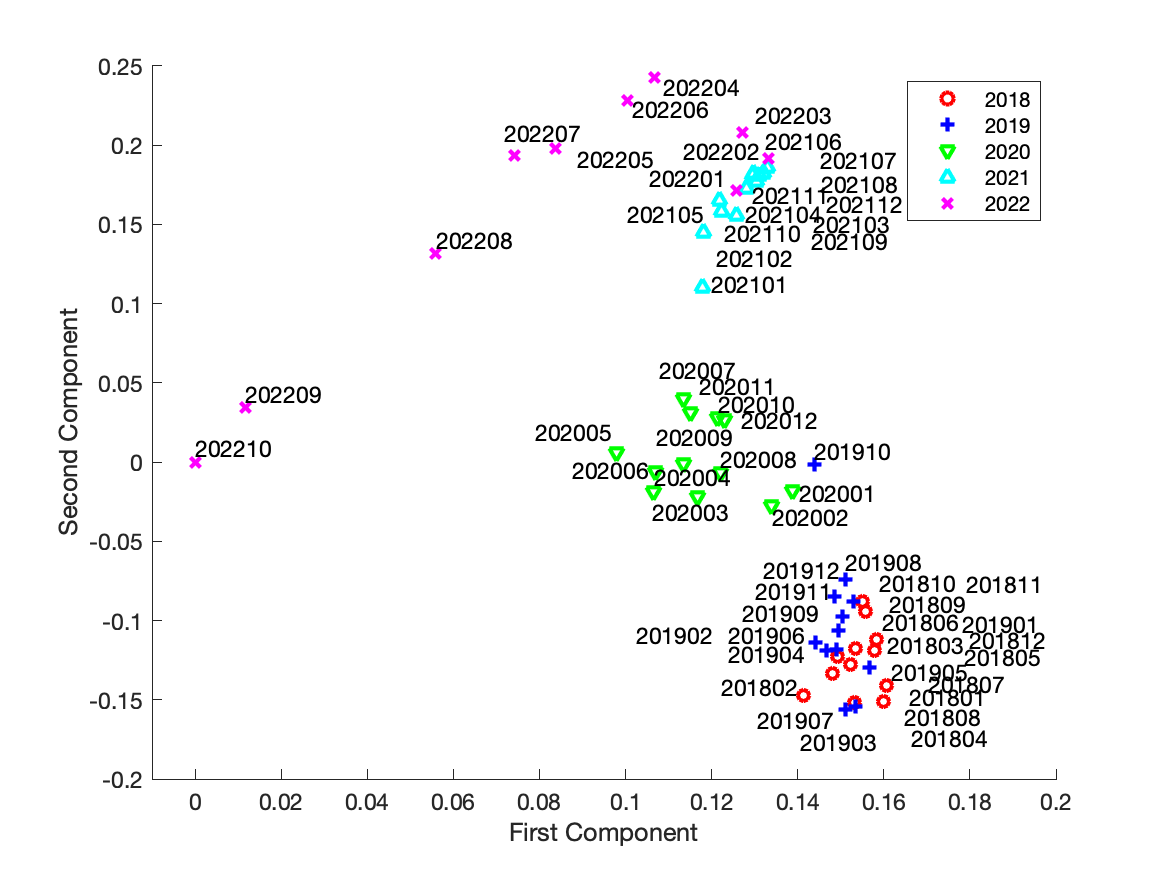}
	\caption{}
	\label{fig:intro4}
\end{subfigure}
	\caption{International trade flow data: node embedding of countries  and months from estimated principal components by tensor decomposition.  Left sub-figures: square-loss tensor decomposition; right sub-figures: absolute-loss tensor decomposition. }
	\label{fig:example}
\end{figure}

The development of statistical methods that are robust to outliers and heavy-tailed noise is garnering increasing significance in today's data-centric world. 
 A variety of these robust methods have been proposed, including the median of means \citep{minsker2015geometric, lecue2020robust,lugosi2019risk,depersin2020robust}, Catoni's method \citep{catoni2016pac,minsker2018sub}, and approaches involving trimming or truncation \citep{fan2016shrinkage,oliveira2019sub,lugosi2021robust}. These methods have proven useful for robust linear regression, mean, and covariance estimation. The issue of robustness against outliers has frequently been examined in theory \citep{depersin2022robust,dalalyan2022all,shen2023computationally,chinot2020robust,thompson2020outlier,minsker2022robust}, often resorting to Huber's contamination model \citep{huber1964robust}. This model posits that a fraction $\alpha\in(0,1)$ of the total samples are corrupted in an arbitrary manner. According to the findings of \cite{chen2016general, chen2018robust}, the minimax optimal error rate for several problems is directly proportional to $\alpha$ under Huber's model. 
Robust methods for matrix data analysis have also been extensively studied in the literature. The seminal work \cite{candes2011robust} examines matrix decomposition in the presence of sparse outliers, a problem known as robust PCA. Several studies \cite{candes2011robust,chandrasekaran2011rank,hsu2011robust,netrapalli2014non,yi2016fast} have demonstrated the possibility of precisely recovering a low-rank matrix corrupted by sparse outliers under specific identifiability conditions. Further, \cite{agarwal2012noisy} and \cite{klopp2017robust} explored the least squares estimator, employing a combination of nuclear norm and $\ell_1$-norm penalties imposing no assumptions over locations of the support, with additional sub-Gaussian noise.  Their derived error rates, proportional to $\alpha^{1/2}$, do not disappear even in the absence of the sub-Gaussian noise. This rate is optimal under arbitrary corruption but sub-optimal under Huber's contamination model where the optimal dependence on the corruption ratio is $\alpha$. A similar sub-optimal rate was exhibited by the non-convex method introduced by \cite{cai2022generalized} and the convex approach based on sorted-Huber loss proposed by \cite{thompson2020outlier}, both with regard to the proportion of corruption. A different perspective was offered by \cite{chen2021bridging}, who presented an alternating minimization algorithm that could attain an optimal error rate under strict conditions: uniformly random location of the outliers, random signs of the outliers, and sub-Gaussian noise. 
Heavy-tailed noise, a common source of outliers, can be treated as a combination of bounded noise and sparse corruption. This approach is generally sub-optimal, as noted by \cite{cai2022generalized}. Fortunately, heavy-tailed noise can usually be handled by robust loss functions including quantile loss, Huber loss, and the absolute loss. For instance, \cite{elsener2018robust,alquier2019estimation,chinot2020robust} showed that statistically optimal low-rank matrix estimators against heavy-tailed noise can be attained by utilizing those robust loss functions. However, all of these methods are based on convex relaxations and the computational aspect of the proposed estimators have not been thoroughly examined. It is important to bear in mind that the optimization process can be quite challenging due to the non-smooth nature of the aforementioned robust loss functions, even when the objective function is convex.

The integrated investigation of the computational and statistical aspects of robust low-rank methods is a somewhat under-explored area.
Both \cite{charisopoulos2021low} and \cite{tong2021low} examined the sub-gradient descent algorithm for matrix decomposition, employing robust loss functions. They demonstrated that the algorithm could achieve linear convergence with a schedule of decaying step sizes. However, the error rates derived from their research are generally sub-optimal, even under Gaussian noise conditions. In their respective works, \cite{cai2022generalized} and \cite{dong2022fast} adopted the square loss and introduced a sparse tensor to accommodate potential outliers resulting from heavy-tailed noise. Although this method ensures rapid computation, it is generally sub-optimal under standard heavy-tailed noise assumptions. The study by \cite{shen2023computationally} revealed that the sub-gradient descent algorithm could be both computationally efficient and statistically optimal for low-rank linear regression under heavy-tailed noise. They observed an intriguing phenomenon termed as ``two-phase convergence". However, it is important to note that the more technically demanding robust tensor decomposition differs significantly from low-rank linear regression, rendering the results of \cite{shen2023computationally} non-transferable. \cite{auddy2022estimating} proposed a one-step power iteration algorithm with Catoni-type initialization for rank-one tensor decomposition under heavy-tailed noise. This method, which only necessitates a finite second moment condition, achieves a near-optimal error rate up to logarithmic factors. The bound remains valid with a probability lower bounded by $1-\Omega(\log^{-1}d)$ for a tensor of size $d\times d\cdots\times d$. However, a strong signal strength condition is also vital for this method. 
Huber matrix completion was studied in \cite{wang2022robust} through the lens of leave-one-out analysis. Due to technical constraints, their analysis framework is not applicable to tensor decomposition, and a significantly large truncate threshold is necessitated by \cite{wang2022robust}. How the methods proposed by \cite{auddy2022estimating} and \cite{wang2022robust} behave in the presence of arbitrary outliers remains unclear.  Robust tensor decomposition in the presence of missing values presents even greater challenges. Shrinkage-based approaches for the matrix case have been studied by \cite{minsker2018sub} and \cite{fan2016shrinkage}.  While their rates are optimal with respect to the dimension and sample size under a minimal second-order moment noise condition, their derived rates are not proportional to the noise level. \cite{wang2022robust} extended the leave-one-out analysis to the vanilla sub-gradient descent algorithm for matrix completion under heavy-tailed noise. However, their entry-wise error rate is still sub-optimal, and it remains unclear whether their method is applicable to tensors and with arbitrary corruptions.  We believe that this sub-optimality is due to technical reasons. We demonstrate this by showing that a simple sample splitting trick can yield statistical optimality for both Frobenius-norm and entry-wise error rates, even in the presence of arbitrary corruptions.

In this paper, we develop computationally fast and statistically optimal methods for tensor decomposition, robust to both heavy-tailed noise and sparse arbitrary corruptions. Our contributions are summarized as follows. 
	\begin{enumerate}[1.]
	\item We propose a tensor decomposition framework that employs quantile loss and pseudo-Huber loss.  Existing works in robust tensor decomposition often falls short in terms of algorithmic development, computational guarantees,  and statistical optimality.  To address this,  we introduce a computationally efficient algorithm grounded in Riemannian (sub-)gradient descent.  We simultaneously explore computational convergence and statistical performance,  demonstrating that our proposed algorithm converges linearly and achieves statistical optimality in handling both heavy-tailed noise and arbitrary corruptions.  Unlike previous works \citep{cai2022generalized, dong2022fast}, our method does not necessitate the specification of a sparsity level in advance.  A phenomenon of two-phase convergence is also observed in the proposed algorithms for robust tensor decomposition.   We apply our methods to the food balance dataset and international trade flow dataset, both of which yield intriguing findings.   
	
	\item Our approach offers several theoretical benefits.  We demonstrate that quantile and pseudo-Huber tensor decomposition can achieve statistical optimality under both dense noise and arbitrary corruptions,  regardless of whether the noise is sub-Gaussian or heavy-tailed. 
Existing works often treat sparse corruptions using heavy-tailed distributions, as seen in \cite{cai2022generalized, fan2016shrinkage,auddy2022estimating,wang2022robust}.  We examine the robustness to sparse corruptions under Huber's contamination model.  Even in the presence of both heavy-tailed noise and Huber's contamination, our approach can still deliver a statistically optimal estimator. 
We are the first to derive the minimax optimal rate of matrix/tensor decomposition under Huber's contamination model.  Previously,  methods by \cite{agarwal2012noisy, klopp2017robust, cai2022generalized} achieved an error rate proportional to $\alpha^{1/2}$,  where $\alpha$ is the proportion of contamination under Huber's model.  We demonstrate that quantile tensor decomposition achieves an error rate proportional to $\alpha$,  which is minimax optimal under Huber's contamination model. 
The left sub-figure in Figure~\ref{fig:Corruption Rate} showcases the achieved error rate by absolute-loss tensor decomposition under Huber's contamination model.  It examines both cases of dense Gaussian noise and Student's t noise.  The plot reveals a linear pattern between the achieved error and the corruption rate.

		\item Robust tensor decomposition poses greater technical challenges than high-dimensional linear regression \citep{shen2023computationally}.  Our key technical contribution lies in demonstrating the so-called two-phase regularity properties of the absolute loss and pseudo-Huber loss.  Particularly noteworthy is the second-phase regularity condition where the size of the projected sub-gradient (namely, the Riemannian sub-gradient of the loss) diminishes as the estimate approaches the true model parameter.  We also prove the first-phase regularity condition that was initially conjectured in \cite{charisopoulos2021low}.  
Robust tensor decomposition becomes even more complex in the presence of missing values, where the powerful leave-one-out framework still yields sub-optimal results. We posit that the sub-optimality is caused by technical difficulty, and demonstrate that a simple sample splitting trick can yield a statistically optimal error rate under missing values and in the presence of arbitrary outliers.
		
	\end{enumerate}

\begin{figure}[t]
	\centering
	\begin{subfigure}[b]{0.45\textwidth}
		\centering
		\includegraphics[width=\textwidth]{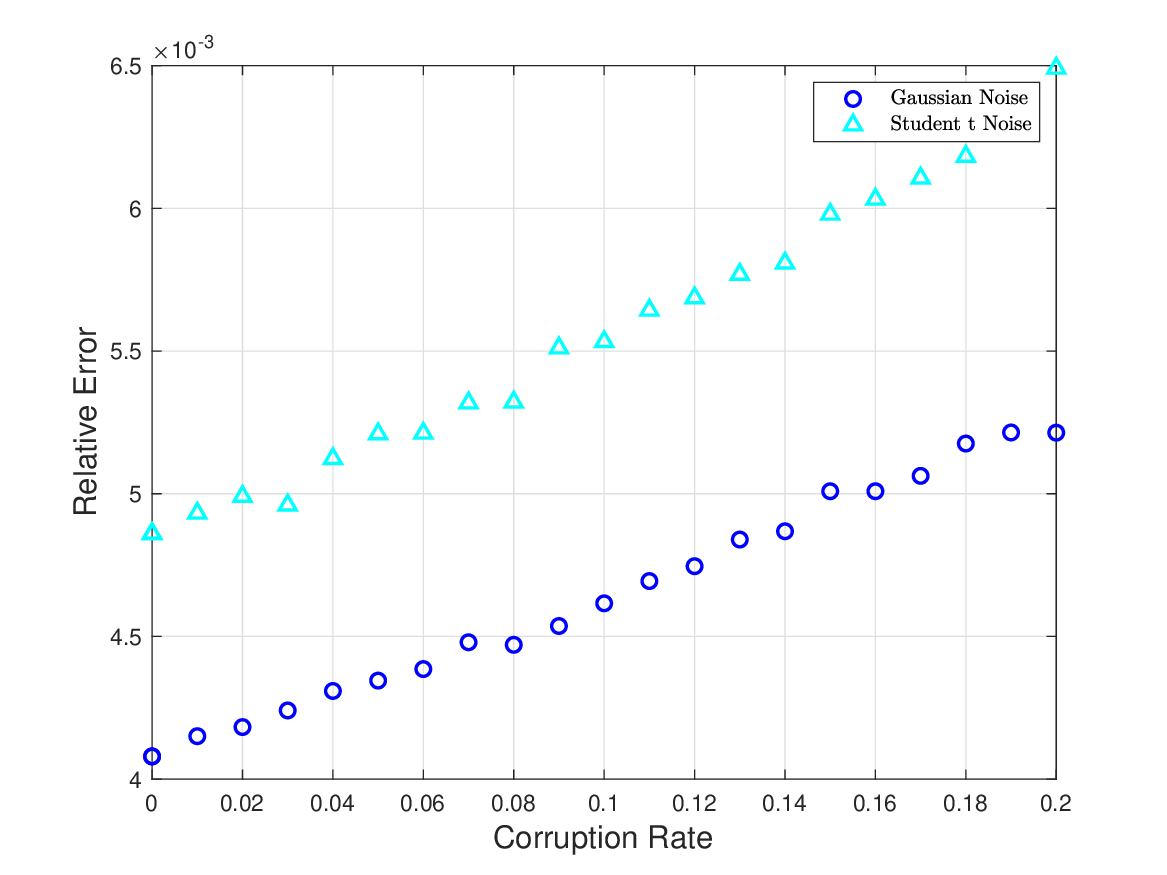}
		\caption{Error against corruption rate}
		\label{fig:Corruption Rate}
	\end{subfigure}
	\hfill
	\begin{subfigure}[b]{0.45\textwidth}
		\centering
		\includegraphics[width=\textwidth]{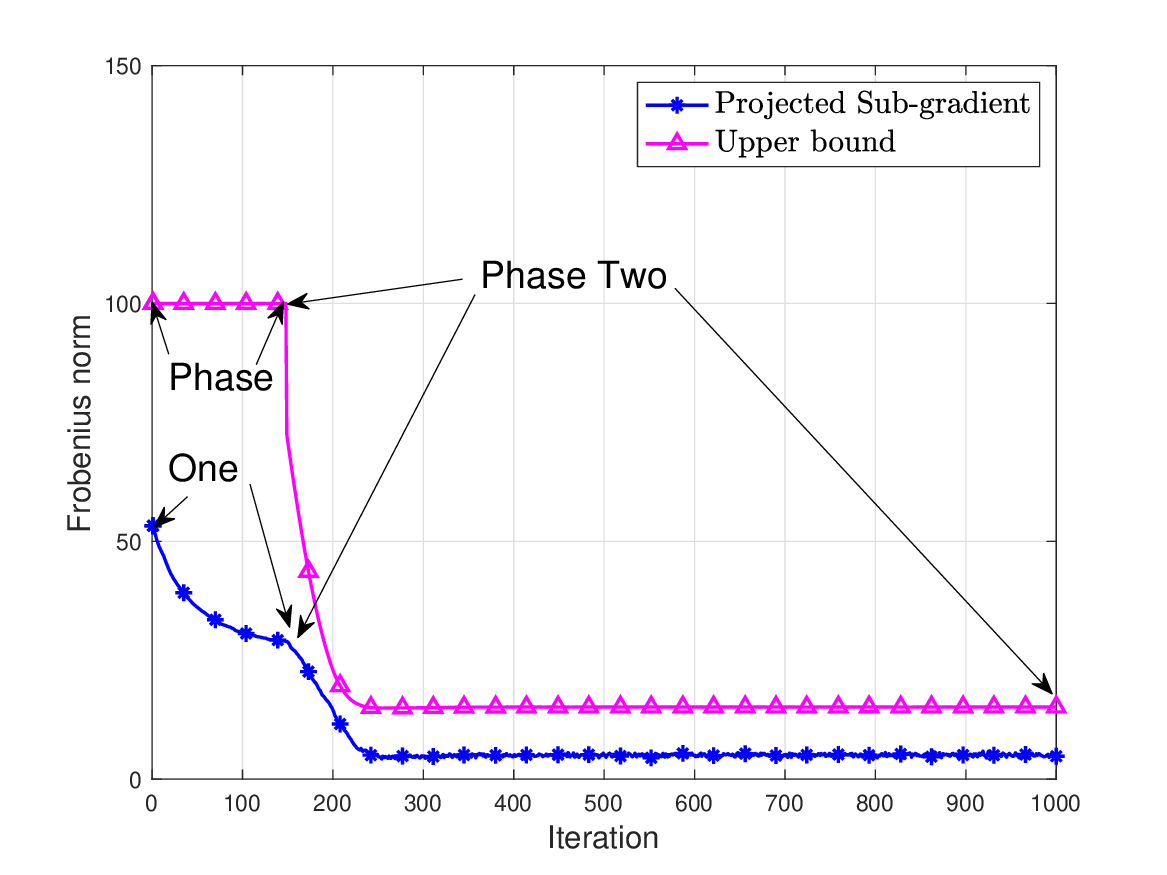}
		\caption{Norm of projected sub-gradient}
		\label{fig:UpperBoundSubGrad}
	\end{subfigure}
	\caption{Optimal rate by and regularity property of absolute loss.  
	Left: relative error $\|\widehat{\bcalT}-\bcalT^*\|_{\rm F}/ \fro{\bcalT^*}$ against the corruption rate $\alpha$ under Huber's contamination model and in the presence of dense Gaussian or Student's t noise.  Plot is based the average over $100$ replications.   Here $\widehat\bcalT$ denotes the estimator produced by our algorithm.  Right: the Frobenius norm of projected sub-gradient of the absolute loss $\|\bcalT_l-\bcalY\|_1$.  Here $\bcalT_l$ denotes the updated estimate after $l$-th iteration.  }
	\label{fig:intro}
\end{figure}

\section{Tensor Decomposition and Robust PCA}\label{sec:tensor-decomp}

We shall write tensors in bold calligraphy font,  such as $\bcalC,\bcalM,\bcalT$ and write matrices in upper-case bold face,  such as $\U,\V,\W$.  Lower-case bold face letters such as $\u, \v,\w$ denote vectors.  
 An $m$-th order tensor $\bcalT\in\RR^{d_1\times\cdots\times d_m}$ is an $m$-dimensional array and $d_j$ is the size in $j$-th dimension.  Denote its mode-$j$ matricization of $\bcalT$ as $\fraM_{j}(\bcalT)\in \RR^{d_j\times d_j^{-}}$, where $d_j^{-}:=\prod_{l\neq j} d_l$. 
	The mode-$j$ marginal multiplication between a tensor $\bcalT$ and a matrix $\U^{\top}\in\RR^{r_j\times d_j}$ results into an $m$-th order tensor of size $d_1\times\cdots d_{j-1}\times r_j\times d_{j+1}\cdots d_m$, whose elements are $(\bcalT\times_j \U^{\top})_{i_1\cdots i_{j-1}l i_{j+1}\cdots i_m}:=\sum_{i_j=1}^{d_j} [\bcalT]_{i_1\cdots i_{j-1} i_{j} i_{j+1}\cdots i_m}\U_{i_jl}.$  A simple and useful fact is $\fraM_{j}\left(\bcalT\times_j\U^{\top}\right)=\U^{\top}\fraM_{j}(\bcalT)$.  
Unlike matrices,  there are multiple definitions of tensor ranks.  Throughout this paper,  tensor ranks are referred to as the Tucker ranks \citep{tucker1966some}.   The $m$-th order tensor $\bcalT$ is said to have Tucker rank $\r:=(r_1,r_2,\cdots, r_m)$ if its mode-$j$ matricization has rank $r_j$, i.e., $r_j=\text{rank}(\fraM_{j}(\bcalT))$.  As a result,  $\bcalT$ admits the so-called Tucker decomposition $\bcalT=\bcalC\cdot \llbracket \U_1,\cdots,\U_m\rrbracket:=\bcalC\times_1\U_1\times_2\cdots\times_m\U_m$ where the core tensor $\bcalC$ is of size $r_1\times\cdots\times r_m$ and $\U_j\in\RR^{d_j\times r_j}$ has orthonormal columns.  Tucker decomposition is conceptually similar to the matrix SVD except that the core tensor is generally not diagonal. 	Interested readers are suggested to refer to \cite{kolda2009tensor,de2008tensor,de2000multilinear} for more details about Tucker ranks and Tucker decomposition.  Tucker decomposition is well-defined and can be fast computed by HOSVD.  
For notational convenience,  we denote $d^*:=d_1\cdots d_m$, $d_k^-:=d^*/d_k$, $r^*:=r_1\cdots r_m$, $r_k^-:= r^*/r_k$ for any $k\in[m]$.  Denote $\r:=(r_1,\cdots,r_m)^{\top}$ and $\MM_{\r}:=\{\bcalT\in\RR^{d_1\times\cdots\times d_m}:\; \text{rank}(\fraM_{k}(\bcalT))\leq r_k\}$ the set of tensors with Tucker rank bounded by $\r$.

Noisy tensor decomposition is concerned with reconstructing a low-rank tensor from noisy observation.  Consider an $m$-th order tensor $\bcalA$ of size $d_1\times\cdots\times d_m$. This could be representative of various types of data, such as international trade flow among countries \citep{cai2022generalized, lyu2023optimal} or a higher-order network \citep{ke2019community,jing2021community}, among others. The fundamental premise of tensor decomposition is the existence of a low-rank ``signal" tensor $\bcalT^{\ast}$ embedded within $\bcalA$.  Here, $\r$ represents the Tucker ranks of $\bcalT^{\ast}$,  satisfying that $r_k\ll d_k$ for all $k\in[m]$.
Throughout this paper, we assume additive noise,  leading to a linear model.  For more context on tensor decomposition in generalized linear models, please refer to \cite{han2022optimal, lyu2023optimal, lyu2023latent}.  With the assumption of additive noise, tensor decomposition strives to find a low-rank approximation for the tensorial data $\bcalA$. If the additive noise is sub-Gaussian, the associated model is often referred to as sub-Gaussian tensor PCA \citep{cai2022generalized} and the signal tensor can be estimated by the least squares estimator
\begin{equation}\label{eq:least-squares}
\widehat\bcalT^{\LS}:=\underset{\bcalT\in\MM_{\r}}{\arg\min}\ \|\bcalT-\bcalA\|_{\rm F}^2:=\sum_{\omega\in[d_1]\times\cdots\times [d_k]} \big([\bcalT]_{\omega}-[\bcalA]_{\omega}\big)^2.
\end{equation}
The optimization problem involved in (\ref{eq:least-squares}) is generally NP-hard. Computationally efficient algorithms have been developed to find locally optimal solutions which are statistically optimal under strong signal-to-noise ratio (SNR) conditions. See, e.g., \cite{zhang2018tensor,liu2022characterizing, cai2022generalized}.

This paper focuses on tensor decomposition in the existence of heavy-tailed noise and arbitrary corruptions/outliers. More specifically, we study the robust tensor PCA model in that the observed tensor data, denoted as $\bcalY$,  consists of three underlying parts:
\begin{equation}\label{eq:model}
\bcalY=\bcalT^{\ast}+\bXi+\bcalS.
\end{equation}
The signal tensor, represented as $\bcalT^{\ast}$, holds a Tucker rank of $\r$. The dense noise tensor, $\bXi$, potentially contains entries with heavy tails, and $\bcalS$ is a sparse tensor that captures arbitrary corruptions or outliers. It's important to note that heavy-tailed noise can result in outliers, and the additional sparse tensor $\bcalS$ accommodates Huber's contamination model. It is possible that $\bcalT^{\ast}$ and $\bcalS$ may be indistinguishable if $\bcalT^{\ast}$ itself also exhibits sparsity. For identifiability, the incoherent condition introduced by \cite{candes2011robust} is often necessary. The set of $\mu$-incoherent rank-$\r$ tensors is denoted by $\MM_{\r,\mu}:=\{\bcalT\in\MM_{\r}: \mu(\bcalT)\leq \mu\}$.

\begin{definition} A tensor $\bcalT=\bcalC\cdot\llbracket\U_1,\dots,\U_m\rrbracket$ with Tucker rank $\r=(r_1,\dots,r_m)$ is said $\mu$-incoherent iff $\mu(\bcalT):=\max_{k=1,\dots,m}\ltinf{\U_k}^2\cdot d_k/r_k\leq\mu,$ or equivalently $\ltinf{\U_k}\leq (\mu r_k/d_k)^{1/2}$ for each $k=1,\dots,m$.
\end{definition}

Heavy-tailed noise and outliers can be handled by robust loss functions. In the following sections, we focus on two specific robust loss functions:
{\it
	\begin{enumerate}[1.]
		\item {\it Pseudo-Huber loss:} $\rho_{H_p,\delta}(x):=(x^2+\delta^2)^{1/2}$ for any $x\in\RR$ where $\delta>0$ is a tuning parameter;
		\item {\it Quantile loss:} $\rho_{Q,\delta}(x):=\delta x\mathbbm{1}(x\geq 0)+(\delta-1)x\mathbbm{1}(x<0)$ for any $x\in \RR$ with $\delta:=\PP(\xi\leq 0)$. Without loss of generality, only the case $\delta=1/2$, i.e, absolute loss $\rho(x)=|x|$,  will be specifically studied. 
	\end{enumerate}
}

A robust low-rank estimator for $\bcalT^{\ast}$ can be achieved through tensor decomposition combined with robust loss functions. More specifically, we define
\begin{align}\label{eq:loss}
		\widehat\bcalT:=\underset{\bcalT\in\MM_{\r,\muT}}{\arg\min}\ f(\bcalT) \quad \textrm{ where } f(\bcalT) := \sum_{\omega\in[d_1]\times\cdots\times[d_m]}\rho\big([\bcalT]_{\omega}-[\bcalY]_{\omega}\big).
\end{align}
Here, $\rho(\cdot)$ can represent either the pseudo-Huber or quantile loss and $\mu^{\ast}$ denotes incoherence parameter of $\bcalT^{\ast}$. The optimization program involved in equation (\ref{eq:loss}) presents a greater challenge than that in equation (\ref{eq:least-squares}) due to the often non-smooth nature of robust loss functions. Our aim is to develop a fast converging algorithm capable of finding a local minimizer for equation (\ref{eq:loss}), which is also statistically optimal w.r.t. the heavy-tailed noise and arbitrary corruptions with high probability.


\section{Pseudo-Huber Tensor Decomposition}\label{sec:warm}
In this section, we study tensor decomposition using the pseudo-Huber loss and demonstrate its robustness to heavy-tailed noise. More specifically, suppose the observed tensor $\bcalY=\bcalT^{\ast}+\bXi$ where $\bXi$ is a noise tensor whose entries are i.i.d. centered random variables.  Denote $\rho_{H_p, \delta}(x):=(x^2+\delta^2)^{1/2}$ the pseudo-Huber loss with a tuning parameter $\delta>0$. The pseudo-Huber loss is a smooth approximation of the absolute loss and Huber loss. We estimate $\bcalT^{\ast}$ by solving the following non-convex program:
\begin{align}
\widehat{\bcalT}= \underset{\bcalT\in\MM_{\r,\mu}}{\arg\min} \lph{\bcalT-\bcalY}:= \sum_{\omega\in[d_1]\times\cdots\times[d_m]} \rho_{H_p,\delta}\big([\bcalT]_{\omega}-[\bcalY]_{\omega}\big).
		\label{eq:loss-pHuber}
\end{align}
Here $\mu$ is some constant larger than the $\mu^{\ast}=\mu(\bcalT^{\ast})$, i.e., the incoherence parameter of the ground truth. Note that \cite{cambier2016robust} has empirically demonstrated the benefit of pseudo-Huber loss in matrix completion. We prove that pseudo-Huber loss is indeed robust to heavy-tailed noise and can deliver a statistically optimal estimator under mild conditions.

\subsection{Projected gradient descent}
Finding the global minimizer of program (\ref{eq:loss-pHuber}) is generally NP-hard. We only intend to find a local minimizer which enjoys statistical optimality. The objective function in  (\ref{eq:loss-pHuber}) is convex, but the feasible set is non-convex. Meanwhile, the set of fixed-rank tensors forms a Riemannian manifold. We apply the projected gradient descent \citep{chen2015fast} algorithm to solving the program (\ref{eq:loss-pHuber}). The vanilla gradient is usually full-rank, rendering the projection step computationally intensive. For computational benefit, we utilize the Riemannian gradient which is also low-rank. This corresponds to the Riemannian gradient descent algorithm extensively studied in the recent decade. See, e.g., \cite{vandereycken2013low,cambier2016robust,wei2016guarantees,cai2022generalized, shen2022computationally} and references therein. The details are in Algorithm~\ref{alg:pseuHuber:RsGrad}. The algorithm consists of two  main steps. First, at the current iterate $\bcalT_l$, Algorithm~\ref{alg:pseuHuber:RsGrad} moves along the Riemannian gradient, which is the projection of the vanilla gradient into the tangent space, denoted as $\TT_l$, of $\MM_{\r}$ at $\bcalT_l$. The second step retracts the updated estimate back to the feasible set $\MM_{\r}$. Although the retraction step seems to require the computation of HOSVD \citep{de2000multilinear} of a $d_1\times\cdots\times d_m$ tensor, which would be rather computational costly, in fact it can be reduced to the HOSVD of a $2r_1\times\cdots\times 2r_m$ tensor. For more details of computation implementation, please refer to \cite{cai2020provable,cai2022generalized,shen2022computationally,luo2022tensor}. Note that Algorithm~\ref{alg:pseuHuber:RsGrad} requires no further steps to ensure the incoherence. Instead, we shall prove that the iterates output by Algorithm~\ref{alg:pseuHuber:RsGrad} maintain the incoherence property if equipped with a good initialization.

\begin{algorithm}[H]
	\caption{Riemannian Gradient Descent for Pseudo-Huber Tensor Decomposition}\label{alg:pseuHuber:RsGrad}
	\begin{algorithmic}
		\STATE{\textbf{Input}: observations $\bcalY$,  max iterations $\lmax$,  step sizes $\{\eta_l\}_{l=0}^{\lmax}$.}
		\STATE{Initialization: $\bcalT_0\in\MM_\r$}
		\FOR{$l = 0,\ldots,\lmax$}
		\STATE{Choose a vanilla gradient:  $\bcalG_l\in\partial \|\bcalT_l -\bcalY\|_{H_p}$}
		\STATE{Compute Riemannian gradient: $\wt\bcalG_l = \calP_{\TT_l}(\bcalG_l)$}
		\STATE{Retraction to $\MM_\r$: $\bcalT_{l+1} = \text{HOSVD}_\r(\bcalT_l - \eta_{l}\wt\bcalG_l)$}
		\ENDFOR
		\STATE{\textbf{Output}: $\widehat\bcalT=\bcalT_{\lmax}$}
	\end{algorithmic}
\end{algorithm}

\subsection{Algorithm convergence and statistical optimality}
Let $\xi$ be a heavy-tailed random variable denote the entrywise error, i.e., the entries of $\bXi$ are i.i.d. and have the same distribution as $\xi$. Denote $h_{\xi}(\cdot)$ and $H_{\xi}(\cdot)$ the density and distribution of  $\xi$, respectively.  Pseudo-Huber tensor decomposition requires the following condition of the noise. 

\begin{assumption}[Noise condition \Romannum{1}]
		There exists an $\eps>0$ such that $\gamma:=\left(\EE|\xi|^{2+\eps}\right)^{1/(2+\eps)}<+\infty$. The density function $h_{\xi}(\cdot)$ is zero symmetric\footnote{The zero-symmetric condition can be slightly relaxed to $\frac{d}{dt}\EE\big(t-\xi)^2+\delta^2\big)^{1/2}\big|_{t=0}=0$, which is equivalent to $\int_{-\infty}^{+\infty} s(s^2+\delta^2)^{-1/2}h_{\xi}(s)\, ds=0$.} in that $h_{\xi}(x)=h_{\xi}(-x)$. There exists $b_0>0$ such that $ h_{\xi}(x)\geq b_0^{-1}$ for all  $|x|\leq C_{m,\muT,r^*}(6\gamma+\delta),$ where $C_{m,\muT,r^*}:=72(5m+1)^23^m\mu^{*m} r^*$ and $\delta$ is the pseudo-Huber loss parameter.
		\label{assu:hub:noise}
\end{assumption}

Basically, Assumption~\ref{assu:hub:noise} requires a finite $2+\eps$ moment bound of noise. The lower bound condition of noise density has appeared in existing literature such as \cite{elsener2018robust,alquier2019estimation,chinot2020robust,wang2020tuning,shen2023computationally}. Note that $b_0$ is only related to the random noise $\xi$ together with pseudo-Huber parameter $\delta$.  Assumption~\ref{assu:hub:noise} also implies a lower bound $b_0\geq C_{m,\muT,r^*}(6\gamma+\delta)$. By choosing a parameter $\delta=O(\gamma)$, the relationship $b_0\asymp \EE|\xi|$ holds for Gaussian noise, Student's t noise, and zero symmetric Pareto noise, etc.

The convergence dynamic of Algorithm~\ref{alg:pseuHuber:RsGrad} and statistical performance are decided by the schedule of step sizes. They are related to regularity properties of the objective function. Interestingly, the following lemma shows that the pseudo-Huber loss exhibits two-phase regularity properties depending on the closeness between $\bcalT$ and the ground truth. Define $\textsf{DoF}_m:=r_1r_2\cdots r_m+\sum_{j=1}^{m}d_jr_j$, reflecting the model complexity. Here the sup-norm $\|\bcalA\|_{\infty}:=\max_{\omega\in[d_1]\times\cdots\times [d_m]} \big|[\bcalA]_{\omega} \big|$ and the $(2,\infty)$-norm of a $d_1\times p_1$ matrix is defined by $\|\A\|_{2,\infty}:=\max_{i\in[d_1]} \|\e_i^{\top}\A\|$ where $\|\cdot\|$ denotes the vector $\ell_2$-norm and $\e_i$ denotes the $i$-th standard basis vector. 

\begin{lemma}[Two-phase regularity properties of pseudo-Huber loss]
Suppose the noise $\bXi$ has i.i.d. entries satisfying Assumption~\ref{assu:hub:noise}. There exist absolute constants $c,c_1,c_2>0$ such that with probability exceeding $1-c\sum_{k=1}^{m} d_k(\dkm)^{-1-\min\{1,\eps\}}-\exp\left(-\textsf{DoF}_m/2\right)$, the following facts hold. 
\begin{enumerate}[(1)]

\item For all $\bcalT\in\RR^{d_1\times\cdots\times d_m}$ and any gradient $\bcalG\in\partial \lph{\bcalT-\bcalY}$, 
$$
\fro{\calP_{\TT}(\bcalG)}\leq (d^*)^{1/2},\quad \lph{\bcalT-\bcalY}-\lph{\bcalT^*-\bcalY}\geq \linft{\bcalT-\bcalT^*}^{-1}\cdot\fro{\bcalT-\bcalT^*}^2-6d^*\gamma-d^*\delta.
$$
Here $\TT$ denotes the tangent space of $\MM_{\r}$ at the point $\bcalT$.  Furthermore, if $\bcalT$ is $\mu$-incoherent, then for each $k\in[m]$ and $j\in[d_k]$, 
$$
\ltinf{\fraM_{k}\left(\calP_{\TT}(\bcalG)\right)}\leq \big(3\mu r_k\cdot \dkm\big)^{1/2},
$$
$$
\lph{\fraM_{k}(\bcalT-\bcalY)_{j,\cdot}}-\lph{\fraM_{k}(\bcalT^*-\bcalY)_{j,\cdot}}\geq\linft{\fraM_{k}(\bcalT-\bcalT^*)_{j,\cdot}}^{-1}\cdot\fro{\fraM_{k}(\bcalT-\bcalT^*)_{j,\cdot}}^2-6\dkm \gamma-\dkm\delta.
$$

\item For all $\bcalT\in\MM_{\r}$ satisfying $\linft{\bcalT-\bcalT^*}\leq C_{m,\muT,r^*}(6\gamma+\delta)$ and $\fro{\bcalT-\bcalT^*}\geq c_1b_0\sqrt{\textsf{DoF}_m}$, 
$$
\fro{\calP_{\TT}(\bcalG)}\leq c_2\delta^{-1}\sqrt{m+1}\cdot \fro{\bcalT-\bcalT^*},\quad \lph{\bcalT-\bcalY}-\lph{\bcalT^*-\bcalY}\geq (4b_0)^{-1}\cdot\fro{\bcalT-\bcalT^*}^2.
$$
\end{enumerate}
		\label{lem:hub:regularity}
\end{lemma}

Lemma~\ref{lem:hub:regularity} admits a sharper characterization of the lower bound on the objective function and the upper bound on the Riemannian gradient when $\bcalT$ is closer to the ground truth $\bcalT^{\ast}$. The loose bound in $(1)$ is derived directly by a triangular inequality, while the bound in $(2)$ relies on techniques from empirical processes \citep{boucheron2013concentration,ludoux1991probability,van1996weak}. The lower bound for Lipschitz objective function such as $\lph{\bcalT-\bcalY}-\lph{\bcalT^*-\bcalY}$ is often referred to as the sharpness condition or margin condition in the literature \citep{elsener2018robust,charisopoulos2021low}. \cite{chinot2020robust} generalizes such lower bounds with a \textit{local Bernstein condition}. The upper bound of the Riemannian gradient plays a critical role in the convergence dynamic of Algorithm~\ref{alg:pseuHuber:RsGrad}. Note that a trivial upper bound of $\rho'_{H_p,\delta}(x)$ is one and thus the upper bound of $\|\calP_{\TT}(\bcalG)\|_{\rm F}$ in $(1)$ is just a trivial bound. However, bound in $(2)$ shows that the Riemannian gradient actually shrinks as $\calT$ approaches closer to the ground truth.  This behavior has been visualized in Figure~\ref{fig:UpperBoundSubGrad}. The polynomial probability term $d_k(d_k^-)^{-1-\min\{1,\eps\}}$ appears from bounding the slice sum of absolute value of random noise, while the negligible exponential probability term is a by-product of applying empirical processes technique. In the special case $d_k\equiv d$, the probability guarantee of Lemma~\ref{lem:hub:regularity} becomes $1-\Omega\big(md^{-\min\{1,\eps\}-(m-2)}-\exp(-\textsf{DoF}_m)\big)$. The one-step power iteration method in \cite{auddy2022estimating} only guarantees a log polynomial probability $1-\Omega(\log^{-1}d)$. Two-phase regularity properties of Lipschitz loss functions have been discovered in robust high-dimensional linear regression        \citep{shen2022computationally,shen2023computationally}. 
We emphasize that establishing two-phase regularity property for tensor decomposition is much more challenging. Towards that end, we need to precisely connect the sup-norm error $\|\bcalT-\bcalT^{\ast}\|_{\infty}$ and the Frobenius-norm error $\|\bcalT-\bcalT^{\ast}\|_{\rm F}$.  Characterizing sup-norm error rate in matrix/tensor decomposition is technically challenging.

Two-phase regularity property from Lemma~\ref{lem:hub:regularity} leads to a two-phase convergence dynamic of Algorithm~\ref{alg:pseuHuber:RsGrad}. Basically, phase-one convergence happens when $\bcalT_l$ is far from $\bcalT^{\ast}$ in that $\|\bcalT_l-\bcalT^{\ast}\|_{\rm F}=\Omega_{m,\mu^{\ast},r^{\ast}}\big((\gamma+\delta)\cdot d^{\ast1/2}\big)$. Algorithm~\ref{alg:pseuHuber:RsGrad} then enters phase-two convergence when $\bcalT_l$  gets closer to $\bcalT^{\ast}$. The precise convergence dynamic is presented in the following theorem.  Note that $\mins^{\ast}:=\min_{k\in[m]}\big\{\sigma_{r_k}\big(\fraM_{k}(\bcalT^*)\big)\big\}$ is referred to as the signal strength, where $\sigma_k(\cdot)$ denotes the $k$-th largest singular value of a matrix. 

\begin{theorem} \label{thm:hub:dynamics} Suppose the noise $\bXi$ has i.i.d. entries satisfying Assumption~\ref{assu:hub:noise} and the pseudo-Huber parameter $\delta\leq \gamma (\log d^* )^{-1/2}$. There exist absolute constants $D_0,c, c', c_1,c_2>0$ such that if the initialization satisfies $d^{\ast 1/2}\linft{\bcalT_0-\bcalT^*}\leq D_0 \leq c\mins^{\ast}\delta^2(b_0^2m^4\mu^{*m}r^*)^{-1}$  and initial stepsize $\eta_{0}\in D_0\cdot (5m+1)^{-2}(\mu^{*m}r^*d^*)^{-1/2}\cdot\left[0.125,\ 0.375\right]$, then, with probability at least $1-c'\sum_{k=1}^{m} d_k(\dkm)^{-1-\min\{1,\eps\}}-\exp\left(-\textsf{DoF}_m/2\right)-c_2 (d^{\ast})^{-7}$, Algorithm~\ref{alg:pseuHuber:RsGrad} exhibits the following dynamics:
\begin{enumerate}[(1)]
\item in phase one, namely for the $l$-th iteration satisfying $\left(1-c_{m,\muT,r^*}/32\right)^{l}D_0\geq2 c_{m,\muT,r^*}^{-1/2}d^{*1/2}(6\gamma+\delta)$, by choosing a stepsize $\eta_{l}=\left(1-c_{m,\muT,r^*}/32\right)^{l}\eta_{0}$ where $c_{m,\muT,r^*}:=(5m+1)^{-2}(3^m\mu^{*m} r^*)^{-1}$, we have
$$\fro{\bcalT_{l+1}-\bcalT^*}\leq \left(1-c_{m,\muT,r^*}/32\right)^{l+1}D_0,$$
$$\linft{\bcalT_{l+1}-\bcalT^*}\leq \frac{1}{\sqrt{c_{m,\muT,r^*}d^*}}\cdot \left(1-c_{m,\muT,r^*}/32\right)^{l+1}D_0;$$

\item in phase two, namely for the $l$-th iteration satisfying $\textsf{DoF}_m^{1/2}\cdot b_0\leq\fro{\bcalT_l-\bcalT^*}\leq 2 c_{m,\muT,r^*}^{-1/2}d^{\ast 1/2}(6\gamma+\delta)$, by choosing a constant stepsize $\eta_{l}=\eta$ such that $8c_1^2(m+1)\eta b_0\delta^{-2}\in [1, 3]$, we have
$$\fro{\bcalT_{l+1}-\bcalT^*}\leq \left(1- \frac{(\delta/b_0)^2}{32c_1^2(m+1)}\right)\fro{\bcalT_l-\bcalT^*}.$$
\end{enumerate}
Therefore, after at most $\tilde{l}=O\big(\log(\mins^*/\sqrt{\mu^mr^*d^*}\gamma)+\log(\gamma/b_0)+\log(d^*/\textsf{DoF}_m)\big)$ iterations, Algorithm~\ref{alg:pseuHuber:RsGrad} outputs an estimator achieving the error rate $\|\bcalT_{\tilde{l}}-\bcalT^{\ast}\|_{\rm F}=O\big(\textsf{DoF}_m^{1/2}\cdot b_0\big)$, which holds with the same aforementioned probability. 
\end{theorem}

Theorem~\ref{thm:hub:dynamics} shows, in both phases, Algorithm~\ref{alg:pseuHuber:RsGrad} enjoys fast linear convergence. Due to technical reasons, the initialization condition is imposed w.r.t. the sup-norm which immediately implies the Frobenius norm bound via the simple fact $\fro{\bcalA}\leq d^{*1/2}\linft{\bcalA}$ for any tensor $\bcalA$ of size $d_1\times\cdots\times d_m$. By Theorem~\ref{thm:hub:dynamics}, the phase-one convergence terminates 
after at most $l_1=O(\log(\mins^*/\sqrt{\mu^mr^*d^*}\gamma))$ iterations and Algorithm~\ref{alg:pseuHuber:RsGrad} reaches an estimate with the Frobenius-norm error rate $d^{\ast1/2}\fro{\bcalT_{l_1}-\bcalT^*}\leq 2c_{m,\muT,r^*}^{-1/2}(6\gamma+\delta)$ and sup-norm error rate $\linft{\bcalT_{l_1}-\bcalT^* }\leq 2c_{m,\muT,r^*}^{-1}(6\gamma+\delta)$. Geometrically decaying stepsizes are required during phase-one iterations, which is typical in non-smooth optimization \citep{charisopoulos2021low, tong2021low, shen2023computationally}. After $\ell_1$ iterations, Algorithm~\ref{alg:pseuHuber:RsGrad} enters the second phase and a constant step size suffices to ensure linear convergence.  The phase-two convergence terminates after at most $l_2=O(\log(\gamma/b_0)+\log(d^*/\textsf{DoF}_m))$ iterations and Algorithm~\ref{alg:pseuHuber:RsGrad} outputs an estimator with error rate $\fro{\bcalT_{l_1+l_2}-\bcalT^*}=O_p\big( \textsf{DoF}_m^{1/2}\cdot b_0\big)$. In total, Algorithm~\ref{alg:pseuHuber:RsGrad} converges within a logarithmic-order number of iterations. Note that $b_0$ is same scale as $\EE|\xi|$ for many examples such as Gaussian, Student's t, and zero symmetric Pareto, etc.  The error rate $\textsf{DoF}_m^{1/2}\cdot b_0$ is minimax optimal \citep{zhang2018tensor} in terms of the model complexity.

We note that our analysis can derive sharp upper bounds for the sup-norm error rate during phase-one convergence.  However,  the analysis framework cannot work for phase-two convergence even by the leave-one-out technique \citep{chen2021bridging,chen2021spectral,cai2022nonconvex}.  This is due to technical issues of treating the derivatives of pseudo-Huber loss function.  The challenge is also observed by the  recent work \cite{wang2022robust} on robust matrix completion using Huber loss.  The Huber parameter set by \cite{wang2022robust} is at the order  $\linft{\bcalT^*}+\gamma d^{1/2}$,  while the  pseudo-Huber parameter in our algorithm should be at the order $\gamma$.  Our Theorem~\ref{thm:hub:dynamics} and \cite{wang2022robust} both yield sub-optimal sup-norm error rates.  We believe the sub-optimality is due to technical issue because Section~\ref{sec:splitting} will present that a sample splitting trick can produce nearly optimal sup-norm error rate.

\section{Quantile Tensor Decomposition}
\label{sec:abs:convergence}
This section addresses the more general setting of robust tensor decomposition that allows both heavy-tailed noise and arbitrary corruptions. More specifically, suppose the observed tensor $\bcalY=\bcalT^{\ast}+\bXi+\bcalS$ where the noise tensor $\bXi$ may have heavy tails and the sparse tensor $\bcalS$ can be arbitrary corruptions. We shall assume that $\bcalS$ is $\alpha$-fraction sparse meaning that $\bcalS$ has at most $\alpha$ fraction non-zero entries in each slice. Here $\alpha\in(0,1)$ is understood as the corruption rate in Huber's contamination model. Basically, for each $k\in[m]$ and $j\in [d_k]$, one has $\|\e_j^{\top}\fraM_{k}(\bcalS)\|_0\leq \alpha\dkm$ where $\e_j$ is the $j$-th  canonical basis vector whose dimension may vary at different appearances. The $\alpha$-fraction sparsity model is also called deterministic sparsity model and has appeared in \cite{hsu2011robust,chandrasekaran2011rank,netrapalli2014non,chen2015fast,cai2022generalized}.  This $\alpha$-fraction sparsity model is less stringent than the one considered in \cite{dong2022fast} that imposes sparsity assumption on each fibers of $\bcalS$ and is more general than the random support model studied in existing literature \citep{candes2011robust,lu2016tensor,chen2021bridging}. In contrast, \cite{agarwal2012noisy,klopp2017robust} impose no assumption over locations of the support but their derived minimax optimal error rates are not proportional to noise level meaning that the low-rank matrix cannot be exactly recovered even if the noise part $\bXi$ is absent. Moreover, the foregoing works mostly focused on the matrix case and it is unclear whether their methods are still applicable for tensors, especially in consideration of the computational aspects of tensor-related problems. 

Our approach is based on quantile tensor decomposition, replacing the square loss by quantile loss.  Without loss of generality, we only present the method and theory for absolute loss, a special case of quantile loss. Let $\rho(x)=|x|$ be the absolute loss and we estimate $\bcalT^{\ast}$ by solving the following non-convex program:
\begin{align}\label{eq:loss-abs}
	\widehat{\bcalT}= \underset{\bcalT\in\MM_{\r,\mu}}{\arg\min}\ \lone{\bcalT-\bcalY}:= \sum_{\omega\in[d_1]\times\cdots\times[d_m]} \big|[\bcalT]_{\omega}-[\bcalY]_{\omega} \big|.
\end{align}
The absolute loss has been proved statistically robust for high-dimensional linear regression \citep{elsener2018robust, moon2022high, shen2023computationally}. Its theoretical analysis for tensor decomposition is more challenging because we must simultaneously investigate the computational and statistical aspects of the minimizers of (\ref{eq:loss-abs}).

\subsection{Projected sub-gradient descent with trimming} 
Our algorithm for finding local minimizers of (\ref{eq:loss-abs}) is essentially the same as the Riemannian-type Algorithm~\ref{alg:pseuHuber:RsGrad} except that now sub-gradient is employed because the absolute loss is non-smooth. The algorithm is thus called Riemannian sub-gradient descent, previously studied in \cite{charisopoulos2021low, shen2023computationally} for low-rank regression. Here the algorithm is more involved because one needs to ensure the incoherence property. Unlike the pseudo-Huber loss used in Algorithm~\ref{alg:pseuHuber:RsGrad}, the absolute loss is non-differentiable so that even the leave-one-out technique cannot help prove the incoherent condition during the phase-two iterations. To enforce incoherence and control sup-norm error rate, an additional trimming and truncation step is utilized.

For a given tensor $\bcalB$ and a truncation threshold $\tau_1$, define the operator $\text{Trun}_{\tau_1,\bcalB}(\cdot):\RR^{d_1\times\cdots\times d_m}\to \RR^{d_1\times\cdots\times d_m}$ as
\begin{align}
	[\text{Trun}_{\tau_1,\bcalB}(\bcalT)]_{\omega}:=[\bcalT]_{\omega}+\text{sign}([\bcalT-\bcalB]_{\omega})\cdot\min\big\{0, \tau_1-|[\bcalT-\bcalB]_{\omega}|\big\},
	\label{eq:trun}
\end{align}
The trimming operator \citep{cai2022generalized, cai2022provable} is defined similarly. For any $\tau_2>0$, define
\begin{align}
	\big[\text{Trim}_{\tau_2}(\bcalT)\big]_{\omega}:=[\bcalT]_{\omega}+\text{sign}\left([\bcalT]_{\omega} \right)\cdot\min\Big\{0, (\tau_2/d^{\ast})^{1/2}\fro{\bcalT}-|[\bcalT]_{i_1\cdots i_m}|\Big\}.
	\label{eq:trim}
\end{align}
The truncation operation ensures a uniform upper bound of $\|\bcalT-\bcalT^{\ast}\|_{\infty}$ during phase-two iterations. The parameter $\tau_1$ is chosen such that $\tau_1=\Omega\big(\|\bcalT_{l_1}-\bcalT^{\ast}\|_{\infty}\big)$ w.h.p. where $\bcalT_{l_1}$ is the output after phase-one iterations.  The trimming operator aims to maintain the incoherence property and the parameter $\tau_2$ can be set at the level  $\mu^{\ast m} r^{\ast}$. The detailed implementations can be found in Algorithm~\ref{alg:abs:RsGrad}. Practical guidelines to the selection of $\tau_1$ and $\tau_2$ shall be discussed in Section~\ref{sec:parameter and init}. Compared to existing algorithms in the literature \citep{chen2021bridging,dong2022fast,cai2022generalized}, our approach does not require any robustness parameters such as the sparsity level.

\begin{algorithm}
	\caption{Riemannian Sub-gradient Descent with Trimming}\label{alg:abs:RsGrad}
	\begin{algorithmic}
		\STATE{\textbf{Input}: observations $\bcalY$,  max iterations $l_{\max}$,  step sizes $\{\eta_l\}_{l=0}^{\lmax}$, parameters $\tau_1.\tau_2$.}
		\STATE{Initialization: $\bcalT_0\in\MM_\r$}
		\FOR{$l = 0,\ldots,l_{\max}$}
		\STATE{Choose a vanilla subgradient:  $\bcalG_l\in\partial \|\bcalT_l-\bcalY\|_1$}
		\STATE{Compute Riemannian sub-gradient: $\wt\bcalG_l = \calP_{\TT_l}(\bcalG_l)$}
		\STATE{Retraction to $\MM_\r$: $\bcalT_{l+1} =\left\{
			\begin{array}{lcl}
				\text{HOSVD}_\r(\bcalT_l - \eta_{l}\wt\bcalG_l)&     &\text{ if in phase one} \\
				\text{HOSVD}_\r\big(\text{Trim}_{\tau_2}(\text{Trun}_{\tau_1,\bcalT_{l_1}}(\bcalT_l - \eta\wt\bcalG_l))\big)&     &\text{ if in phase two}
			\end{array}
			\right., $}
		\STATE{where $\bcalT_{l_1}$ is phase one output and $\text{Trun}_{\tau_1,\bcalT_{l_1}}(\cdot)$, $\text{Trim}_{\tau_2}(\cdot)$ are defined in \eqref{eq:trun}  and\eqref{eq:trim}}, respectively.
		\ENDFOR
		\STATE{\textbf{Output}: $\widehat\bcalT=\bcalT_{l_{\max}}$}
	\end{algorithmic}
\end{algorithm}

\subsection{Algorithm convergence and error bound}
Assume that the noise tensor $\bXi$ has i.i.d. entries whose density and distribution functions are denoted as  $h_{\xi}(\cdot)$ and $H_{\xi}(\cdot)$, respectively.  It turns out that absolute loss requires a lightly different condition on the noise,  detailed in the following assumption.  Here the tensor condition number $\kappa$ is defined as $\kappa:=\kappa(\bcalT^{\ast}):=\mins^{\ast -1}\maxs^{\ast}$ where $\maxs^*:=\max_{k=1,\dots,m}\left\{\sigma_{1}\left(\fraM_{k}(\bcalT^*)\right)\right\}$. 
 
\begin{assumption}[Noise condition \Romannum{2}]
There exists an $\eps>0$ such that $\gamma:=\left(\EE|\xi|^{2+\eps}\right)^{1/(2+\eps)}<+\infty$ and the noise term has median zero $H_{\xi}(0)=\frac{1}{2}$. Also, there exist $b_0,b_1>0$ such that\footnote{The lower bound can be slightly relaxed to $|H_\xi(x)-H_\xi(0)|\geq|x|/b_0$ for all $|x|\leq C_{m,\muT,r^^*,\kappa}\gamma$.}
	\begin{align*}
		h_{\xi}(x)\geq b_0^{-1}, &\ \ \  \textrm{ for all } |x|\leq C_{m,\muT,r^*,\kappa}\gamma;\\
		h_{\xi}(x)\leq b_1^{-1}, &\ \ \ \textrm{ for all } x\in \RR,
	\end{align*}
where $C_{m,\muT,r^*,\kappa}:=(5m+1)^26^m\kappa^m\mu^{*m(m+1)/2}(r^*)^{(m+1)/2}$.
	\label{assu:abs:noise}
\end{assumption}

A simple fact of Assumption~\ref{assu:abs:noise} is $b_1\leq b_0$ and $b_0\geq C_{m,\muT,r^*,\kappa}\gamma$.  Compared with the noise condition in Assumption~\ref{assu:hub:noise},  an additional upper bound of the noise density is imposed but the symmetry requirement is waived.  
See \cite{alquier2019estimation,elsener2018robust,shen2023computationally} for comparable noise assumptions for treating various types of loss functions.   
The constant $C_{m,\muT,r^*,\kappa}$ does not depend on the tensor dimensions.  If $m,\muT,r^*,\kappa$ are regarded as constants,  we have $b_0\asymp b_1\asymp\gamma\asymp\EE|\xi|$  for Gaussian, Student's t,  and zero-symmetric Pareto distributions, etc.

The absolute loss also exhibits a two-phase regularity property even in the existence of the additional sparse corruptions.  These properties play an essential role in characterizing the convergence dynamics of Algorithm~\ref{alg:abs:RsGrad}.  Here $\mu$ is any positive constant. 

\begin{lemma}[Two-phase regularity properties of absolute loss] 	\label{lem:abs:regularity}
Suppose $\bXi$ contains i.i.d.  entries satisfying Assumption~\ref{assu:abs:noise} and $\bcalS$ is $\alpha$-fraction sparse with its non-zero entries being arbitrary values.  	Then there exist absolute constants $c, c_1,c_2>0$ such that with probability exceeding $1-c\sum_{k=1}^{m} d_k(\dkm)^{-1-\min\{1,\eps\}}-\exp\left(-\textsf{DoF}_m/2\right)$, the following facts hold.  

\begin{enumerate}[(1)]
\item For all $\bcalT\in\RR^{d_1\times \cdots\times d_m}$ and any sub-gradient $\bcalG\in\partial \|\bcalT-\bcalY\|_1$,  we have
$$
	\fro{\calP_{\TT}(\bcalG)}\leq d^{\ast 1/2},
$$ 
$$
	\lone{\bcalT-\bcalY}-\lone{\bcalT^*-\bcalY}\geq \linft{\bcalT-\bcalT^*}^{-1}\cdot\left(\fro{\bcalT-\bcalT^*}^2-2\alpha d^*\linft{\bcalT-\bcalT^*}^2\right)-6d^*\gamma
$$
Furthermore,  for each $k\in[m]$ and $j\in[d_k]$,  if $\bcalT\in\MM_{\r,\mu}$,  then
\begin{align*}
			&{~~~~~~~~~~~~~~~~~~~~~~~~~~~~~~~~~~}\ltinf{\fraM_{k}\left(\calP_{\TT}(\bcalG)\right)}\leq (3\mu r_k\cdot \dkm)^{1/2},\\
			&\lone{\fraM_{k}(\bcalT-\bcalY)_{j,\cdot}}-\lone{\fraM_{k}(\bcalT^*-\bcalY)_{j,\cdot}}\\&{~~~~~~~~}\geq\linft{\fraM_{k}(\bcalT-\bcalT^*)_{j,\cdot}}^{-1}\left(\fro{\fraM_{k}(\bcalT-\bcalT^*)_{j,\cdot}}^2-2\alpha\dkm \linft{\fraM_{k}(\bcalT-\bcalT^*)_{j,\cdot}}^2\right)-6\dkm \gamma.
\end{align*}
		
\item For all $\bcalT\in\MM_{\r,\mu}$ and any sub-gradient $\bcalG\in\partial \|\bcalT-\bcalY\|_1$ with $\bcalT$ satisfying $\linft{\bcalT-\bcalT^*}\leq C_{m,\muT,r^*,\kappa}\gamma$ and $\fro{\bcalT-\bcalT^*}\geq c_1b_0\cdot\max\big\{\textsf{DoF}_m^{1/2},\ \alpha\big((m+1)(\muT\vee\mu)^mr^*d^*\big)^{1/2}\big\}$,  we have 
$$
\fro{\calP_{\TT}(\bcalG)}\leq c_2 (m+1)^{1/2}\cdot b_1^{-1}\cdot\fro{\bcalT-\bcalT^*},\quad  \lone{\bcalT-\bcalY}-\lone{\bcalT^*-\bcalY}\geq (2b_0)^{-1}\cdot\fro{\bcalT-\bcalT^*}^2.
$$
	\end{enumerate}

\end{lemma}

Compared with Lemma~\ref{lem:hub:regularity},  the second phase property (2) in Lemma~\ref{lem:abs:regularity} only holds in the restricted subset over $\mu$-incoherent tensors.  This additional restriction comes from dealing with the presence of arbitrary sparse outliers.  We note that the probability can be improved to $1-\Omega\big(\sum_{k=1}^{m}d_k\exp(-d_k)-\exp(-\textsf{DoF}_m/2)\big)$ if the random noise $\xi$ has sub-Gaussian tails. 

 
\begin{theorem}
Suppose $\bXi$ contains i.i.d.  entries satisfying Assumption~\ref{assu:abs:noise} and $\bcalS$ is $\alpha$-fraction sparse with its non-zero entries being arbitrary values. Let $c_{m,\muT,r^*}:=(5m+1)^{-2}(3^m\mu^{\ast m} r^*)^{-1}$ and set $\tau_1\in c_{m,\muT, r^{\ast}}^{-1}\cdot [12, 24]$ and $\tau_2\in \mu^{\ast m}r^{\ast}\cdot [1,\ 2]$. There exist absolute constants $D_0, c, c', c_1, c_2>0$ such that if the initialization satisfies 	$\linft{\bcalT_0-\bcalT^*}\leq D_0/d^{*1/2}\leq c (b_1/b_0)^2(m^43^m\mu^{*m}r^*)^{-1}\mins^{*}/d^{*1/2}$, initial stepsize satisfies $\eta_{0}\in D_0\cdot (5m+1)^{-2}(3^m\mu^mr^*d^*)^{-1/2}\cdot\left[0.125,\ 0.375\right]$ and corruption rate is bounded with $\alpha\leq \big(12(5m+1)^23^m\mu^{*m}r^*\big)^{-1}$, then with probability at least $1-c'\sum_{k=1}^{m} d_k(\dkm)^{-1-\min\{1,\eps\}}-\exp\left(-\textsf{DoF}_m/2\right)$, Algorithm~\ref{alg:abs:RsGrad} exhibits the following dynamics:
\begin{enumerate}[(1)]
		\item in phase one, namely for the $l$-th iteration satisfying $\left(1-c_{m,\muT,r^*}/32\right)^{l}D_0\geq12c_{m,\muT,r^*}^{-1/2}d^{*1/2}\gamma $, by choosing a stepsize $\eta_{l}=\left(1-c_{m,\muT,r^*}/32\right)^{l}\eta_{0}$, we have
$$
\fro{\bcalT_{l+1}-\bcalT^*}\leq \left(1-c_{m,\muT,r^*}/32\right)^{l+1}D_0,
$$
$$
\linft{\bcalT_{l+1}-\bcalT^*}\leq \frac{1}{\sqrt{c_{m,\muT,r^*}d^*}}\cdot \left(1-c_{m,\muT,r^*}/32\right)^{l+1}D_0;
$$
		\item in phase two, namely for the $l$-th iteration satisfying $c_1b_0\cdot\max\big\{ \textsf{DoF}_m^{1/2},\alpha\big((m+1)\mu^{*m}r^*d^*\big)^{1/2} \}\leq\fro{\bcalT_l-\bcalT^*}\leq 12 c_{m,\muT,r^*}^{-1/2}d^{*1/2}\gamma$, by choosing a constant step size $\eta_{l}=\eta\in b_0^2\big(c_1^2b_1(m+1)\big)^{-1} [1,\ 3]$, we have
		$$
		\fro{\bcalT_{l+1}-\bcalT^*}\leq \left(1- \frac{(b_1^2/b_0^2)}{32c_1^2(m+1)}\right)\fro{\bcalT_l-\bcalT^*}.
		$$
\end{enumerate}
Therefore, after at most $\tilde l=O\big(\log(\mins^{\ast}/\sqrt{d^{\ast}}\gamma)+\log(\gamma/b_0)+\min\{\log(d^*/\textsf{DoF}_m),\log(1/\alpha)\}\big)$ iterations, Algorithm~\ref{alg:abs:RsGrad} outputs an estimator achieving the error rate $\|\bcalT_{\tilde{l}}-\bcalT^{\ast}\|_{\rm F}^2=O\big(b_0^2\cdot (\textsf{DoF}_m+\alpha^2d^*)\big)$ if treating $\mu^{\ast}, m$ as constants, holding with the aforementioned probability. 
	\label{thm:abs:dynamics}
\end{theorem}

Basically, Algorithm~\ref{alg:abs:RsGrad} enjoys a two-phase linear convergence with the scheduled step sizes. The phase-one convergence terminates after $l_1=O(\log(\mins^*/\sqrt{d^*}\gamma))$ iterations and the output satisfies $\fro{\bcalT_{l_1}-\bcalT^*}\leq 12\big(c_{m,\muT,r^*}^{-1} d^*\big)^{1/2}\gamma$ and $\linft{\bcalT_{l_1}-\bcalT^* }\leq 12c_{m,\muT,r^*}^{-1}\gamma$. The phase-two convergence lasts for at most  $l_2=O(\log(\gamma/b_0)+\min\{\log(d^*/\textsf{DoF}_m),\log(1/\alpha)\})$ iterations and  the algorithm finally outputs an estimator with error rate $\fro{\bcalT_{l_1+l_2}-\bcalT^*}^2= O_p\big(b_0^2\cdot (\textsf{DoF}_m+\alpha^2d^*)\big)$ where $\muT,m,r^*$ are regarded as some constants. The first term $b_0^2\cdot \textsf{DoF}_m$ is sharp in terms of the model complexity.  The model complexity $\textsf{DoF}_m$ dominates $\alpha^2 d^{\ast}$ if the corruption rate $\alpha=O\big((\textsf{DoF}_m/d^{\ast})^{1/2}\big)$, improving the prior work \cite{cai2022generalized}. 
Note that if the random noise $\bXi$ is absent so that $\gamma=0$,  Theorem~\ref{thm:abs:dynamics} implies that Algorithm~\ref{alg:abs:RsGrad} can exactly recovers the ground truth $\bcalT^{\ast}$ after phase-one iterations, enjoying both Frobenius norm and sup norm convergence guarantees. It cannot be achieved by the convex approaches studied in \cite{agarwal2012noisy} and \cite{klopp2017robust}. 

\paragraph*{Optimality w.r.t. corruption rate} The support size of $\bcalS$ is at most $\alpha d^{\ast}$ implying that the associated model complexity is $O(\alpha d^{\ast})$. Thus a seemingly natural outlook on the optimal error rate should emerge as $O_p(b_0^2\cdot \alpha d^{\ast})$. This is indeed what has appeared in the existing literature. See, e.g., \cite{agarwal2012noisy, klopp2017robust, cai2022generalized} and references therein. Intriguingly, Theorem~\ref{thm:abs:dynamics} shows that Algorithm~\ref{alg:abs:RsGrad} achieves an error rate with a faster dependence of the corruption rate, which is $O_p(b_0^2\cdot \alpha^2 d^{\ast})$. This rate turns out to be minimax optimal with a comparable lower bound to be established in the next section. The improvement comes from the benefit of absolute loss,  compared with the square loss used in the foregoing works. Denote $\tilde{\Omega}$ the support of $\bcalS$ and an upper bound for $\|[\bcalT-\bcalT^{\ast}]_{\tilde{\Omega}}\|_{\rm F}$ is often needed for incoherent matrices/tensors $\bcalT$ and $\bcalT^{\ast}$. \cite{cai2022generalized} bounds this term by $\|[\bcalT-\bcalT^{\ast}]_{\tilde{\Omega}}\|_{\rm F}=O\big(\alpha^{1/2}\cdot \|\bcalT-\bcalT^{\ast}\|_{\rm F}\big)$. An additional factor $\alpha^{1/2}$ will appear by considering the absolute loss in that $\|[\bcalT-\bcalT^{\ast}]_{\tilde{\Omega}}\|_1=O\big(\alpha d^{\ast 1/2}\cdot \|\bcalT-\bcalT^{\ast}\|_{\rm F}\big)$.

\subsection{Minimax lower bound}
We now establish the minimax lower bounds of robust tensor decomposition in the existence of both dense noise and sparse corruptions. For simplicity, we assume the dense noise tensor $\bXi$ comprises of i.i.d.  Gaussian entries and the support of $\bcalS$ is randomly sampled with probability $\alpha$,  following the typical scheme used in \cite{candes2011robust,yi2016fast,chen2021bridging}. The proof of Theorem~\ref{thm:lowerbound}  borrows the idea used in studying Huber's contamination model \citep{chen2018robust}.

\begin{theorem}
Suppose the entries of $\bXi$ are i.i.d. with distribution $N(0,\sigma^2)$. Let $\alpha\in(0,1)$,  suppose the entries of $\bcalS$ follow the distribution $[\bcalS]_{\omega}\sim (1-\alpha)\delta_0+\alpha Q_{\omega}$, where $Q_{\omega}$ is an arbitrary distribution and $\delta_0$ is the zero distribution for all $\omega\in[d_1]\times\cdots\times [d_m]$. Then there exists absolute constants $c, C>0$ such that 
	\begin{align*}
		\inf_{\widehat{\bcalT}}\sup_{\bcalT^*\in\MM_{\r,\muT}}\sup_{\{Q_{\omega}\}}\PP\left(\fro{\widehat{\bcalT}-\bcalT^*}^2\geq\sigma^2 \max\left\{ \textsf{DoF}_m,\; C\alpha^2 d^*/(\mu^{*m}r^*)\right\}\right)\geq c,
	\end{align*}
where $\widehat{\bcalT}$ is any estimator of $\bcalT^*$ based on an observation $\bcalY=\bcalT^*+\bXi+\bcalS$.
\label{thm:lowerbound}
\end{theorem}

\section{Algorithmic Parameter Selection and Initialization}
\label{sec:parameter and init}

\paragraph{Algorithmic parameter selection} The initial stepsize and two-phase stepsizes can be selected similarly to \cite{shen2023computationally}.  We only need to discuss the selection of truncation parameters $\tau_1,\tau_2$ in the second phase of Algorithm~\ref{alg:abs:RsGrad}.  It's important to note that $\tau_1,\tau_2$ are determined by the incoherence $\muT$ and the noise level $\gamma$.  We can estimate $\muT$ and $\gamma$ based on the phase-one output $\bcalT_{l_1}$.  
In fact, according to the proof of Theorem~\ref{thm:abs:dynamics},  we have $\muT/2\leq\mu(\bcalT_{l_1})\leq 2\muT$.  This allows us to obtain a satisfactory estimation of the oracle $\muT$.  As for $\gamma$,  we have $\linft{\bcalT_{l_1}-\bcalT^{\ast}}\asymp \gamma$ with high probability.  Thus,  the median $\text{med}(|\bcalT^{\ast}-\bcalY |)$ is a rough estimation of the noise scale $\tau_2$.  
Moreover, in simulations, the sequence $\{\bcalT_{l}\}_{l\geq 1}$ maintains incoherence automatically and in practice, we don't need the truncation or trimming steps. The proof of $\ell_1$-loss maintaining incoherence implicitly is left for future study.

\paragraph{Initialization}
We now present an initialization method that works under both dense noise and sparse arbitrary corruptions.  See model (\ref{eq:model}).  Note that \cite{auddy2022estimating} proposed an initialization method based on Catoni's estimator \citep{minsker2018sub} where only the case of heavy-tailed noise is considered. The robust low-rank matrix work of \cite{wang2022robust,cai2022generalized} uses the truncation method as an initialization, providing the guarantees of heavy-tailed noise case and sub-Gaussian noise plus sparse corruptions respectively. And \cite{dong2022fast} provides the noiseless case initialization guarantees. Our initialization approach is inspired by the truncation method \citep{fan2016shrinkage}.
We begin with truncating the observed tensor $\bcalY$ with a threshold that is selected at the level $\tau\asymp\left(\linft{\bcalT^*}+d^{*1/8}\op{\xi}_4\right)$.   Here we write $\op{\xi}_4:=(\EE\xi^4)^{1/4}$ in short.   The truncation step yields 
$$
[\hat{\bcalY}]_{\omega}:=[\bcalY]_{\omega}\cdot 1_{\left\{|[\bcalY]_{\omega}|\leq\tau\right\}}+\tau\cdot\text{sign}\left([\bcalY]_{\omega}\right)\cdot 1_{\left\{|[\bcalY]_{\omega}|>\tau\right\}},\quad \forall \omega\in[d_1]\times\cdots\times[d_m].
$$
Finally,  we apply spectral initialization and obtain $\bcalT_0:=\text{HOSVD}_{\r}(\hat{\bcalY})$. 

\begin{theorem}\label{thm:init}
Suppose the noise tensor $\bXi$ has i.i.d. entries with a finite $(4+\eps)$ moment for any $\eps>0$ and $\bcalS$ has independent entries with $[\bcalS]_{\omega}\sim (1-\alpha)\delta_0+\alpha Q_{\omega}$ where $Q_{\omega}$ is an arbitrary distribution.  There exist $c_0,c,c_1,C,C_1,C_2,C_3>0$ such that if $d^*\geq \mu^{*m}r^{*}\kappa\dmax\log \dmax$, truncation level $\tau\in(\linft{\bcalT^*}+d^{*1/8}\|\xi\|_4)\cdot[C_1,C_2]$, signal strength $\mins^*/\|\xi\|_4\geq C_3m\kappa\sqrt{r^*} \max\{(\dmax\log\dmax)^{1/2},\ d^{*1/4} (\log\dmax)^{1/4}\}$, and corruption rate $\alpha\leq c_1\min\{(\mins^*/\|\xi\|_4)/d^{*5/8}, 1/(\mu^{*m}r^*)\}/(m\kappa^2\sqrt{r^*})$, then with probability at least $1-cd^{*-\eps/4}-\sum_{k=1}^{m}\dkm\exp(-\alpha d_k)$, we have
\begin{align*}
		\fro{\bcalT_0-\bcalT^*}&\leq C_3m\kappa\sqrt{r^*}\left((\|\xi\|_4+\linft{\bcalT^*})\cdot\left(\sqrt{\dmax\log\dmax}+4d^{*1/4}(\log\dmax)^{1/4} \right)+2\alpha\tau\sqrt{d^*}\right),\\
		\linft{\bcalT_0-\bcalT^*}&\leq C_3m^2\kappa^2\sqrt{r^*}\sqrt{\frac{\mu^{*m}r^*}{d^*}}\left((\|\xi\|_4+\linft{\bcalT^*})\cdot\left(\sqrt{\dmax\log\dmax}+4d^{*1/4}(\log\dmax)^{1/4} \right)+2\alpha\tau\sqrt{d^*}\right).
	\end{align*}
\end{theorem}

For ease of exposition, suppose that $m,\muT,r^*,\kappa\asymp O(1)$.  Theorem~\ref{thm:init} shows that $\bcalT_0$ satisfies the initialization condition required in Theorem~\ref{thm:abs:dynamics} if the signal strength satisfies $\mins^*/\|\xi\|_4=\Omega\big(\max\{\sqrt{\dmax\log \dmax}, (d^*\log\dmax)^{1/4}\}\big)$ and the corruption rate is bounded as $\alpha=O\big( \min\{(\mins^*/\|\xi\|_4)/d^{*5/8}, \allowbreak 1/(\mu^{*m}r^*)\}\big)$.  The signal-to-noise ratio is near optimal with an extra $\log^{1/2}\dmax$ factor \citep{zhang2018tensor}. The corruption rate requirement is weaker than \cite{cai2022generalized}. Initialization guarantee of Theorem~\ref{thm:hub:dynamics} can be attained in a similar fashion.

\section{Missing Values,  Sample Splitting and Optimality}\label{sec:splitting}
While Theorems~\ref{thm:hub:dynamics} and \ref{thm:abs:dynamics} demonstrate that both pseudo-Huber tensor decomposition and quantile tensor decomposition can yield estimators that are minimax optimal in Frobenius norm, the derived entry-wise error rates are generally sub-optimal. This remains the case even though powerful techniques like leave-one-out have been utilized.
This sub-optimality, which is due to the non-smoothness of loss functions, has also been observed in \cite{wang2022robust}. However, we believe that this sub-optimality is a result of technical difficulty and can be addressed using a simple sample splitting trick.
We hope that the positive insights from this section can inspire future research to tackle this technically unresolved problem.

For technical simplicity,  we focus on the sampling with replacement model,  commonly used in matrix and tensor completion literature  \citep{cai2016matrix,elsener2018robust,xia2021statistically,cai2022provable}.  Let $\{(Y_i,  \bcalX_i)\}_{i=1}^N$ be independent observations where $\bcalX_i$ is uniformly sampled from the set $\bcalX:=\{\e_{\omega}: \omega\in[d_1]\times\cdots\times [d_m]\}$.  Here the tensor $\e_{\omega}$ has value $1$ on its entry $\omega$ and $0$'s everywhere else.   The response $Y_i$ satisfies the trace-regression model
\begin{align*}
	Y_i=\inp{\bcalX_i}{\bcalT^*}+\xi_{i}+s_i,
\end{align*}
where $\xi_{i}$'s are i.i.d.  (potentially) heavy-tailed noise and $s_i\sim (1-\alpha)\delta_0+\alpha Q_{\omega_i}$ represents a potentially arbitrary corruption.  Here $Q_{\omega_i}$ denotes  an arbitrary distribution and $\alpha\in[0,1)$ is the corruption rate,  following the Huber's contamination model \citep{chen2016general,chen2018robust}. 
We split the data into $M+1$ non-overlapping sub-samples and,  without loss of generality,  assume $N=(M+1)n$ for some integer $n$.  Here $M+1$ denotes the total number of iterations of our algorithm.  Denote the $M+1$ sub-samples as $\calD_l=\cup_{i=1}^n\{(Y_i^{(l)},  \bcalX_i^{(l)})\}$ and $\cup_{l=0}^M \calD_l=\{(Y_i,  \bcalX_i)\}_{i=1}^N$.   We still apply the Riemannian sub-gradient descant algorithm to minimize the absolute loss,  but at the $l$-the iteration,  the algorithm is only implemented on the $l$-th sub-sample data.  The sample splitting ensures the independence across iterations.  The detailed implementation can be found in Algorithm~\ref{alg:abs:completion}.

\begin{algorithm}
	\caption{Riemannian Sub-gradient Descent with Sample Splitting}\label{alg:abs:completion}
	\begin{algorithmic}
		\STATE{\textbf{Input}: observations $\{\calD_l\}_{l=0}^M$,  max iterations $M+1$,  step sizes $\{\eta_l\}_{l=0}^{M}$.}
		\STATE{Initialization: $\bcalT_0\in\MM_\r$ is based on $\calD_0$}
		\FOR{$l = 0,\ldots,M-1$}
		\STATE{Choose a vanilla sub-gradient:  $\bcalG_l\in\partial \sum_{i=1}^{n}|Y_i^{(l+1)}-\langle\bcalX_i^{(l+1)},\bcalT_l\rangle|$.}
		\STATE{Compute Riemannian sub-gradient: $\wt\bcalG_l = \calP_{\TT_l}(\bcalG_l)$}
		\STATE{Retraction to $\MM_\r$: $\bcalT_{l+1} =
				\text{HOSVD}_\r(\bcalT_l - \eta_{l}\wt\bcalG_l)$ }
		\ENDFOR
		\STATE{\textbf{Output}: $\widehat\bcalT=\bcalT_{M}$}
	\end{algorithmic}
\end{algorithm}

\begin{assumption}[Noise condition \Romannum{3}]
	There exists an $\eps>0$ such that $\gamma:=\left(\EE|\xi|^{1+\eps}\right)^{1/(1+\eps)}<+\infty$ and the noise term has median zero $H_{\xi}(0)=\frac{1}{2}$. Also, there exist $b_0,b_1>0$ such that the noise density satisfies \footnote{The lower bound can be slightly relaxed to $|H_\xi(x)-H_\xi(0)|\geq|x|/b_0$ for all $|x|\leq C_{m,\muT,r^^*,\kappa}\gamma$.}
	\begin{align*}
		h_{\xi}(x)\geq b_0^{-1}, &\ \ \  \textrm{ for all } |x|\leq C_{m,\muT,r^*}\gamma;\\
		h_{\xi}(x)\leq b_1^{-1}, &\ \ \ \textrm{ for all } x\in \RR,
	\end{align*}
	where $C_{m,\muT,r^*}:=(5m+1)^26^m\mu^{*m}r^*$.
	\label{assu:abs:noise:completion}
\end{assumption}

Compared with Assumption~\ref{assu:hub:noise} and \ref{assu:abs:noise}, here we only require a finite $1+\eps$ moment. The following theorem established the convergence dynamic of Algorithm~\ref{alg:abs:completion}.   Recall that $\bar d$ denotes $\max_{j\in[m]} d_j$.  

\begin{theorem}\label{thm:missing-values}
Suppose Assumption~\ref{assu:abs:noise:completion} holds.  There exist positive constants $D_0,  \big\{c_{m,\muT,r^*}^{(i)}\big\}_{i=1}^5,  \big\{C_{m,\muT,r^*}^{(j)}\big\}_{j=1}^5$ depending only on $m,  \mu^{\ast},  r^{\ast}$ such that 
if $n\geq C_{m,\muT,r^*}^{(1)} \dmax\log \dmax$,  the initialization satisfies $\linft{\bcalT_0-\bcalT^*}\leq D_{0}/d^{*1/2}\leq c_{m,\muT,r^*}^{(1)}(b_1/b_0)^2\mins^*/d^{*1/2}$,  the initial stepsize $\eta_{0}\in d^{*1/2}D_0/n\cdot [c_{m,\muT,r^*}^{(2)},c_{m,\muT,r^*}^{(3)}]$,  and corruption rate is bounded by $\alpha\leq c_{m,\muT,r^*}^{(4)}$,   then with probability at least $1-c_mMd^{\ast-10}$,  Algorithm~\ref{alg:abs:completion} exhibits the following dynamics:
	\begin{enumerate}[(1)]
		\item in phase one, namely for the $l$-th iteration satifying $(1-c_{m,\muT,r^*}^{(5)})^lD_0\geq C_{m,\muT,r^*}^{(2)} \sqrt{d^*}\gamma$, by specifying a stepsize $\eta_l=(1-c_{m,\muT,r^*}^{(5)})^l \eta_{0}$,  we have
		\begin{align*}
			\fro{\bcalT_{l+1}-\bcalT^*}&\leq (1-c_{m,\muT,r^*}^{(5)})^{l+1} D_0,\\
			\linft{\bcalT_{l+1}-\bcalT^*}&\leq \frac{C_{m,\muT,r^*}^{(3)}}{\sqrt{d^*}}\cdot (1-c_{m,\muT,r^*}^{(5)})^{l+1} D_0;
		\end{align*}
	\item in phase two, namely for the $l$-th iteration satisfying $C_{m,\muT,r^*}^{(4)} b_0\cdot\max\{(n^{-1}\cdot \textsf{DoF}\log\dmax )^{1/2}, \ \alpha \}\leq \fro{\bcalT_{l}-\bcalT^*}/d^{\ast1/2} \leq C_{m,\muT,r^*}^{(1)}\gamma$,  by choosing a constant stepsize satisfying $\eta_l=\eta\in (b_1^2/b_0)d^*/n\cdot[c_{m,\muT,r^*}^{(6)},c_{m,\muT,r^*}^{(7)}]$,  we have
	\begin{align*}
		\fro{\bcalT_{l+1}-\bcalT^*}&\leq (1-c_{m,\muT,r^*}^{(8)})^{l+1-l_1} \fro{\bcalT_{l_1}-\bcalT^*},\\
		\linft{\bcalT_{l+1}-\bcalT^*}&\leq \frac{C_{m,\muT,r^*}^{(5)}}{\sqrt{d^*}}\cdot (1-c_{m,\muT,r^*}^{(8)})^{l+1-l_1} \fro{\bcalT_{l_1}-\bcalT^*},
	\end{align*}
where $\bcalT_{l_1}$ is the output of the first phase and $l_1=O\big(\log(\mins^{\ast}/\sqrt{d^{\ast}}\gamma)\big)$. Therefore, by choosing $M= \Omega\big(\log(\mins^{\ast}/\sqrt{d^{\ast}}\gamma)+\log(\gamma/b_0)+\min\{\log(n/\textsf{DoF}_m),\log(1/\alpha)\}\big)$, Algorithm~\ref{alg:abs:completion} outputs an estimator $\widehat\bcalT=\bcalT_M$ achieving the error rate 
$$
d^{\ast-1}\|\widehat\bcalT-\bcalT^{\ast}\|_{\rm F}^2=O\Big(b_0^2\cdot \Big(\frac{\textsf{DoF}_m\log\dmax}{n}+\alpha^2\Big)\Big);
$$
$$
\|\widehat\bcalT-\bcalT^{\ast}\|_{\infty}^2=O\Big(b_0^2\cdot \Big(\frac{\textsf{DoF}_m\log\dmax}{n}+\alpha^2\Big)\Big),
$$
 if treating $\mu^{\ast}, m$ as constants, holding with the aforementioned probability.   
	\end{enumerate}
\end{theorem}
By ignoring the log terms involved in $M$,  the established rates of $\widehat\bcalT$ in Frobenius norm and sup-norm are minimax optimal with respect to the sample size $n$, degree of freedom $\textsf{DoF}_m$, and the corruption rate $\alpha$.  The sample size requirement $n=\Omega_{m,\mu^{\ast},r^{\ast}}(\dmax \log \dmax)$ is sharp in view of existing works \citep{xia2019polynomial, cai2022provable}. Theorem~\ref{thm:missing-values} also allows a wide range of corruption rate under Huber's contamination model.

\section{Numerical Simulations}
We evaluate the convergence of our algorithm (written as RsGrad in short) and the error rate of the estimator, comparing them with two recent methods \cite{cai2022generalized,auddy2022estimating}. We present the simulation results from two perspectives: convergence dynamics and the accuracy of the output.
In fact, Algorithms~\ref{alg:pseuHuber:RsGrad} and \ref{alg:abs:RsGrad} demonstrate considerable tolerance with respect to parameter selections. Specifically, the stepsize decaying rate in the first phase can take values in the range $0.8<q<1$, all of which lead to roughly similar performance. Furthermore, a selection of $\eta\in[0.01,0.1]$ for the second phase stepsize is acceptable and does not significantly influence the accuracy.

\paragraph{Algorithm convergence}  We assess the convergence dynamics of our algorithm in comparison with RGrad \citep{cai2022generalized}, for which algorithmic parameters are exhaustively searched. Dimensions are set as $d_1=d_2=d_3=100$ and Tucker rank as $r_1=r_2=r_3=2$. Figure~\ref{fig:cov} represents the scenario under Student's $t$-distributed noise with degrees of freedom $\nu=2.01$, in the absence of sparse corruptions. 
The left figure \ref{fig:11} illustrates a low signal-to-noise ratio scenario where $\fro{\bcalT^{\ast}}/\EE|\xi|=300$. In this setting, the signal-to-noise ratio fulfills the condition $\mins^\leq \gamma d^{\ast1/2}$; according to Theorem~\ref{thm:hub:dynamics} and Theorem~\ref{thm:abs:dynamics}, it should bypass phase one and directly enter phase two. As expected, Figure~\ref{fig:11} shows that the iterations do enter the second phase after a few steps, aligning with our theoretical analysis. 
Conversely, Figure~\ref{fig:12} demonstrates a high signal-to-noise ratio setting where $\fro{\bcalT^{\ast}}/\EE|\xi|=1500$, clearly exhibiting the two-phase convergence of RsGrad. In both cases (figures \ref{fig:11} and \ref{fig:12}), RsGrad performs better. 
Figure~\ref{fig:cov-corrup} is plotted under conditions of both dense noise and sparse corruptions. For achieving the typical PCA optimal rate $\textsf{DoF}_m^{1/2}$ \citep{zhang2018tensor}, the corruption rate should be bounded by $(\textsf{DoF}_m/d^{\ast})^{1/2}\approx 0.02$ according to Theorem~\ref{thm:hub:dynamics}. Therefore, we fix the corruption rate $\alpha$ to be either $0.01$ or $0.02$. To differentiate from the scheme in \cite{chen2021bridging}, we set all the non-zero entries of the corruptions to large positive values, such as exceeding $100\times\linft{\bcalT^{\ast}}$. The top two figures \ref{fig:61} and \ref{fig:62} depict the scenario under Student's t noise with degrees of freedom $\nu=2.01$. The bottom two figures \ref{fig:63} and \ref{fig:64} illustrate the scenario under Gaussian noise. The results show that under heavy-tailed noise, RsGrad significantly outperforms RGrad. Conversely, under Gaussian noise, RGrad and RsGrad exhibit similar performance.

\begin{figure}
	\centering
	\begin{subfigure}[b]{0.45\textwidth}
		\centering
		\includegraphics[width=\textwidth]{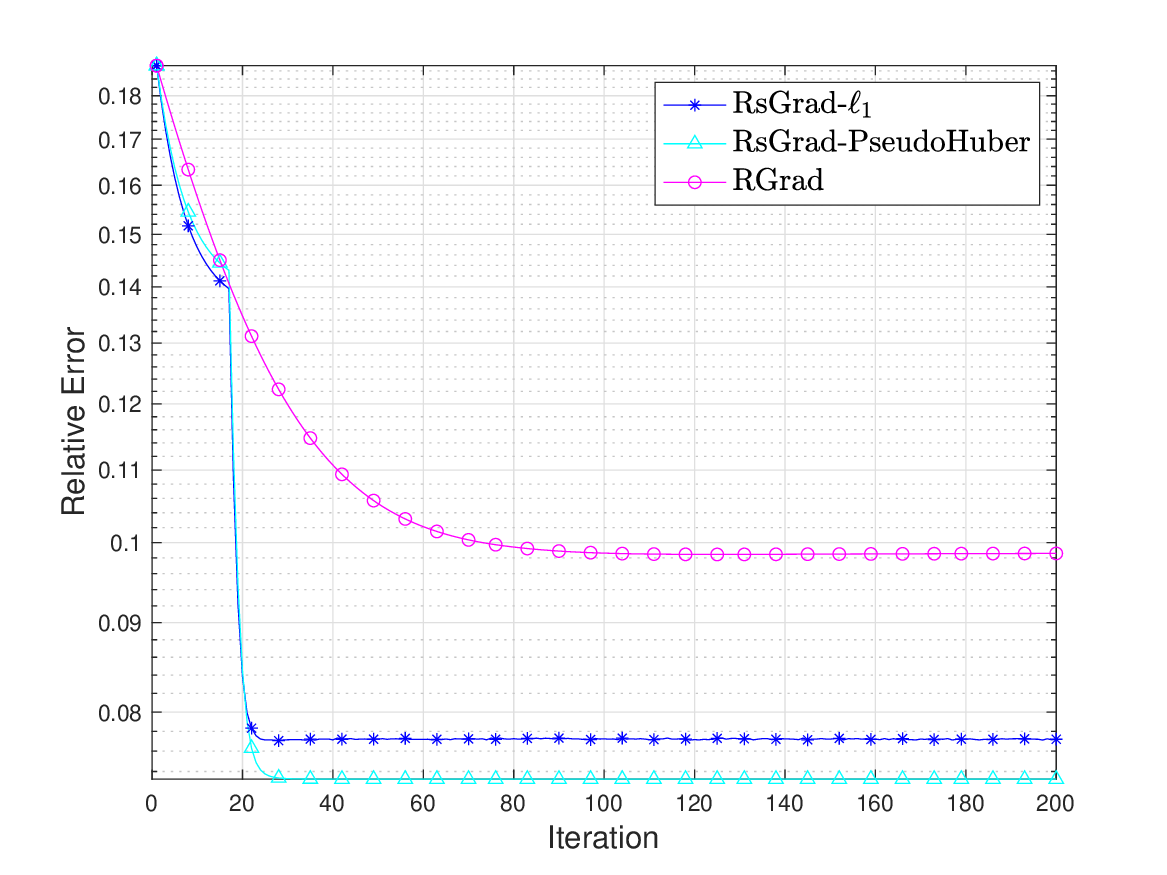}
		\caption{$\frac{\fro{\bcalT^*}}{\EE|\xi|}=300$}
		\label{fig:11}
	\end{subfigure}
	\hfill
	\begin{subfigure}[b]{0.45\textwidth}
		\centering
		\includegraphics[width=\textwidth]{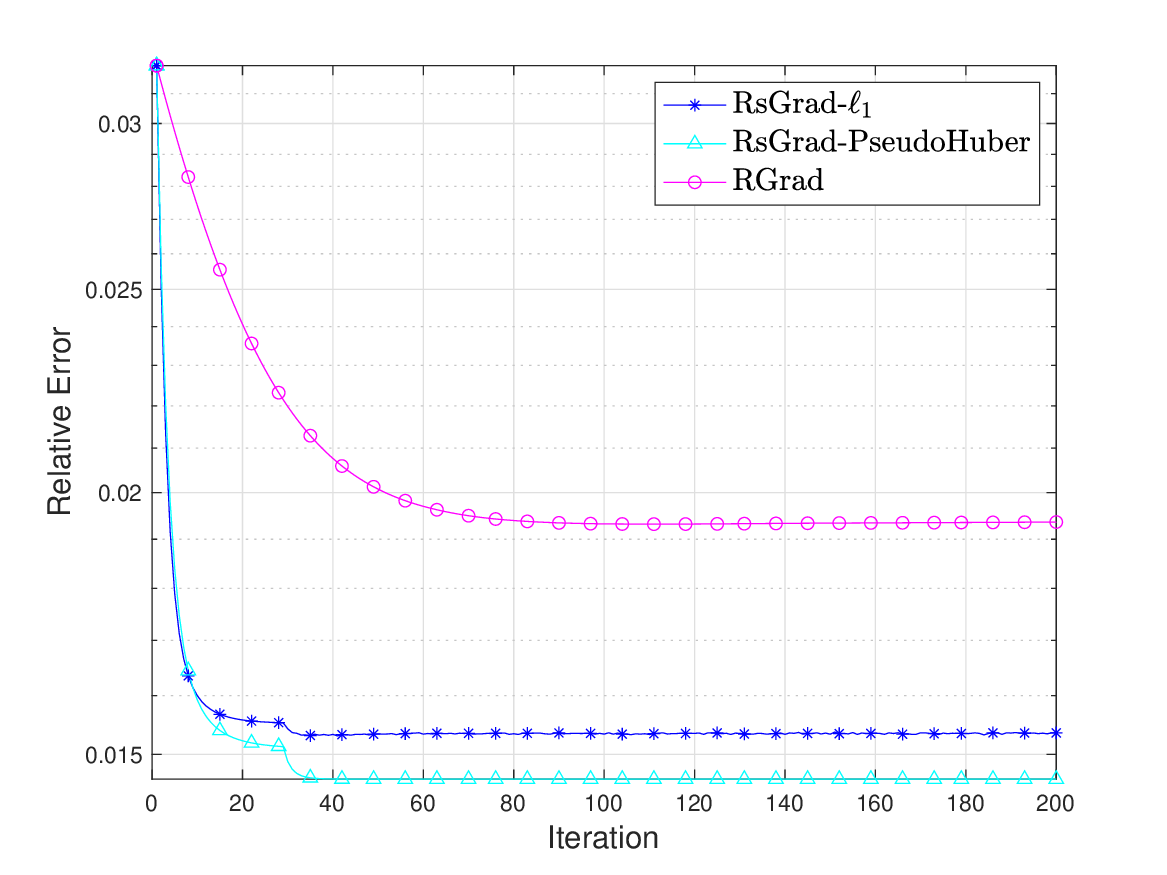}
		\caption{$\frac{\fro{\bcalT^*}}{\EE|\xi|}=1500$}
		\label{fig:12}
	\end{subfigure}
	\caption{Convergence dynamics of RGrad \citep{cai2022generalized}, RsGrad-$\ell_1$ (Algorithm~\ref{alg:abs:RsGrad}) and RsGrad-Pseudo Huber (Algorithm~\ref{alg:pseuHuber:RsGrad}) under Student $t$ noise with d.f. $\nu=2.01$. Dimension $d_1=d_2=d_3=100$, Tucker rank $r_1=r_2=r_3=2$.}
	\label{fig:cov}
\end{figure}

\begin{figure}
	\centering
	\begin{subfigure}[b]{0.45\textwidth}
		\centering
		\includegraphics[width=\textwidth]{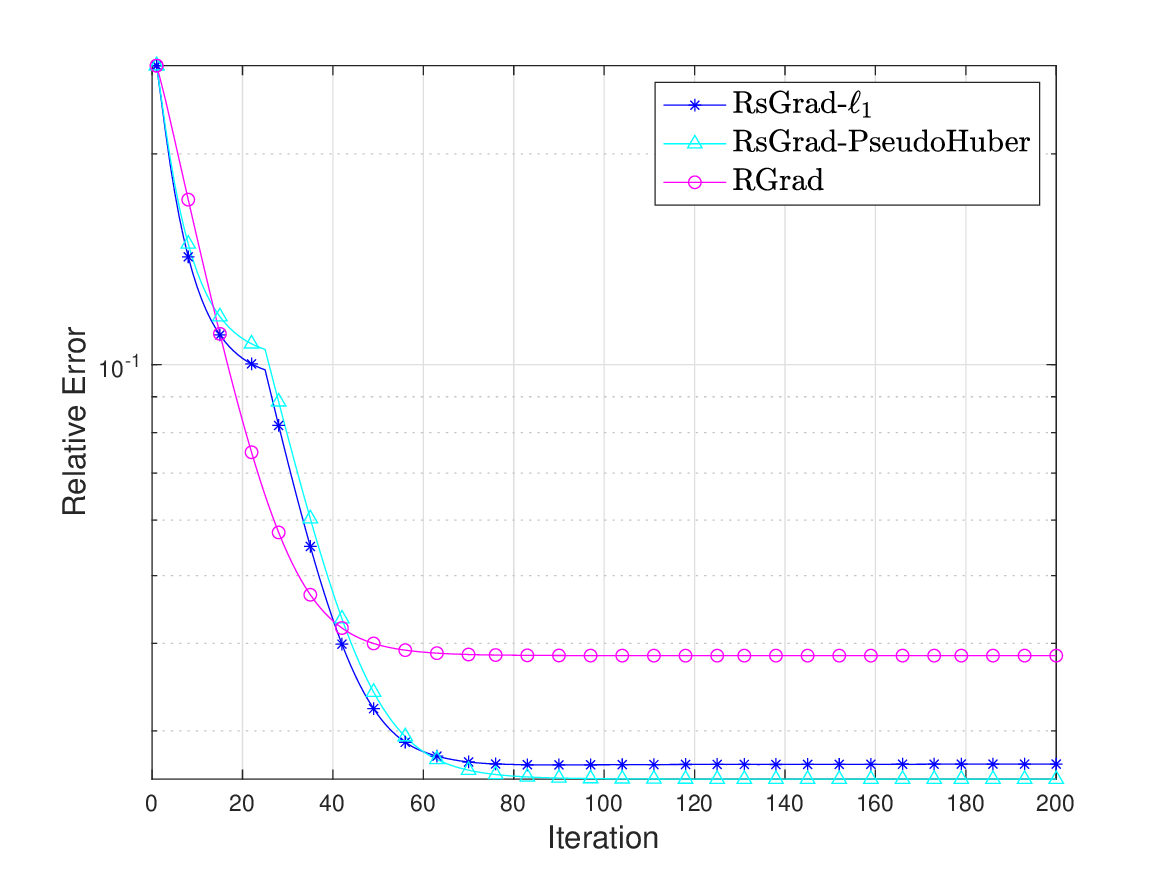}
		\caption{Student's t noise $\nu=2.01$, Corruption rate $\alpha=0.01$}
		\label{fig:61}
	\end{subfigure}
	\hfill
	\begin{subfigure}[b]{0.45\textwidth}
		\centering
		\includegraphics[width=\textwidth]{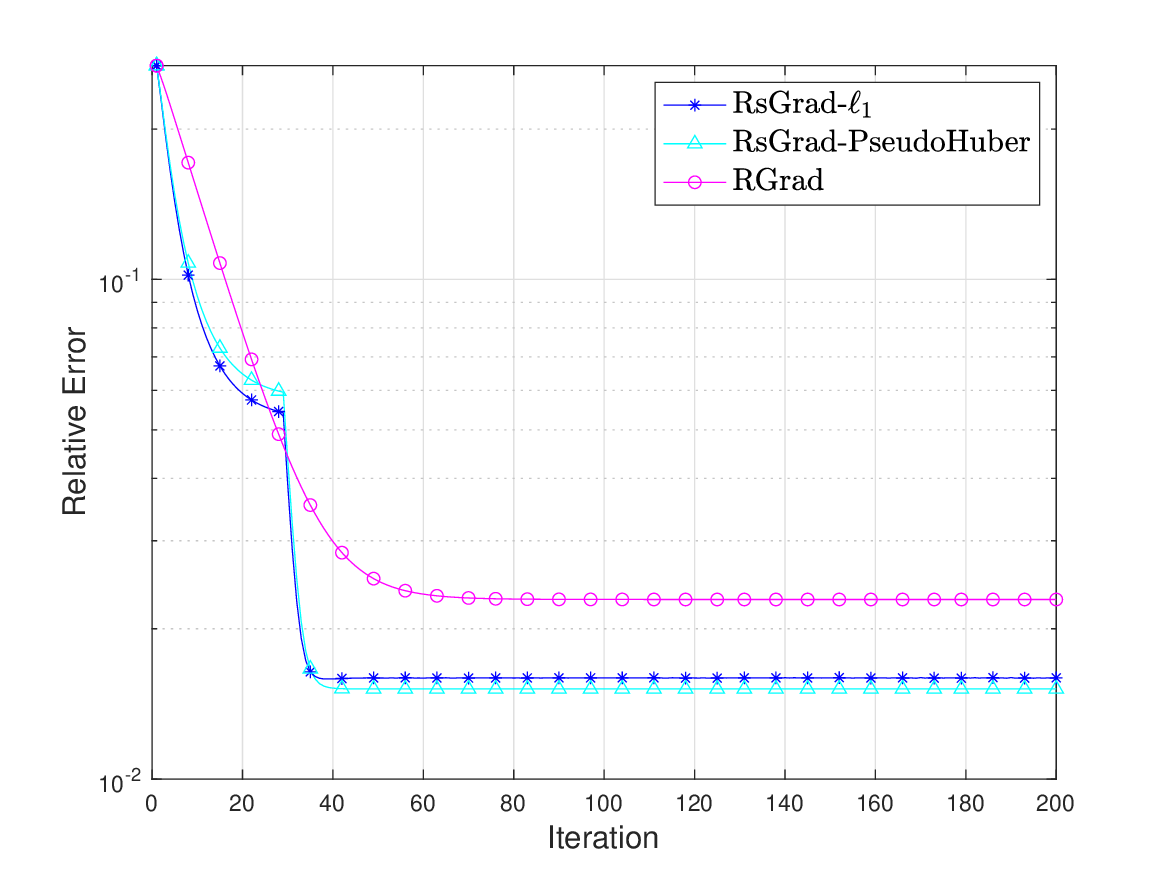}
		\caption{Student's t noise $\nu=2.01$, Corruption Rate $\alpha=0.02$}
		\label{fig:62}
	\end{subfigure}

\begin{subfigure}[b]{0.45\textwidth}
	\centering
	\includegraphics[width=\textwidth]{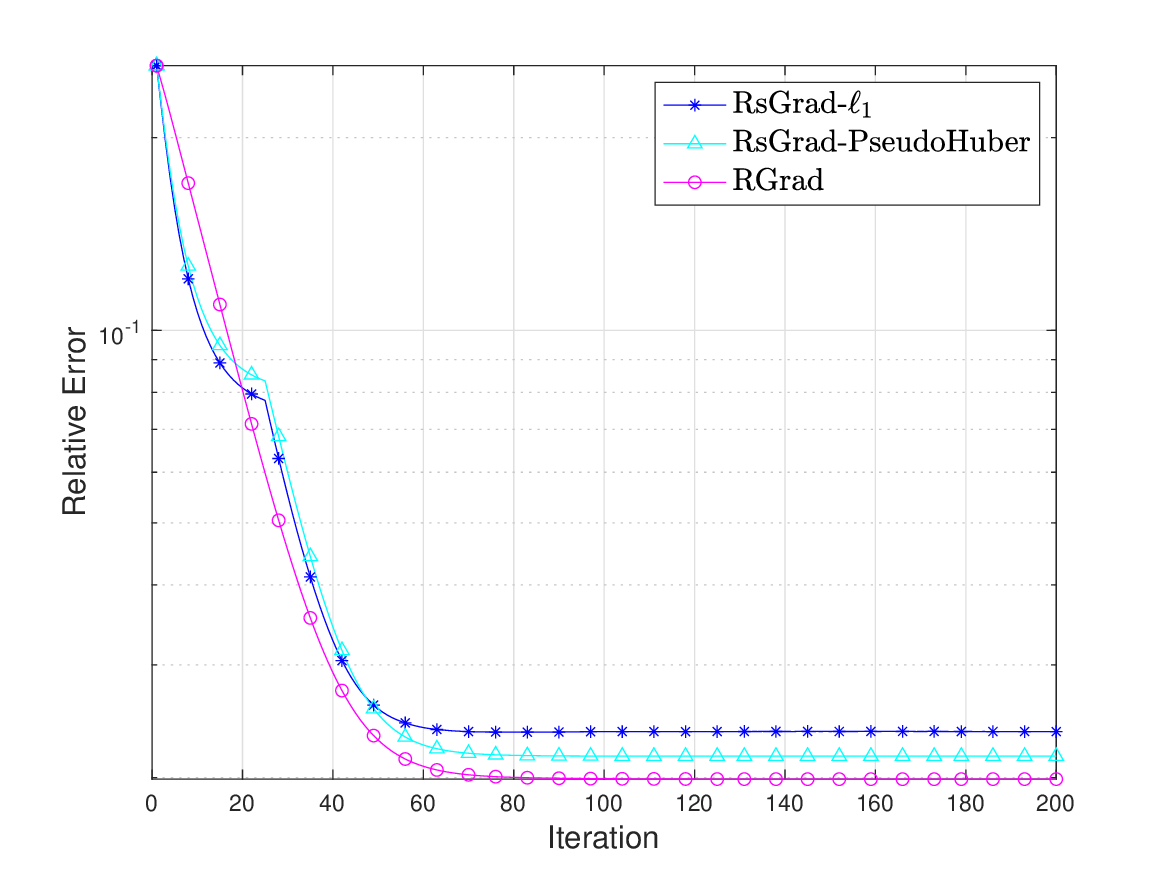}
	\caption{Gaussian noise, Corruption rate $\alpha=0.01$}
	\label{fig:63}
\end{subfigure}
\hfill
\begin{subfigure}[b]{0.45\textwidth}
	\centering
	\includegraphics[width=\textwidth]{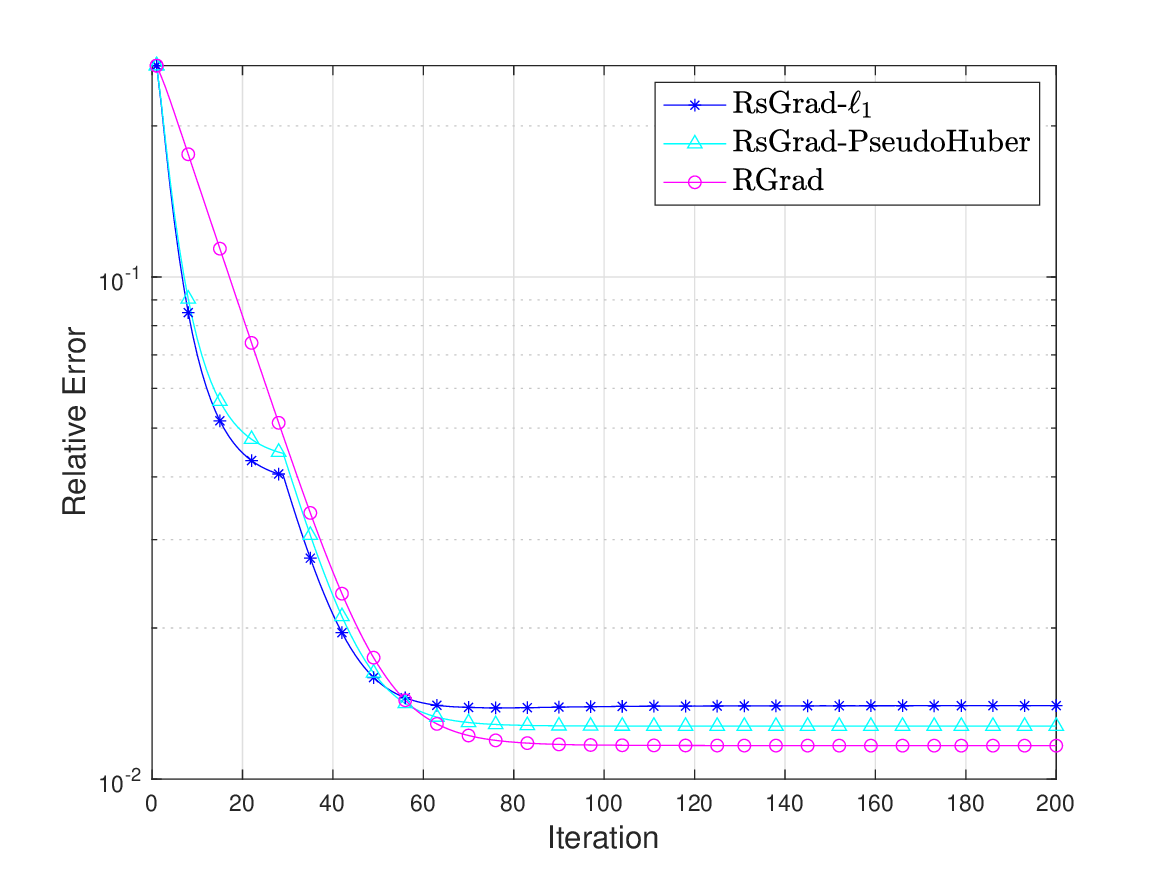}
	\caption{Gaussian noise, Corruption rate $\alpha=0.02$}
	\label{fig:64}
\end{subfigure}
	\caption{Convergence dynamics of RGrad \citep{cai2022generalized}, RsGrad-$\ell_1$ (Algorithm~\ref{alg:abs:RsGrad})  and RsGrad-PseudoHuber (Algorithm~\ref{alg:pseuHuber:RsGrad})  under dense noise and sparse corruptions, with dimension $d_1=d_2=d_3=100$, Tucker rank $r_1=r_2=r_3=2$ and a high signal-to-noise ratio $\fro{\bcalT^{\ast}}/\EE|\xi|=1500$.}
	\label{fig:cov-corrup}
\end{figure}

\paragraph{Accuracy} We assess the accuracy of output estimators by comparing them with the robust HOSVD approach \citep{auddy2022estimating}. The robust HOSVD method employs Catoni's estimator for initialization, followed by a one-step power iteration. This approach achieves statistically optimal accuracy up to a logarithmic factor with a smaller probability $1-\Omega\big((\log d)^{-1}\big)$. It's important to note that the robust HOSVD approach primarily provides eigenvector estimations for rank-one tensors under heavy-tailed noise conditions. Consequently, we have fixed the setting to $d_1=d_2=d_3=100$, $r_1=r_2=r_3=1$, with Student's t noise with a degree of freedom $\nu=2.01$, and we are comparing the accuracy of eigenvector estimation using the $\sin\Theta$ distance. 
Figure~\ref{fig:acc} presents a box-plot based on 50 replications. The left figure pertains to a low signal-to-noise ratio setting, where $\fro{\bcalT^{\ast}}/\EE|\xi|=150$, while the right figure corresponds to a scenario where $\fro{\bcalT^{\ast}}/\EE|\xi|=1000$. The results demonstrate that RsGrad exhibits greater robustness against heavy-tailed noise, along with superior accuracy and reduced deviation, which aligns with established theories.

\begin{figure}
	\centering
	\begin{subfigure}[b]{0.45\textwidth}
		\centering
		\includegraphics[width=\textwidth]{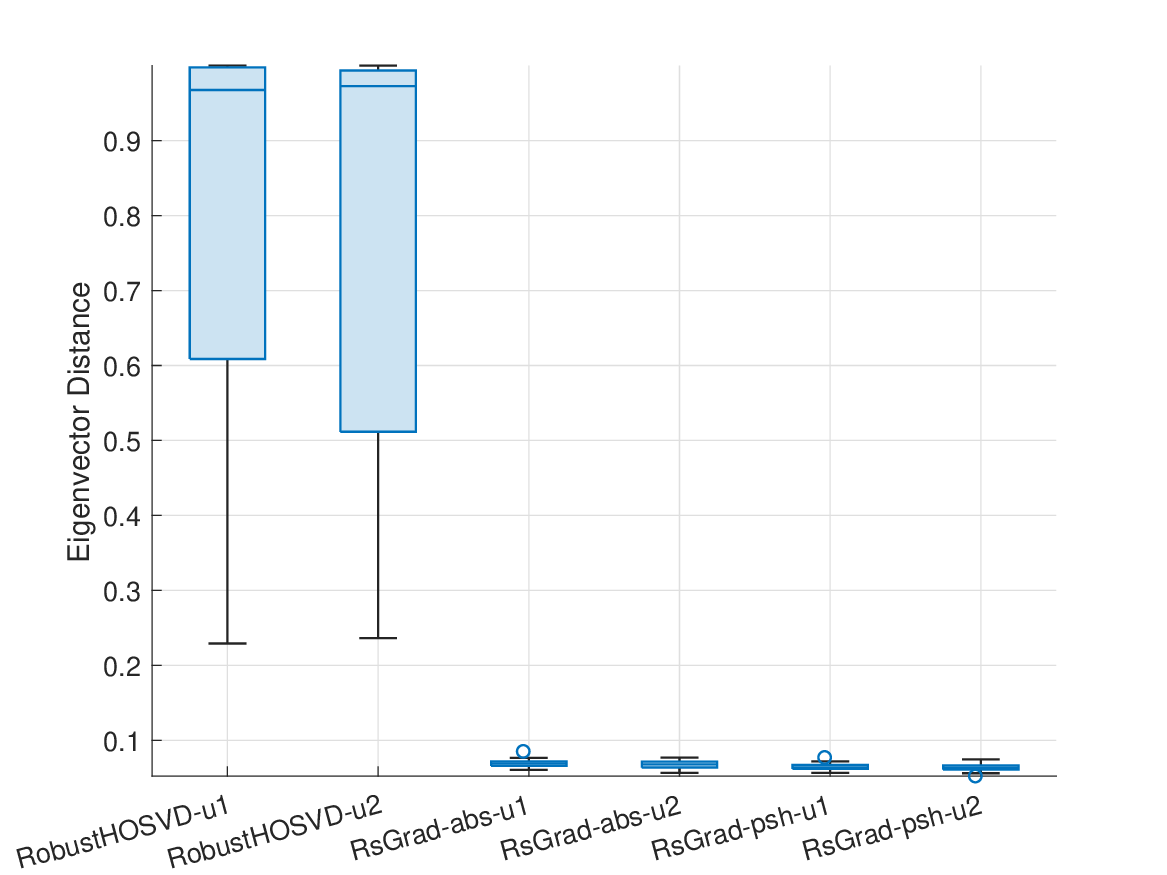}
		\caption{$\frac{\fro{\bcalT^*}}{\EE|\xi|}=150$}
		\label{fig:21}
	\end{subfigure}
	\hfill
	\begin{subfigure}[b]{0.45\textwidth}
		\centering
		\includegraphics[width=\textwidth]{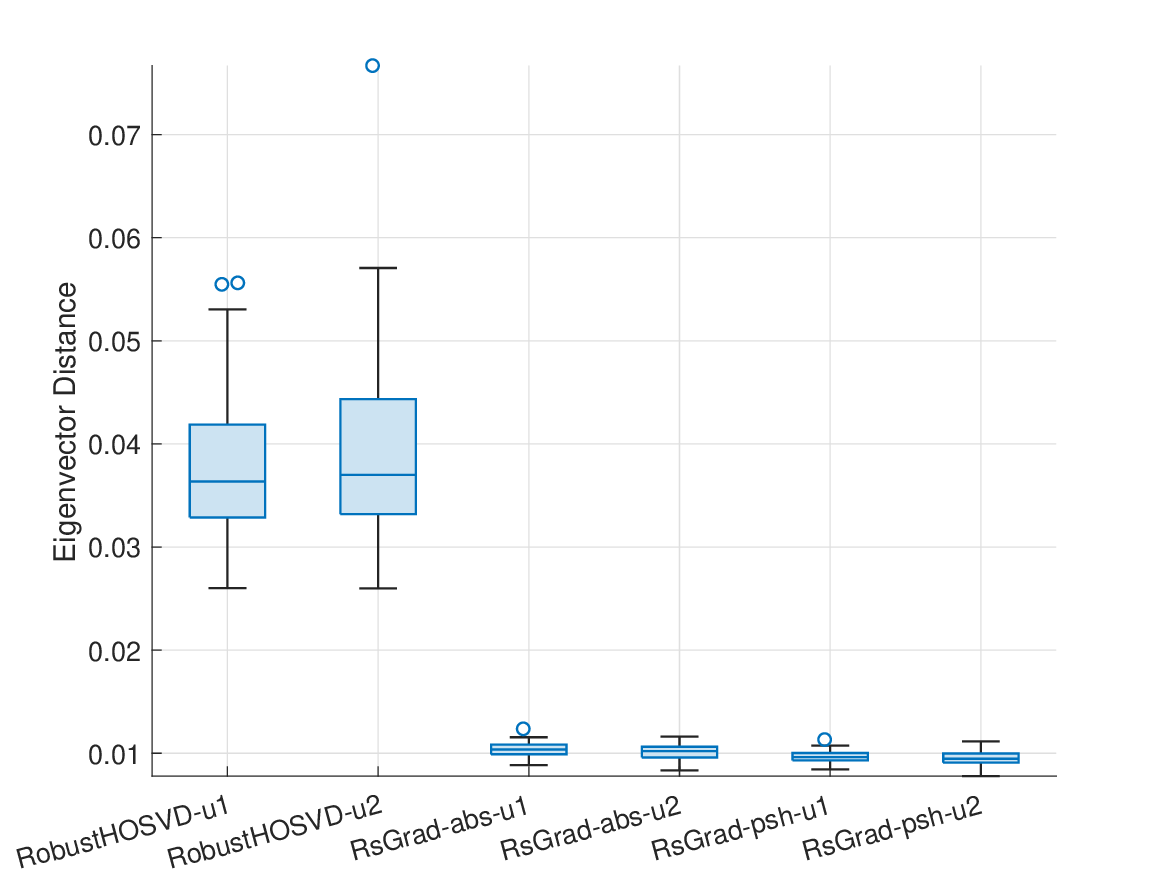}
		\caption{$\frac{\fro{\bcalT^*}}{\EE|\xi|}=1000$}
		\label{fig:22}
	\end{subfigure}
	\caption{Accuracy Comparisons of Robust HOSVD \citep{auddy2022estimating}, RsGrad-$\ell_1$ (Algorithm~\ref{alg:abs:RsGrad}) and RsGrad-Pseudo Huber (Algorithm~\ref{alg:pseuHuber:RsGrad})  under Student's $t$ noise with d.f. $\nu=2.01$, replicated $50$ times, dimension $d_1=d_2=d_3=100$ and Tucker rank $r_1=r_2=r_3=1$.}
	\label{fig:acc}
\end{figure}

\section{Real Data Applications}
\label{sec:real}
\subsection{Food balance dataset}
We collected the \textsf{Food Balance Dataset} from \url{https://www.fao.org/faostat/en/#data/FBS}. This dataset provides an intricate breakdown of a country or region's food supply during a specified period. Our analysis focuses on the food balance data in the year $2018$. We have incorporated all metrics for all items, excluding population, such as `production quantity', `import quantity', and `food supply' for `wheat and products', `apples and products'. It is crucial to acknowledge that some values in the dataset are imputed, while others are estimated, as per the notes on its website. This necessitates the use of robust statistical methods.

We first analyze the food balance data in Asian regions, consisting of $45$ countries or regions, such as Yemen, Viet Nam and so on. Consequently, we procure a three-way tensor $\textsf{Region}\times\textsf{Measurment}\times\textsf{Items}$, sized $45\times20\times 97$. It's worth noting that some of the measurements are the total value for the entire country for the year, while some represent per capita value per day; some indicate fat supply quantity, while others denote protein supply quantity.  To unify different measurements and negate the influence of population size, we scale the $45\times 20$ vectors of size $97$ such that each vector has a unit Euclidean length. The entries of the scaled tensor depict the proportion of a specific food type overall, and the entire tensor can reflect the dietary habits of a country or a region. For instance, different regions may have preferences for various kinds of meat or oil, despite each type providing protein or fat. We employ the RsGrad algorithm with an input Tucker rank of $(r_1,r_2,r_3)=(5,2,5)$, as increasing ranks do not significantly reduce the residuals. In fact, choices within the region $(2,2,2)-(10,5,10)$ yield similar results. We obtained Figure~\ref{fig:71} by plotting the second component eigenvector against the first one along the $\textsf{Region}$ trajectory. Southeast Asian countries, renowned for their Southeast Asian cuisine, occupy the top left of the figure. The center of the figure primarily consists of East Asian and South Asian countries or regions, which share similar dietary habits. The bottom right clusters West Asian countries that are geographically proximate. The figure effectively encapsulates the differences and similarities in dietary habits across Asia.

\begin{figure}
	\centering
	\begin{subfigure}[b]{0.45\textwidth}
		\centering
		\includegraphics[width=\textwidth]{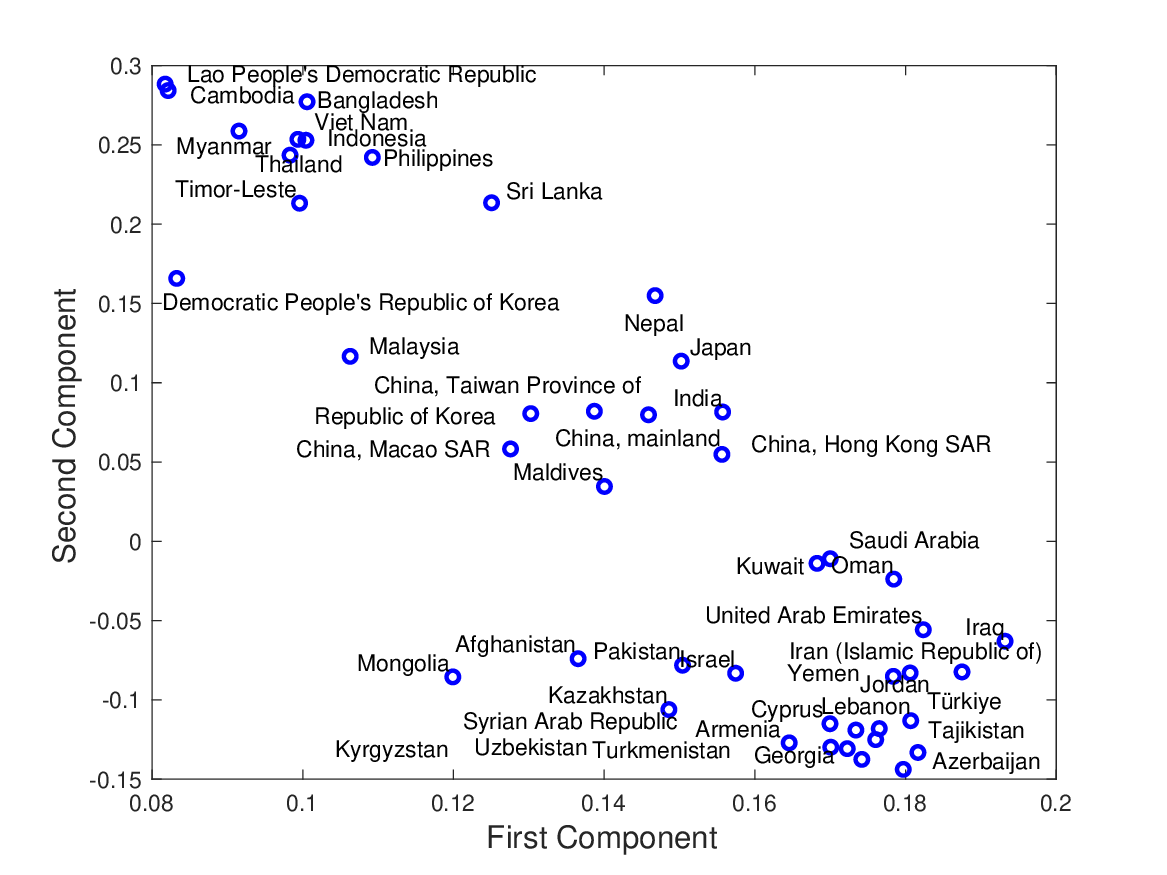}
		\caption{Regions in Asia}
		\label{fig:71}
	\end{subfigure}
	\hfill
\begin{subfigure}[b]{0.45\textwidth}
	\centering
	\includegraphics[width=\textwidth]{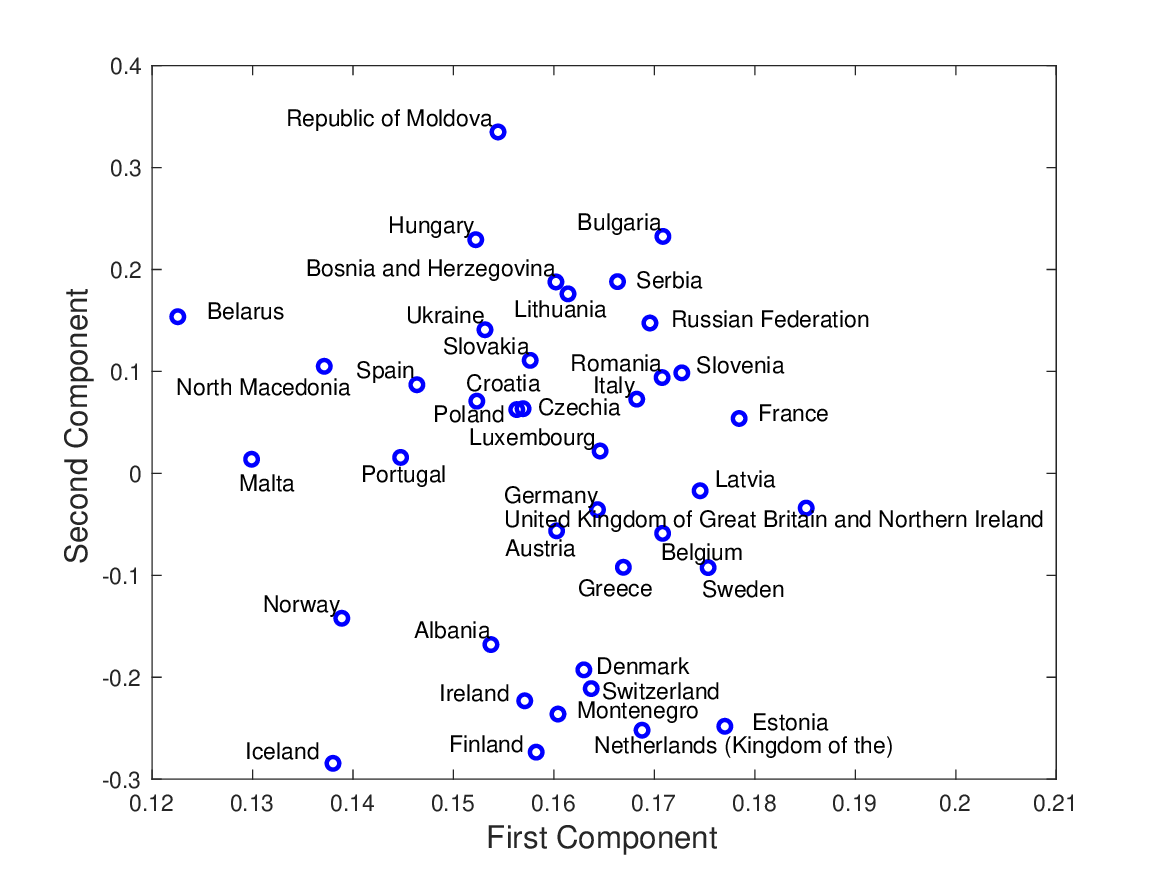}
	\caption{Regions in Europe}
	\label{fig:73}
\end{subfigure}
	\caption{Food balance in Asia and Europe. Node embedding by the leading two eigenvectors are presented. In the left figure, Southeast Asian, East Asian and South Asian, West Asian countries or regions are clustered, respectively, consistent with Asian culture. The right figure is obtained from European data and is also able to demonstrate the country habitat similarities.}
\end{figure}

Studies by \cite{cai2022generalized,dong2022fast} have indicated that varying robustness parameters can yield significantly different results. In our case, such confusion is not an issue. Although soft thresholding \citep{dong2022fast} or quantile thresholding \citep{cai2022generalized} can be employed to identify outliers, we provide a heatmap of absolute residuals measured with `food supply' in Figure~\ref{fig:72}. This method demonstrates that, barring a few outlying entries, the remaining values are sufficiently small. It reveals notable deviations in the supply of soybean oil in Taiwan, as well as maize supply in the Democratic People's Republic of Korea and Timor-Leste. 
Figure~\ref{fig:75} presents a heatmap of the scaled dataset within the `food supply' slice. However, it cannot identify the outlying entries, and can only illustrate which types of food are in high demand. Particularly, some staple food columns such as rice and wheat stand out.

In parallel, Figures~\ref{fig:73}, \ref{fig:74}, and \ref{fig:76} are derived from the European Food Balance Dataset. They also illustrate dietary similarities in Europe, where geographically close countries tend to cluster, such as Iceland, Finland, Norway, and others. Similar to the Asian dataset, the absolute residuals here can pinpoint outlying entries like maize supply in Albania and olive oil in Greece and Spain. However, the scaled original data can't provide this information, only indicating that wheat, milk, and sugar are in substantial demand across Europe..

\begin{figure}
	\begin{subfigure}[b]{0.45\textwidth}
		\centering
		\includegraphics[width=\textwidth]{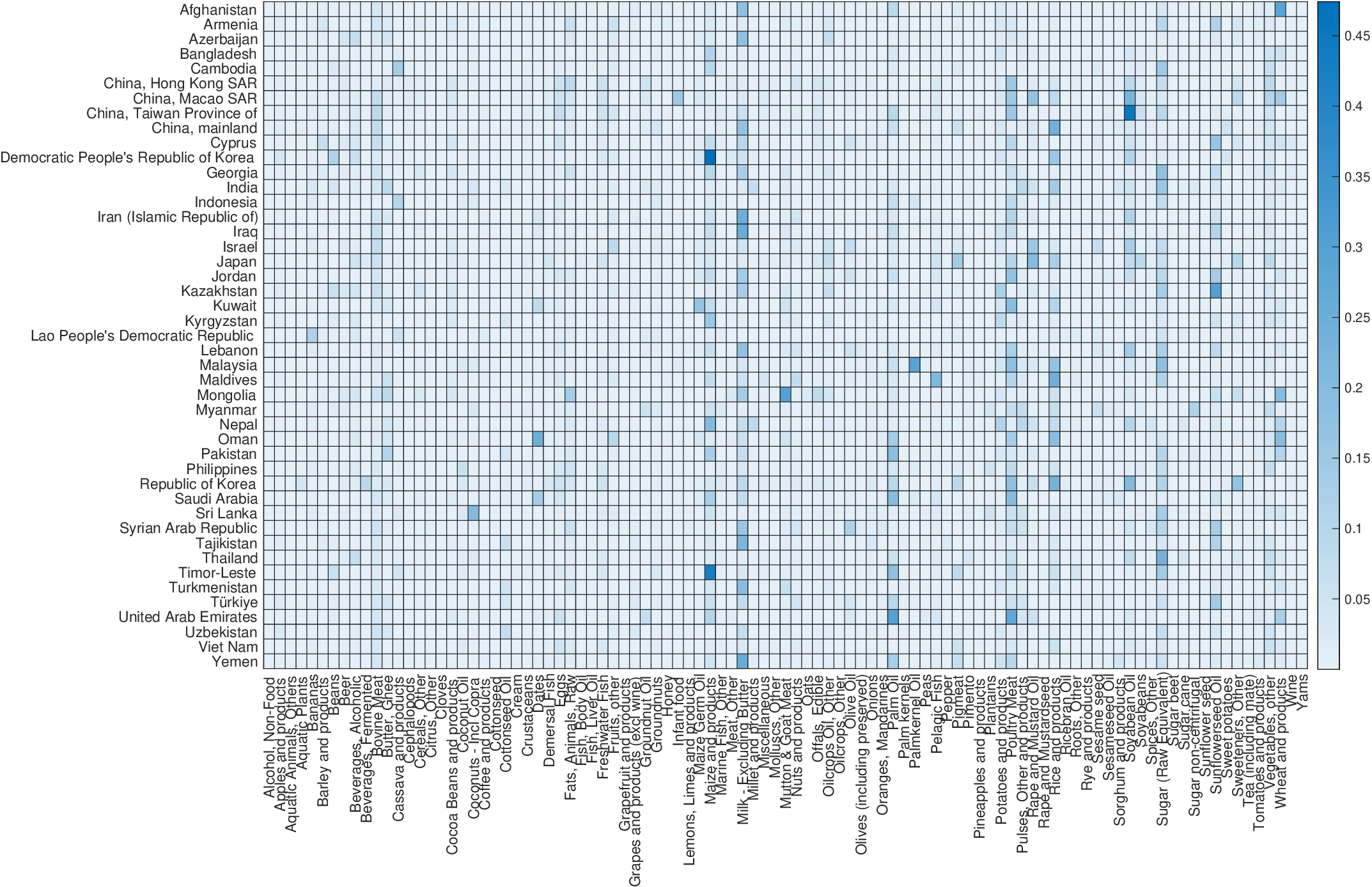}
		\caption{Asian Regions: Absolute Residuals}
		\label{fig:72}
	\end{subfigure}
\hfill
\begin{subfigure}[b]{0.45\textwidth}
	\centering
	\includegraphics[width=\textwidth]{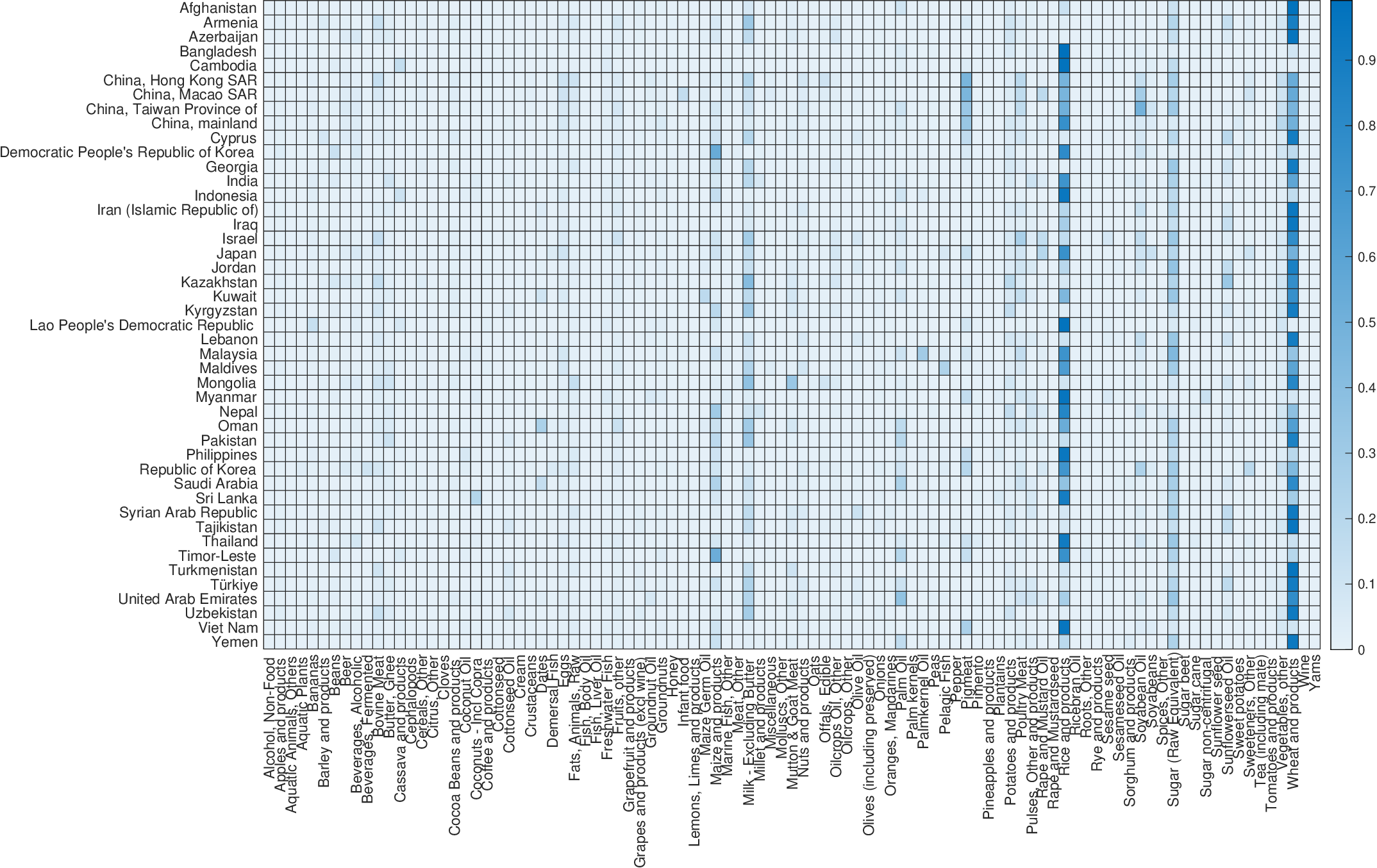}
	\caption{Asian Regions: Scaled Original Data}
	\label{fig:75}
\end{subfigure}

	\begin{subfigure}[b]{0.45\textwidth}
		\centering
		\includegraphics[width=\textwidth]{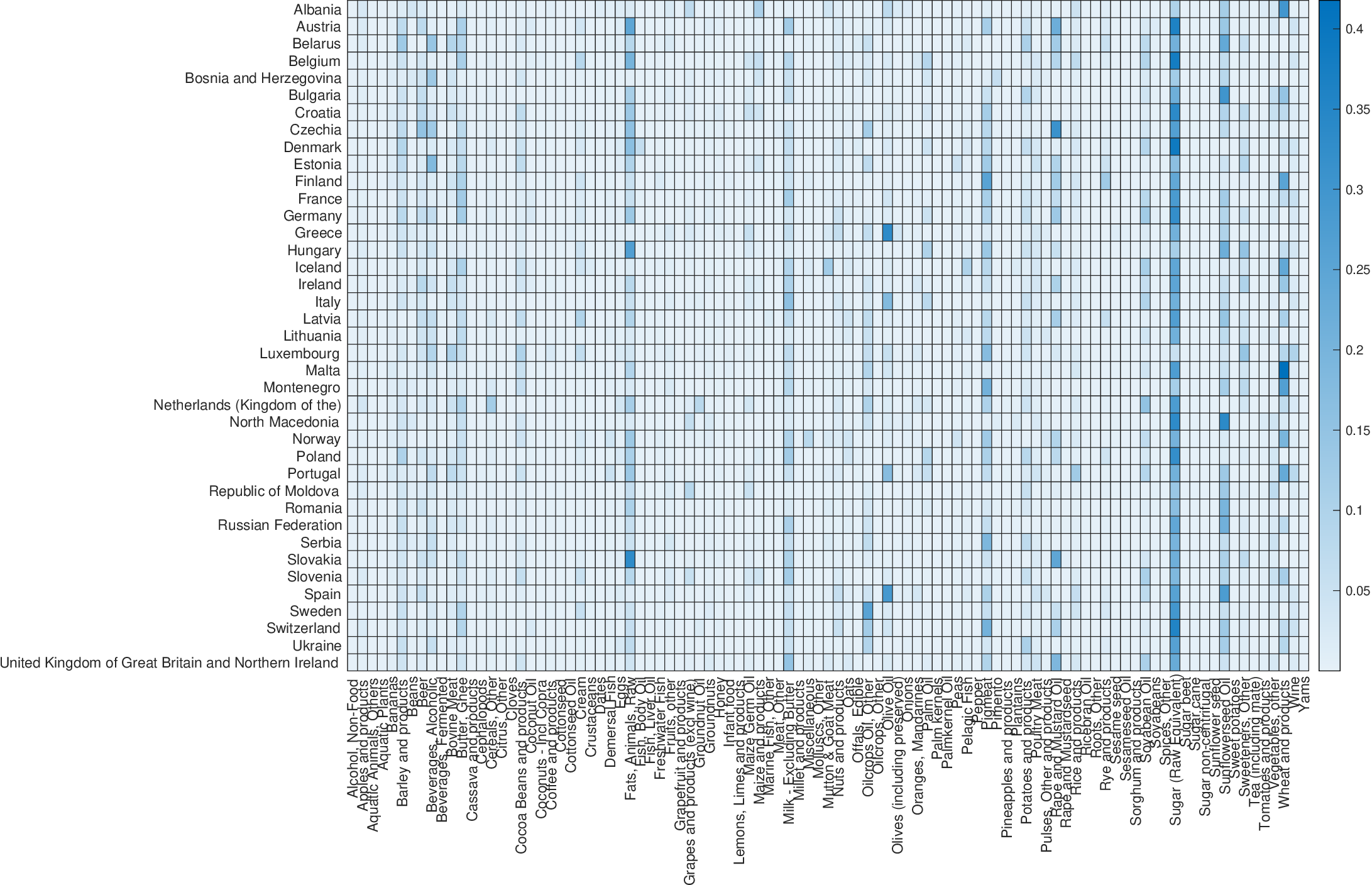}
		\caption{European Regions: Absolute Residuals}
		\label{fig:74}
	\end{subfigure}
\hfill
	\begin{subfigure}[b]{0.45\textwidth}
	\centering
	\includegraphics[width=\textwidth]{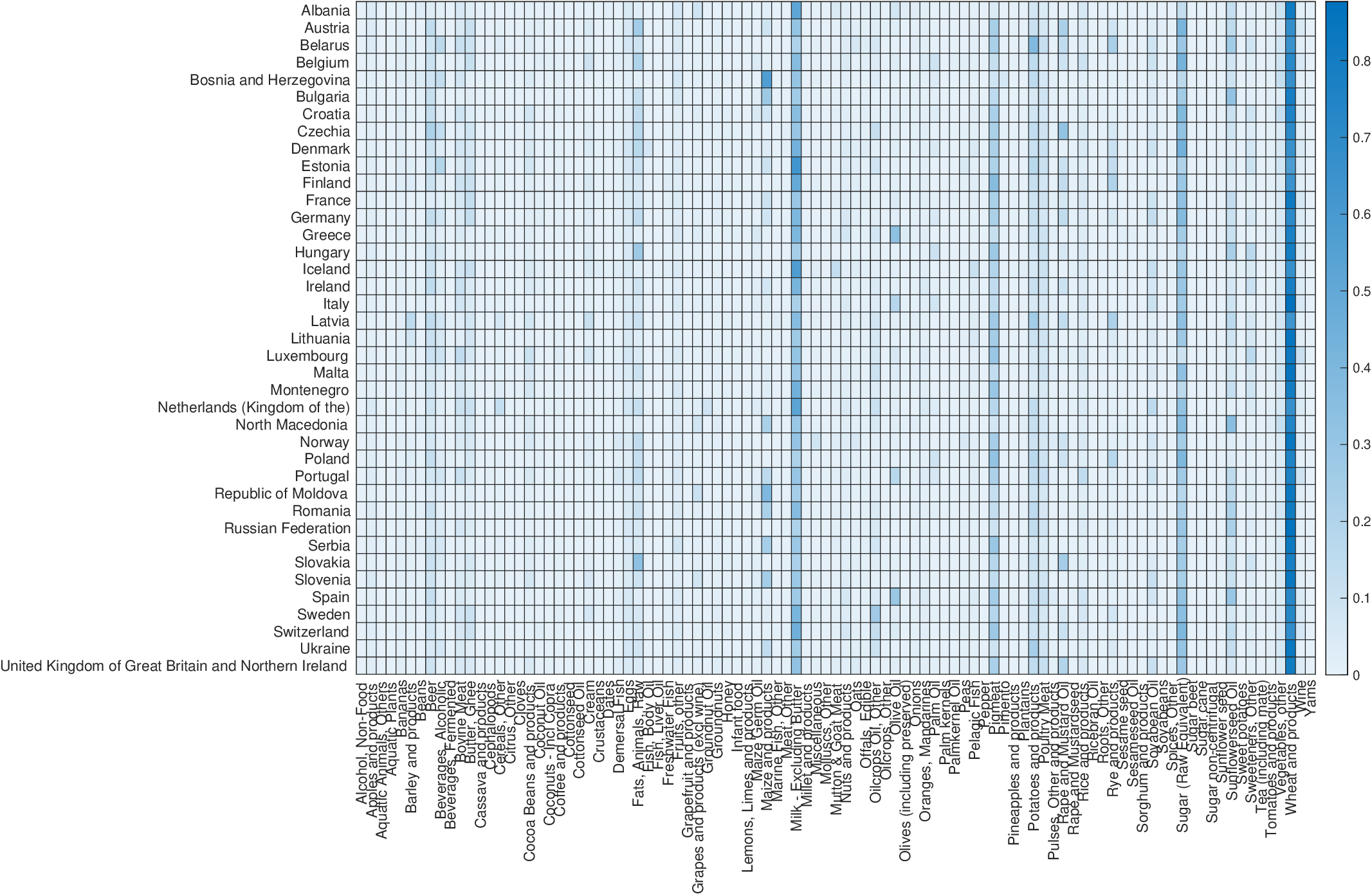}
	\caption{European Regions: Scaled Original Data}
	\label{fig:76}
\end{subfigure}
\caption{Slice of food supply measurement. The aforementioned figures illustrate that the scaled original data can indicate which types of food are in high demand. On the other hand, the outlying entries visible in the absolute residuals plot represent data that cannot be approximated by a low-rank structure, essentially indicating deviations from the pattern. This demonstrates the ability of our methods to uncover structures that may not be immediately discernible from the original data. Moreover, it underscores the robustness of the RsGrad method in handling outliers.}
\label{fig:7}
\end{figure}

\subsection{Trade flow dataset}
We colloected trade flow data from \url{https://comtradeplus.un.org/TradeFlow}, containing the trading quantity among countries. The goods are categorized according to HS code which could be found in \url{https://www.foreign-trade.com/reference/hscode.htm}. We focus on the import data among $47$ countries or regions. Specifically, $12$ of the countries are from Asia, $17$ from Europe, and $6$ from American.

The import amount is measured using the `CIF value', and we examine the trade of all goods categories (encoded as HS codes 01-97) during the year 2018. This results in a $47\times47\times97$ tensor, corresponding to $\textsf{Import Places}\times\textsf{Export Places}\times\textsf{Goods Category}$. After discarding the zero slices, we are left with a $45\times47\times96$ tensor. Given that population size significantly influences the quantity of imported goods, we scale the $45$ slices of the $47\times96$ matrices, ensuring each slice has a unit Frobenius norm. Consequently, each entry now represents the import proportion of certain goods from a specific country over the total import quantity. This scaled tensor can reflect a country's goods requirements or economic structure, and demonstrate whether two countries maintain a close trade relationship.
We input this tensor into the RsGrad algorithm with a Tucker rank of $(r_1,r_2,r_3)=(3,3,8)$, aiming to uncover the latent low-rank structure. Notably, the visualization is insensitive to rank selections: we have experimented with ranks in the region $(2,2,2)-(8,8,8)$, all of which produce similar outputs. Figure~\ref{fig:81} and \ref{fig:82} display the leading three eigenvectors in the $\textsf{Import Places}$ direction. Countries from the Americas, Asia, and Europe are denoted with blue circles, red triangles, and cyan plus signs respectively. In both figures, European countries cluster together, while Asian countries merge with American countries. This outcome aligns with the fact that a significant amount of trade occurs within Europe \citep{cai2022generalized}.
 
\begin{figure}
	\centering
	\begin{subfigure}[b]{0.45\textwidth}
		\centering
		\includegraphics[width=\textwidth]{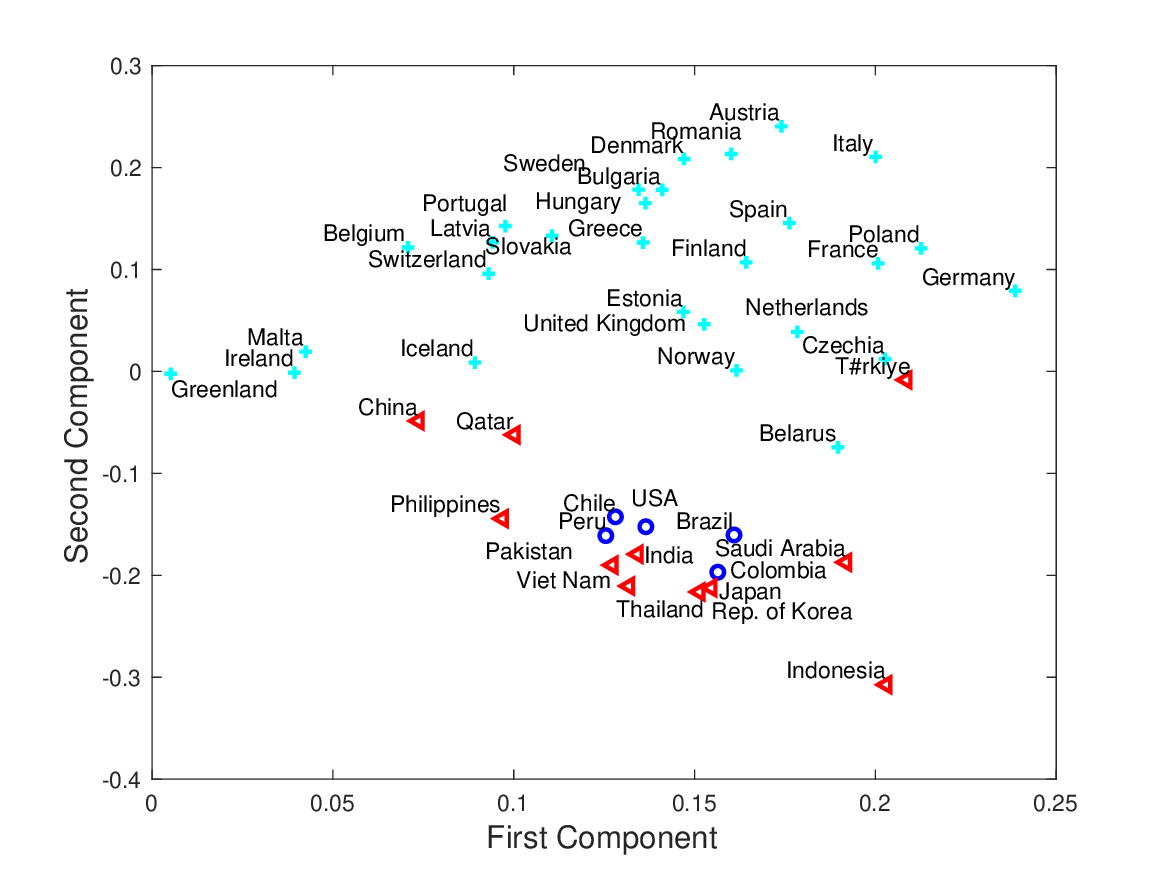}
		\caption{First Component against Second Component}
		\label{fig:81}
	\end{subfigure}
	\hfill
	\begin{subfigure}[b]{0.45\textwidth}
		\centering
		\includegraphics[width=\textwidth]{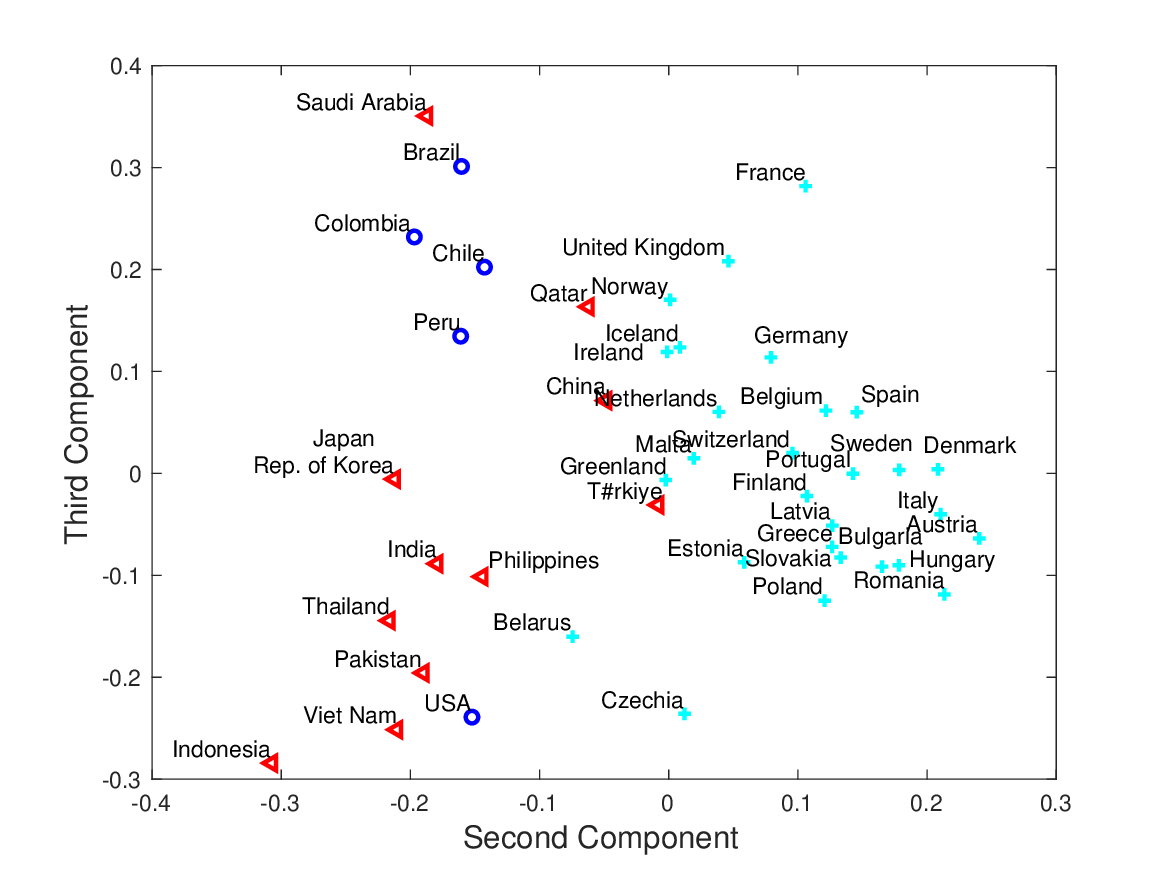}
		\caption{Second Component agains Third Component}
		\label{fig:82}
	\end{subfigure}
	\caption{Trade flow visualization of the $45$ countries. Blue circles represent American countries; Asian countries are plotted with red triangulars; cyan plus signs are used to mark European countries. In both figures, European countries cluster.}
\end{figure}

We also illustrate four slices of absolute residuals, corresponding to `clocks and watches and parts thereof', `glass and glassware', `mineral fuels, mineral oils and products of their distillation; bituminous substances; mineral waxes', and `printed books, newspapers, pictures and other products of the printing industry; manuscripts, typescripts and plans' (encoded as HS codes $91$, $70$, $27$ and $49$ respectively). In Figure~\ref{fig:83}, we observe that the import of glass and glassware from Portugal constitutes a significant portion of Spain's total imports. This is understandable given that Marinha Grande, a city in Portugal known as `The Crystal City', is renowned for its glass and glassware manufacturing. Figure~\ref{fig:84} shows that the import proportions of clocks and watches from Switzerland are notably high in China and France, reflecting Switzerland's prestige in watch manufacturing. Figure~\ref{fig:85} depicts the absolute residual plot in the mineral products slice, corroborating the fact that Norway is a major importer of mineral fuels. Finally, Figure~\ref{fig:86} reveals that the import of printed books and newspapers is significant in Germany, particularly from Austria and Switzerland.

\begin{figure}
	\centering
	\begin{subfigure}[b]{0.45\textwidth}
		\centering
		\includegraphics[width=\textwidth]{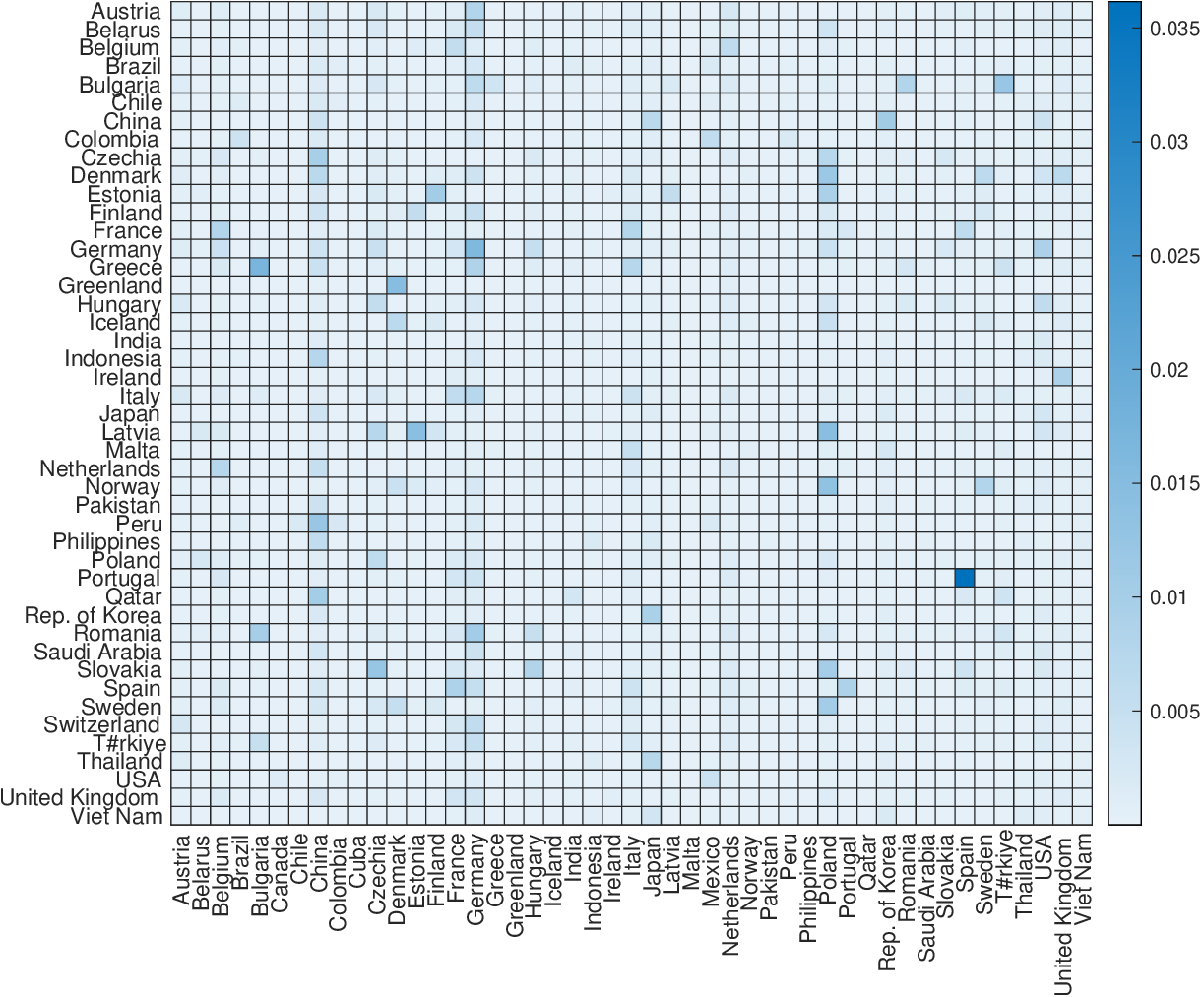}
		\caption{Glass and galssware}
		\label{fig:83}
	\end{subfigure}
	\hfill
	\begin{subfigure}[b]{0.45\textwidth}
		\centering
		\includegraphics[width=\textwidth]{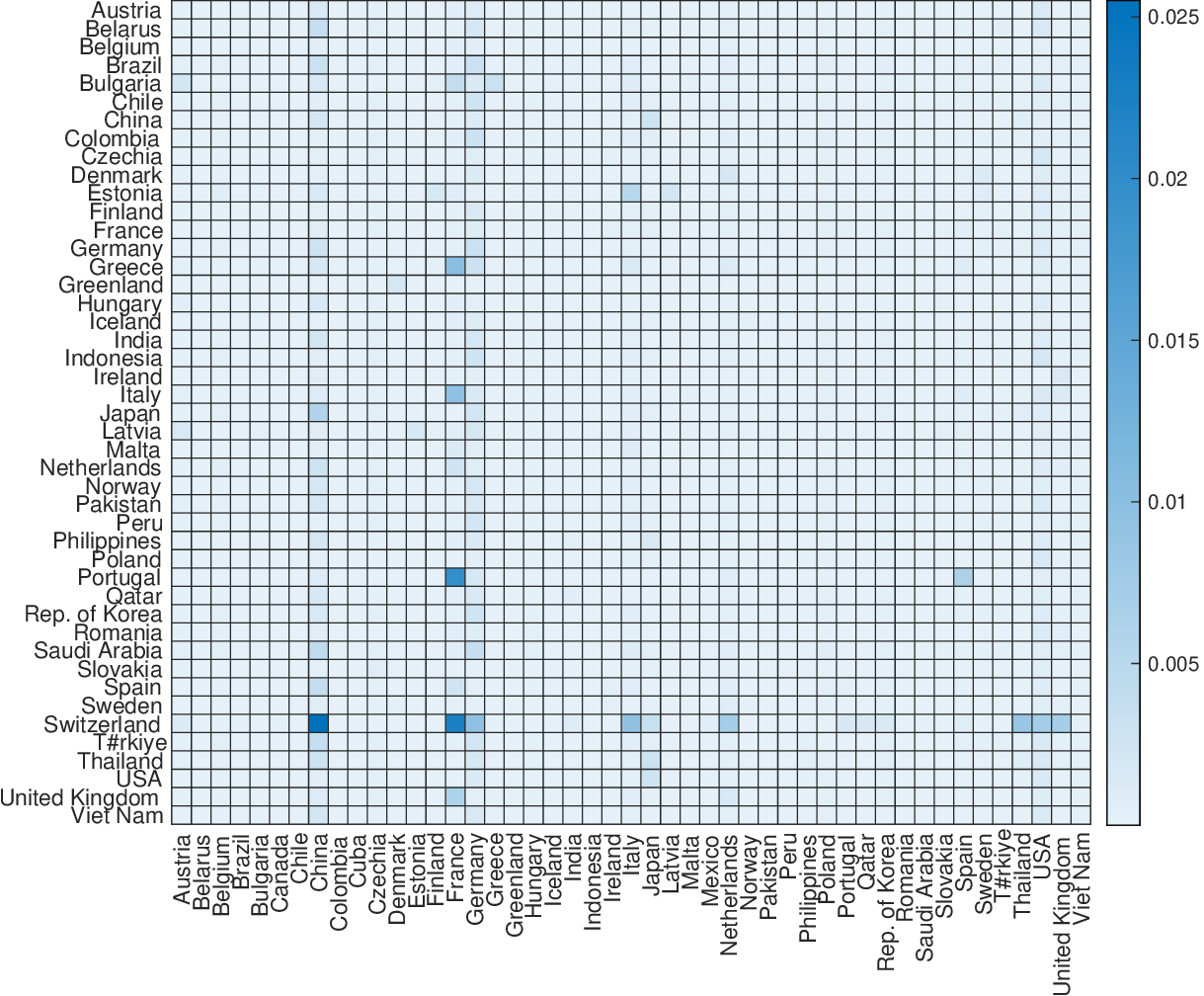}
		\caption{Clocks and watches and parts thereof}
		\label{fig:84}
	\end{subfigure}

\begin{subfigure}[b]{0.45\textwidth}
	\centering
	\includegraphics[width=\textwidth]{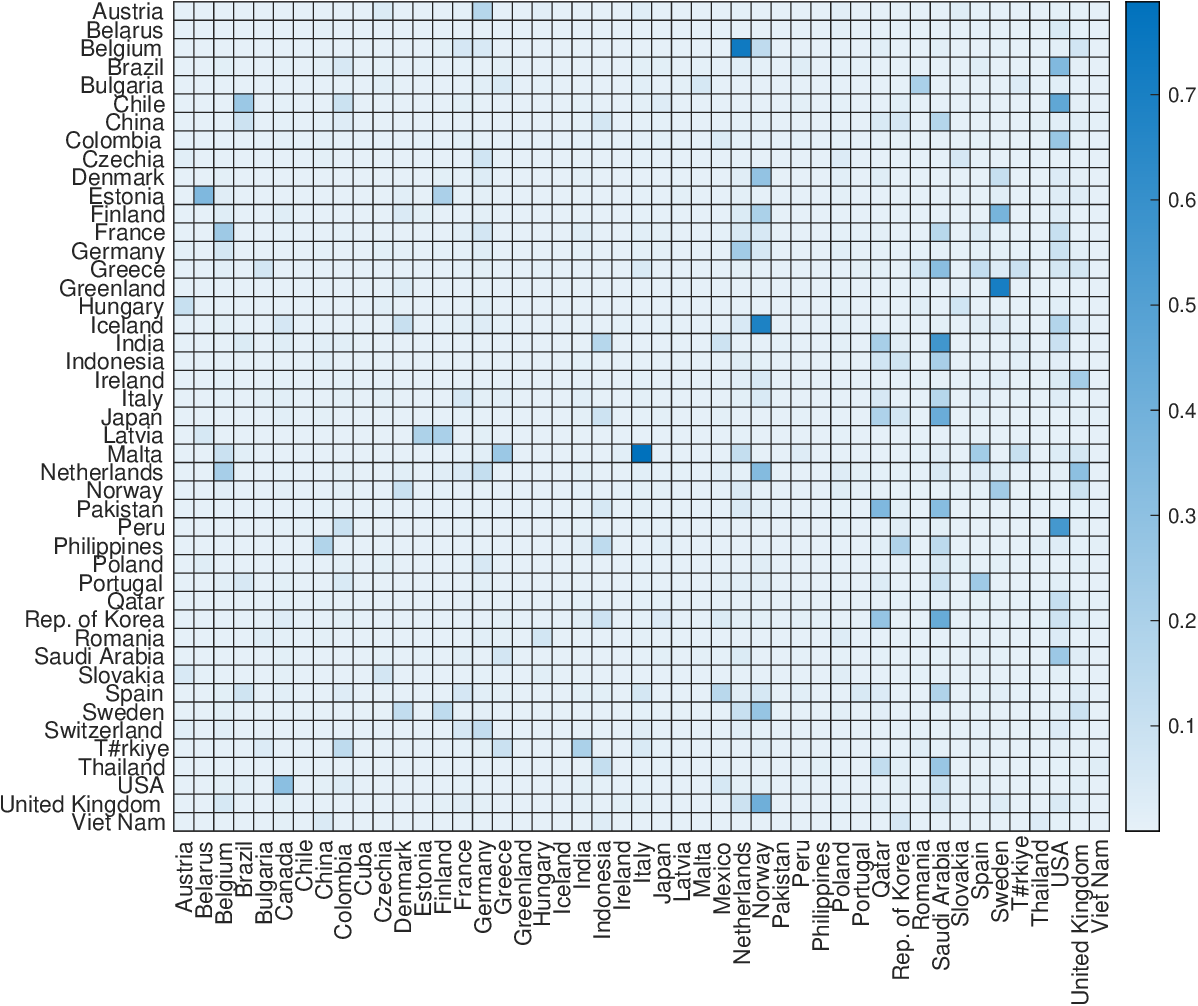}
	\caption{Mineral fuels, mineral oils and products of their distillation; bituminous substances; mineral waxes}
	\label{fig:85}
\end{subfigure}
\hfill
\begin{subfigure}[b]{0.45\textwidth}
	\centering
	\includegraphics[width=\textwidth]{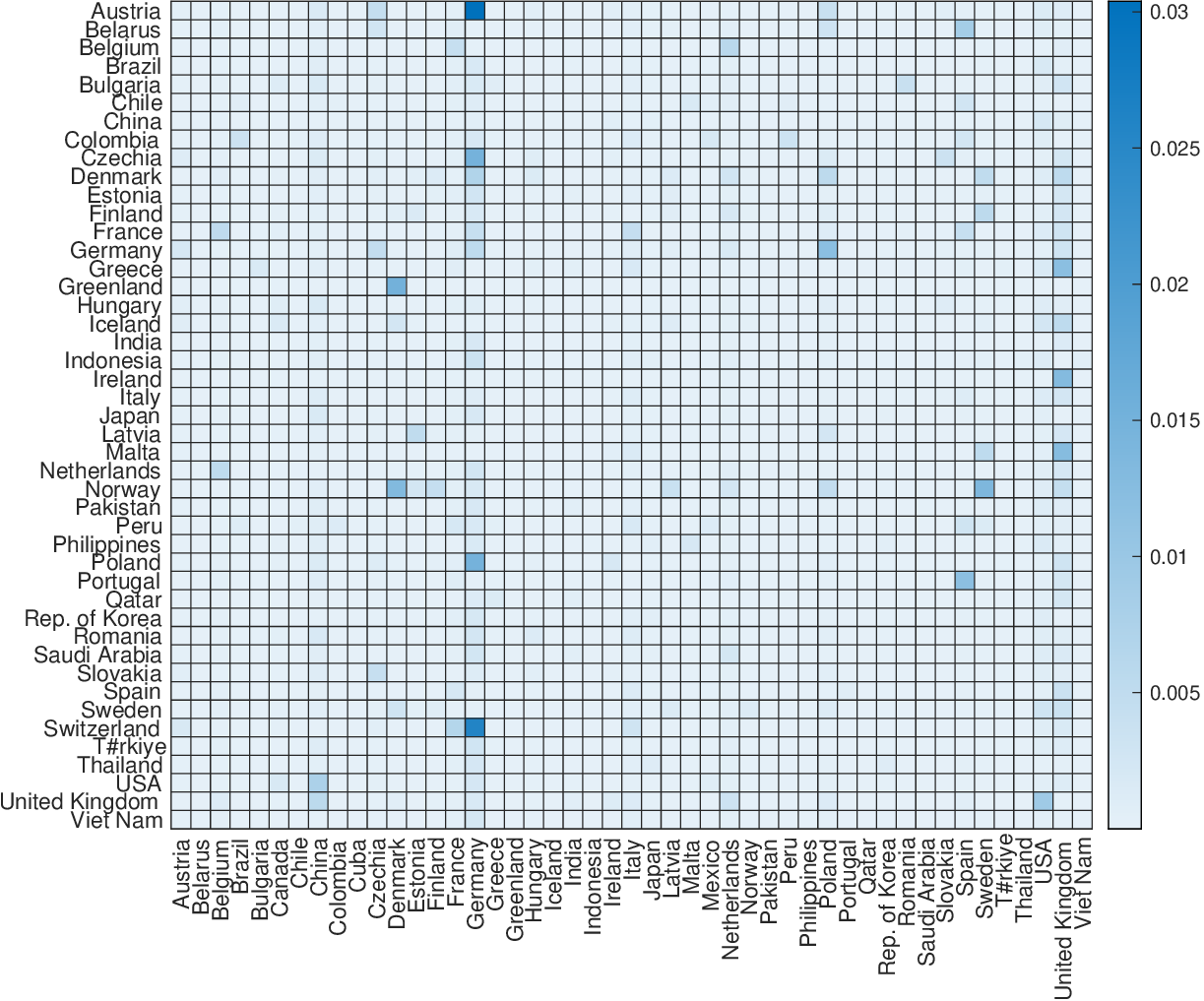}
	\caption{Printed books, newspapers, pictures and other products of the printing industry; manuscripts, typescripts and plans}
	\label{fig:86}
\end{subfigure}
	\caption{Absolute Residuals of Specific Slices.}
\end{figure}

\bibliographystyle{plainnat}
\bibliography{reference} 

\newpage

\appendix

\begin{center}
	\textbf{\Large{Appendix of ``Quantile and pseudo-Huber Tensor Decomposition"}}\\
	
\medskip	
	
	{Yinan Shen and Dong Xia}\\
	
\medskip

Department of Mathematics, Hong Kong University of Science and Technology	
	
\end{center}
	\section{Proofs under Heavy-Tailed Noise}
In this section, we are going to prove Lemma~\ref{lem:hub:regularity} and Theorem~\ref{thm:hub:dynamics}, where pseudo-Huber loss is taken. To simplify the writing, we introduce mask operaters $\calP_{\Omega_{j}^{(k)}}(\cdot)$,
\begin{align*}
	\left[\calP_{\Omega_{j}^{(k)}}(\bcalT)\right]_{i_1\dots i_m}:=\left\{
	\begin{array}{lcl}
		\left[\bcalT\right]_{i_1\cdots i_m}&     &\text{ if } i_k=j\\
		0&     &\text{ if } i_k\neq j
	\end{array}
	\right.,
\end{align*}
and $\calP_{\Omega_{-j}^{(k)}}(\bcalT):=\bcalT-\calP_{\Omega_{j}^{(k)}}(\bcalT)$.
Then $\lph{\fraM_{k}(\bcalT-\bcalY)_{j,\cdot}}-\lph{\fraM_{k}(\bcalT^*-\bcalY)_{j,\cdot}}$ has a simpler expression
\begin{align*}
	\lph{\fraM_{k}(\bcalT-\bcalY)_{j,\cdot}}-\lph{\fraM_{k}(\bcalT^*-\bcalY)_{j,\cdot}}&=\lph{\calP_{\Omega_{j}^{(k)}}\left(\bcalT-\bcalT^*-\bXi\right)}-\lph{\calP_{\Omega_{j}^{(k)}}\left(\bXi\right)}\\
	&=f\left(\calP_{\Omega_{j}^{(k)}}(\bcalT)\right)-f\left(\calP_{\Omega_{j}^{(k)}}(\bcalT^*)\right).
\end{align*}
\subsection{Proof of Lemma~\ref{lem:hub:regularity}}
\paragraph{Phase One Analyses}
We shall prove phase one properties under event $\bcalE_1$,
$$\bcalE_1:=\left\{\lone{\calP_{\Omega_{j}^{(k)}}(\bXi)}\leq 3\dkm\gamma,\quad\text{ for all } k=1,\dots,m,\; j=1,\dots,d_k\right\}.$$
Specifically, Lemma~\ref{teclem:Contraction of Heavy Tailed Random Variables:slice} proves $\PP(\bcalE_1)\geq 1-c\sum_{k=1}^{m} d_k(\dkm)^{-1-\min\{1,\eps\}}$. First consider Frobenius norm of the projected sub-gradient term. Notice that absolute values of entries in $\bcalG$ are not larger than $1$, which infers
\begin{align*}
	\fro{\calP_{\TT}(\bcalG)}^2&=\fro{\bcalG}^2-\fro{\calP_{\TT}^{\perp}(\bcalG)}^2\leq \fro{\bcalG}^2\leq d^*.
\end{align*}
It verifies $\fro{\calP_{\TT}(\bcalG)}\leq\sqrt{d^*}$. Then consider the function difference,
\begin{align*}
	f(\bcalT)-f(\bcalT^*)&=\sum_{i_1=1}^{d_1}\cdots \sum_{i_m=1}^{d_m}\sqrt{\left([\bcalT]_{i_1\cdots i_m}-[\bcalT^*]_{i_1\cdots i_m}-\xi_{i_1\cdots i_m}\right)^2+\delta^2}-\sum_{i_1=1}^{d_1}\cdots \sum_{i_m=1}^{d_m}\sqrt{\xi_{i_1\cdots i_m}^2+\delta^2}\\
	&\geq \lone{\bcalT-\bcalT^*}-2\lone{\bXi}-d^*\delta.
\end{align*}
which uses $\sqrt{(a-b)^2+\delta^2}\geq|a|-|b| $ and $\sqrt{b^2+\delta^2}\leq|b|+\delta $. On the other hand, event $\bcalE_1$ infers that $\lone{\bXi}\leq 3d^*\gamma$  and Lemma~\ref{teclem:norm-relation} shows $\lone{\cdot}\geq \linft{\cdot}^{-1}\fro{\cdot}^2$. Thus we have
\begin{align*}
	f(\bcalT)-f(\bcalT^*)\geq \linft{\bcalT-\bcalT^*}^{-1}\cdot\fro{\bcalT-\bcalT^*}^2-6d^*\gamma-d^*\delta.
\end{align*}
Next, consider slice of the projected sub-gradient. The matricization of $\calP_{\TT}(\bcalG)$ has the expression,
\begin{align*}
	&{~~~}\fraM_k(\calP_{\TT}(\bcalG))\\
	&=\fraM_k(\bcalG)\left(\otimes_{i\neq k}\U_i\right)\fraM_k(\bcalC)^{\dagger}\fraM_k(\bcalC)\left(\otimes_{i\neq k}\U_i\right)^{\top}+\U_k\U_k^{\top}\fraM_k(\bcalG)\left(\otimes_{i\neq k}\U_i\right)\left(\I-\fraM_k(\bcalC)^{\dagger}\fraM_k(\bcalC)\right)\left(\otimes_{i\neq k}\U_i\right)^{\top}\\
	&{~~~}+\sum_{i\neq k} \U_k\fraM_k(\bcalC\times_{j\neq i,k}\U_j\times\V_i),
\end{align*}
where $\V_i:=\left(\I_{d_i}-\U_i\U_i^{\top}\right)\fraM_k(\bcalG)\left(\otimes_{j\neq i}\U_j\right)\fraM_k(\bcalC)^{\dagger}$. Hence we have,
\begin{align*}
	\ltinf{\fraM_k(\calP_{\TT}(\bcalG))}^2&\leq 2\ltinf{\U_k}^2\fro{\bcalG}^2+\ltinf{\fraM_k(\bcalG)\left(\otimes_{i\neq k}\U_i\right)\fraM_k(\bcalC)^{\dagger}\fraM_k(\bcalC)\left(\otimes_{i\neq k}\U_i\right)^{\top}}^2\\
	&\leq 2\frac{\mu r_k}{d_k}\cdot d_1\cdots d_m+\dkm\leq3\frac{\mu r_k}{d_k}d^*.
\end{align*}
As for the slice function value difference, under event $\bcalE_1$, it has for each $k=1,\dots, m$, $j=1,\dots,
d_k$,
\begin{align*}
	\lph{\calP_{\Omega_{j}^{(k)}}(\bcalT-\bcalT^*-\bXi)}-\lph{\calP_{\Omega_{j}^{(k)}}(\bXi)}&\geq\lone{\calP_{\Omega_{j}^{(k)}}(\bcalT-\bcalT^*)}-2\lone{\calP_{\Omega_{j}^{(k)}}(\bXi)}-\dkm\delta\\
	&\geq \frac{1}{\linft{\calP_{\Omega_{j}^{(k)}}(\bcalT-\bcalT^*)}}\fro{\calP_{\Omega_{j}^{(k)}}(\bcalT-\bcalT^*)}^2-6d_k^- \gamma-\dkm\delta.
\end{align*}
Hence, we finish phase one analyses.
\paragraph{Phase Two Analysis} In phase two analyses, we shall assume the event $$\bcalE_2:=\left\{\sup_{\bcalT\in\RR^{d_1\times\cdots\times d_m}, \Delta\bcalT\in\MM_{2\r}}\left|f(\bcalT+\Delta\bcalT)-f(\bcalT)-\EE\left(f(\bcalT+\Delta\bcalT)-f(\bcalT) \right)\right|\cdot\fro{\Delta\bcalT}^{-1}\leq C\sqrt{\textsf{DoF}_m}\right\}$$
holds. Specifically, 
Lemma~\ref{teclem:empirical} proves $\PP(\bcalE_2)\geq1-\exp(-\textsf{DoF}/2)$. By event $\bcalE_2$ and loss function expectation Lemma~\ref{teclem:pseudo:function expectation}, when $\linft{\bcalT-\bcalT^*}\leq C_{m,\muT,r^*}(6\gamma+\delta)$ and $\fro{\bcalT-\bcalT^*}\geq cb_0\sqrt{\textsf{DoF}_m}$, we have, 
\begin{align*}
	f(\bcalT)-f(\bcalT^*)&\geq \EE[f(\bcalT)-f(\bcalT^*)]-C\sqrt{\textsf{DoF}_m}\fro{\bcalT-\bcalT^*}\\
	&\geq \frac{1}{3b_0}\fro{\bcalT-\bcalT^*}^2-C\sqrt{\textsf{DoF}_m}\fro{\bcalT-\bcalT^*}\\
	&\geq \frac{1}{4b_0}\fro{\bcalT-\bcalT^*}^2,
\end{align*}
where the last inequality is due to $\fro{\bcalT-\bcalT^*}\geq C_1 \sqrt{\textsf{DoF}_m}\cdot b_0$. The following lemma analyzes Frobenius norm for projected sub-gradient and completes proving Lemma~\ref{lem:hub:regularity}.
\begin{lemma}[Upper bound for sub-gradient]
	Let $\bcalT$ be Tucker rank at most $\r$ tensor. Suppose it satisfies $\fro{\bcalT-\bcalT^*}\geq\sqrt{\textsf{DoF}_m}$$\cdot b_0$. Let $\bcalG\in\partial f(\bcalT)$ be the sub-gradient and $\TT$ be the tangent space of $\MM_{\r}$ at point $\bcalT$. Then under $\bcalE_2$, we have $$\fro{\calP_{\TT}(\bcalG)}\leq c_1\cdot\sqrt{m+1}\cdot \delta^{-1}\fro{\bcalT-\bcalT^*}.$$
\end{lemma}
\begin{proof}
	Note that $\fro{\calP_{\TT}(\bcalG)}$ has the upper bound
	\begin{align*}
		\fro{\calP_{\TT}(\bcalG)}^2&=\fro{\bcalG\times_1\U_1\U_1^{\top}\times_2\cdots\times_m\U_m\U_m^{\top}}^2\\
		&{~~~~~}+\sum_{k=1}^{m}\fro{\left(\I_{d_k}-\U_k\U_k^{\top}\right)\fraM_k(\bcalG)\left(\otimes_{i\neq k}\U_i\right)\fraM_k(\bcalC^*)^{\dag} \fraM_k(\bcalC^*)\left(\otimes_{i\neq k}\U_i\right)^{\top} }^2\\
		&\leq \underbrace{\frorr{\bcalG}^2}_{=A_1}+\underbrace{\sum_{k=1}^{m}\fro{\fraM_k(\bcalG)\left(\otimes_{i\neq k}\U_i\right)\fraM_k(\bcalC)^{\dag} \fraM_k(\bcalC)\left(\otimes_{i\neq k}\U_i\right)^{\top} }^2}_{A_2},
	\end{align*}
	where $\frorr{\bcalG}:=\sup_{\W_j\in\OO_{d_j,r_j}}\fro{\bcalG\times_1\W_1\W_1^{\top}\times_2\cdots\times_m\W_m\W_m^{\top} }$.
	
	\textbf{First consider $A_1$.}
	Suppose $\bcalG$ achieves $\frorr{\cdot}$ with orthonormal matrices $\V_k\in\OO_{d_k,r_k}$, namely, $$\frorr{\bcalG}=\fro{\bcalG\times_1\V_1\V_1^{\top}\times_2\cdots\times_m\V_m\V_m^{\top}},$$ and then take $\bcalS=\bcalT+\frac{1}{2}\delta\cdot\bcalG\times_1\V_1\V_1^{\top}\times_2\cdots\times_m\V_m\V_m^{\top}$. Then we have $\rank(\bcalS-\bcalT)\leq\r$. Hence by definition of sub-gradient and by event $\bcalE_2$, we have
	\begin{align}
		\inp{\bcalS-\bcalT}{\bcalG}\leq f(\bcalS)-f(\bcalT)\leq\EE f(\bcalS)-\EE f(\bcalT)+C\sqrt{\textsf{DoF}_m}\fro{\bcalS-\bcalT}.
		\label{eq3}
	\end{align}
	With Lemma~\ref{teclem:pseudo:function expectation} we have $$ \EE f(\bcalS)-\EE f(\bcalT)\leq \frac{1}{2\delta}\fro{\bcalS-\bcalT}^2+\frac{1}{\delta}\fro{\bcalS-\bcalT}\fro{\bcalT-\bcalT^*}=\frac{\delta}{8}\frorr{\bcalG}^2+\frac{1}{2}\frorr{\G}\fro{\bcalT-\bcalT^*}.$$ Note that insert $\bcalS=\bcalT+\frac{1}{2}\delta\cdot\bcalG\times_1\V_1\V_1^{\top}\times_2\cdots\times_m\V_m\V_m^{\top}$ into Equation~\eqref{eq3} and with $b_0\geq \delta$, $\fro{\bcalT-\bcalT^*}\geq \delta\cdot\sqrt{\textsf{DoF}_m}$, we have 
	\begin{align*}
		\frac{1}{2}\delta \frorr{\bcalG}^2\leq \frac{1}{8}\delta\frorr{\bcalG}^2+C \fro{\bcalT-\bcalT^*}\frorr{\bcalG},
	\end{align*}
	By solving the sbove quadratic inequality, we get \begin{align*}
		\frorr{\bcalG}\leq c_1 \delta^{-1}\cdot\fro{\bcalT-\bcalT^*}.
	\end{align*}
	
	\textbf{Second consider $A_2$.} Note that $\fraM_k(\bcalG)\left(\otimes_{i\neq k}\U_i\right)\fraM_k(\bcalC^*)^{\dag} \fraM_k(\bcalC^*)\left(\otimes_{i\neq k}\U_i\right)^{\top}$ is the $k$-th matricization of some Tucker rank at most $\r$ tensor. Then by same analysis as $A_1$, we have \begin{align*}
		\fro{\fraM_k(\bcalG)\left(\otimes_{i\neq k}\U_i\right)\fraM_k(\bcalC^*)^{\dag} \fraM_k(\bcalC^*)\left(\otimes_{i\neq k}\U_i\right)^{\top}}\leq c_1\cdot \delta^{-1}\cdot\fro{\bcalT-\bcalT^*}.
	\end{align*}
	Finally, we have $\fro{\calP_{\TT}(\bcalG)}^2\leq (m+1)c_1^2\delta^{-2}\fro{\bcalT-\bcalT^*}^2$, which leads to
	\begin{align*}
		\fro{\calP_{\TT}(\bcalG)}\leq c_1\cdot \sqrt{m+1}\cdot \delta^{-1}\fro{\bcalT-\bcalT^*}.
	\end{align*}
\end{proof}
\subsection{Proof of Theorem~\ref{thm:hub:dynamics}}
\subsubsection{Leave-one-out Sequence}
Entrywise normed error in phase one could be obtained directly. However, in phase two, in order to have delicate bound of entrywise normed error, we turn to the powerful leave-one-out framework \citep{chen2021spectral}. Introduce two sets of the auxiliary loss function $\check{f}_{j}^{(k)}$ and $\hat{f}_j^{(k)}$, for each $k=1,\dots,m$ and $j=1,\dots,d_k$,
\begin{align*}
	\check{f}_j^{(k)}(\bcalT):=\lph{\calP_{\Omega_{-j}^{(k)}}\left(\bcalT-\bcalT^*-\bXi\right)}+\lph{\calP_{\Omega_{j}^{(k)}}\left(\bcalT-\bcalT^*\right)},
\end{align*}
and
\begin{align}
	\hat{f}_{j}^{(k)}(\bcalT):=\lph{\calP_{\Omega_{-j}^{(k)}}\left(\bcalT-\bcalT^*-\bXi\right)}+\EE\lph{\calP_{\Omega_{j}^{(k)}}\left(\bcalT-\bcalT^*-\bXi\right)}.
\end{align}
Both $\check{f}_{j}^{(k)}$ and $\hat{f}_j^{(k)}$ are free of noise randomness for the $j$-th slice by order $k$ and we define the leave-one-out sequence $\{\bcalT_l^{(k),j}\}$ accordingly, see Algorithm~\ref{alg:leave:RsGrad}. Here, $\check{f}_{j}^{(k)}$ is used in phase one while in phase two the leave-one-out sequence is based on $\hat{f}_{j}^{(k)}$, see Algorithm~\ref{alg:leave:RsGrad}.
\begin{algorithm}
	\caption{Leave-one-out Sequence}\label{alg:leave:RsGrad}
	\begin{algorithmic}
		\STATE{\textbf{Input}: Same $\bcalY$, $l_{\max}$, $\eta_{l}$ as Algorithm~\ref{alg:pseuHuber:RsGrad}}
		\STATE{Initialization: $\bcalT_0^{(k),j}\in\MM_\r$}
		\FOR{$l = 0,\ldots,l_{\max}$}
		\STATE{Choose a vanilla subgradient:  $\G_l^{(k),j}\in\left\{
			\begin{array}{lcl}
				\partial \check{f}_j^{(k)}(\bcalT_l^{(k),j})&     &\text{ if in phase one}\\
				\partial \hat{f}_j^{(k)}(\bcalT_l^{(k),j})&     &\text{ if in phase two}
			\end{array}
			\right.$}
		\STATE{Compute Riemannian sub-gradient: $\wt\G_l^{(k),j} = \calP_{\TT_l^{(k),j}}(\G_l^{(k),j})$}
		\STATE{Retraction to $\MM_\r$: $\bcalT_{l+1}^{(k),j} = \text{HOSVD}_\r(\bcalT_l^{(k),j} - \eta_{l}\wt\G_l^{(k),j})$}
		\ENDFOR
	\end{algorithmic}
\end{algorithm}

Even though, in phase one, we don't need the leave-one-out sequence to obtain sharp entrwise norm, in order to have a sequence not related with slice noise in the second phase, we need such a sequence in the first phase. Besides, notice that here for Pseudo-Huber loss, we have two different methods in removing the slice randomness, ignoring the noise or taking expectation and these two methods are equivalent in $\ell_2$ loss \cite{chen2021bridging,chen2021spectral}. Due to phase one and phase two have different analysis framework, the proper type of leave-one-out sequence is taken accordingly.
\subsubsection{Phase One}
\label{proof:pseudo-Huber:phaseone}
For convenience, denote $D_l:=\left(1-\frac{1}{32}(5m+1)^{-2}(3^m\mu^{*m} r^*)^{-1}\right)^{l}\cdot D_0$. We shall prove the following Euqation~\eqref{eq61}-\eqref{eq65} and \eqref{eq66}-\eqref{eq610} by induction. It's obvious that it holds for the initialization $\bcalT_0$. Suppose it holds for iteration $l$ and we consider the $(l+1)$-th iteration. As for the original sequence, we are going to prove 
\begin{subequations}
	\begin{align}
		\fro{\bcalT_{l+1}-\bcalT^*}&\leq D_{l+1},
		\label{eq61}\\
		\ltinf{\bcalT_{l+1}-\bcalT^*}&\leq 3\sqrt{\frac{\muT r_k}{d_k}}\cdot D_{l+1},
		\label{eq62}\\
		\ltinf{\left(\U_k^{(l+1)}\H_k^{(l+1)}-\U_k^*\right)\fraM_{k}(\bcalC^*)}&\leq5\sqrt{\frac{\muT r_k}{d_k}}\cdot D_{l+1},
		\label{eq63}\\
		\linft{\bcalT_{l+1}-\bcalT^*}&\leq (5m+1)\sqrt{\frac{\mu^{*m} r^*}{d^*}}\cdot D_{l+1},
		\label{eq64}\\
		\ltinf{\U_k^{(l)}}&\leq \sqrt{\frac{3\muT r_k}{d_k}},\label{eq65}
	\end{align}
	\label{eq6}
\end{subequations}
where $\bcalT_{l+1}=\bcalC_{l+1}\cdot\llbracket\U_1^{(l+1)},\dots,\U_m^{(l+1)}\rrbracket$ is the Tucker decomposition and $\H_k^{(l+1)}:=\U_k^{(l+1)\top}\U_k^*$. As for the leave-one-out sequence, we are going to prove
\begin{subequations}
	\begin{align}
		\fro{\bcalT_{l+1}^{(k),j}-\bcalT^*}&\leq D_{l+1}
		\label{eq66}\\
		\ltinf{\bcalT_{l+1}^{(k),j}-\bcalT^*}&\leq 3\sqrt{\frac{\muT r_k}{d_k}}\cdot D_{l+1}
		\label{eq67}\\
		\ltinf{\left(\U_k^{(l+1),(k),j}\H_k^{(l+1),(k),j}-\U_k^*\right)\fraM_{k}(\bcalC^*)}&\leq5\sqrt{\frac{\muT r_k}{d_k}}\cdot D_{l+1}
		\label{eq68}\\
		\linft{\bcalT_{l+1}^{(k),j}-\bcalT^*}&\leq (5m+1)\sqrt{\frac{\mu^{*m} r^*}{d^*}}\cdot D_{l+1},
		\label{eq69}\\
		\ltinf{\U_k^{(l+1),(k),j}}&\leq\sqrt{\frac{3\muT r_k}{d_k}},\label{eq610}
	\end{align}
\end{subequations}
where $\bcalT_{l+1}^{(k),j}=\bcalC_{l+1}^{(k),j}\cdot\llbracket\U_1^{(l+1),(k),j},\dots,\U_m^{(l+1),(k),j}\rrbracket$ is the Tucker decomposition and $\H_k^{(l+1),(k),j}:=\left(\U_k^{(l+1),(k),j}\right)^{\top}\U_k^* $. Notice that phase one regularity conditions Lemma~\ref{lem:hub:regularity} also holds for the leave-one-out sequences $\{\bcalT_{l+1}^{(k),j}\}$ under event $\bcalE_1$ in parallel and its convergence analyses are same as the original sequence. Hence we shall only show detailed proof of original sequence and skip the leave-one-out analysis in the first phase.
\paragraph{Frobenius norm}
First consider $\fro{\bcalT_l-\bcalT-\eta_l\calP_{\TT_l}(\bcalG_l)}$,
\begin{align*}
	\fro{\bcalT_l-\eta_l\calP_{\TT_l} (\bcalG_l)-\bcalT^*}^2=\fro{\bcalT_l-\bcalT^*}^2-2\eta_l\inp{\bcalT_l-\bcalT^*}{\calP_{\TT_l}(\bcalG_l)} + \eta_l^2\fro{\calP_{\TT_l}(\bcalG_l)}^2.
\end{align*}
We have analyzed the last term $\fro{\calP_{\TT_l}(\bcalG_l)}^2$ in Lemma~\ref{lem:hub:regularity} that $\fro{\calP_{\TT_l}(\bcalG_l)}^2\leq d^*$.
Note that by definition of sub-gradient and analyses of $f(\bcalT)-f(\bcalT^*)$ in Lemma~\ref{lem:hub:regularity}, the intermediate term has the following lower bound
\begin{align*}
	\inp{\bcalT_l-\bcalT^*}{\calP_{\TT_l}(\bcalG_l)}&=\inp{\bcalT_l-\bcalT^*}{\bcalG_l}-\inp{\calP_{\TT_l}^{\perp}(\bcalT_l-\bcalT^*)}{\bcalG_l}\\
	&\geq f(\bcalT_l)-f(\bcalT^*)-\inp{\calP_{\TT_l}^{\perp}\bcalT^*}{\bcalG_l}\\
	&\geq  \linft{\bcalT_l-\bcalT^*}^{-1}\fro{\bcalT_l-\bcalT^*}^2-6d^*\gamma-d^*\delta-\inp{\calP_{\TT_l}^{\perp}\bcalT^*}{\bcalG_l}.
\end{align*}
Besides, Lemma~\ref{teclem:rieman-orthogonal-project} shows that $\left|\inp{\calP_{\TT_l}^{\perp}\bcalT^*}{\bcalG_l}\right|\leq 8m^2\mins^{*-1}\fro{\bcalT_l-\bcalT^*}^2\cdot \fro{\bcalG_l}$ and absolute values of $\bcalG_l$ entries are bounded by $1$, which implies $\fro{\bcalG_l}\leq\sqrt{d^*}$. Thus, we have 
\begin{align*}
	\fro{\bcalT_l-\eta_l\calP_{\TT_l} (\bcalG_l)-\bcalT^*}^2&\leq \fro{\bcalT_l-\bcalT^*}^2-2\eta_l \linft{\bcalT_l-\bcalT^*}^{-1}\cdot\fro{\bcalT_l-\bcalT^*}^2+12\eta_{l}d^*\gamma+2\eta_{l}d^*\delta\\
	&{~~~~~~~~~~~~~~~~~~~~~~~~}+16\eta_{l}m^2\mins^{*-1}\fro{\bcalT_l-\bcalT^*}^2\cdot \sqrt{d^*}+\eta_{l}^2 d^*.
\end{align*}
Then insert $\fro{\bcalT_l-\bcalT^*}\leq D_l$ and $\linft{\bcalT_l-\bcalT^*}\leq (5m+1)\cdot\sqrt{\frac{3^m\mu^{*m} r^*}{d^*}}\cdot D_l $ into the above equation,
\begin{align*}
	&{~~~}\fro{\bcalT_l-\eta_l\calP_{\TT_l} (\bcalG_l)-\bcalT^*}^2\\
	&\leq\left(1-2\eta_{l}\linft{\bcalT_l-\bcalT^*}^{-1}\right)\fro{\bcalT_l-\bcalT^*}^2+12\eta_{l}d^*\gamma+2\eta_{l}d^*\delta+16\eta_{l}m^2\mins^{*-1}\fro{\bcalT_l-\bcalT^*}^2\cdot \sqrt{d^*}+\eta_{l}^2 d^*\\
	&\leq \left(1-2\eta_{l}\linft{\bcalT_l-\bcalT^*}^{-1}\right)D_l^2+12\eta_{l}d^*\gamma+2\eta_{l}d^*\delta+16\eta_{l}m^2\mins^{*-1}D_l^2\cdot \sqrt{d^*}+\eta_{l}^2 d^*\\
	&\leq D_l^2-2\eta_l (5m+1)^{-1}\sqrt{\frac{d^*}{3^m\mu^{*m} r^*}}D_l+12\eta_{l}d^*\gamma+2\eta_{l}d^*\delta+16\eta_{l}m^2\mins^{*-1}D_l^2\cdot \sqrt{d^*}+\eta_{l}^2 d^*,
\end{align*}
where the second inequality also uses $1-2\eta_{l}\linft{\bcalT_l-\bcalT^*}^{-1}> 0$. Then with phase one region constraint and initialization condition $D_l\leq D_0\leq c_m\mins^*$, we have
\begin{align*}
	\fro{\bcalT_l-\eta_l\calP_{\TT_l} (\bcalG_l)-\bcalT^*}^2&\leq D_l^2-\eta_l (5m+1)^{-1}\sqrt{\frac{d^*}{3^m\mu^{*m} r^*}}D_l+\eta_{l}^2 d^*.
\end{align*}
Note that the stepsize $\eta_l \in\frac{1}{8(5m+1)\sqrt{3^m\mu^mr^*d^*}}\cdot D_l\cdot \left[1,3\right]$ and we could have 
$$\fro{\bcalT_l-\eta_l\calP_{\TT_l} (\bcalG_l)-\bcalT^*}^2\leq \left( 1-\frac{3}{64}(5m+1)^{-2}(3^m\mu^{*m} r^*)^{-1}\right)D_l^2. $$
Recall that $\bcalT_{l+1}=\textrm{HOSVD}(\bcalT_l-\eta_{l}\calP_{\TT_l}(\bcalG_l))$ and by Theorem~\ref{teclem:perturbation:tensor}, we have \begin{align*}
	\fro{\bcalT_{l+1}-\bcalT^*}\leq \left( 1-\frac{1}{64}(5m+1)^{-2}(3^m\mu^{*m} r^*)^{-1}\right)D_l=D_{l+1},
\end{align*}
where initialization condition $D_{l}\leq D_0\leq c\mins^*\cdot (5m+1)^{-2}(3^m\muT r^*)^{-1}$ is used.
\paragraph{Entrywise norm}
Consider $\ltinf{\fraM_k\left(\bcalT_l-\bcalT^*-\eta_{l}\calP_{\TT_l}(\bcalG_l)\right)}$, for $k=1,\dots,m$ or equivalently, consider $$\fro{\calP_{\Omega_{j}^{(k)}}\left(\bcalT_l-\bcalT^*-\eta_{l}\calP_{\TT_l}(\bcalG_l)\right) },\quad \text{ for each }j=1,\dots,d_k.$$
Note that
\begin{align*}
	\fro{\calP_{\Omega_{j}^{(k)}}\left(\bcalT_l-\bcalT^*-\eta_{l}\calP_{\TT_l}(\bcalG_l)\right) }^2&= \fro{\calP_{\Omega_{j}^{(k)}}\left(\bcalT_l-\bcalT^*\right)}^2\\
	&{~~~~~~~~~~}-2\eta_{l}\inp{\calP_{\Omega_{j}^{(k)}}(\bcalT_l-\bcalT^*)}{\calP_{\Omega_{j}^{(k)}}(\calP_{\TT_l}(\bcalG_l))}+\eta_{l}^2\fro{\calP_{\Omega_{j}^{(k)}}(\calP_{\TT_l}(\bcalG_l))}^2.
\end{align*}
Insert induction $\ltinf{\U_k^{(l)}}\leq\sqrt{\frac{3\muT r_k}{d_k}}$ into Lemma~\ref{lem:hub:regularity} and it provides an upper bound for the last term $$\ltwo{\calP_{\Omega_{j}^{(k)}}(\calP_{\TT_l}(\bcalG_l))}^2\leq9\frac{\muT r_k}{d_k}d^*.$$ Then consider the intermidiate term $\inp{\calP_{\Omega_{j}^{(k)}}(\bcalT_l-\bcalT^*)}{\calP_{\Omega_{j}^{(k)}}(\calP_{\TT_l}(\bcalG_l))}= \inp{\calP_{\Omega_{j}^{(k)}}(\bcalT_l)}{\calP_{\Omega_{j}^{(k)}}(\calP_{\TT_l}(\bcalG_l))}-\inp{\calP_{\Omega_{j}^{(k)}}(\bcalT^*)}{\calP_{\Omega_{j}^{(k)}}(\calP_{\TT_l}(\bcalG_l))}$. Note that with simple calculations, we obtain
$$\inp{\calP_{\Omega_{j}^{(k)}}(\bcalT_l)}{\calP_{\Omega_{j}^{(k)}}(\calP_{\TT_l}(\bcalG_l))}=\inp{\calP_{\Omega_{j}^{(k)}}(\bcalT_l)}{\calP_{\Omega_{j}^{(k)}}(\bcalG_l)},$$
and
\begin{equation}
	\label{eq8}
	\begin{split}
		\inp{\calP_{\Omega_{j}^{(k)}}(\bcalT^*)}{\calP_{\TT_l}(\bcalG_l)}&=\inp{\calP_{\TT_l}\calP_{\Omega_{j}^{(k)}}(\bcalT^*)}{\bcalG_l}\\
		&=\inp{\calP_{\TT_l}\calP_{\Omega_{j}^{(k)}}\calP_{\TT_l}(\bcalT^*)}{\bcalG_l}+\inp{\calP_{\TT_l}\calP_{\Omega_{j}^{(k)}}\calP_{\TT_l}^{\perp}(\bcalT^*)}{\bcalG_l}\\
		&=\inp{\calP_{\Omega_{j}^{(k)}}\calP_{\TT_l}(\bcalT^*)}{\bcalG_l}+\inp{\calP_{\TT_l}\calP_{\Omega_{j}^{(k)}}\calP_{\TT_l}^{\perp}(\bcalT^*)}{\bcalG_l}\\
		&=\inp{\calP_{\Omega_{j}^{(k)}}(\bcalT^*)}{\bcalG_l}-\inp{\calP_{\Omega_{j}^{(k)}}\calP_{\TT_l}^{\perp}(\bcalT^*)}{\bcalG_l}+\inp{\calP_{\Omega_{j}^{(k)}}\calP_{\TT_l}^{\perp}(\bcalT^*)}{\calP_{\TT_l}(\bcalG_l)}\\
		&=\inp{\calP_{\Omega_{j}^{(k)}}(\bcalT^*)}{\calP_{\Omega_{j}^{(k)}}(\bcalG_l)}-\inp{\calP_{\Omega_{j}^{(k)}}\calP_{\TT_l}^{\perp}(\bcalT^*)}{\calP_{\Omega_{j}^{(k)}}\bcalG_l}+\inp{\calP_{\Omega_{j}^{(k)}}\calP_{\TT_l}^{\perp}(\bcalT^*)}{\calP_{\Omega_{j}^{(k)}}\calP_{\TT_l}(\bcalG_l)},
	\end{split}
\end{equation}
where $\calP_{\TT_l}\calP_{\Omega_{j}^{(k)}}\calP_{\TT_l}=\calP_{\Omega_{j}^{(k)}}\calP_{\TT_l}$ is used. With Lemma~\ref{teclem:rieman-orthogonal-project}, we have
\begin{align*}
	&{~~~~}\bigg| \inp{\calP_{\Omega_{j}^{(k)}}\calP_{\TT_l}^{\perp}(\bcalT^*)}{\calP_{\Omega_{j}^{(k)}}(\bcalG_l)}\bigg|\\
	&\leq\fro{\calP_{\Omega_{j}^{(k)}}\calP_{\TT_l}^{\perp}(\bcalT^*)}\fro{\calP_{\Omega_{j}^{(k)}}(\bcalG_l)}\\
	&\leq \sqrt{\dkm}\fro{\bcalT_l-\bcalT^*}\left(m^2\ltinf{\U_k^{(l)}}\frac{\fro{\bcalT_l-\bcalT^*}}{\mins^*}+m\ltinf{\U_k^{(l)}\U_k^{(l)\top}-\U_k^*\U_k^{*\top}}\right)=:B_1,
\end{align*}
and
\begin{align*}
	&{~~~~}\bigg|\inp{\calP_{\Omega_{j}^{(k)}}\calP_{\TT_l}^{\perp}(\bcalT^*)}{\calP_{\Omega_{j}^{(k)}}\calP_{\TT_l}(\bcalG_l)} \bigg|\\
	&\leq \fro{\calP_{\Omega_{j}^{(k)}}\calP_{\TT_l}^{\perp}(\bcalT^*)}\fro{\calP_{\Omega_{j}^{(k)}}\calP_{\TT_l}(\bcalG_l)}\\
	&\leq 3\sqrt{\frac{\muT r_k}{d_k}\cdot d^*}\fro{\bcalT_l-\bcalT^*}\left(m^2\ltinf{\U_k^{(l)}}\frac{\fro{\bcalT_l-\bcalT^*}}{\mins^*}+m\ltinf{\U_k^{(l)}\U_k^{(l)\top}-\U_k^*\U_k^{*\top}}\right)=:B_2.
\end{align*}
Note that by induction $\ltinf{\left(\U_k^{(l)}\H_k^{(l)}-\U_k^*\right)\fraM_k(\bcalC^*) }\leq 5\sqrt{\frac{\muT r_k}{d_k}}\cdot D_{l}$, we have $\ltinf{\U_k^{(l)}\H_k^{(l)}-\U_k^* }\leq 5\sqrt{\frac{\muT r_k}{d_k}}\cdot\mins^{*-1}D_{l}$. Lemma~\ref{teclem:ltinf-transformation} and Lemma~\ref{teclem:perturbation:matrix} infer that $$\ltinf{\U_k^{(l)}\U_k^{(l)\top}-\U_k^*\U_k^{*\top} }\leq 8\mins^{*-1}\sqrt{\frac{\muT r_k}{d_k}}D_l.$$ In this way, we have
\begin{align*}
	B_1\vee B_2\leq 16m^2\mins^{*-1}\sqrt{d^*}\frac{\muT r_k}{d_k}D_l^2.
\end{align*}
Also, by definition of sub-gradient and by regularity properties in Lemma~\ref{lem:hub:regularity}, we have \begin{align*}
	\inp{\calP_{\Omega_{j}^{(k)}}(\bcalT_l-\bcalT^*)}{\calP_{\Omega_{j}^{(k)}}(\bcalG_l)}&\geq\lph{\calP_{\Omega_{j}^{(k)}}(\bcalT-\bcalT^*-\bXi)}-\lph{\calP_{\Omega_{j}^{(k)}}(\bXi)}\\
	&\geq\linft{\calP_{\Omega_{j}^{(k)}}(\bcalT_l-\bcalT^*)}^{-1}\cdot\fro{\calP_{\Omega_{j}^{(k)}}(\bcalT_l-\bcalT^*)}^2- 6\dkm\gamma-\dkm\delta.
\end{align*}
Thus, the intermediate term has the lower bound
\begin{align*}
	&\inp{\calP_{\Omega_{j}^{(k)}}(\bcalT_l-\bcalT^*)}{\calP_{\Omega_{j}^{(k)}}(\calP_{\TT_l}(\bcalG_l))} \geq \linft{\calP_{\Omega_{j}^{(k)}}(\bcalT_l-\bcalT^*)}^{-1}\cdot\fro{\calP_{\Omega_{j}^{(k)}}(\bcalT_l-\bcalT^*)}^2- 6\dkm\gamma- (B_1+B_2).
\end{align*}
Hence combine the above euqations and then we have upper bound for the slice \begin{align*}
	&{~~~~}\fro{\calP_{\Omega_{j}^{(k)}}\left(\bcalT_l-\bcalT^*-\eta_{l}\calP_{\TT_l}(\bcalG_l)\right) }^2\\
	&\leq \fro{\calP_{\Omega_{j}^{(k)}}\left(\bcalT_l-\bcalT^*\right)}^2-2\eta_l \linft{\calP_{\Omega_{j}^{(k)}}(\bcalT_l-\bcalT^*)}^{-1}\cdot\fro{\calP_{\Omega_{j}^{(k)}}(\bcalT_l-\bcalT^*)}^2 +12\eta_{l} \dkm \gamma+2\eta_{l}(B_1+B_2)+9\eta_l^2 \frac{\muT r_k}{d_k}d^*\\
	&\leq 9\frac{\muT r_k}{d_k}D_l^2-18\eta_{l}\frac{\muT r_k}{d_k}(5m+1)^{-1}(3^m\mu^{*m} r^* d^*)^{-1/2}D_l+ 12\eta_{l} \dkm \gamma+2\eta_{l}(B_1+B_2)+9\eta_l^2 \frac{\muT r_k}{d_k}d^*\\
	&\leq 9\frac{\mu r_k}{d_k}\cdot\left(1- \frac{3}{64}(5m+1)^{-2}(3^m\mu^m r^*)^{-1}\right)D_{l}^2,
\end{align*}
where the second inequality uses induction of $\fro{\calP_{\Omega_{j}^{(k)}}(\bcalT_{l}-\bcalT^*)}$ and last line uses phase one region constraint and step size selection, similar to Frobenius norm analyses. The above equation infers that $$\fro{\calP_{\Omega_{j}^{(k)}}\left(\bcalT_l-\bcalT^*-\eta_{l}\calP_{\TT_l}(\bcalG_l)\right) }\leq3\sqrt{\frac{\muT r_k}{d_k}} \left(1-\frac{3}{128}(5m+1)^{-2}(3^m\mu^{*m} r^*)^{-1}\right) D_l.$$
Take maximum over $j=1,\dots,d_k$, it is exactly $$\ltinf{\fraM_{k}\left(\bcalT_l-\bcalT^*-\eta_{l}\calP_{\TT_l}(\bcalG_l)\right) }\leq 3\sqrt{\frac{\muT r_k}{d_k}}\left(1-\frac{3}{128}(5m+1)^{-2}(3^m\mu^{*m} r^*)^{-1}\right) D_l.$$
Besides, with Lemma~\ref{teclem:rieman-orthogonal-project} and Lemma~\ref{teclem:remian-perturb-subgradient}, we have
\begin{align*}
	\ltinf{\fraM_k(\calP_{\TT^*}^{\perp}(\bcalT_l-\eta_l\calP_{\TT_l}(\bcalG_l)))}&\leq \ltinf{\fraM_k(\calP_{\TT^*}^{\perp}(\bcalT_l))}+\eta_{l}\ltinf{\fraM_k(\calP_{\TT^*}^{\perp}(\calP_{\TT_l}(\bcalG_l)))}\\
	&\leq5m^2 \sqrt{\frac{\muT r_k}{d_k}}\cdot\mins^{*-1}D_l^2.
\end{align*}
Then it arrives at
\begin{align*}
	&{~~~~}\ltinf{\fraM_k(\calP_{\TT^*}(\bcalT_l-\bcalT^*-\eta_l\calP_{\TT_l}(\bcalG_l)))}\\
	&\leq \ltinf{\fraM_k(\bcalT_l-\bcalT^*-\eta_l\calP_{\TT_l}(\bcalG_l))}+\ltinf{\fraM_k(\calP_{\TT^*}^{\perp}(\bcalT_l-\eta_l\calP_{\TT_l}(\bcalG_l)))}\\
	&\leq 3\sqrt{\frac{\muT r_k}{d_k}}\left(1-\frac{3}{128}(5m+1)^{-2}(3^m\mu^{*m} r^*)^{-1}\right) D_l+5m^2 \sqrt{\frac{\muT r_k}{d_k}}\cdot\mins^{*-1}D_l^2,
\end{align*}
where $\calP_{\TT^*}^{\perp}(\bcalT^*)=0$ is used. Then by Lemma~\ref{teclem:perturbation:tensor}, we have \begin{align*}
	\ltinf{\fraM_k(\bcalT_{l+1}-\bcalT^*)}&\leq\ltinf{\fraM_k(\calP_{\TT^*}(\bcalT_l-\bcalT^*-\eta_l\calP_{\TT_l}(\bcalG_l)))}+32m\sqrt{\frac{\muT r_k}{d_k}}\frac{\fro{\bcalT_l-\bcalT^*-\eta_l\calP_{\TT_l}(\bcalG_l)}^2}{\mins^*}\\
	&{~~~}+32m\ltinf{\fraM_k(\bcalT_l-\bcalT^*-\eta_l\calP_{\TT_l}(\bcalG_l))}\frac{\fro{\bcalT_l-\bcalT^*-\eta_l\calP_{\TT_l}(\bcalG_l)}}{\mins^*}\\
	&\leq3\sqrt{\frac{\muT r_k}{d_k}}\cdot\left(1-\frac{3}{128}(5m+1)^{-2}(\mu^m r^*)^{-1}\right)D_{l}+32m\mins^{*-1}D_{l+1}^2\cdot\sqrt{\frac{\mu r_k}{d_k}}\\
	&{~~~}+5m^2 \sqrt{\frac{\muT r_k}{d_k}}\cdot\mins^{*-1}D_l^2\\
	&\leq3\sqrt{\frac{\muT r_k}{d_k}} \cdot D_{l+1},
\end{align*}
and Lemma~\ref{teclem:perturbation:tensor} also infers
\begin{align*}
	&{~~~~}\ltinf{\left(\U_k^{(l+1)}\H_k^{(l+1)}-\U_k^*\right)\fraM_k(\bcalC^*)}\\
	&\leq \ltinf{\U_{k\perp}^*\U_{k\perp}^*\fraM_k(\bcalT_l-\bcalT^*-\eta_l\calP_{\TT_l}(\bcalG_l))}+64\ltinf{\U_k^*}\frac{\fro{\bcalT_l-\bcalT^*-\eta_{l}\calP_{\TT_l}(\bcalG_l)}^2}{\mins^*}\\
	&{~~~}+16\ltinf{\U_{k\perp}^*\U_{k\perp}^*\fraM_k(\bcalT_l-\bcalT^*-\eta_l\calP_{\TT_l}(\bcalG_l))}\cdot \frac{\fro{\bcalT_l-\bcalT^*-\eta_{l}\calP_{\TT_l}(\bcalG_l)}}{\mins^*}\\
	&\leq \left(1+16D_{l+1}\cdot \mins^{*-1}\right)\ltinf{\fraM_k(\bcalT_l-\bcalT^*-\eta_l\calP_{\TT_l}(\bcalG_l)) }+1.1\ltinf{\U_k^*}D_{l+1}\\
	&\leq 5D_{l+1}\cdot\sqrt{\frac{\muT r_k}{d_k}},
\end{align*} where the second ineuqality uses
\begin{align*}
	&{~~~}\ltinf{\U_{k\perp}^*\U_{k\perp}^*\fraM_k(\bcalT_l-\bcalT^*-\eta_l\calP_{\TT_l}(\bcalG_l))}\\
	&\leq\ltinf{\fraM_k(\bcalT_l-\bcalT^*-\eta_l\calP_{\TT_l}(\bcalG_l)) }+\ltinf{\U_k^*}\fro{\bcalT_l-\bcalT^*-\eta_l\calP_{\TT_l}(\bcalG_l)}.
\end{align*}
Note that it implies $\bcalT_{l+1}$ is incoherent with $3\muT$, namely due to,
\begin{align*}
	\ltinf{\U_k^{(l+1)}}&\leq\sqrt{2}\ltinf{\U_k^{(l+1)}\H_k^{(l+1)}}\leq \sqrt{2}\linft{\U_k^{(l+1)}\H_k^{(l+1)}-\U_k^* }+\sqrt{2}\linft{\U_k^*}\\
	&\leq \sqrt{2}\mins^{*-1} \ltinf{\left(\U_k^{(l+1)}\H_k^{(l+1)}-\U_k^*\right)\fraM_k(\bcalC^*)}D_{l+1}+\sqrt{2}\linft{\U_k^*}\\
	&\leq\sqrt{\frac{3\muT r_k}{d_k}}.
\end{align*}
Finally, by Lemma~\ref{teclem:entrynorm-expansion}, we have bound of entrywise normed bound
\begin{align*}
	&{~~~~}\linft{\bcalT_{l+1}-\bcalT^*}\\
	&\leq \sqrt{\frac{3^m\mu^{*m} r^*}{d^*}}\fro{\bcalT_{l+1}-\bcalT^*}+\sum_{k=1}^m\sqrt{\frac{3^m\mu^{*m-1}r_k^-}{\dkm}}\ltinf{\left(\U_k^{(l+1)}\H_k^{(l+1)}-\U_k^*\right)\fraM_k(\bcalC^*)}\\
	&\leq (5m+1)\sqrt{\frac{3^m\mu^{*m} r^*}{d^*}}D_{l+1}.
\end{align*}
\paragraph{Phase One Output}
Notice that if the signal-to-noise ratio is smaller than $O(\sqrt{d^*})$ and then the initialization already guarantees error of scale $O(\sqrt{d^*}\gamma)$, in which case it enters phase two directly and doesn't need the first phase.
Anyway, phase two starts with the error rate of
\begin{align*}
	\fro{\bcalT_{l_1}^{(k),j}-\bcalT^*}	\vee\fro{\bcalT_{l_1}-\bcalT^*}\leq \min\left\{2(5m+1)\sqrt{3^m\mu^{*m}r^*d^*}(6\gamma+\delta), D_0 \right\},
\end{align*}
By traingular inequality, we have upper bound of distance between the origanl sequence and leave-one-out sequence,
\begin{align*}
	\fro{\bcalT_{l_1}^{(k),j}-\bcalT_{l_1}}\leq2\min\left\{2(5m+1)\sqrt{3^m\mu^{*m}r^*d^*}(6\gamma+\delta), D_0 \right\}.
\end{align*}
Also, it has
\begin{align*}
	\fro{\calP_{\Omega_{j}^{(k)}}\left(\bcalT_{l_1}^{(k),j}-\bcalT^*\right)}	\vee\fro{\calP_{\Omega_{j}^{(k)}}\left(\bcalT_{l_1}-\bcalT^*\right)}\leq 3\sqrt{\frac{\mu r_k}{d_k}}\cdot\min\left\{2(5m+1)\sqrt{3^m\mu^{*m}r^*d^*}(6\gamma+\delta), D_0 \right\},
\end{align*}
which infers
\begin{align*}
	\fro{\calP_{\Omega_{j}^{(k)}}\left(\bcalT_{l_1}-\bcalT_{l_1}^{(k),j}\right)}\leq 6\sqrt{\frac{\mu r_k}{d_k}}\cdot\min\left\{2(5m+1)\sqrt{3^m\mu^{*m}r^*d^*}(6\gamma+\delta), D_0 \right\}.
\end{align*}
The entry-wise normed distance has the following bound,
\begin{align*}
	\linft{\bcalT_{l_1}-\bcalT^*}\leq 2(5m+1)^23^m\mu^{*m}r^*(6\gamma+\delta),\quad \linft{\bcalT_{l_1}^{(k),j}-\bcalT^*}\leq 2(5m+1)^23^m\mu^{*m}r^*(6\gamma+\delta).
\end{align*}

\subsubsection{Phase Two}
Analysis of phase two is more delicate. We shall continue from the output of phase one $\bcalT_{l_1}$ and prove via induction.  Denote $D_l:=\left(1-\frac{3}{c_1^264(m+1)}\cdot\frac{\delta^2}{b_0^2}\right)^{l-l_1}\fro{\bcalT_{l_1}-\bcalT^*}$.  Suppose Equation~\eqref{eq81}-\eqref{eq85} hold for iteration $l$ and we shall prove Equation~\eqref{eq81}-\eqref{eq85} with iteration $l+1$ for all $k,v=1,\dots,m$ and $j=1,\dots,d_k$, $i=1,\dots,d_v$,
\begin{subequations}
	\begin{align}
		\fro{\bcalT_{l+1}-\bcalT^*}\vee \fro{\bcalT_{l+1}^{(k),j}-\bcalT^*}&\leq D_{l+1}\label{eq81}\\
		\fro{\calP_{\Omega_{j}^{(k)}}\left(\bcalT_{l+1}^{(k),j}-\bcalT^*\right)}&\leq3\sqrt{\frac{\muT r_k}{d_k}}D_{l+1}\label{eq82}\\
		\fro{\calP_{\Omega_{j}^{(k)}}\left(\bcalT_{l+1}^{(k),j}-\bcalT_{l+1}\right)}&\leq3\sqrt{\frac{\muT r_k}{d_k}}\min\left\{36D_0,C_{m,\muT,r^*}(6\gamma+\delta)\right\}+2\frac{\mins^*}{\sqrt{\textsf{DoF}_m}b_0}\delta\label{eq83}\\
		\fro{\calP_{\Omega_{j}^{(k)}}\left(\bcalT_{l+1}^{(v),i}-\bcalT_{l+1}^{(j),k}\right)}&\leq3\sqrt{\frac{\muT r_k}{d_k}}\min\left\{36D_0,C_{m,\muT,r^*}(6\gamma+\delta)\right\}+2\frac{\mins^*}{\sqrt{\textsf{DoF}_m}b_0}\delta\label{eq84}\\
		\ltinf{\U_k^{(l+1)}}\vee\ltinf{\U_k^{(l+1),(v),i}}&\leq\sqrt{\frac{3\muT r_k}{d_k}}\label{eq85}\\
		\linft{\bcalT_{l+1}-\bcalT^*}\vee\linft{\bcalT_{l+1}^{(k),j}-\bcalT^*}&\leq72(5m+1)^23^m\mu^{*m}r^*(6\gamma+\delta)\label{eq86}
	\end{align}
\end{subequations}
\paragraph{Frobenius norm}
First consider $\fro{\bcalT_l-\bcalT-\eta_l\calP_{\TT_l}(\bcalG_l)}$,
\begin{align*}
	\fro{\bcalT_l-\eta_l\calP_{\TT_l} (\bcalG_l)-\bcalT^*}^2=\fro{\bcalT_l-\bcalT^*}^2-2\eta_l\inp{\bcalT_l-\bcalT^*}{\calP_{\TT_l}(\bcalG_l)} + \eta_l^2\fro{\calP_{\TT_l}(\bcalG_l)}^2.
\end{align*}
According to Lemma~\ref{lem:hub:regularity}, the last term has the upper bound $\fro{\calP_{\TT_l}(\bcalG_l)}^2\leq c_1^2(m+1)\delta^{-2}\fro{\bcalT_l-\bcalT^*}^2$.
By definition of sub-gradient and analysis of $f(\bcalT)-f(\bcalT^*)$ in Lemma~\ref{lem:hub:regularity}, the intermediate term has the lower bound
\begin{align*}
	\inp{\bcalT_l-\bcalT^*}{\calP_{\TT_l}(\bcalG_l)}&=\inp{\bcalT_l-\bcalT^*}{\bcalG_l}-\inp{\calP_{\TT_l}^{\perp}(\bcalT_l-\bcalT^*)}{\bcalG_l}\\
	&\geq f(\bcalT_l)-f(\bcalT^*)-\inp{\calP_{\TT_l}^{\perp}\bcalT^*}{\bcalG_l}\\
	&\geq \frac{1}{2b_0} \fro{\bcalT_l-\bcalT^*}^2-\inp{\calP_{\TT_l}^{\perp}\bcalT^*}{\bcalG_l}
\end{align*}
Besides, by Lemma~\ref{teclem:rieman-orthogonal-project} and proofs of Lemma~\ref{lem:hub:regularity}, we have $ \left|\inp{\calP_{\TT_l}^{\perp}\bcalT^*}{\bcalG_l} \right|\leq \fro{\calP_{\TT_l}^{\perp}\bcalT^*}\|\bcalG_{l}\|_{\mathrm{F, 2\r}}\leq8m^2 c_1\delta^{-1}\mins^{*-1}\fro{\bcalT_l-\bcalT^*}^3$ and hence we have
\begin{align*}
	\fro{\bcalT_l-\eta_l\calP_{\TT_l} (\bcalG_l)-\bcalT^*}^2&\leq \fro{\bcalT_l-\bcalT^*}^2-\eta_{l}\frac{1}{b_0} \fro{\bcalT_l-\bcalT^*}^2+16\eta_{l}m^2c_1\delta^{-1}\mins^{*-1}\fro{\bcalT_l-\bcalT^*}^3\\
	&{~~~~~~~~~~~~~~~~~~~~~~~~~~~~~~~~~~~~~~~~~~~~~~~~~~~~~~}+\eta_{l}^2c_1^2(m+1)\delta^{-2}\fro{\bcalT_l-\calT^*}^2\\
	&\leq  \fro{\bcalT_l-\bcalT^*}^2-\eta_{l}\frac{1}{2b_0} \fro{\bcalT_l-\bcalT^*}^2 + +\eta_{l}^2c_1^2(m+1)\delta^{-2}\fro{\bcalT_l-\calT^*}^2\\
	&\leq \left(1-\frac{3}{c_1^264(m+1)}\cdot\frac{\delta^2}{b_0^2}\right)\fro{\bcalT_l-\bcalT^*}^2,
\end{align*}
where the second inequality is because of $\fro{\bcalT_l-\bcalT^*}\leq \fro{\bcalT_0-\bcalT^*}\leq c_1^{-1}m^{-2}\frac{\delta}{b_0}\mins^{*}$ and the last inequality uses stepsize selection $\eta_{l}\in\left[\frac{1}{8c_1^2(m+1)}\cdot\frac{\delta^2}{b_0}, \frac{3}{8c_1^2(m+1)}\cdot\frac{\delta^2}{b_0}\right]$.
By tensor perturbation Lemma~\ref{teclem:perturbation:tensor}, we have
\begin{align*}
	\fro{\bcalT_{l+1}-\bcalT^*}&\leq \fro{\bcalT_l-\eta_l\calP_{\TT_l} (\bcalG_l)-\bcalT^*}+\mins^{*-1}\fro{\bcalT_l-\eta_l\calP_{\TT_l} (\bcalG_l)-\bcalT^*}^2\\
	&\leq \left(1-\frac{1}{c_1^232(m+1)}\cdot\frac{\delta^2}{b_0^2}\right)\fro{\bcalT_l-\bcalT^*}\leq D_{l+1}.
\end{align*}
We could have parallel results under event $\bcalE_2$ for  leave-one-out sequence $k=1,\dots,m$, $j=1,\dots,d_k$, $\fro{\bcalT_{l+1}^{(k),j}-\bcalT^*}\leq D_{l+1}$ and hence we skip its proof.
\paragraph{Entrywise norm} We shall prove Equation~\eqref{eq82}-\eqref{eq86} step by step.

\textbf{Step One} First, consider the $j$-th slice of order $k$ in the leave-one-out sequence,
\begin{align*}
	&\fro{\calP_{\Omega_{j}^{(k)}}\left(\bcalT_{l}^{(k),j}-\bcalT^*-\eta_{l}\calP_{\TT_l^{(k),j}}(\bcalG_l^{(k),j})\right)}^2=\fro{\calP_{\Omega_{j}^{(k)}}\left(\bcalT_{l}^{(k),j}-\bcalT^*\right)}^2\\
	&{~~~~~~~~~}-2\eta_{l}\inp{\calP_{\Omega_{j}^{(k)}}\left(\bcalT_{l}^{(k),j}-\bcalT^*\right)}{\calP_{\Omega_{j}^{(k)}}\calP_{\TT_l^{(k),j}}\left(\bcalG_l^{(k),j}\right)}+\eta_l^2\fro{\calP_{\Omega_{j}^{(k)}}\calP_{\TT_l^{(k),j}}\left(\bcalG_l^{(k),j}\right)}^2.
\end{align*}
According to leave-one-out sequence construction and with expectation calculations in proof of  Lemma~\ref{teclem:pseudo:function expectation} (that is $|\EE\dot{\rho}_{H_p}(t-\xi)|\leq t/\delta$), we know $\fro{\calP_{\Omega_{j}^{(k)}}\left(\bcalG_l^{(k),j}\right) }^2\leq\delta^{-2}\cdot \fro{\calP_{\Omega_{j}^{(k)}}\left(\bcalT_{l}^{(k),j}-\bcalT^*\right)}^2$. Hence, by the induction of $\ltwo{(\U_k^{(l),(k),j})_{j,\cdot}}\leq\sqrt{\frac{3\muT r_k}{d_k}}$ and the regularity properties, the slice of projected sub-gradient term has the following upper bound $$\fro{\calP_{\Omega_{j}^{(k)}}\calP_{\TT_l^{(k),j}}\left(\bcalG_l^{(k),j}\right)}^2\leq\delta^{-2} \fro{\calP_{\Omega_{j}^{(k)}}\left(\bcalT_{l}^{(k),j}-\bcalT^*\right)}^2+6\delta^{-2}\frac{\muT r_k}{d_k}\fro{\bcalT_l^{(k),j}-\bcalT^*}^2.$$
As for the intermediate term, it has\begin{align*}
	&\inp{\calP_{\Omega_{j}^{(k)}}\left(\bcalT_{l}^{(k),j}-\bcalT^*\right)}{\calP_{\Omega_{j}^{(k)}}\calP_{\TT_l^{(k),j}}\left(\bcalG_l^{(k),j}\right)}\\
	&{~~~~~~~~~~~~~~~~~~~~~}=\inp{\calP_{\Omega_{j}^{(k)}}\left(\bcalT_{l}^{(k),j}\right)}{\calP_{\Omega_{j}^{(k)}}\left(\bcalG_l^{(k),j}\right)}-\inp{\calP_{\Omega_{j}^{(k)}}\left(\bcalT^*\right)}{\calP_{\Omega_{j}^{(k)}}\calP_{\TT_l^{(k),j}}\left(\bcalG_l^{(k),j}\right)},
\end{align*}
where the latter term could be expanded in the following way (see details in phase one analyses Section~\ref{proof:pseudo-Huber:phaseone}),
\begin{align*}
	\inp{\calP_{\Omega_{j}^{(k)}}\left(\bcalT^*\right)}{\calP_{\Omega_{j}^{(k)}}\calP_{\TT_l^{(k),j}}\left(\bcalG_l^{(k),j}\right)}&=\inp{\calP_{\Omega_{j}^{(k)}}\left(\bcalT^*\right)}{\calP_{\Omega_{j}^{(k)}}\bcalG_l^{(k),j}}-\underbrace{\inp{\calP_{\Omega_{j}^{(k)}}\calP_{\TT_l^{(k),j}}^{\perp}\left(\bcalT^*\right)}{\calP_{\Omega_{j}^{(k)}}\bcalG_l^{(k),j}}}_{E_1}\\
	&{~~~~~~~~~~~~~~~~~~}+\underbrace{\inp{\calP_{\Omega_{j}^{(k)}}\calP_{\TT_l^{(k),j}}^{\perp}\left(\bcalT^*\right)}{\calP_{\Omega_{j}^{(k)}}\calP_{\TT_l^{(k),j}}\bcalG_l^{(k),j}}}_{E_2}.
\end{align*}
Similar to phase one analyses in Section~\ref{proof:pseudo-Huber:phaseone}, we have bound for $|E_1|$ and $|E_2|$,
\begin{align*}
	|E_1|\vee|E_2|\leq 16m^2\mins^{*-1}\frac{\muT r_k}{d_k}D_{l}^2,
\end{align*}
where the induction of $\bcalT_l^{(k),j}$ is used. Also according to leave-one-out sequence definition and loss function expectation Lemma~\ref{teclem:pseudo:function expectation}, it has $$\inp{\calP_{\Omega_{j}^{(k)}}\left(\bcalT_l^{(k),j}-\bcalT^*\right)}{\calP_{\Omega_{j}^{(k)}}\bcalG_l^{(k),j}}\geq \hat{f}_j^{(k)}(\calP_{\Omega_{j}^{(k)}}(\bcalT^{(k),j}))-\hat{f}_j^{(k)}(\calP_{\Omega_{j}^{(k)}}(\bcalT^*))\geq b_0^{-1}\fro{\calP_{\Omega_{j}^{(k)}}\left(\bcalT_l^{(k),j}-\bcalT^*\right) }^2. $$
Thus the intermediate term has the lower bound
$$\inp{\calP_{\Omega_{j}^{(k)}}(\bcalT_{l}-\bcalT^*)}{\calP_{\Omega_{j}^{(k)}}(\calP_{\TT_l^{(k),j}}(\bcalG_l^{(k),j}))}\geq b_0^{-1}\fro{\calP_{\Omega_{j}^{(k)}}(\bcalT_{l}^{(k),j}-\bcalT^*)}^2-32m^2\mins^{*-1} \frac{\muT r_k}{d_k}\cdot D_l^2. $$
Hence, just like pase one entrywise normed analyses in Section~\ref{proof:pseudo-Huber:phaseone}, it has \begin{align}
	\fro{\calP_{\Omega_{j}^{(k)}}(\bcalT_{l}^{(k),j}-\bcalT^*-\eta_{l}\calP_{\TT_l^{(k),j}}(\bcalG_l^{(k),j}))}\leq 3\sqrt{\frac{\muT r_k}{d_k}}\left(1- \frac{3\delta^2}{64(m+1)b_0^2}\right)\cdot D_{l}.
	\label{eq11}
\end{align}
Remark that even though results of Lemma~\ref{teclem:perturbation:tensor} are measured $\ltinf{\cdot}$, they also hold if it's constrained to certain slice which is a byproduct in its proof. Then with $\bcalT_{l+1}^{(k),j}=\text{HOSVD}_{\r}\left(\bcalT_l^{(k),j}-\eta_l\calP_{\TT_l^{(k),j}}\left(\bcalG_l^{(k),j}\right)\right)$, it has
\begin{align}
	\fro{\calP_{\Omega_{j}^{(k)}}\left(\bcalT_{l+1}^{(k),j}-\bcalT^*\right)}\leq 3\sqrt{\frac{\muT r_k}{d_k}}\cdot D_{l+1}.
	\label{eq13}
\end{align} 
\textbf{Step Two} Consider distance between the original sequence and the leave-one-out sequence,
\begin{equation}
	\begin{split}
		&{~~~}\fro{\calP_{\Omega_{j}^{(k)}}\left(\bcalT_{l}-\bcalT_l^{(k),j}-\eta_{l}\cdot\left(\calP_{\TT_l}(\bcalG_l)-\calP_{\TT_l^{(k),j}}(\bcalG_l^{(k),j})\right)\right)}^2\\
		&=\ltwo{\calP_{\Omega_{j}^{(k)}}\left(\bcalT_{l}-\bcalT_l^{(k),j}\right) }^2-2\eta_{l}\inp{\calP_{\Omega_{j}^{(k)}}\left(\bcalT_l-\bcalT_l^{(k),j}\right)}{\calP_{\Omega_{j}^{(k)}}\left(\calP_{\TT_l}\bcalG_l-\calP_{\TT_l^{(k),j}}\bcalG_l^{(k),j}\right)}\\&{~~~}+\eta_{l}^2\fro{\calP_{\Omega_{j}^{(k)}}\calP_{\TT_l}\bcalG_l-\calP_{\Omega_{j}^{(k)}}\calP_{\TT_l^{(k),j}}\bcalG_l^{(k),j} }^2.
	\end{split}
	\label{eq7}
\end{equation}
Denote the sub-gradient of the original loss function at leave-one-out iterative as $\bar{\bcalG}_{l}^{(k),j}\in\partial f(\bcalT_{l}^{(k),j})$. Notice that entries of $\bcalG_{l}^{(k),j}$ are same as $\bar{\bcalG}_{l}^{(k),j}$ except the $j$ th slice of order $k$.
The last term of Equation~\eqref{eq7} could be upper bounded with
\begin{align*}
	&{~~~}\fro{\calP_{\Omega_{j}^{(k)}}\calP_{\TT_l}\bcalG_l-\calP_{\Omega_{j}^{(k)}}\calP_{\TT_l^{(k),j}}\bcalG_l^{(k),j} }\\
	&\leq \fro{\calP_{\Omega_{j}^{(k)}}\calP_{\TT_l}\left(\bcalG_l-\bar{\bcalG_l}^{(k),j}\right) }+\fro{\calP_{\Omega_{j}^{(k)}}\left(\calP_{\TT_l}-\calP_{\TT_l^{(k),j}}\right)\bar{\bcalG}_l^{(k),j}}+\fro{\calP_{\Omega_{j}^{(k)}}\calP_{\TT_l^{(k),j}}\left(\bcalG_l^{(k),j}-\bar{\bcalG_l}^{(k),j}\right)}.
\end{align*}
Note that with definition of Riemannian projections and induction over $\U_k^{(l)}$, we have \begin{align*}
	\fro{\calP_{\Omega_{j}^{(k)}}\calP_{\TT_l}\left(\bcalG_l-\bar{\bcalG}_l^{(k),j}\right) }^2\leq\fro{\calP_{\Omega_{j}^{(k)}}\left(\bcalG_l-\bar{\bcalG}_l^{(k),j}\right) }^2+3\delta^{-2}\frac{\muT r_k}{d_k}\cdot \fro{\bcalG_l-\bar{\bcalG}_l^{(k),j}}^2,
\end{align*}
and by Lemma~\ref{teclem:remian-perturb-subgradient}, we have
\begin{align}
	\fro{\calP_{\Omega_{j}^{(k)}}\left(\calP_{\TT_l}-\calP_{\TT_l^{(k),j}}\right)\bar{\bcalG}_l^{(k),j}}\leq m^2 \sqrt{\frac{\muT r_k}{d_k}}\delta^{-1}\mins^{*-1}D_l^2+\mins^{*-1}\delta^{-1}D_l\fro{\calP_{\Omega_{j}^{(k)}}\left(\bcalT_{l}^{(k),j}-\bcalT_l\right)}.
	\label{eq10}
\end{align}
\begin{claim} With probability exceeding $1-cd^{*-7}$, the following holds for each $k=1,\dots,m$ and $j=1,\dots,d_k$,
	\begin{align*}
		\fro{\calP_{\Omega_{j}^{(k)}}\calP_{\TT_l^{(k),j}}\left(\bcalG_l^{(k),j}-\bar{\bcalG_l}^{(k),j}\right)}\leq C(m+1) \sqrt{r^*\log d^*},
	\end{align*}
where $c,C>0$ are two constants.
\label{claim:leave-one-out}
\end{claim}
According to Claim~\ref{claim:leave-one-out}, we have \begin{align}
	\fro{\calP_{\Omega_{j}^{(k)}}\calP_{\TT_l^{(k),j}}\left(\bcalG_l^{(k),j}-\bar{\bcalG_l}^{(k),j}\right)}\leq C(m+1) \sqrt{r^*\log d^*}.
	\label{eq9}
\end{align}Thus the last term of Equation~\eqref{eq7} has upper bound,
\begin{align*}
	\fro{\calP_{\Omega_{j}^{(k)}}\calP_{\TT_l}\bcalG_l-\calP_{\Omega_{j}^{(k)}}\calP_{\TT_l^{(k),j}}\bcalG_l^{(k),j} }^2&\leq 4m^4 \frac{\muT r_k}{d_k}\delta^{-2}\mins^{*-2}D_l^4+4\delta^{-2}\mins^{*-2}D_l^2\fro{\calP_{\Omega_{j}^{(k)}}\left(\bcalT_{l}^{(k),j}-\bcalT_l\right)^2}\\
	&+\fro{\calP_{\Omega_{j}^{(k)}}\left(\bcalG_l-\bar{\bcalG}_l^{(k),j}\right) }^2+2\frac{\muT r_k}{d_k}\cdot \fro{\bcalG_l-\bar{\bcalG}_l^{(k),j}}^2+C(m+1)^2r^*\log d^*.
\end{align*}
As for the intermediate term of Equation~\eqref{eq7}, we have 
\begin{align*}
	&{~~~}\inp{\calP_{\Omega_{j}^{(k)}}\left(\bcalT_l-\bcalT_l^{(k),j}\right)}{\calP_{\Omega_{j}^{(k)}}\left(\calP_{\TT_l}\bcalG_l-\calP_{\TT_l^{(k),j}}\bcalG_l^{(k),j}\right)}\\
	&=\inp{\calP_{\Omega_{j}^{(k)}}\left(\bcalT_l-\bcalT_l^{(k),j}\right)}{\calP_{\Omega_{j}^{(k)}}\calP_{\TT_l}\left(\bcalG_l-\bar{\bcalG}_l^{(k),j} \right)}+\inp{\calP_{\Omega_{j}^{(k)}}\left(\bcalT_l-\bcalT_l^{(k),j}\right)}{\calP_{\Omega_{j}^{(k)}}\left(\calP_{\TT_l}-\calP_{\TT_l^{(k),j}}\right) \bcalG_l^{(k),j}}\\
	&{~~~}+\inp{\calP_{\Omega_{j}^{(k)}}\left(\bcalT_l-\bcalT_l^{(k),j}\right)}{\calP_{\Omega_{j}^{(k)}}\calP_{\TT_l^{(k),j}}\left(\bar{\bcalG}_{l}^{(k),j} - \bcalG_l^{(k),j}\right)}.
\end{align*}
Remark that second term and last term of the above equation could be upper bounded with
\begin{align*}
	&{~~~}\left|\inp{\calP_{\Omega_{j}^{(k)}}\left(\bcalT_l-\bcalT_l^{(k),j}\right)}{\calP_{\Omega_{j}^{(k)}}\left(\calP_{\TT_l}-\calP_{\TT_l^{(k),j}}\right) \bcalG_l^{(k),j}}\right|\\
	&\leq\fro{\calP_{\Omega_{j}^{(k)}}\left(\bcalT_l-\bcalT_l^{(k),j}\right)}\fro{\calP_{\Omega_{j}^{(k)}}\left(\calP_{\TT_l}-\calP_{\TT_l^{(k),j}}\right) \bcalG_l^{(k),j}},
\end{align*}
where $\fro{\calP_{\Omega_{j}^{(k)}}\left(\calP_{\TT_l}-\calP_{\TT_l^{(k),j}}\right) \bcalG_l^{(k),j}}$ is analyzed in Equation~\eqref{eq10} and similarly \begin{align*}
	&{~~~}\left|\inp{\calP_{\Omega_{j}^{(k)}}\left(\bcalT_l-\bcalT_l^{(k),j}\right)}{\calP_{\Omega_{j}^{(k)}}\calP_{\TT_l^{(k),j}} \left(\bar{\bcalG}_{l}^{(k),j} - \bcalG_l^{(k),j}\right) }\right|\\
	&\leq\fro{\calP_{\Omega_{j}^{(k)}}\left(\bcalT_l-\bcalT_l^{(k),j}\right)}\fro{\calP_{\Omega_{j}^{(k)}}\calP_{\TT_l^{(k),j}} \left(\bar{\bcalG}_{l}^{(k),j} - \bcalG_l^{(k),j}\right)},
\end{align*}
where $\fro{\calP_{\Omega_{j}^{(k)}}\calP_{\TT_l^{(k),j}}\left(\bar{\bcalG}_{l}^{(k),j} - \bcalG_l^{(k),j}\right) }$ is bounded in Equation~\eqref{eq9}. Note that with some simple calculations of the remaining term, we have
\begin{align*}
	&{~~~}\inp{\calP_{\Omega_{j}^{(k)}}\left(\bcalT_l-\bcalT_l^{(k),j}\right)}{\calP_{\Omega_{j}^{(k)}}\calP_{\TT_l}\left(\bcalG_l-\bar{\bcalG}_l^{(k),j} \right)}\\
	&= \inp{\calP_{\Omega_{j}^{(k)}}\left(\bcalT_l\right)}{\calP_{\Omega_{j}^{(k)}}\calP_{\TT_l}\left(\bcalG_l-\bar{\bcalG}_l^{(k),j} \right)}-\inp{\calP_{\Omega_{j}^{(k)}}\left(\bcalT_l^{(k),j}\right)}{\calP_{\Omega_{j}^{(k)}}\calP_{\TT_l}\left(\bcalG_l-\bar{\bcalG}_l^{(k),j} \right)}\\
	&=\inp{\calP_{\Omega_{j}^{(k)}}\left(\bcalT_l\right)}{\calP_{\Omega_{j}^{(k)}}\left(\bcalG_l-\bar{\bcalG}_l^{(k),j} \right)}-\inp{\calP_{\Omega_{j}^{(k)}}\left(\bcalT_l^{(k),j}\right)}{\calP_{\Omega_{j}^{(k)}}\calP_{\TT_l}\left(\bcalG_l-\bar{\bcalG}_l^{(k),j} \right)},
\end{align*}
where the second term has the following expressions (see details in Section~\ref{proof:pseudo-Huber:phaseone}),
\begin{align*}
	&{~~~}\inp{\calP_{\Omega_{j}^{(k)}}\left(\bcalT_l^{(k),j}\right)}{\calP_{\Omega_{j}^{(k)}}\calP_{\TT_l}\left(\bcalG_l-\bar{\bcalG}_l^{(k),j} \right)}\\
	&=\inp{\calP_{\Omega_{j}^{(k)}}\left(\bcalT_l^{(k),j}\right)}{\calP_{\Omega_{j}^{(k)}}\left(\bcalG_l-\bar\bcalG_l^{(k),j}\right)}-\underbrace{\inp{\calP_{\Omega_{j}^{(k)}}\calP_{\TT_l}^{\perp}\left(\bcalT_l^{(k),j}\right)}{\calP_{\Omega_{j}^{(k)}}\left(\bcalG_l-\bar\bcalG_l^{(k),j}\right)}}_{F_1}\\&{~~~~~~~~~}+\underbrace{\inp{\calP_{\Omega_{j}^{(k)}}\calP_{\TT_l}^{\perp}\left(\bcalT_l^{(k),j}\right)}{\calP_{\Omega_{j}^{(k)}}\calP_{\TT_l}\left(\bcalG_l-\bar\bcalG_l^{(k),j}\right)}}_{F_2}.
\end{align*}
By same analyses in Section~\ref{proof:pseudo-Huber:phaseone}, we have
\begin{align*}
	&{~~~}\left|F_1\right|\vee\left|F_2\right|\\
	&\leq m^2 \fro{\calP_{\Omega_{j}^{(k)}}\left(\bcalG_l-\bar\bcalG_l^{(k),j}\right) } \left(\sqrt{\frac{\muT r_k}{d_k}}\mins^{*-1}D_l^2+\mins^{*-1}D_l\fro{\calP_{\Omega_{j}^{(k)}}\left(\bcalT_{l}^{(k),j}-\bcalT_l\right)}\right)\\
	&\leq 0.25\delta \fro{\calP_{\Omega_{j}^{(k)}}\left(\bcalG_l-\bar\bcalG_l^{(k),j}\right) }^2+ m^4\delta^{-1} \left(\sqrt{\frac{\muT r_k}{d_k}}\mins^{*-1}D_l^2+\mins^{*-1}D_l\fro{\calP_{\Omega_{j}^{(k)}}\left(\bcalT_{l}^{(k),j}-\bcalT_l\right)}\right)^2,
\end{align*}
whose last line uses Cauchy-Schwarz inequality. On the other hand, by Lemma~\ref{teclem:hubfunction}, we have
\begin{align*}
	\inp{\calP_{\Omega_{j}^{(k)}}(\bcalT_{l}-\bcalT_l^{(k),j})}{\calP_{\Omega_{j}^{(k)}}(\bcalG_l-\bar\bcalG_l^{(k),j})}\geq \delta \fro{\calP_{\Omega_{j}^{(k)}}(\bcalG_l-\bcalG_l^{(k),j}) }^2.
\end{align*}
Thus altogether the intermediate term of Equation~\eqref{eq7} has lower bound,
\begin{align*}
	&{~~~}\inp{\calP_{\Omega_{j}^{(k)}}\left(\bcalT_l-\bcalT_l^{(k),j}\right)}{\calP_{\Omega_{j}^{(k)}}\left(\calP_{\TT_l}\bcalG_l-\calP_{\TT_l^{(k),j}}\bcalG_l^{(k),j}\right)}\\
	&\geq \frac{1}{2}\delta \fro{\calP_{\Omega_{j}^{(k)}}\left(\bcalG_l-\bcalG_l^{(k),j}\right) }^2-8m^4\frac{\muT r_k}{d_k}\delta^{-1}\mins^{*-2}D_l^{4}-8m^4\delta^{-1}\mins^{*-2}D_l^{2}\fro{\calP_{\Omega_{j}^{(k)}}\left(\bcalT_l-\bcalT_l^{(k),j}\right) }^2.
\end{align*}
Notice that with $8\eta_{l}\leq\delta$, we have $$8\eta_{l}^2\fro{\calP_{\Omega_{j}^{(k)}}\left(\bcalG_l-\bcalG_l^{(k),j}\right)}^2\leq \eta_{l} \delta \fro{\calP_{\Omega_{j}^{(k)}}(\bcalG_l-\bcalG_l^{(k),j}) }^2.$$
Hence after inserting value of stepsize $\eta_{l}\in\left[\frac{1}{8c_1^2(m+1)}\cdot\frac{\delta^2}{b_0}, \frac{3}{8c_1^2(m+1)}\cdot\frac{\delta^2}{b_0}\right]$, we have the upper bound for Equation~\eqref{eq7},
\begin{align*}
	&{~~~~}\fro{\calP_{\Omega_{j}^{(k)}}\left(\bcalT_{l}-\bcalT_l^{(k),j}-\eta_{l}\cdot\left(\calP_{\TT_l}(\bcalG_l)-\calP_{\TT_l^{(k),j}}(\bcalG_l^{(k),j})\right)\right)}^2\\
	&\leq \left(1+4m^2\frac{D_l^2}{\mins^{*2}}\cdot\frac{\delta}{b_0} \right)\ltwo{\calP_{\Omega_{j}^{(k)}}\left(\bcalT_{l}-\bcalT_l^{(k),j}\right) }^2+32m^4\frac{\delta}{b_0}\frac{\muT  r_k}{d_k}\cdot\frac{D_l^4}{\mins^{*2}}+\frac{\delta^4}{b_0^2}r^*\log d^*.
\end{align*}
Note that Assumption~\ref{assu:hub:noise} infers that $b_0\geq C_{m,\muT,r^*}(6\gamma+\delta)\geq r^*\gamma\geq r^*\delta\sqrt{\log d^*}$. Furthermore, take square root of the above equation and it has
\begin{equation}
	\begin{split}
		&\fro{\calP_{\Omega_{j}^{(k)}}\left(\bcalT_{l}-\bcalT_l^{(k),j}-\eta_{l}\cdot\left(\calP_{\TT_l}(\bcalG_l)-\calP_{\TT_l^{(k),j}}(\bcalG_l^{(k),j})\right)\right)}\\
		&{~~~~~~~~~~~~~~~~~~~~~~~}\leq \left(1+2m^4\frac{\delta}{b_0}\cdot\frac{D_l^2}{\mins^{*2}}\right)\ltwo{\calP_{\Omega_{j}^{(k)}}\left(\bcalT_{l}-\bcalT_l^{(k),j}\right) }+8m^2\sqrt{\frac{\muT r_k}{d_k}}\frac{D_l^2}{\mins^{*}}+\delta\\
		&{~~~~~~~~~~~~~~~~~~~~~~~}\leq \left(1+m^2\frac{D_l}{\mins^{*}}\right)\ltwo{\calP_{\Omega_{j}^{(k)}}\left(\bcalT_{l}-\bcalT_l^{(k),j}\right) }+8m^2\sqrt{\frac{\muT r_k}{d_k}}\frac{D_l^2}{\mins^{*}}+\delta,
	\end{split}
	\label{eq12}
\end{equation}
where the last line is due to initialization condition. Then, with slice perturbations (same as Section~\ref{proof:pseudo-Huber:phaseone} analyses), we have
\begin{align*}
	\ltwo{\calP_{\Omega_{j}^{(k)}}\left(\bcalT_{l+1}-\bcalT_{l+1}^{(k),j}\right) }\leq\left(1+m^2\frac{D_l}{\mins^{*}}\right)\ltwo{\calP_{\Omega_{j}^{(k)}}\left(\bcalT_{l}-\bcalT_l^{(k),j}\right) }+32m^2\sqrt{\frac{\muT r_k}{d_k}}\frac{D_l^2}{\mins^{*}}+\delta.
\end{align*}
By $D_l\geq \sqrt{\textsf{DoF}_m}b_0$ and $D_{l+1}\leq D_l$, the above equation implies \begin{align*}
	&{~~~}\fro{\calP_{\Omega_{j}^{(k)}}\left(\bcalT_{l+1}-\bcalT_{l+1}^{(k),j}\right) }+32\sqrt{\frac{\muT r_k}{d_k}}D_{l+1}+ \frac{\mins^*}{\sqrt{\textsf{DoF}_m}b_0}\delta\\
	&\leq\left(1+m^2\frac{D_l}{\mins^{*}}\right)\ltwo{\calP_{\Omega_{j}^{(k)}}\left(\bcalT_{l}-\bcalT_l^{(k),j}\right) }+32m^2\sqrt{\frac{\muT r_k}{d_k}}\frac{D_l^2}{\mins^{*}}+ 32\sqrt{\frac{\muT r_k}{d_k}}D_{l}+\frac{D_l}{\sqrt{\textsf{DoF}_m}b_0}\delta+\frac{\mins^*}{\sqrt{\textsf{DoF}_m}b_0}\delta\\
	&\leq \left(1+m^2\cdot\frac{D_l}{\mins^{*}}\right)\left(\fro{\calP_{\Omega_{j}^{(k)}}\left(\bcalT_{l}-\bcalT_l^{(k),j}\right) } +32\sqrt{\frac{\muT r_k}{d_k}} D_l+ \frac{\mins^*}{\sqrt{\textsf{DoF}_m}b_0}\delta\right)\\
	&\leq\prod_{h=l_1}^{l} \left(1+m^2\frac{D_{h}}{\mins^{*}}\right)\left(\fro{\calP_{\Omega_{j}^{(k)}}\left(\bcalT_{l_1}-\bcalT_{l_1}^{(k),j}\right) } +32\sqrt{\frac{\muT r_k}{d_k}}D_{l_1}+\frac{\mins^*}{\sqrt{\textsf{DoF}_m}b_0}\delta\right)\\
	&\leq 3\left(\fro{\calP_{\Omega_{j}^{(k)}}\left(\bcalT_{l_1}-\bcalT_{l_1}^{(k),j}\right) } +32\sqrt{\frac{\muT r_k}{d_k}}D_{l_1}+\frac{\mins^*}{\sqrt{\textsf{DoF}_m}b_0}\delta\right),
\end{align*}
where the last inequality is due to
\begin{align*}
	\prod_{h=l_1}^{+\infty} \left(1+m^2\frac{D_{h}}{\mins^{*}}\right)\leq\exp\left(\sum_{h=l_1}^{+\infty}\log\left(1+m^2\frac{D_l}{\mins^*}\right)\right)\leq\exp\left(\sum_{h=l_1}^{+\infty}m^2\frac{D_l}{\mins^*}\right)&\leq \exp\left(\mins^{*-1}\cdot D_{l_1}\cdot 64\frac{b_0^2}{\delta^2}\right)\\
	&\leq 3.
\end{align*}
Also, notice that $\fro{\calP_{\Omega_{j}^{(k)}}\left(\bcalT_{l_1}-\bcalT_{l_1}^{(k),j}\right) }+32\sqrt{\frac{\muT r_k}{d_k}}D_{l_1} \leq \sqrt{\frac{\muT r_k}{d_k}}\min\left\{c_m\sqrt{\mu^{*m}r^*d^*}(6\gamma+\delta), 36D_0 \right\} $, where $c_m:= 72(5m+1)\sqrt{3^m}$.
In this way, we completes showing
\begin{align}
	\fro{\calP_{\Omega_{j}^{(k)}}\left(\bcalT_{l+1}-\bcalT_{l+1}^{(k),j}\right) }\leq 3\sqrt{\frac{\muT r_k}{d_k}}\min\{36D_0,c_m\sqrt{\mu^{*m}r^*d^*}(6\gamma+\delta)\}+2\frac{\mins^*}{\sqrt{\textsf{DoF}_m}b_0}\delta. 
	\label{eq15}
\end{align} 
\textbf{Step Three} Combine step one by-product Equation~\eqref{eq11} and step two Equation~\eqref{eq12} and then we have
\begin{align*}
	&{~~~}\fro{\calP_{\Omega_{j}^{(k)}}\left(\bcalT_l-\bcalT^*-\eta_{l}\calP_{\TT_l}\bcalG_l\right)}\\
	&\leq\fro{\calP_{\Omega_{j}^{(k)}}\left(\bcalT_{l}-\bcalT_l^{(k),j}-\eta_{l}\cdot\left(\calP_{\TT_l}(\bcalG_l)-\calP_{\TT_l^{(k),j}}(\bcalG_l^{(k),j})\right)\right)} + 	\fro{\calP_{\Omega_{j}^{(k)}}(\bcalT_{l}^{(k),j}-\bcalT^*-\eta_{l}\calP_{\TT_l^{(k),j}}(\bcalG_l^{(k),j}))}\\
	&\leq 3 \sqrt{\frac{\muT r_k}{d_k}}\min\{36D_0,c_m\sqrt{\mu^{*m}r^*d^*}(6\gamma+\delta)\}+3\sqrt{\frac{\muT r_k}{d_k}}D_{l+1}+\delta.
\end{align*}
Similarly, Equation \eqref{eq13} and Equation \eqref{eq15} lead to
\begin{align*}
	&{~~~}\fro{\calP_{\Omega_{j}^{(k)}}\left(\bcalT_{l+1}-\bcalT^*\right) }\\
	&\leq\fro{\calP_{\Omega_{j}^{(k)}}\left(\bcalT_{l+1}^{(k,j)}-\bcalT^*\right) }+\fro{\calP_{\Omega_{j}^{(k)}}\left(\bcalT_{l+1}-\bcalT_{l+1}^{(k),j}\right) }\\
	&\leq 3\sqrt{\frac{\muT r_k}{d_k}}\min\{36D_0,c_m\sqrt{\mu^{*m}r^*d^*}(6\gamma+\delta)\}+3\sqrt{\frac{\muT r_k}{d_k}}D_{l+1}+2\frac{\mins^*}{\sqrt{\textsf{DoF}_m}b_0}\delta.
\end{align*}
Hence after taking maximum over $j=1,\dots,d_k$, we obtain
\begin{align*}
	\ltinf{\fraM_{k}\left(\bcalT_l-\bcalT^*-\eta_{l}\calP_{\TT_l}\bcalG_l \right)}\leq3\sqrt{\frac{\muT r_k}{d_k}}\min\{36D_0,c_m\sqrt{\mu^{*m}r^*d^*}(6\gamma+\delta)\}+3.1\sqrt{\frac{\muT r_k}{d_k}}D_{l_1},
\end{align*}
and
\begin{align*}
	\ltinf{\fraM_{k}(\bcalT_{l+1}-\bcalT^*)}\leq3\sqrt{\frac{\muT r_k}{d_k}}\min\{36D_0,c_m\sqrt{\mu^{*m}r^*d^*}(6\gamma+\delta)\}+3\sqrt{\frac{\muT r_k}{d_k}}D_{l_1}+2\frac{\mins^*}{\sqrt{\textsf{DoF}_m}b_0}\delta.
\end{align*}
By Lemma~\ref{teclem:perturbation:tensor}, we have
\begin{align*}
	\ltinf{\left(\U_k^{(l+1)}\H_k^{(l+1)}-\U_k^*\right)\fraM_{k}(\bcalC^*)}\leq3.1 \sqrt{\frac{\muT r_k}{d_k}}\min\{36D_0,c_m\sqrt{\mu^{*m}r^*d^*}(6\gamma+\delta)\}+6\sqrt{\frac{\muT r_k}{d_k}}D_{l_1}+2\frac{\mins^*}{\sqrt{\textsf{DoF}_m}b_0}\delta.
\end{align*}
Then combined with $D_0\leq C\mins^{*}$ and with same analyses in Section~\ref{proof:pseudo-Huber:phaseone}, we obtain $\linft{\U_k^{(l+1)}}\leq\sqrt{\frac{3\muT r_k}{d_k}}$ for each $k=1,\dots,m$. Finally, by Lemma~\ref{teclem:entrynorm-expansion}, we have upper bound of the error with respect to $\linft{\cdot}$ norm,
\begin{align*}
	\linft{\bcalT_{l+1}-\bcalT^*}\leq72(5m+1)^23^m\mu^{*m}r^*(6\gamma+\delta).
\end{align*}
\textbf{Final Step} We still need to show that leave-one-out sequences also stay in the phase two regions which is characterized in Lemma~\ref{lem:hub:regularity}, namely, $\linft{\bcalT_{l+1}^{(k),j}-\bcalT^*}\lesssim \gamma+\delta$. The proof procedure is similar to bounding $\linft{\bcalT_{l+1}-\bcalT^*}$. Hence, details are omitted and we only present the steps. Similar to Step Two, we could prove for any $v=1,\dots,k-1$ and any $v=1,\dots, d_v$,
\begin{align*}
	&\fro{\calP_{\Omega_{i}^{(v)}}\left(\bcalT_{l}^{(v),i}-\bcalT_l^{(k),j}-\eta_{l}\cdot\left(\calP_{\TT_l^{(v),i}}(\bcalG_l^{(v),i})-\calP_{\TT_l^{(k),j}}(\bcalG_l^{(k),j})\right)\right)}\\
	&{~~~~~~~~~~~~~~~~~~~~~~~~~~~~~}\leq 3\sqrt{\frac{\muT r_v}{d_v}}\min\left\{c_m\sqrt{\mu^{*m}r^*d^*}(6\gamma+\delta), 36D_0 \right\},
\end{align*}
and
\begin{align*}
	\fro{\calP_{\Omega_{i}^{(v)}}\left(\bcalT_{l+1}^{(v),i}-\bcalT_{l+1}^{(k),j}\right) }\leq 3\sqrt{\frac{\muT r_v}{d_v}}\min\left\{c_m\sqrt{\mu^{*m}r^*d^*}(6\gamma+\delta), 36D_0 \right\},
\end{align*}
by which we have
\begin{align*}
	&{~~~}\fro{\calP_{\Omega_{i}^{(v)}}\left(\bcalT^*-\bcalT_l^{(k),j}-\eta_{l}\calP_{\TT_l^{(k),j}}(\bcalG_l^{(k),j})\right)}\\
	&\leq\fro{\calP_{\Omega_{i}^{(v)}}\left(\bcalT_{l}^{(v),i}-\bcalT_l^{(k),j}-\eta_{l}\cdot\left(\calP_{\TT_l^{(v),i}}(\bcalG_l)-\calP_{\TT_l^{(k),j}}(\bcalG_l^{(k),j})\right)\right)}+\fro{\calP_{\Omega_{i}^{(v)}}\left(\bcalT^*-\bcalT_l^{(v),i}-\eta_{l}\calP_{\TT_l^{(v),i}}(\bcalG_l^{(v),i})\right)}\\
	&\leq 3\sqrt{\frac{\muT r_v}{d_v}}\min\left\{c_m\sqrt{\mu^{*m}r^*d^*}(6\gamma+\delta), 36D_0 \right\}+6\sqrt{\frac{\muT r_v}{d_v}}D_{l+1}+\delta,
\end{align*}
and \begin{align*}
	\fro{\calP_{\Omega_{i}^{(v)}}\left(\bcalT^*-\bcalT_{l+1}^{(k),j}\right) } &\leq \fro{\calP_{\Omega_{i}^{(v)}}\left(\bcalT_{l+1}^{(v),i}-\bcalT^*\right) } + \fro{\calP_{\Omega_{i}^{(v)}}\left(\bcalT_{l+1}^{(v),i}-\bcalT_{l+1}^{(k),j}\right) }\\
	&\leq 3\sqrt{\frac{\muT r_v}{d_v}}\min\left\{c_m\sqrt{\mu^mr^*d^*}(6\gamma+\delta), 36D_0 \right\}+6\sqrt{\frac{\muT r_v}{d_v}}D_{l+1}+2\frac{\mins^*}{\sqrt{\textsf{DoF}_m}b_0}\delta.
\end{align*}
By taking maximum over $i=1,\dots,d_v$, we aobtain
\begin{align*}
	\ltinf{\fraM_{v}(\bcalT_{l}^{(k),j}-\bcalT^*-\eta_l\calP_{\TT_l^{(k),j}}(\bcalG_l))}\leq 3\sqrt{\frac{\muT r_v}{d_v}}\min\left\{c_m\sqrt{\mu^{*m}r^*d^*}(6\gamma+\delta), 36D_0 \right\}+6\sqrt{\frac{\muT r_v}{d_v}}D_{l+1}+\delta.
\end{align*}
Then by Lemma~\ref{teclem:perturbation:tensor}, we have for each $v=1,\dots,m$,
\begin{align*}
	\ltinf{\left(\U_v^{(l+1),(k,j)}\H_v^{(l+1),(k),j}-\U_k^*\right)\fraM_{k}(\bcalC^*)}\leq5\sqrt{\frac{\muT r_v}{d_v}}\min\left\{c_m\sqrt{\mu^{*m}r^*d^*}(6\gamma+\delta), 36D_0 \right\},
\end{align*}
which infers $\ltinf{\U_v^{(l+1),(k),j}}\leq\sqrt{\frac{3\muT r_v}{d_v}}$
and by Lemma~\ref{teclem:entrynorm-expansion}, we have
\begin{align*}
	\linft{\bcalT_{l+1}^{(k),j}-\bcalT^*}\leq72(5m+1)^23^m\mu^{*m}r^*(6\gamma+\delta).
\end{align*}
\paragraph{Phase Two Output}
At the end of phase two, it reaches the error rate \begin{align*}
	\fro{\bcalT_{l_1+l_2}-\bcalT^*}\leq C\sqrt{\textsf{DoF}_m}\cdot b_0,\quad\linft{\bcalT_{l_1+l_2}-\bcalT^*}\leq C(5m+1)\sqrt{3^m\mu^{*m}r^*}\max_{k=1,\dots,m}(\dkm)^{-1/2}\cdot \sqrt{\textsf{DoF}_m}\cdot b_0.
\end{align*}

\begin{proof}[Proof of Claim~\ref{claim:leave-one-out}]
	First consider fixed $j$ and $k$. Notice that $\bcalG_l^{(k),j}-\bar{\bcalG_l}^{(k),j}$ is a mean zero tensor and only has non-zero entries on the $j$-th slice of order $k$. For simiplicity, we denote
	\begin{align*}
		\fraM_{k}(\bcalG_l^{(k),j}-\bar{\bcalG_l}^{(k),j})=\left(\begin{matrix}
			0&0&\ldots&0\\
			\vdots&\vdots& &\vdots\\
			0&0&\cdots&0\\
			x_{1}&x_{2} & &x_{\dkm}\\
			0&0&\cdots&0\\
			\vdots&\vdots& &\vdots\\
			0&0&\ldots&0\\
		\end{matrix}\right)\in\RR^{d_k\times\dkm}.
	\end{align*}
	 And denote $\x=(x_1,\dots,x_{\dkm})^{\top}$. Recall that $\calP_{\Omega_{j}^{(k)}}(\bXi)$ and $\bcalT_l^{(k),j}$ are independent. Then consider $\calP_{\TT_l^{(k),j}}\left(\bcalG_l^{(k),j}-\bar{\bcalG_l}^{(k),j}\right)$ and by Riemannian projection definition, we obtian
	 \begin{align*}
	 	\fro{\calP_{\Omega_{j}^{(k)}}\calP_{\TT_l^{(k),j}}\left(\bcalG_l^{(k),j}-\bar{\bcalG_l}^{(k),j}\right)}&\leq\sum_{i=1}^{m+1}c_i\ltwo{\x^{\top}\V_i},
	 \end{align*}
 where $c_i\leq 1$ and orthogonal matrices $\V_i\in\RR^{\dkm\times r_k^-}$ are independent of $\calP_{\Omega_{j}^{(k)}}(\bXi)$. Also notice that $|x_i|\leq2$. Suppose $\V_i=[\v_{i1},\dots,\v_{ir_k^-}]$ are columns of $\V_i$. The Orlicz norm could be bounded with $\op{\x^{\top}\v_{il}}_{\Psi_2}\leq 2$, where Hoeffeding Inequality is used. It leads to $\op{\ltwo{\x^{\top}\V_i}^2}_{\Psi_1}\leq \sum_{l=1}^{r_k^-} \op{\x^{\top}\v_{il}}_{\Psi_2}^2\leq 4r_k^-$. Thus, we have $$\fro{\calP_{\Omega_{j}^{(k)}}\calP_{\TT_l^{(k),j}}\left(\bcalG_l^{(k),j}-\bar{\bcalG_l}^{(k),j}\right)}\leq C(m+1)\sqrt{r_k^-\log d^*}$$ holds with probability exceeding $1-cd^{*-8}$. Taking the union over $k=1,\dots,m$ and $j=1,\dots,d_k$ and then we obtain Claim~\ref{claim:leave-one-out}, where $r_k^-<r^*$ is used.
\end{proof}
\section{Proofs under Heavy-Tailed Noise and Sparse Arbitrary Corruptions}
To simplify the notation, we use $\calP_{\Omega_{j}^{(k)}}(\cdot)$ to represent mask operator of the $j$ th slice of a tensor by the $k$ th order, namely,
\begin{align*}
	\left[\calP_{\Omega_{j}^{(k)}}(\bcalT)\right]_{i_1\dots i_m}:=\left\{
	\begin{array}{lcl}
		\left[\bcalT\right]_{i_1\cdots i_m}&     &\text{ if } i_k=j\\
		0&     &\text{ if } i_k\neq j
	\end{array}
	\right..
\end{align*}
Hence, after simple calculations, we have
\begin{align*}
	\lone{\fraM_{k}(\bcalT-\bcalY)_{j,\cdot}}-\lone{\fraM_{k}(\bcalT^*-\bcalY)_{j,\cdot}}&=f\left(\calP_{\Omega_{j}^{(k)}}(\bcalT)\right)-f\left(\calP_{\Omega_{j}^{(k)}}(\bcalT^*)\right)\\
	&=\lone{\calP_{\Omega_{j}^{(k)}}\left(\bcalT-\bcalT^*-\bXi\right)}-\lone{\calP_{\Omega_{j}^{(k)}}\left(\bXi\right)}.
\end{align*}
\subsection{Proof of Lemma~\ref{lem:abs:regularity}}
\paragraph{Phase One Analysis}
We shall prove phase one properties under event $\bcalE_1$,
$$\bcalE_1:=\left\{\lone{\calP_{\Omega_{j}^{(k)}}(\bXi)}\leq 3\dkm\gamma,\quad\text{ for all } k=1,\dots,m,\; j=1,\dots,d_k\right\}.$$
Specifically, Lemma~\ref{teclem:Contraction of Heavy Tailed Random Variables:slice} proves $\PP(\bcalE_1)\geq 1-c\sum_{k=1}^{m} d_k(\dkm)^{-1-\min\{1,\eps\}}$. First consider the projected sub-gradient term. Notice that absolute values of the sub-gradient entries $\bcalG$ are bounded by $1$, which leads to
\begin{align*}
	\fro{\calP_{\TT}(\bcalG)}^2&=\fro{\bcalG}^2-\fro{\calP_{\TT}^{\perp}(\bcalG)}^2\leq \fro{\bcalG}^2\leq d^*.
\end{align*}
It verifies $\fro{\calP_{\TT}(\bcalG)}\leq \sqrt{d^*}$. Then consider $f(\bcalT)-f(\bcalT^*)$,
\begin{align*}
	f(\bcalT)-f(\bcalT^*)&=\sum_{(i_1,\dots,i_m)\in\Omega}\left(\left|[\bcalT]_{i_1\cdots i_m}-[\bcalT^*]_{i_1\cdots i_m}-\xi_{i_1\cdots i_m}-[\bcalS]_{i_1\cdots i_m}\right|-\left|\xi_{i_1\cdots i_m}+[\bcalS]_{i_1\cdots i_m}\right|\right)\\
	&{~~~~~}+\sum_{(i_1,\dots,i_m)\notin\Omega}\left(\left\vert[\bcalT]_{i_1\cdots i_m}-[\bcalT^*]_{i_1\cdots i_m}-\xi_{i_1\cdots i_m}\right\vert-|\xi_{i_1\cdots i_m}|\right)\\
	&\geq -\lone{\calP_{\Omega}(\bcalT-\bcalT^*)}+\lone{\calP_{\Omega^{C}}(\bcalT-\bcalT^*)}-2\lone{\bXi}\\
	&= \lone{\bcalT-\bcalT^*}-2\lone{\calP_{\Omega}(\bcalT-\bcalT^*)}-2\lone{\bXi},
\end{align*}
where the inequality use triangle inequality. Note that under event $\bcalE_1$, it has $\lone{\bXi}\leq 3d^*\gamma$. Then by relationship among $\lone{\cdot}$, $\linft{\cdot}$, $\fro{\cdot}$ in Lemma~\ref{teclem:norm-relation}, we get
\begin{align*}
	f(\bcalT)-f(\bcalT^*)\geq \linft{\bcalT-\bcalT^*}^{-1}\cdot\fro{\bcalT-\bcalT^*}^2-2 \lone{\calP_{\Omega}(\bcalT-\bcalT^*)}-6d^*\gamma.
\end{align*}
Also, note that $\# \Omega\leq \alpha d^*$ and it infers that $\lone{\calP_{\Omega}(\bcalT-\bcalT^*)}\leq\alpha\linft{\bcalT-\bcalT^*}$. In conclusion, we obtain $$f(\bcalT)-f(\bcalT^*)\geq \linft{\bcalT-\bcalT^*}^{-1}\cdot\left(\fro{\bcalT-\bcalT^*}^2-\alpha d^*\linft{\bcalT-\bcalT^*}^2\right)-6d^*\gamma.$$
When $\bcalT$ is low-rank and incoherent, we could have delicate bound for slice of the projected sub-gradient. Suppose $\bcalT=\bcalC\cdot\llbracket\U_1,\dots,\U_m\rrbracket$ is the Tucker decomposition. In this way, matricization of the projected sub-gradient is,
\begin{align*}
	&{~~~}\fraM_k(\calP_{\TT}(\bcalG))\\
	&=\fraM_k(\bcalG)\left(\otimes_{i\neq k}\U_i\right)\fraM_k(\bcalC)^{\dagger}\fraM_k(\bcalC)\left(\otimes_{i\neq k}\U_i\right)^{\top}+\U_k\U_k^{\top}\fraM_k(\bcalG)\left(\otimes_{i\neq k}\U_i\right)\left(\I-\fraM_k(\bcalC)^{\dagger}\fraM_k(\bcalC)\right)\left(\otimes_{i\neq k}\U_i\right)^{\top}\\
	&{~~~}+\sum_{i\neq k} \U_k\fraM_k(\bcalC\times_{j\neq i,k}\U_j\times\V_i),
\end{align*}
where $\V_i:=\left(\I_{d_i}-\U_i\U_i^{\top}\right)\fraM_k(\bcalG)\left(\otimes_{j\neq i}\U_j\right)\fraM_k(\bcalC)^{\dagger}$. Then with the inequality $\ltinf{\A\B}\leq\ltinf{\A}\fro{\B}$, we have
\begin{align*}
	\ltinf{\fraM_k(\calP_{\TT}(\bcalG))}^2&\leq 2\ltinf{\U_k}^2\fro{\bcalG}^2+\ltinf{\fraM_k(\bcalG)\left(\otimes_{i\neq k}\U_i\right)\fraM_k(\bcalC)^{\dagger}\fraM_k(\bcalC)\left(\otimes_{i\neq k}\U_i\right)^{\top}}^2\\
	&\leq 3\frac{\mu r_k}{d_k}\cdot d_1\cdots d_m=3\mu r_k\cdot \dkm.
\end{align*}
On the other hand, by triangle inequality, the slice loss function has a lower bound
\begin{align*}
	\lone{\calP_{\Omega_{j}^{(k)}}(\bcalT-\bcalY)}&-\lone{\calP_{\Omega_{j}^{(k)}}(\bcalT^*-\bcalY)}=\lone{\calP_{\Omega_{j}^{(k)}}\calP_{\Omega^C}(\bcalT-\bcalT^*-\bXi)}-\lone{\calP_{\Omega_{j}^{(k)}}\calP_{\Omega^C}(\bXi)}\\
	&{~~~}+\lone{\calP_{\Omega_{j}^{(k)}}\calP_{\Omega}(\bcalT-\bcalT^*-\bXi-\bcalS)}-\lone{\calP_{\Omega_{j}^{(k)}}\calP_{\Omega}(\bXi+\bcalS)}\\
	&\geq\lone{\calP_{\Omega_{j}^{(k)}}(\bcalT-\bcalT^*)}-2\lone{\calP_{\Omega_{j}^{(k)}}\calP_{\Omega}(\bcalT-\bcalT^*)}-2\lone{\calP_{\Omega_{j}^{(k)}}(\bXi)}\\
	&\geq \frac{1}{\linft{\calP_{\Omega_{j}^{(k)}}(\bcalT-\bcalT^*)}}\left(\fro{\calP_{\Omega_{j}^{(k)}}(\bcalT-\bcalT^*)}^2-2\alpha\dkm\linft{\calP_{\Omega_{j}^{(k)}}(\bcalT-\bcalT^*)}^2\right)-6d_k^- \gamma,
\end{align*}
where the last line uses Lemma~\ref{teclem:norm-relation} and event $\bcalE_1$.
\paragraph{Phase Two Analysis}
Denote $f_0(\bcalT):=\lone{\bcalT-\bcalT^*-\bXi}$ for simplicity. In phase two analyses, we shall assume the event $$\bcalE_2:=\left\{\sup_{\bcalT\in\RR^{d_1\times\cdots\times d_m}, \Delta\bcalT\in\MM_{2\r}}\left|f_0(\bcalT+\Delta\bcalT)-f_0(\bcalT)-\EE\left(f_0(\bcalT+\Delta\bcalT)-f_0(\bcalT) \right)\right|\cdot\fro{\Delta\bcalT}^{-1}\leq C\sqrt{\textsf{DoF}_m}\right\}$$
holds. Specifically, 
Lemma~\ref{teclem:empirical} proves $\PP(\bcalE_2)\geq1-\exp(-\textsf{DoF}/2)$. Then under event $\bcalE_2$, we have a lower bound of $f(\bcalT)-f(\bcalT^*)$,
\begin{align*}
	f(\bcalT)-f(\bcalT^*)&\geq \EE[f(\bcalT)-f(\bcalT^*)]-C\sqrt{\textsf{DoF}_m}\fro{\bcalT-\bcalT^*}.
\end{align*}
Besides, \begin{align*}
	&{~~~}\EE f(\bcalT)-\EE f(\bcalT^*)\\
	&=\EE\left[\lone{\calP_{\Omega^{C}}(\bcalT-\bcalT^*-\bXi)}-\lone{\calP_{\Omega^C}(\bXi)}\right]+\EE\left[\lone{\calP_{\Omega}(\bcalT-\bcalT^*-\bXi-\bcalS)}-\lone{\calP_{\Omega}(\bXi+\bcalS)}\right]\\
	&=\EE\left[\lone{\bcalT-\bcalT^*-\bXi}-\lone{\bXi}\right]-\EE\left[\lone{\calP_{\Omega}(\bcalT-\bcalT^*-\bXi)}-\lone{\calP_{\Omega}(\bXi)}\right]\\
	&{~~~}+\EE\left[\lone{\calP_{\Omega}(\bcalT-\bcalT^*-\bXi-\bcalS)}-\lone{\calP_{\Omega}(\bXi+\bcalS)}\right].
\end{align*}
Note that Lemma~\ref{teclem:abs:function expectation} proves $\EE\left[\lone{\bcalT-\bcalT^*-\bXi}-\lone{\bXi}\right]\geq b_0^{-1}\fro{\bcalT-\bcalT^*}^2$. By triangle inequality, Holder Inequality and Lemma~\ref{teclem:alpha-bound}, we have
\begin{align*}
	\left|\EE\left[\lone{\calP_{\Omega}(\bcalT-\bcalT^*-\bXi)}-\lone{\calP_{\Omega}(\bXi)}\right]\right| &\leq\lone{\calP_{\Omega}(\bcalT-\bcalT^*)}\leq \sqrt{\alpha d^*}\fro{\calP_{\Omega}(\bcalT-\bcalT^*)}\\
	&\leq 2\alpha\sqrt{(m+1)(\muT\vee\mu)^mr^*d^*}\fro{\bcalT-\bcalT^*},
\end{align*}
and
\begin{align*}
	\left|\EE\left[\lone{\calP_{\Omega}(\bcalT-\bcalT^*-\bXi-\bcalS)}-\lone{\calP_{\Omega}(\bXi+\bcalS)}\right] \right|&\leq\lone{\calP_{\Omega}(\bcalT-\bcalT^*)}\\
	&\leq 2\alpha\sqrt{(m+1)(\muT\vee\mu)^mr^*d^*}\fro{\bcalT-\bcalT^*}.
\end{align*}
Thus, we obtain the following lower bound of $f(\bcalT-\bcalT^*)$,
\begin{align*}
	f(\bcalT-\bcalT^*)&\geq b_0^{-1}\fro{\bcalT-\bcalT^*}^2-4\alpha\sqrt{(m+1)(\muT\vee\mu)^mr^*d^*}\fro{\bcalT-\bcalT^*}-C\sqrt{\textsf{DoF}_m}\fro{\bcalT-\bcalT^*}\\
	&\geq\frac{1}{2b_0}\fro{\bcalT-\bcalT^*}^2,
\end{align*}
where the last inequality is due to the phase two region $$\fro{\bcalT-\bcalT^*}\geq Cb_0\cdot\max\{\sqrt{\textsf{DoF}_m}, \alpha \sqrt{(m+1)(\muT\vee\mu)^m r^*d^*}\}.$$ The following lemma shall inherit notations and assumptions in Lemma~\ref{lem:abs:regularity}. It provides upper bound for the projected sub-gradient and finishes the proof.
\begin{lemma}[Upper bound for projected sub-gradient]
	Let $\bcalT\in\MM_{\r,\mu}$ statisfy $\fro{\bcalT-\bcalT^*}\geq Cb_0\cdot\max\{\sqrt{\textsf{DoF}_m},\alpha \sqrt{(m+1)(\muT\vee\mu)^m r^*d^*}\}$. Let $\bcalG\in\partial f(\bcalT)$ be the sub-gradient and $\TT$ be the tangent space of $\MM_{\r}$ at point $\bcalT$. Then under event $\bcalE_2$, we have $$\fro{\calP_{\TT}(\bcalG)}\leq c_1\cdot \sqrt{m+1}\cdot b_1^{-1}\fro{\bcalT-\bcalT^*}.$$
\end{lemma}
\begin{proof}
	Note that $\fro{\calP_{\TT}(\bcalG)}$ has the upper bound
	\begin{align*}
		\fro{\calP_{\TT}(\bcalG)}^2&=\fro{\bcalG\times_1\U_1\U_1^{\top}\times_2\cdots\times_m\U_m\U_m^{\top}}^2\\
		&{~~~~~}+\sum_{k=1}^{m}\fro{\left(\I_{d_k}-\U_k\U_k^{\top}\right)\fraM_k(\bcalG)\left(\otimes_{i\neq k}\U_i\right)\fraM_k(\bcalC^*)^{\dag} \fraM_k(\bcalC^*)\left(\otimes_{i\neq k}\U_i\right)^{\top} }^2\\
		&\leq \underbrace{\frorr{\bcalG}^2}_{A_1}+\underbrace{\sum_{k=1}^{m}\fro{\fraM_k(\bcalG)\left(\otimes_{i\neq k}\U_i\right)\fraM_k(\bcalC^*)^{\dag} \fraM_k(\bcalC^*)\left(\otimes_{i\neq k}\U_i\right)^{\top} }^2}_{A_2},
	\end{align*}
	where $\frorr{\bcalG}:=\sup_{\W_j\in\OO_{d_j,r_j}}\fro{\bcalG\times_1\W_1\W_1^{\top}\times_2\cdots\times_m\W_m\W_m^{\top} }$.
	
	\textbf{First consider $A_1$.}
	Suppose $\bcalG$ achieves $\frorr{\cdot}$ with orthogonal matrices $\V_k\in\OO_{d_k,r_k}$, namely, $$\frorr{\bcalG}=\fro{\bcalG\times_1\V_1\V_1^{\top}\times_2\cdots\times_m\V_m\V_m^{\top}},\quad \text{ for all } k=1,\cdots,m.$$ Then take $\bcalS=\bcalT+\frac{1}{2}b_1\cdot\bcalG\times_1\V_1\V_1^{\top}\times_2\cdots\times_m\V_m\V_m^{\top}$. 
	By definition of sub-gradient and by triangular inequality, we have
	\begin{equation}
		\begin{split}
		\inp{\bcalM-\bcalT}{\bcalG}\leq f(\bcalM)-f(\bcalT)&=\lone{\bcalM-\bcalT^*-\bXi}-\lone{\bcalT-\bXi}\\
		&~~~+\lone{\calP_{\Omega}(\bcalM-\bcalT^*-\bXi-\bcalS)}- \lone{\calP_{\Omega}(\bcalT-\bcalT^*-\bXi-\bcalS)}\\
		&~~~-\lone{\calP_{\Omega}(\bcalM-\bcalT^*-\bXi)}+ \lone{\calP_{\Omega}(\bcalT-\bcalT^*-\bXi)}\\
		&\leq \lone{\bcalM-\bcalT^*-\bXi}-\lone{\bcalT-\bXi}+2\lone{\calP_{\Omega}(\bcalM-\bcalT)}.
		\label{eq2}
		\end{split}
	\end{equation}
On the other hand, event $\bcalE_2$ and  Lemma~\ref{teclem:abs:function expectation} imply that \begin{align*}
	\lone{\bcalM-\bcalT^*-\bXi}-\lone{\bcalT-\bXi}&\leq \frac{1}{b_1}\left(\fro{\bcalM-\bcalT}^2+2\fro{\bcalM-\bcalT}\fro{\bcalT-\bcalT^*}\right)+C\sqrt{\textsf{DoF}_m}\fro{\bcalM-\bcalT}.
\end{align*}
 Also note that 
	\begin{align*}
		\lone{\calP_{\Omega}(\bcalM-\bcalT)}\leq \sqrt{\alpha d^*}\fro{\bcalM-\bcalT}&=0.5b_1\sqrt{\alpha d^*}\fro{\bcalG\times_1\V_1\V_1^{\top}\times_2\cdots\times_m\V_m\V_m^{\top}}\\
		&= 0.5b_1\sqrt{\alpha d^*}\frorr{\bcalG}.
	\end{align*}
 Insert $\bcalS=\bcalT+\frac{1}{2}b_1\cdot\bcalG\times_1\V_1\V_1^{\top}\times_2\cdots\times_m\V_m\V_m^{\top}$ into Equation~\eqref{eq2} and with $\fro{\bcalT-\bcalT^*}\geq b_0\cdot\max\{\sqrt{\textsf{DoF}_m}, \alpha\sqrt{(m+1)\bar{\mu}^m r^*d^*}\}$, $b_0\geq b_1$, we obtain
	\begin{align*}
		\frac{1}{2}b_1 \frorr{\bcalG}^2& \leq \frac{1}{4}b_1\frorr{\bcalG}^2+\fro{\bcalT-\bcalT^*}\frorr{\bcalG}+0.5b_1\sqrt{\alpha d^*}\frorr{\bcalG}+C b_1 \sqrt{\textsf{DoF}_m}\frorr{\bcalG}\\
		&\leq \frac{1}{4}b_1\frorr{\bcalG}^2+C \fro{\bcalT-\bcalT^*}\frorr{\bcalG}.
	\end{align*}
	By solving the sbove quadratic inequality of $\frorr{\bcalG}$, we get \begin{align*}
		\frorr{\bcalG}\leq c_1 b_1^{-1}\cdot\fro{\bcalT-\bcalT^*}.
	\end{align*}
	
	\textbf{Second consider $A_2$.} Note that $\fraM_k(\bcalG)\left(\otimes_{i\neq k}\U_i\right)\fraM_k(\bcalC^*)^{\dag} \fraM_k(\bcalC^*)\left(\otimes_{i\neq k}\U_i\right)^{\top}$ is the $k$-th matricization of some Tucker rank at most $\r$ tensor.
	Then by same analysis trick as $A_1$, we have \begin{align*}
		\fro{\fraM_k(\bcalG)\left(\otimes_{i\neq k}\U_i\right)\fraM_k(\bcalC^*)^{\dag} \fraM_k(\bcalC^*)\left(\otimes_{i\neq k}\U_i\right)^{\top}}\leq c_1\cdot b_1^{-1}\cdot\fro{\bcalT-\bcalT^*}.
	\end{align*}
	Finally, we have $\fro{\calP_{\TT}(\bcalG)}^2\leq (m+1)c_1^2b_{1}^{-2}\fro{\bcalT-\bcalT^*}^2$, which leads to
	\begin{align*}
		\fro{\calP_{\TT}(\bcalG)}\leq c_1\cdot \sqrt{m+1}\cdot b_{1}^{-1}\fro{\bcalT-\bcalT^*}
	\end{align*}
\end{proof}

\subsection{Proof of Theorem~\ref{thm:abs:dynamics}}
\subsubsection{Phase One}
For convenience, denote $D_l:=\left(1-\frac{1}{32}(5m+1)^{-2}(3^m\mu^{*m} r^*)^{-1}\right)^{l}\cdot D_0$. We shall prove the following Euqation~\eqref{eq51}-\eqref{eq54} by induction. It's obvious that it holds for the initialization $\bcalT_0$. Suppose it holds for iteration $l$ and we consider the $(l+1)$-th iteration. We need to prove 
\begin{subequations}
	\begin{align}
		\fro{\bcalT_{l+1}-\bcalT^*}&\leq D_{l+1}
		\label{eq51}\\
		\ltinf{\bcalT_{l+1}-\bcalT^*}&\leq 3\sqrt{\frac{\muT r_k}{d_k}}\cdot D_{l+1}
		\label{eq52}\\
		\ltinf{\left(\U_k^{(l+1)}\H_k^{(l+1)}-\U_k^*\right)\fraM_{k}(\bcalC^*)}&\leq5\sqrt{\frac{\muT r_k}{d_k}}\cdot D_{l+1}
		\label{eq53}\\
		\linft{\bcalT_{l+1}-\bcalT^*}&\leq (5m+1)\sqrt{\frac{3^m\mu^{*m} r^*}{d^*}}\cdot D_{l+1}
		\label{eq54}\\
		\linft{\U_k^{(l+1)}}&\leq\sqrt{\frac{3\muT r_k}{d_k}}.\label{eq55}
	\end{align}
	\label{eq5}
\end{subequations}

\paragraph*{Frobenius norm}
First consider $\fro{\bcalT_l-\bcalT-\eta_l\calP_{\TT_l}(\bcalG_l)}$,
\begin{align*}
	\fro{\bcalT_l-\eta_l\calP_{\TT_l} (\bcalG_l)-\bcalT^*}^2=\fro{\bcalT_l-\bcalT^*}^2-2\eta_l\inp{\bcalT_l-\bcalT^*}{\calP_{\TT_l}(\bcalG_l)} + \eta_l^2\fro{\calP_{\TT_l}(\bcalG_l)}^2.
\end{align*}
We have analyzed the last term in Lemma~\ref{lem:abs:regularity}, which has $\fro{\calP_{\TT_l}(\bcalG_l)}^2\leq d^*$.
Note that by definition of sub-gradient and analyses of $f(\bcalT)-f(\bcalT^*)$ in Lemma~\ref{lem:abs:regularity}, the intermediate term has the lower bound
\begin{align*}
	\inp{\bcalT_l-\bcalT^*}{\calP_{\TT_l}(\bcalG_l)}&=\inp{\bcalT_l-\bcalT^*}{\bcalG_l}-\inp{\calP_{\TT_l}^{\perp}(\bcalT_l-\bcalT^*)}{\bcalG_l}\\
	&\geq f(\bcalT_l)-f(\bcalT^*)-\inp{\calP_{\TT_l}^{\perp}\bcalT^*}{\bcalG_l}\\
	&\geq  \linft{\bcalT_l-\bcalT^*}^{-1}\cdot\left(\fro{\bcalT_l-\bcalT^*}^2-2\alpha d^* \linft{\bcalT_l-\bcalT^*}^2\right)-6d^*\gamma-\inp{\calP_{\TT_l}^{\perp}\bcalT^*}{\bcalG_l}
\end{align*}
Besides, Lemma~\ref{teclem:rieman-orthogonal-project} shows $\left|\inp{\calP_{\TT_l}^{\perp}\bcalT^*}{\bcalG_l}\right|\leq \fro{\calP_{\TT_l}^{\perp}\bcalT^*}\cdot \fro{\bcalG_l}\leq 8m^2\sqrt{d^*}\mins^{*-1}\fro{\bcalT_l-\bcalT^*}^2$. Hence, we have 
\begin{align*}
	\fro{\bcalT_l-\eta_l\calP_{\TT_l} (\bcalG_l)-\bcalT^*}^2&\leq \fro{\bcalT_l-\bcalT^*}^2-2\eta_l\linft{\bcalT_l-\bcalT^*}^{-1}\cdot\fro{\bcalT_l-\bcalT^*}^2+4\eta_{l}\alpha d^*\linft{\bcalT_{l}-\bcalT^*}\\
	&{~~~~~~~~~~~~~~~~~~~~~~~~~~}+12\eta_{l}d^*\gamma+16\eta_{l}m^2\sqrt{d^*}\mins^{*-1}\fro{\bcalT_l-\bcalT^*}^2+\eta_{l}^2 d^*,
\end{align*}
Then insert the induction of $\bcalT_{l}$ into the above equation $\fro{\bcalT_l-\bcalT^*}\leq D_l$ and $\linft{\bcalT_l-\bcalT^*}\leq (5m+1)\sqrt{\frac{3^m\mu^{*m} r^*}{d^*}}\cdot D_l $ into the above equation, which is slimiar to pseudo-Huber loss case Section~\ref{proof:pseudo-Huber:phaseone},
\begin{align*}
	\fro{\bcalT_l-\eta_l\calP_{\TT_l} (\bcalG_l)-\bcalT^*}^2&\leq D_l^2-2\eta_{l}(5m+1)^{-1}\sqrt{\frac{d^*}{3^m\mu^{*m} r^*}}D_l+4\eta_{l}\alpha (5m+1)\sqrt{3^m\mu^{*m} r^*d^*}D_l+12\eta_{l}d^*\gamma\\
	&{~~~~~~~~~~~~~~~~~~~~~~~~~~~~~~~~~~~~~~~~~~~~~~~~~~~~~}+16\eta_{l}m^2\sqrt{d^*}\mins^{*-1}D_l^2+\eta_{l}^2 d^*\\
	&\leq D_l^2-\frac{2}{3}\eta_l(5m+1)^{-1}\sqrt{\frac{d^*}{3^m\mu^{*m} r^*}}D_l +16\eta_{l}m^2\sqrt{d^*}\mins^{*-1}D_l^2+\eta_{l}^2 d^*\\
	&\leq D_l^2-\frac{1}{2}\eta_l(5m+1)^{-1}\sqrt{\frac{d^*}{3^m\mu^{*m} r^*}}D_l +\eta_{l}^2 d^*,
\end{align*}
where the second inequality uses phase one region constraints $D_l\geq 12(5m+1)\sqrt{3^m\mu^{*m} r^*d^*}\gamma$, corruption rate $\alpha\leq\frac{1}{12(5m+1)^23^m\mu^{*m}r^*}$ and the last line is from initialization condition. Then with the stepsize $\eta_l \in\frac{1}{8(5m+1)\sqrt{3^m\mu^{*m}r^*d^*}}\cdot D_l\cdot \left[1,3\right]$, we obtain
$$\fro{\bcalT_l-\eta_l\calP_{\TT_l} (\bcalG_l)-\bcalT^*}^2\leq \left( 1-\frac{3}{64}(5m+1)^{-2}(3^m\mu^{*m} r^*)^{-1}\right)D_l^2. $$
Note that $\bcalT_{l+1}=\textrm{HOSVD}(\bcalT_l-\eta_{l}\calP_{\TT_l}(\bcalG_l))$ and perturbation bound Theorem~\ref{teclem:perturbation:tensor} implies \begin{align*}
	\fro{\bcalT_{l+1}-\bcalT^*}\leq \left( 1-\frac{1}{64}(5m+1)^{-2}(3^m\mu^{*m} r^*)^{-1}\right)D_l=D_{l+1},
\end{align*}
where we use $D_{l}< D_0\leq c\mins^*\cdot (5m+1)^{-2}(3^m\mu^{*m}r^*)^{-1}$.
\paragraph*{Entrywise norm} \label{subsec:abs-phaseone-entry}
For each $k=1,\dots,m$, consider $\ltinf{\fraM_k\left(\bcalT_l-\bcalT^*-\eta_{l}\calP_{\TT_l}(\bcalG_l)\right)}$,  or equivalently, consider $$\fro{\calP_{\Omega_{j}^{(k)}}\left(\bcalT_l-\bcalT^*-\eta_{l}\calP_{\TT_l}(\bcalG_l)\right) },\quad \text{ for each }j=1,\dots,d_k.$$
Note that
\begin{align*}
	\fro{\calP_{\Omega_{j}^{(k)}}\left(\bcalT_l-\bcalT^*-\eta_{l}\calP_{\TT_l}(\bcalG_l)\right) }^2&= \fro{\calP_{\Omega_{j}^{(k)}}\left(\bcalT_l-\bcalT^*\right)}^2\\
	&{~~~~~~~~~~}-2\eta_{l}\inp{\calP_{\Omega_{j}^{(k)}}(\bcalT_l-\bcalT^*)}{\calP_{\Omega_{j}^{(k)}}(\calP_{\TT_l}(\bcalG_l))}+\eta_{l}^2\ltwo{\calP_{\Omega_{j}^{(k)}}(\calP_{\TT_l}(\bcalG_l))}^2.
\end{align*}
With $\ltinf{\U_k^{(l)}}\leq\sqrt{\frac{3\muT r_k}{d_k}}$, Lemma~\ref{lem:abs:regularity} provides an upper bound for the last term $$\ltwo{\calP_{\Omega_{j}^{(k)}}(\calP_{\TT_l}(\bcalG_l))}^2\leq9\frac{\muT r_k}{d_k}d^*.$$ Then consider the intermidiate term $\inp{\calP_{\Omega_{j}^{(k)}}(\bcalT_l-\bcalT^*)}{\calP_{\Omega_{j}^{(k)}}(\calP_{\TT_l}(\bcalG_l))}= \inp{\calP_{\Omega_{j}^{(k)}}(\bcalT_l)}{\calP_{\Omega_{j}^{(k)}}(\calP_{\TT_l}(\bcalG_l))}-\inp{\calP_{\Omega_{j}^{(k)}}(\bcalT^*)}{\calP_{\Omega_{j}^{(k)}}(\calP_{\TT_l}(\bcalG_l))}$. Note that simple calculations lead to
$$\inp{\calP_{\Omega_{j}^{(k)}}(\bcalT_l)}{\calP_{\Omega_{j}^{(k)}}(\calP_{\TT_l}(\bcalG_l))}=\inp{\calP_{\Omega_{j}^{(k)}}(\bcalT_l)}{\calP_{\Omega_{j}^{(k)}}(\bcalG_l)},$$
and
\begin{equation}
	\label{eq8}
	\begin{split}
		\inp{\calP_{\Omega_{j}^{(k)}}(\bcalT^*)}{\calP_{\TT_l}(\bcalG_l)}&=\inp{\calP_{\TT_l}\calP_{\Omega_{j}^{(k)}}(\bcalT^*)}{\bcalG_l}\\
		&=\inp{\calP_{\TT_l}\calP_{\Omega_{j}^{(k)}}\calP_{\TT_l}(\bcalT^*)}{\bcalG_l}+\inp{\calP_{\TT_l}\calP_{\Omega_{j}^{(k)}}\calP_{\TT_l}^{\perp}(\bcalT^*)}{\bcalG_l}\\
		&=\inp{\calP_{\Omega_{j}^{(k)}}\calP_{\TT_l}(\bcalT^*)}{\bcalG_l}+\inp{\calP_{\TT_l}\calP_{\Omega_{j}^{(k)}}\calP_{\TT_l}^{\perp}(\bcalT^*)}{\bcalG_l}\\
		&=\inp{\calP_{\Omega_{j}^{(k)}}(\bcalT^*)}{\bcalG_l}-\inp{\calP_{\Omega_{j}^{(k)}}\calP_{\TT_l}^{\perp}(\bcalT^*)}{\bcalG_l}+\inp{\calP_{\Omega_{j}^{(k)}}\calP_{\TT_l}^{\perp}(\bcalT^*)}{\calP_{\TT_l}(\bcalG_l)}\\
		&=\inp{\calP_{\Omega_{j}^{(k)}}(\bcalT^*)}{\calP_{\Omega_{j}^{(k)}}(\bcalG_l)}-\inp{\calP_{\Omega_{j}^{(k)}}\calP_{\TT_l}^{\perp}(\bcalT^*)}{\calP_{\Omega_{j}^{(k)}}\bcalG_l}+\inp{\calP_{\Omega_{j}^{(k)}}\calP_{\TT_l}^{\perp}(\bcalT^*)}{\calP_{\Omega_{j}^{(k)}}\calP_{\TT_l}(\bcalG_l)}.
	\end{split}
\end{equation}
With Lemma~\ref{teclem:rieman-orthogonal-project}, we have
\begin{align*}
	&{~~~~}\bigg| \inp{\calP_{\Omega_{j}^{(k)}}\calP_{\TT_l}^{\perp}(\bcalT^*)}{\calP_{\Omega_{j}^{(k)}}(\bcalG_l)}\bigg|\\
	&\leq \sqrt{\dkm}\fro{\bcalT_l-\bcalT^*}\left(m^2\ltinf{\U_k^{(l)}}\frac{\fro{\bcalT_l-\bcalT^*}}{\mins^*}+m\ltinf{\U_k^{(l)}\U_k^{(l)\top}-\U_k^*\U_k^{*\top}}\right)=:B_1,
\end{align*}
and
\begin{align*}
	&{~~~~}\bigg|\inp{\calP_{\Omega_{j}^{(k)}}\calP_{\TT_l}^{\perp}(\bcalT^*)}{\calP_{\Omega_{j}^{(k)}}\calP_{\TT_l}(\bcalG_l)} \bigg|\\
	&\leq 3\sqrt{\frac{\muT r_k}{d_k}\cdot d^*}\fro{\bcalT_l-\bcalT^*}\left(m^2\ltinf{\U_k^{(l)}}\frac{\fro{\bcalT_l-\bcalT^*}}{\mins^*}+m\ltinf{\U_k^{(l)}\U_k^{(l)\top}-\U_k^*\U_k^{*\top}}\right)=:B2.
\end{align*}
Note that with induction $\ltinf{\left(\U_k^{(l)}\H_k^{(l)}-\U_k^*\right)\fraM_k(\bcalC^*) }\leq 4\sqrt{\frac{\muT r_k}{d_k}}\cdot D_{l}$, we have $\ltinf{\U_k^{(l)}\H_k^{(l)}-\U_k^* }\leq 4\sqrt{\frac{\muT r_k}{d_k}}\cdot\mins^{*-1}D_{l}$. Also, Lemma~\ref{teclem:ltinf-transformation} shows $$\ltinf{\U_k^{(l)}\U_k^{(l)\top}-\U_k^*\U_k^{*\top} }\leq 8\mins^{*-1}\sqrt{\frac{\muT r_k}{d_k}}D_l.$$ In this way, we have
\begin{align*}
	B_1\vee B_2\leq 16m^2\mins^{*-1}\sqrt{d^*}\frac{\muT r_k}{d_k}D_l^2.
\end{align*}
Also, by definition of sub-gradient and by analysis in Lemma~\ref{lem:abs:regularity}, we have \begin{align*}
	\inp{\calP_{\Omega_{j}^{(k)}}(\bcalT_l-\bcalT^*)}{\calP_{\Omega_{j}^{(k)}}(\bcalG_l)}&\geq \lone{\calP_{\Omega_{j}^{(k)}}(\bcalT_{l}-\bcalY)}-\lone{\calP_{\Omega_{j}^{(k)}}(\bcalT^*-\bcalY)}\\
	&\geq\linft{\calP_{\Omega_{j}^{(k)}}(\bcalT_l-\bcalT^*)}^{-1}\cdot\fro{\calP_{\Omega_{j}^{(k)}}(\bcalT_l-\bcalT^*)}^2-2\alpha\dkm \linft{\calP_{\Omega_{j}^{(k)}}(\bcalT-\bcalT^*)}- 6\dkm\gamma.
\end{align*}
Thus, the intermediate term has the lower bound,
\begin{align*}
	\bigg|\inp{\calP_{\Omega_{j}^{(k)}}(\bcalT_l-\bcalT^*)}{\calP_{\Omega_{j}^{(k)}}(\calP_{\TT_l}(\bcalG_l))} \bigg|&\geq \linft{\calP_{\Omega_{j}^{(k)}}(\bcalT_l-\bcalT^*)}^{-1}\cdot\fro{\calP_{\Omega_{j}^{(k)}}(\bcalT_l-\bcalT^*)}^2- 6\dkm\gamma- (B_1+B_2)\\
	&~~~~~-2\alpha\dkm \linft{\calP_{\Omega_{j}^{(k)}}(\bcalT-\bcalT^*)}.
\end{align*}
Hence combine the above euqations and then we have upper bound for the slice \begin{align*}
	&{~~~~}\fro{\calP_{\Omega_{j}^{(k)}}\left(\bcalT_l-\bcalT^*-\eta_{l}\calP_{\TT_l}(\bcalG_l)\right) }^2\\
	&\leq \fro{\calP_{\Omega_{j}^{(k)}}\left(\bcalT_l-\bcalT^*\right)}^2-2\eta_l \linft{\calP_{\Omega_{j}^{(k)}}(\bcalT_l-\bcalT^*)}^{-1}\cdot\fro{\calP_{\Omega_{j}^{(k)}}(\bcalT_l-\bcalT^*)}^2 \\
	&{~~~~~~~~~~~~~~~~~~~~~~~~~~~~~~~~~~~~~~}+4\eta_{l} \alpha\dkm\linft{\bcalT_l-\bcalT^*}+12\eta_{l} \dkm \gamma+2\eta_{l}(B_1+B_2)+9\eta_l^2 \frac{\muT r_k}{d_k}d^*.
\end{align*}
Then insert inductions of $\bcalT_l$ into the above equation,
\begin{align*}
	&{~~~~}\fro{\calP_{\Omega_{j}^{(k)}}\left(\bcalT_l-\bcalT^*-\eta_{l}\calP_{\TT_l}(\bcalG_l)\right) }^2\\
	&\leq 9\frac{\mu r_k}{d_k}D_l^2-18\eta_{l}\frac{\muT r_k}{d_k}(5m+1)^{-1}(3^m\mu^{*m} r^* d^*)^{-1/2}D_l^2\\
	&{~~~~~~~~~~~~~~~~~~~~~~~~~~~~~~~~~~~~}+8\eta_{l}\alpha \frac{1}{d_k} (5m+1)\sqrt{3^m\mu^{*m}r^*d^*}D_{l}+ 12\eta_{l} \dkm \gamma+2\eta_{l}(B_1+B_2)+9\eta_l^2 \frac{\muT r_k}{d_k}d^*\\
	&\leq 9\frac{\muT r_k}{d_k}\cdot\left(1- \frac{3}{64}(5m+1)^{-2}(3^m\mu^{*m} r^*)^{-1}\right)D_{l}^2,
\end{align*}
where the last line uses phase one region constraints, corruption rate and initialization guarantees. It shows $$\fro{\calP_{\Omega_{j}^{(k)}}\left(\bcalT_l-\bcalT^*-\eta_{l}\calP_{\TT_l}(\bcalG_l)\right) }\leq 3\sqrt{\frac{\muT r_k}{d_k}}\left(1-\frac{3}{128}(5m+1)^{-2}(3^m\mu^{*m} r^*)^{-1}\right) D_l.$$
It also infers $$\ltinf{\fraM_{k}\left(\bcalT_l-\bcalT^*-\eta_{l}\calP_{\TT_l}(\bcalG_l)\right) }\leq 3\sqrt{\frac{\muT r_k}{d_k}}\left(1-\frac{3}{128}(5m+1)^{-2}(3^m\mu^{*m} r^*)^{-1}\right) D_l.$$
Also, notice that,
\begin{align*}
	\ltinf{\fraM_k(\calP_{\TT^*}(\bcalT_l-\bcalT^*-\eta_l\calP_{\TT_l}(\bcalG_l)))}&\leq \ltinf{\fraM_k(\bcalT_l-\bcalT^*-\eta_l\calP_{\TT_l}(\bcalG_l))}+\ltinf{\fraM_k(\calP_{\TT^*}^{\perp}(\bcalT_l-\eta_l\calP_{\TT_l}(\bcalG_l)))},
\end{align*}
furthermore, with Lemma~\ref{teclem:rieman-orthogonal-project} and Lemma~\ref{teclem:remian-perturb-subgradient}, we could have the following upper bound for the latter term, (same as phase one under pseudo-Huber loss)
\begin{align*}
	\ltinf{\fraM_k(\calP_{\TT^*}^{\perp}(\bcalT_l-\eta_l\calP_{\TT_l}(\bcalG_l)))}\leq 8m^2 \sqrt{\frac{\muT r_k}{d_k}}\cdot\mins^{*-1}D_l^2.
\end{align*}
Then it arrives at
\begin{align*}
	&{~~~~}\ltinf{\fraM_k(\calP_{\TT^*}(\bcalT_l-\bcalT^*-\eta_l\calP_{\TT_l}(\bcalG_l)))}\\
	&\leq \ltinf{\fraM_k(\bcalT_l-\bcalT^*-\eta_l\calP_{\TT_l}(\bcalG_l))}+\ltinf{\fraM_k(\calP_{\TT^*}^{\perp}(\bcalT_l-\eta_l\calP_{\TT_l}(\bcalG_l)))}\\
	&\leq 3\sqrt{\frac{\muT r_k}{d_k}}\left(1-\frac{3}{128}(5m+1)^{-2}(3^m\mu^{*m} r^*)^{-1}\right) D_l+5m^2 \sqrt{\frac{\muT r_k}{d_k}}\cdot\mins^{*-1}D_l^2.
\end{align*}
Then by Lemma~\ref{teclem:perturbation:tensor}, we have \begin{align*}
	\ltinf{\fraM_k(\bcalT_{l+1}-\bcalT^*)}&\leq\ltinf{\fraM_k(\calP_{\TT^*}(\bcalT_l-\bcalT^*-\eta_l\calP_{\TT_l}(\bcalG_l)))}+32m\sqrt{\frac{\muT r_k}{d_k}}\frac{\fro{\bcalT_l-\bcalT^*-\eta_l\calP_{\TT_l}(\bcalG_l)}^2}{\mins^*}\\
	&{~~~}+32m\ltinf{\fraM_k(\bcalT_l-\bcalT^*-\eta_l\calP_{\TT_l}(\bcalG_l))}\frac{\fro{\bcalT_l-\bcalT^*-\eta_l\calP_{\TT_l}(\bcalG_l)}}{\mins^*}\\
	&\leq\left(1-\frac{3}{128}(4m+1)^{-2}(\mu^{*m} r^*)^{-1}\right)D_{l}\cdot\sqrt{\frac{3\muT r_k}{d_k}}+32m\mins^{*-1}D_{l+1}^2\cdot\sqrt{\frac{3\muT r_k}{d_k}}\\
	&{~~~}+8m^2 \sqrt{\frac{\muT r_k}{d_k}}\cdot\mins^{*-1}D_l^2\\
	&\leq 3D_{l+1}\cdot\sqrt{\frac{\muT r_k}{d_k}},
\end{align*}
where $D_l^2/\mins^*\leq D_l\cdot D_{0}/\mins^*\leq cD_l/(m^4\mu^{*m}r^*)\leq 2cD_{l+1}/(m^4\mu^{*m}r^*)$ is used, and
\begin{align*}
	&\ltinf{\left(\U_k^{(l+1)}\H_k^{(l+1)}-\U_k^*\right)\fraM_k(\bcalC^*)}\\
	&{~~~~~~~~}\leq \ltinf{\U_{k\perp}^*\U_{k\perp}^*\fraM_k(\bcalT_l-\bcalT^*-\eta_l\calP_{\TT_l}(\bcalG_l))}+64\ltinf{\U_k^*}\frac{\fro{\bcalT_l-\bcalT^*-\eta_{l}\calP_{\TT_l}(\bcalG_l)}^2}{\mins^*}\\
	&{~~~~~~~~~~~}+16\ltinf{\U_{k\perp}^*\U_{k\perp}^*\fraM_k(\bcalT_l-\bcalT^*-\eta_l\calP_{\TT_l}(\bcalG_l))}\cdot \frac{\fro{\bcalT_l-\bcalT^*-\eta_{l}\calP_{\TT_l}(\bcalG_l)}}{\mins^*}\\
	&{~~~~~~~~}\leq \left(1+16D_{l+1}\cdot \mins^{*-1}\right)\ltinf{\fraM_k(\bcalT_l-\bcalT^*-\eta_l\calP_{\TT_l}(\bcalG_l)) }+1.1\ltinf{\U_k^*}D_{l+1}\\
	&{~~~~~~~~}\leq 5D_{l+1}\cdot\sqrt{\frac{\muT r_k}{d_k}},
\end{align*}
where the second ineuqality is because $$\ltinf{\U_{k\perp}^*\U_{k\perp}^*\fraM_k(\bcalT_l-\bcalT^*-\eta_l\calP_{\TT_l}(\bcalG_l))}\leq\ltinf{\fraM_k(\bcalT_l-\bcalT^*-\eta_l\calP_{\TT_l}(\bcalG_l)) }+\ltinf{\U_k^*}\fro{\bcalT_l-\bcalT^*-\eta_l\calP_{\TT_l}(\bcalG_l)}.$$
Note that it implies $\bcalT_{l+1}$ is incoherent with $3\muT$, namely due to,
\begin{align*}
	\ltinf{\U_k^{(l+1)}}&\leq\sqrt{2}\ltinf{\U_k^{(l+1)}\H_k^{(l+1)}}\leq \sqrt{2}\linft{\U_k^{(l+1)}\H_k^{(l+1)}-\U_k^* }+\sqrt{2}\linft{\U_k^*}\\
	&\leq \sqrt{2}\mins^{*-1} \ltinf{\left(\U_k^{(l+1)}\H_k^{(l+1)}-\U_k^*\right)\fraM_k(\bcalC^*)}D_{l+1}+\sqrt{2}\linft{\U_k^*}\\
	&\leq\sqrt{\frac{3\muT r_k}{d_k}},
\end{align*}
where the first inequality is from $\op{(\H_k^{(l+1)})^{-1}}\leq \sqrt{2}$, see Lemma 4.6.3 \citep{chen2021spectral}. Finally, by Lemma~\ref{teclem:entrynorm-expansion}, we have the upper bound for the entrywise norm of $\bcalT_{l+1}-\bcalT^*$,
\begin{align*}
	&{~~~~}\linft{\bcalT_{l+1}-\bcalT^*}\\
	&\leq \sqrt{\frac{3^m\mu^{*m} r^*}{d^*}}\fro{\bcalT_{l+1}-\bcalT^*}+\sum_{k=1}^m\sqrt{\frac{3^m\mu^{*m-1}r_k^-}{\dkm}}\ltinf{\left(\U_k^{(l+1)}\H_k^{l+1}-\U_k^*\right)\fraM_k(\bcalC^*)}\\
	&\leq (5m+1)\sqrt{\frac{3^m\mu^{*m} r^*}{d^*}}D_{l+1}.
\end{align*}

\subsubsection{Phase Two}
\paragraph*{Frobenius norm}
First consider $\fro{\bcalT_l-\bcalT-\eta_l\calP_{\TT_l}(\bcalG_l)}$,
\begin{align*}
	\fro{\bcalT_l-\eta_l\calP_{\TT_l} (\bcalG_l)-\bcalT^*}^2=\fro{\bcalT_l-\bcalT^*}^2-2\eta_l\inp{\bcalT_l-\bcalT^*}{\calP_{\TT_l}(\bcalG_l)} + \eta_l^2\fro{\calP_{\TT_l}(\bcalG_l)}^2.
\end{align*}
We have analyzed the last term in Lemma~\ref{lem:abs:regularity} that  $\fro{\calP_{\TT_l}(\bcalG_l)}^2\leq c_1^2(m+1)b_1^{-2}\fro{\bcalT_l-\bcalT^*}^2$.
By definition of sub-gradient and analysis of $f(\bcalT)-f(\bcalT^*)$ in Lemma~\ref{lem:abs:regularity}, the intermediate term has the lower bound
\begin{align*}
	\inp{\bcalT_l-\bcalT^*}{\calP_{\TT_l}(\bcalG_l)}&=\inp{\bcalT_l-\bcalT^*}{\bcalG_l}-\inp{\calP_{\TT_l}^{\perp}(\bcalT_l-\bcalT^*)}{\bcalG_l}\\
	&\geq f(\bcalT_l)-f(\bcalT^*)-\inp{\calP_{\TT_l}^{\perp}\bcalT^*}{\bcalG_l}\\
	&\geq \frac{1}{2b_0} \fro{\bcalT_l-\bcalT^*}^2-\inp{\calP_{\TT_l}^{\perp}\bcalT^*}{\bcalG_l}.
\end{align*}
Lemma~\ref{teclem:rieman-orthogonal-project} and bound of $\frorr{\bcalG}$ in proofs of Lemma~\ref{lem:abs:regularity} infer $ \bigg|\inp{\calP_{\TT_l}^{\perp}\bcalT^*}{\bcalG_l} \bigg|\leq8m^2 c_1b_1^{-1}\mins^{*-1}\fro{\bcalT_l-\bcalT^*}^3$ and then we have
\begin{align*}
	\fro{\bcalT_l-\eta_l\calP_{\TT_l} (\bcalG_l)-\bcalT^*}^2&\leq \fro{\bcalT_l-\bcalT^*}^2-\eta_{l}\frac{1}{2b_0} \fro{\bcalT_l-\bcalT^*}^2+\eta_{l}^2c_1^2(m+1)b_1^{-2}\fro{\bcalT_l-\calT^*}^2\\
	&\leq \left(1-\frac{3}{64c_1^2(m+1)}\cdot\frac{b_1^2}{b_0^2}\right)\fro{\bcalT_l-\bcalT^*}^2,
\end{align*}
where the last inequality is due to $\eta_{l}\in\left[\frac{1}{8c_1^2(m+1)}\cdot\frac{b_1^2}{b_0}, \frac{3}{8c_1^2(m+1)}\cdot\frac{b_1^2}{b_0}\right]$. Then note that since $\linft{\bcalT^*-\bcalT_{l_1}}\leq\tau_1$, for each entry $i_1,\dots,i_m$ it has
\begin{align*}
	\left|[\text{Trun}_{\tau_1,\bcalT_{l_1}}(\bcalT_{l}-\eta_{l} \calP_{\TT_l}(\bcalG_{l}))-\bcalT^*]_{i_1\cdots i_m}\right|\leq \left|[\bcalT_l-\eta_l\calP_{\TT_l} (\bcalG_l)-\bcalT^*]_{i_1\cdots i_m}\right|.
\end{align*}
Besides $\linft{\bcalT^*}\leq \sqrt{\frac{\tau_2}{d^*}}\fro{\text{Trun}_{\tau_1,\bcalT_{l_1}}(\bcalT_{l}-\eta_{l} \calP_{\TT_l}(\bcalG_{l}))}$. Thus altogether we have
\begin{align*}
	\left|[\text{Trim}_{\tau_2}(\text{Trun}_{\tau_1,\bcalT_{l_1}}(\bcalT_{l}-\eta_{l} \calP_{\TT_l}(\bcalG_{l})))-\bcalT^*]_{i_1\cdots i_m}\right|&\leq \left|[\text{Trun}_{\tau_1,\bcalT_{l_1}}(\bcalT_{l}-\eta_{l} \calP_{\TT_l}(\bcalG_{l}))-\bcalT^*]_{i_1\cdots i_m}\right|\\
	&\leq \left|[\bcalT_l-\eta_l\calP_{\TT_l} (\bcalG_l)-\bcalT^*]_{i_1\cdots i_m}\right|,
\end{align*}
which is also used in \cite{cai2022generalized}. As a consequence, we have \begin{align*}
	\fro{\text{Trim}_{\tau_2}(\text{Trun}_{\tau_1,\bcalT_{l_1}}(\bcalT_{l}-\eta_{l} \calP_{\TT_l}(\bcalG_{l})))-\bcalT^*}^2&\leq\fro{\bcalT_{l}-\eta_{l}\calP_{\TT_l}(\bcalG_{l})-\bcalT^*}^2\\
	&\leq \left(1-\frac{3}{64c_1^2(m+1)}\cdot\frac{b_1^2}{b_0^2}\right)\fro{\bcalT_l-\bcalT^*}^2.
\end{align*}
Then by perturbation bound Lemma~\ref{teclem:perturbation:tensor}, we have
\begin{align*}
	\fro{\bcalT_{l+1}-\bcalT^*}&\leq \fro{\text{Trim}_{\tau_2}(\text{Trun}_{\tau_1,\bcalT_{l_1}}(\bcalT_{l}-\eta_{l} \calP_{\TT_l}(\bcalG_{l})))-\bcalT^*}\\
	&{~~~~~~~~~~~~~~~~~~~}+\mins^{*-1}\fro{\text{Trim}_{\tau_2}(\text{Trun}_{\tau_1,\bcalT_{l_1}}(\bcalT_{l}-\eta_{l} \calP_{\TT_l}(\bcalG_{l})))-\bcalT^*}^2\\
	&\leq \left(1-\frac{1}{32c_1^2(m+1)}\cdot\frac{b_1^2}{b_0^2}\right)\fro{\bcalT_l-\bcalT^*}.
\end{align*}
\paragraph*{Entrywise norm}
Note that with the trimming operations, the entrywise normed error is guaranteed
\begin{align*}
	\left|[\text{Trim}_{\tau_1,\bcalT_{l_1}}(\bcalT_l-\eta_l\calP_{\TT_l}\bcalG_l)-\bcalT^*]_{i_1\cdots i_m}\right|&\leq \left|[\text{Trim}_{\tau_1,\bcalT_{l_1}}(\bcalT_l-\eta_l\calP_{\TT_l}\bcalG_l)-\bcalT_{l_1}]_{i_1\cdots i_m}\right|+\left|[\bcalT_{l_1}-\bcalT^*]_{i_1\cdots i_m}\right|\\
	&\leq 2\tau_1,
\end{align*}
and
\begin{align*}
	\left|[\text{Trim}_{\tau_2}(\text{Trun}_{\tau_1,\bcalT_{l_1}}(\bcalT_{l}-\eta_{l} \calP_{\TT_l}(\bcalG_{l})))-\bcalT^*]_{i_1\cdots i_m}\right|&\leq \left|[\text{Trun}_{\tau_1,\bcalT_{l_1}}(\bcalT_{l}-\eta_{l} \calP_{\TT_l}(\bcalG_{l}))-\bcalT^*]_{i_1\cdots i_m}\right|\leq 2\tau_1.
\end{align*}
Thus we have $$	\fro{\calP_{\Omega_{j}^{(k)}}(\text{Trim}_{\tau_2}(\text{Trun}_{\tau_1,\bcalT_{l_1}}(\bcalT_{l}-\eta_{l} \calP_{\TT_l}(\bcalG_{l})))-\bcalT^*)}\leq 2\sqrt{\dkm}\tau_1.$$
Furthermore, by Lemma~\ref{teclem:perturbation:tensor}, we get (details of calculations are same as Section~\ref{proof:pseudo-Huber:phaseone})
\begin{align*}
	\ltinf{(\U_k^{(l+1)}\H_k^{(l+1)}-\U_k^*)\fraM_{k}(\bcalC)}\leq 5\sqrt{\dkm}\tau_1.
\end{align*}
 Also $\text{Trim}_{\tau_2}(\cdot)$ guarantees the incoherence of $\bcalT_{l+1}$, see Lemma B.6 of \cite{cai2022generalized}, namely,
$$\ltinf{\U_k^{(l+1)}}\leq 2\kappa \sqrt{\frac{\tau_2}{d_k}}.$$
Finally, by Lemma~\ref{teclem:entrynorm-expansion}, we obtain the entrywised norm
\begin{align*}
	\linft{\bcalT_{l+1}-\bcalT^*}&\leq (5m+1)2^m\kappa^m\tau_2^{m/2}\tau_1.
\end{align*}
\subsection{Proof of Lower Bound Theorem~\ref{thm:lowerbound}}
The proof follows Theorem~5.1 in \cite{chen2018robust}. Define \begin{align*}
	\omega(\alpha,\MM_{\r,\muT}):=\sup\left\{\fro{\bcalT_1-\bcalT_2}^2:\; \max_{i_1,\cdots,i_m}\textsf{TV}(P_{[\bcalT_1]_{i_1\cdots i_m}}, P_{[\bcalT_2]_{i_1\cdots i_m}})\leq \sigma\frac{\alpha}{1-\alpha},\; \bcalT_1,\bcalT_2\in\MM_{\r,\muT}\right\},
\end{align*} 
where $P_{[\bcalT_j]_{i_1\cdots i_m}}:=N([\bcalT_j]_{i_1\cdots i_m},\sigma^2)$, $j=1,2$ is the Gaussian distribution and $\textsf{TV}(\cdot,\cdot)$ is the total variation. We shall first prove \begin{align*}
	\inf_{\widehat{\bcalT}}\sup_{\bcalT^*\in\MM_{\r,\muT}}\sup_{\{Q_{i_1\dots i_m}\}}\PP\left(\fro{\widehat{\bcalT}-\bcalT^*}^2\geq \left(\sum_{k=1}^{m} r_kd_k+r_1\cdots r_m\right)\sigma^2\vee \omega(\alpha,\MM_{\r,\muT})\right)\geq c,
\end{align*}
for some constant $c$ and then prove $\omega(\alpha,\MM_{\r,\muT})\geq C\alpha^2d^*/\mu^{*m}r^*$ for some $C>0$.

\textbf{Step One} If the corruption rate satisfies $\omega(\alpha,\MM_{\r,\muT})\leq \left(\sum_{k=1}^{m} r_kd_k+r_1\cdots r_m\right)\sigma^2$, then the lower bound is $\sigma^2\cdot(\sum_{k=1}^{m}r_kd_k+r_1\cdots r_m)$, which is shown in \cite{zhang2018tensor}. We only need to prove when $\omega(\alpha,\MM_{\r,\muT})\geq \left(\sum_{k=1}^{m} r_kd_k+r_1\cdots r_m\right)\sigma^2$, it has \begin{align}
	\inf_{\widehat{\bcalT}}\sup_{\bcalT^*\in\MM_{\r,\muT}}\sup_{\{Q_{i_1\dots i_m}\}}\PP\left(\fro{\widehat{\bcalT}-\bcalT^*}^2\geq \omega(\alpha,\MM_{\r,\muT})\right)\geq c.
	\label{eq:corrupt1}
\end{align}
There exist $\bcalT_1,\bcalT_2\in\MM_{\r,\muT}$ such that
\begin{align*}
	\fro{\bcalT_1-\bcalT_2}^2=\omega(\alpha,\MM_{\r,\muT}),\quad \max_{i_1,\cdots,i_m}\textsf{TV}(P_{[\bcalT_1]_{i_1\cdots i_m}}, P_{[\bcalT_2]_{i_1\cdots i_m}})\leq \frac{\alpha'}{1-\alpha'}\sigma,
\end{align*}
for some $0<\alpha'\leq\alpha$. Note that for each entry $(i_1,\dots,i_m)\in [d_1]\times\cdots\times [d_m]$, there is $0<\alpha_{i_1\dots i_m}\leq\alpha'$ such that
\begin{align*}
	\textsf{TV}(P_{[\bcalT_1]_{i_1\cdots i_m}}, P_{[\bcalT_2]_{i_1\cdots i_m}})=\frac{\alpha_{i_1\dots i_m}}{1-\alpha_{i_1\dots i_m}}\sigma.
\end{align*}
Besides, according to \cite{chen2018robust}, there exist distributions $\tilde{Q}_{i_1\dots i_m}^{(1)}$ and $\tilde{Q}_{i_1\dots i_m}^{(2)}$ such that \begin{align*}
	(1-\alpha_{i_1\dots i_m})P_{[\bcalT_1]_{i_1\cdots i_m}}+\alpha_{i_1\dots i_m} \tilde{Q}_{i_1\dots i_m}^{(1)}=(1-\alpha_{i_1\dots i_m})P_{[\bcalT_2]_{i_1\cdots i_m}}+\alpha_{i_1\dots i_m} \tilde{Q}_{i_1\dots i_m}^{(2)}.
\end{align*}
There exist distributions $Q_{i_1\dots i_m}^{(j)}$, $j=1,2$, such that if random variable $\omega\sim Q_{i_1\dots i_m}^{(j)}$ then $\omega+[\bcalT_j+\bXi_j]_{i_1\dots i_m}\sim \tilde{Q}_{i_1\dots i_m}^{(j)}$, where $\bXi_j$ comprises i.i.d. $N(0,\sigma^2)$ entries. Then construct the corruptions with 
$$[\bcalS]_{i_1\dots i_m}^{(j)}\sim (1-\alpha_{i_1\dots i_m})\delta_0+\alpha_{i_1\dots i_m}Q_{i_1\dots i_m}^{(j)},\quad j=1,2,$$
where $\delta_0$ is the zero distribution. Specifically, if a random variable follows $\delta_0$, then it is a.s. zero. Under such corruptions, $\bcalY_1:=\bcalT_1+\bXi_1+\bcalS_1$ and $\bcalY_2\bcalT_2+\bXi_2+\bcalS_2$ have the same distribution, in which case $\bcalT_1$ and $\bcalT_2$ are not identifiable based on observations $\bcalY_j$, $j=1,2$. Then Le Cam’s two point testing method \cite{yu1997assouad} leads to Equation~\eqref{eq:corrupt1}.

\textbf{Step Two} We have
\begin{align*}
	\omega(\alpha,\MM_{\r,\muT})&=\sup\left\{\fro{\bcalT_1-\bcalT_2}^2:\; \max_{i_1,\cdots,i_m}\textsf{TV}(P_{[\bcalT_1]_{i_1\cdots i_m}}, P_{[\bcalT_2]_{i_1\cdots i_m}})\leq \sigma\frac{\alpha}{1-\alpha},\; \bcalT_1,\bcalT_2\in\MM_{\r,\muT}\right\}\\
	&\geq\sup\left\{\fro{\bcalT_1-\bcalT_2}^2:\; \linft{\bcalT_1-\bcalT_2}^2\leq 4\sigma^2\alpha^2,\; \bcalT_1,\bcalT_2\in\MM_{\r,\muT}\right\}\\
	&\geq C \sigma^2\cdot\alpha^2d^*/\mu^{*m}r^*,
\end{align*}
where the last equation follows from \cite{chen2021bridging} and the proof completes.
\section{Proofs of Initialization Theorem~\ref{thm:init}}
\label{proof:init}
Recall that $\tilde{\Omega}$ is the support of sparse corruption term $\bcalS$. Denote $$\bcalE:=\left\{\|\fraM_{k}(\bcalS)_{j,\cdot}\|_{0}\leq 3\alpha\dkm,\quad k\in[m],\; j\in[d_k] \right\}$$ as the event of $\bcalS$ to be an $\alpha$-fiber sparse tensor. By Chernoff bounds, we have $\PP(\bcalE)\geq 1-\sum_{k=1}^{m}\dkm\exp(-\alpha d_k)$. We shall use the fact that for all $\bcalX$, its operator is not larger than the one with its entrywise absolute value, namely, $\op{\bcalX}\leq \op{\bcalY}$ where $[\bcalY]_{\omega}=|[\bcalX]_{\omega}|$. First consider entries of $\hat{\bcalY}-\bcalT^*$, for any $(i_1,\dots,i_m)\in[d_1]\times\cdots\times[d_m]$, it has
\begin{align*}
	[\hat{\bcalY}-\bcalT^*]_{i_1\cdots i_m}&= \left(\xi_{i_1\cdots i_m}+[\bcalS]_{i_1\cdots i_m}\right)\cdot1_{\left\{|[\bcalY]_{i_1\cdots i_m}|\leq\tau\right\}}+\left(\tau\cdot\text{sign}\left([\bcalY]_{i_1\dots i_m}\right)-[\bcalT^*]_{i_1\dots i_m}\right)\cdot 1_{\left\{|[\bcalY]_{i_1\cdots i_m}|>\tau\right\}}\\
	&=\xi_{i_1\cdots i_m}\cdot1_{\left\{|[\bcalY]_{i_1\cdots i_m}|\leq\tau,(i_1,\dots,i_m)\notin \tilde\Omega\right\}}+\left([\bcalS]_{i_1\cdots i_m}+\xi_{i_1\cdots i_m}\right)\cdot1_{\left\{|[\bcalY]_{i_1\cdots i_m}|\leq\tau, (i_1,\dots,i_m)\in \tilde\Omega\right\}}\\
	&{~~~}+\left(\tau\cdot\text{sign}\left([\bcalY]_{i_1\dots i_m}\right)-[\bcalT^*]_{i_1\dots i_m}\right)\cdot 1_{\left\{|[\bcalY]_{i_1\cdots i_m}|>\tau, (i_1,\dots,i_m)\notin \tilde\Omega\right\}}\\
	&{~~~}+\left(\tau\cdot\text{sign}\left([\bcalY]_{i_1\dots i_m}\right)-[\bcalT^*]_{i_1\dots i_m}\right)\cdot 1_{\left\{|[\bcalY]_{i_1\cdots i_m}|>\tau, (i_1,\dots,i_m)\in \tilde\Omega\right\}}.
\end{align*}
After simple calculations, we have
\begin{align*}
	\hat{\bcalY}-\bcalT^*&=\bXi-\bXi\odot 1_{\{\omega\in\tilde{\Omega}\}}+\left(\bcalS+\bXi\right)\odot1_{\{|\bcalY|\leq\tau,\omega\in\tilde{\Omega}\}}\\
	&~~~+\left(\tau\text{sign}(\bcalT^*+\bXi)-\bcalT^*-\bXi\right)\odot1_{\{|\bcalY|>\tau,\omega\notin\tilde{\Omega}\}}+\left(\tau\text{sign}(\bcalY)-\bcalT^*\right)\odot1_{\{|\bcalY|>\tau,\omega\in\tilde{\Omega}\}}.
\end{align*}
Notice that $\bXi-\bXi\odot 1_{\{\omega\in\tilde{\Omega}\}}=\calP_{\tilde{\Omega}^{C}}(\bXi)$ is a mean zero term. Then by Theorem~2.1 in \cite{auddy2022estimating}, with probability exceeding $1-c_m\dmax^{-\eps/4}$, the first two terms have the bounded operator norm,
\begin{align*}
	&~~~~\op{\bXi-\bXi\odot 1_{\{\omega\in\tilde{\Omega}\}}} \leq C\|\xi\|_2\left(\sqrt{\dmax\log\dmax}+d^{*1/4}(\log\dmax)^{1/4}\right).
\end{align*}
 Then consider the fourth term and with probability exceeding $1-c_m\dmax^{-\eps/4}$
\begin{align*}
	&~~~\op{\left(\tau\text{sign}(\bcalY)-\bcalT^*-\bXi\right)\odot1_{\{|\bcalY|>\tau,\omega\notin\tilde{\Omega}\}}}\\
	&\leq \op{\left(\tau\text{sign}(\bcalT^*+\bXi)-\bcalT^*-\bXi\right)\odot1_{\{|\bcalT^*+\bXi|>\tau,\omega\notin\tilde{\Omega}\}}-\EE\left(\tau\text{sign}(\bcalT^*+\bXi)-\bcalT^*-\bXi\right)\odot1_{\{|\bcalT^*+\bXi|>\tau,\omega\notin\tilde{\Omega}\}}}\\
	&~~~+\op{\EE\left(\tau\text{sign}(\bcalT^*+\bXi)-\bcalT^*-\bXi\right)\odot1_{\{|\bcalT^*+\bXi|>\tau,\omega\notin\tilde{\Omega}\}}}\\
	&\leq \|\xi\|_4\left(\sqrt{\dmax\log\dmax}+d^{*1/4}(\log\dmax)^{1/4}\right)+\sqrt{d^*}\left(\|\xi\|_4+\linft{\bcalT^*} \right)\frac{\|\xi\|_4^2}{\tau^2},
\end{align*}
where the last line is due to Theorem~2.1 of \cite{auddy2022estimating} and also $\EE|(\xi+t-\tau\cdot\text{sign}(\xi+t))\cdot1_{\{|\xi+t|\geq \tau\}}|\leq \EE|(\xi+t)\cdot1_{\{|\xi+t|\geq \tau\}}|\leq \sqrt{t^2+\EE\xi^2}\sqrt{\PP(|\xi|\geq \tau/2)}$. And as for the third and the last term, which is an $\alpha$-fraction fiber sparse term, according to Lemma~\ref{teclem:sparse}, we have $\|\left(\tau\text{sign}(\bcalY)-\bcalT^*\right)\odot1_{\{|\bcalY|>\tau,\omega\in\tilde{\Omega}\}}\|_{\muT}\leq 2\tau\alpha \sqrt{d^*}$.
Thus altogether we have
\begin{align*}
	\|\hat{\bcalY}-\bcalT^*\|&\leq 2(\|\xi\|_4+\linft{\bcalT^*})\cdot\left(\sqrt{\dmax\log\dmax}+d^{*1/4}(\log\dmax)^{1/4}+\sqrt{d^*}\frac{\|\xi\|_4^2}{\tau^2} \right)+2\alpha\tau\sqrt{d^*}=:\Lambda\\
	&\leq 2(\|\xi\|_4+\linft{\bcalT^*})\cdot\left(\sqrt{\dmax\log\dmax}+4d^{*1/4}(\log\dmax)^{1/4} \right)+2\alpha\tau\sqrt{d^*},
\end{align*}
where $\|\cdot\|_{\mu}\leq\|\cdot\|$ is used.
Note that the initialization is $\bcalT_0=\bcalC^{(0)}\cdot\llbracket\U_1^{(0)},\dots,\U_m^{(0)}\rrbracket=\text{HOSVD}(\hat{\bcalY})$. By tensor perturbation bound \citep{cai2022generalized} or modifications of Theorem~3 in \cite{cai2018rate} with analyses similar to the above one, for each $k=1,\dots,m$, we have
\begin{align*}
	&~~~\op{\U_k^{(0)}\H_k^{(0)}-\U_k^*}\vee\min_{\Q\in\OO_{r_k,r_k}}\op{\U_k^{(0)}\Q-\U_k^*}\leq\left\Vert \U_k^{(0)}\U_k^{(0)\top}-\U_k^*\U_k^{*\top}\right\Vert\leq C\frac{\Lambda}{\mins^{*}},
\end{align*}
where $\H_k^{(0)}:=\U_k^{(0)}\top\U_k^*$. Furthermore, consider $\hat{\bcalT}_0-\bcalT^*$,
\begin{align*}
	\bcalT_0-\bcalT^*&=\hat{\bcalY}\times_{k=1,\dots,m}\U_k^{(0)}\U_k^{(0)\top}-\bcalT^*\times_{k=1,\dots,m}\U_k^*\U_k^{*\top}\\
	&=\sum_{k=1}^m \hat\bcalY\times_{i<k}\U_i^*\U_i^{*\top}\times_k\left(\U_k^{(0)}\U_k^{(0)\top}-\U_k^*\U_k^{*\top} \right)\times_{j>k}\U_j^{(0)}\U_j^{(0)\top}+\left(\hat{\bcalY}-\bcalT^*\right)\times_{k=1,\dots,m}\U_k^{*}\U_k^{*\top}.
\end{align*}
Then we have
\begin{align*}
	\fro{\bcalT_0-\bcalT^*}&\leq m\op{\hat\bcalY}\fro{\U_k^{(0)}\U_k^{(0)\top}-\U_k^*\U_k^{*\top}}+ \sqrt{r^*}\|\hat{\bcalY}-\bcalT^*\|_{\mu}\leq Cm\kappa\sqrt{r^*}\Lambda.
\end{align*}
Also, by Lemma~\ref{teclem:perturbation-coretensor}, we have
\begin{align*}
	\fro{\bcalC^{(0)}\cdot\llbracket\U_1^{(0)},\dots,\U_m^{(0)}\rrbracket-\bcalC^*}\leq Cm\kappa\sqrt{r^*}\Lambda.
\end{align*}
\subsection{Leave-one-out Sequence}
Introduce leave-one-out sequence: for each $k=1,\dots,m$ and $j=1,\dots,d_k$, denote $\hat{\bcalY}^{(k),j}:=\calP_{\Omega_{-j}^{(k)}}(\hat{\bcalY})+\calP_{\Omega_{j}^{(k)}}(\bcalT^*)$ and $\bcalT_0^{(k),j}:=\bcalC_0^{(k),j}\cdot\llbracket\U_1^{(0),(k),j},\dots,\U_m^{(0),(k),j}\rrbracket=\text{HOSVD}_{\r}(\hat\bcalY^{(k),j})$. Notice that $\hat{\bcalY}-\hat{\bcalY}^{(k),j}=\calP_{\Omega_{j}^{(k)}}(\hat{\bcalY}-\bcalT^*)$. 

By \cite{cai2018rate}, we have\begin{align*}
	\left\Vert\U_k^{(0)}\U_k^{(0)\top}- \U_k^{(0),(k),j}\U_k^{(0),(k),j\top}\right\Vert \leq c\sqrt{\frac{\muT r_k}{d_k}}\frac{\Lambda}{\mins^*}.
\end{align*}
On the other hand, notice that $\calP_{\Omega_{j}^{(k)}}(\hat{\bcalY}^{(k),j})=\calP_{\Omega_{j}^{(k)}}(\bcalT^*)$. Then by Lemma~\ref{teclem:yuxinchen} and Lemma~\ref{teclem:coretensor1}, we have \begin{align*}
	\ltwo{\left(\U_k^{(0),(k),j}\H_k^{(0),(k),j}-\U_k^*\right)_{j,\cdot}}\leq cm\kappa\sqrt{\frac{\muT r_k}{d_k}}\frac{\Lambda}{\mins^*},
\end{align*}
where $\H_k^{(0),(k),j}=\U_k^{(0),(k),j\top}\U_k^*$. Combine the above two inequalities, it has
\begin{align*}
	\ltwo{\left(\U_k^{(0)}\H_k^{(0)}-\U_k^*\right)_{j,\cdot}}&\leq\ltwo{\left(\U_k^{(0),(k),j}\H_k^{(0),(k),j}-\U_k^*\right)_{j,\cdot}}+\ltwo{\left(\U_k^{(0),(k),j}\H_k^{(0),(k),j}-\U_k^{(0)}\H_k^{(0)}\right)_{j,\cdot}}\\
	&\leq\ltwo{\left(\U_k^{(0),(k),j}\H_k^{(0),(k),j}-\U_k^*\right)_{j,\cdot}}+\ltwo{\left(\U_k^{(0),(k),j}\U_k^{(0),(k),j\top}-\U_k^{(0)}\U_k^{(0)\top}\right)_{j,\cdot}}\\
	&\leq \ltwo{\left(\U_k^{(0),(k),j}\H_k^{(0),(k),j}-\U_k^*\right)_{j,\cdot}}+\op{\U_k^{(0),(k),j}\U_k^{(0),(k),j\top}-\U_k^{(0)}\U_k^{(0)\top}}\\
	&\leq cm\kappa\sqrt{\frac{\muT r_k}{d_k}}\frac{\Lambda}{\mins^*}.
\end{align*}
Take maximum over $j=1,\dots,d_k$ and then we obtain
\begin{align*}
	\ltinf{\U_k^{(0)}\H_k^{(0)}-\U_k^*}\leq cm\kappa\sqrt{\frac{\muT r_k}{d_k}}\frac{\Lambda}{\mins^*}.
\end{align*}
\subsection{Entrywise norm}
By Lemma~\ref{teclem:entrynorm-expansion}, we have the upper bound of the entrywise norm \begin{align*}
	\linft{\bcalT_0-\bcalT^*}&\leq cm^2\kappa^2\sqrt{r^*}\sqrt{\frac{\mu^{*m} r^*}{d^*}}\frac{\Lambda}{\mins^*},
\end{align*}
which finishes the proof.
\begin{lemma}[\cite{yi2016fast}]
	Suppose $\bcalS\in\RR^{d_1\times\cdots\times d_m}$ is an $\alpha$-fiber sparse tensor. Then we have $$\|\bcalS\|\leq\alpha\sqrt{d^*}\linft{\bcalS}.$$
	\label{teclem:sparse}
\end{lemma}
\section{Proofs under Missing Values}
We shall only prove the following regularity properties with which the convergence dynamics could be obtained easily following the framework of PCA.
\begin{lemma}[Two-phase regularity properties with missing data]
	Suppose $\{\xi_{i}\}_{i=1}^n$ are i.i.d. following Assumption~\ref{assu:abs:noise:completion} and independent corruptions $\{s_i\}_{i=1}^n$ are non-zero with probability $\alpha$. 
	Then there exist $c_1,c_2,c_3,C_0,C_1,C_2,C_3$ such that if $n\geq C_0\dmax\log \dmax$, then for any fixed $\mu$-incoherent tensor $\bcalT\in\RR^{d_1\times \cdots\times d_m}$ such that $\ltinf{\fraM_k(\bcalT-\bcalT^*)}\gtrsim C_3b_0\sqrt{\log\dmax\cdot d^*/n}$ for all $k\in[m]$ and for any sub-gradient $\bcalG\in\partial f(\bcalT)$, with probability exceeding $1-c_1\sum_{l=2,4}\exp(-t_l^2/(n\|\fraM_{k}(\bcalT-\bcalT^*)_{j,\cdot}\|_{\mathrm{F}}^2/d^*+t_l\|\fraM_{k}(\bcalT-\bcalT^*)_{j,\cdot} \|_{\infty}))-c_2\sum_{l=1,3}\exp(-t_l^2/(n\|\bcalT-\bcalT^*\|_{\mathrm{F}}^2/d^*+t_l\|\bcalT-\bcalT^* \|_{\infty}))-c_3md^{*-10}$,
	\begin{enumerate}[(1).]
		\item we have
		\begin{align*}
			\fro{\calP_{\TT}(\bcalG)}^2&\leq C_1(m+1)n^2\frac{\mu^m r^*}{d^*},\\
			\sum_{i=1}^{n}\left|Y_i-\inp{\bcalX_i}{\bcalT}\right|&-\sum_{i=1}^{n}\left|Y_i-\inp{\bcalX_i}{\bcalT^*}\right|\\
			&\geq \frac{n}{2d^*}\linft{\bcalT-\bcalT^*}^{-1}\cdot\left(\fro{\bcalT-\bcalT^*}^2-2\alpha d^*\linft{\bcalT-\bcalT^*}^2\right)-2n\gamma-t_1;
		\end{align*}
		and for any $k\in[m]$ and $j\in[d_k]$, we have
		\begin{align*}
			\ltinf{\fraM_{k}\calP_{\TT}(\bcalG)}^2&\leq C_1(m+1)n^2\frac{\mu^m r^*}{d^*}\cdot\frac{\mu r_k}{d_k},\\
			\sum_{i=1}^{n}|Y_i-\langle\bcalX_i,\calP_{\Omega_{j}^{(k)}}(\bcalT)\rangle|&-\sum_{i=1}^{n}|Y_i-\langle\bcalX_i,\calP_{\Omega_{j}^{(k)}}(\bcalT^*)\rangle|\geq \frac{n}{2d^*}\linft{\fraM_{k}(\bcalT-\bcalT^*)_{j,\cdot}}^{-1}\\&\times\left(\fro{\fraM_{k}(\bcalT-\bcalT^*)_{j,\cdot}}^2-2\alpha \dkm\linft{\fraM_{k}(\bcalT-\bcalT^*)_{j,\cdot}}^2\right)-2\frac{n}{d_k}\gamma-t_2;
		\end{align*}
		\item we have
		\begin{align*}
			\fro{\calP_{\TT}(\bcalG)}^2&\leq C_2(m+1)n^2\frac{\mu^m r^*}{d^{*2}b_1^2}\fro{\bcalT-\bcalT^*}^2\\
			\sum_{i=1}^{n}\left|Y_i-\inp{\bcalX_i}{\bcalT}\right|-\sum_{i=1}^{n}\left|Y_i-\inp{\bcalX_i}{\bcalT^*}\right|&\geq \frac{1}{4b_0}\frac{n}{d^*}\fro{\bcalT-\bcalT^*}^2-\alpha n\linft{\bcalT-\bcalT^*}-t_3;
		\end{align*}
		and for any $k\in[m]$ and $j\in[d_k]$, we have
		\begin{align*}
			\ltinf{\fraM_{k}\calP_{\TT}(\bcalG)}^2&\leq C_2mn^2\frac{\mu^m r^*}{d^{*2}}\cdot\frac{\mu r_k}{d_k}\fro{\bcalT-\bcalT^*}^2+Cn^2\frac{\mu^m r^*}{d^{*2}}\cdot\frac{\mu r_k}{d_k}\ltinf{\fraM_{k}(\bcalT-\bcalT^*)}^2,\\
			\sum_{i=1}^{n}|Y_i-\langle\bcalX_i,\calP_{\Omega_{j}^{(k)}}(\bcalT)\rangle|&-\sum_{i=1}^{n}|Y_i-\langle\bcalX_i,\calP_{\Omega_{j}^{(k)}}(\bcalT^*)\rangle|\geq\frac{1}{4b_0}\frac{n}{d^*}\fro{\fraM_{k}(\bcalT-\bcalT^*)_{j,\cdot}}^2 \\&-\alpha \frac{n}{d_k}\linft{\fraM_{k}(\bcalT-\bcalT^*)_{j,\cdot}}-t_4;
		\end{align*}
	\end{enumerate}
	\label{lem:completion}
\end{lemma}
\subsection{Proof of Lemma~\ref{lem:completion}}
\subsubsection{Phase One Analysis}
\paragraph{Analysis of $f(\bcalT)-f(\bcalT^*)$} 
Note that with triangle inequality we have
\begin{align*}
	\EE f(\bcalT)-\EE f(\bcalT^*)&=\sum_{i=1}^{n} \EE  \left[\left|\inp{\bcalX_i}{\bcalT-\bcalT^*}-\xi_{i}-s_i\right|-\left|\xi_{i}+s_i\right|\right]\cdot 1_{\{s_i=0\}}\\
	&~~~+\sum_{i=1}^{n} \EE  \left[\left|\inp{\bcalX_i}{\bcalT-\bcalT^*}-\xi_{i}-s_i\right|-\left|\xi_{i}+s_i\right|\right]\cdot 1_{\{s_i\neq0\}}\\
	&\geq \frac{(1-\alpha)n}{d^*}\lone{\bcalT-\bcalT^*}-2(1-\alpha)n\gamma-\alpha n\linft{\bcalT-\bcalT^*}.
\end{align*}
Denote the event $$\bcalE:=\left\{	\left|f(\bcalT)-f(\bcalT^*)-\EE\left[f(\bcalT)-f(\bcalT^*)\right]\right|\leq t\right\},$$ where $c<1/4$ is some constant. Specifically, Proposition~\ref{tecprop:concentration completion indepen} proves that $\PP(\bcalE)\geq 1-2\exp\left(-\frac{t^2}{n\fro{\bcalT-\bcalT^*}^2/d^*+t\linft{\bcalT-\bcalT^*}}\right)$. And event $\bcalE$ implies that
\begin{align*}
	f(\bcalT)- f(\bcalT^*)&\geq \EE f(\bcalT)-\EE f(\bcalT^*)-t\\
	&\geq \frac{n}{2d^*}\cdot\linft{\bcalT-\bcalT^*}^{-1}\cdot\left(\fro{\bcalT-\bcalT^*}^2- 2\alpha d^*\linft{\bcalT-\bcalT^*}^2\right)-2n\gamma-t,
\end{align*}
where Lemma~\ref{teclem:norm-relation} is used.
\paragraph{Analysis of $\fro{\calP_{\TT}(\bcalG)}$}
Note that $\fro{\calP_{\TT}(\bcalG)}$ has the expansion, \begin{align*}
	\fro{\calP_{\TT}(\bcalG)}^2&\leq\underbrace{\fro{\bcalG\times_1\U_1\U_1^{\top}\times_2\cdots\times_m\U_m\U_m^{\top}}^2}_{=A_1}\\
	&{~~~~~}+\sum_{k=1}^{m}\underbrace{\fro{\fraM_k(\bcalG)\left(\otimes_{i\neq k}\U_i\right)\fraM_k(\bcalC)^{\dag} \fraM_k(\bcalC)\left(\otimes_{i\neq k}\U_i\right)^{\top} }^2}_{=A_2}.
\end{align*}
First analyze $A_1$. By sub-gradient definition, we have
\begin{align*}
	\fro{\bcalG\times_1\U_1\U_1^{\top}\times_2\cdots\times_m\U_m\U_m^{\top}}^2&\leq f(\bcalT+\bcalG\times_1\U_1\U_1^{\top}\times_2\cdots\times_m\U_m\U_m^{\top})-f(\bcalT)\\
	&\leq \sum_{i=1}^{n} \left|\inp{\bcalX_i}{\bcalG\times_1\U_1\U_1^{\top}\times_2\cdots\times_m\U_m\U_m^{\top}}\right|\\
	&\leq \sum_{i=1}^{n} \lone{\bcalX_i}\linft{\bcalG\times_1\U_1\U_1^{\top}\times_2\cdots\times_m\U_m\U_m^{\top}}.
\end{align*}
Notice that $\lone{\X_i}=1$ and $\linft{\bcalG\times_1\U_1\U_1^{\top}\times_2\cdots\times_m\U_m\U_m^{\top}}\leq\sqrt{\frac{\mu^m r^*}{d^*}}\fro{\bcalG\times_1\U_1\times_2\cdots\times_m\U_m}=\sqrt{\frac{\mu^m r^*}{d^*}}\fro{\bcalG\times_1\U_1\U_1^{\top}\times_2\cdots\times_m\U_m\U_m^{\top}}$. Thus we have
\begin{align*}
	\fro{\bcalG\times_1\U_1\U_1^{\top}\times_2\cdots\times_m\U_m\U_m^{\top}}\leq n\sqrt{\frac{\mu^m r^*}{d^*}}.
\end{align*}
In this way we prove $A_1\leq n^2\frac{\mu^m r^*}{d^*}$. Before bounding term $A_2$, we introduce the orthogonal matrix $\V_k\in\RR^{\dkm\times r_k}$ which denotes $\V_k\V_k^{\top}= \left(\otimes_{i\neq k}\U_i\right)\fraM_k(\bcalC)^{\dag} \fraM_k(\bcalC)\left(\otimes_{i\neq k}\U_i\right)^{\top}$ and satisfies $\ltinf{\V_k}\leq\sqrt{\mu^{m-1}r_k^-/\dkm}$. Then we have
\begin{align*}
	A_2&=\sum_{i=1}^{n} \underbrace{\text{trace}(\fraM_{k}(\bcalX_i) \V_k\V_k^{\top} \fraM_{k}(\bcalX_i)^{\top})}_{B_1}\\
	&~~~~~~~~+\underbrace{\sum_{i\neq j} \text{trace}(\fraM_{k}(\bcalX_i) \V_k\V_k^{\top} \fraM_{k}(\bcalX_j)^{\top})\times \text{sign}(\inp{\X_i}{\bcalT}-Y_i)\times\text{sign}(\inp{\X_j}{\bcalT}-Y_j)}_{B_2}.
\end{align*}
We shall only provide the detailed bound of the leading term $B_2$. Suppose $\bcalX_j^{i},\xi_j^{i}, s_j^{i}$ is an i.i.d. copy of $\bcalX_j, \xi_j, s_j$ respectively. Denote $C_{ij}':= \text{trace}(\fraM_k(\bcalX_i)\V_k\V_k^{\top} \fraM_k(\bcalX_j^i)^{\top})\cdot \text{sign}(\inp{\bcalX_i}{\bcalT-\bcalT^*}-\xi_i-s_i)\cdot \text{sign}(\inp{\bcalX_j^i}{\bcalT-\bcalT^*}-\xi_j^i-s_j^i)$. Then by decoupling technique \citep{de2012decoupling}, we have
\begin{align*}
	\PP\left(\left|B_2\right|\geq t\right)\leq C\PP\left(\left|\sum_{i\neq j} C_{ij}'\right|\geq t\right).
\end{align*}
First consider $\EE C_{ij}'$,
\begin{align*}
	\EE C_{ij}'&\leq \EE\text{trace}(\fraM_{k}(\bcalX_i) \V_k\V_k^{\top} \fraM_{k}(\bcalX_j^i)^{\top})
	\leq \frac{\mu^{m}r^*}{d^*}.
\end{align*}
Also, we have \begin{align*}
	\EE(C_{ij}')^2=\frac{1}{(d^*)^2}\sum_{\bcalX_i\in\bcalX,\, \bcalX_j\in\bcalX} \text{trace}(\fraM_{k}(\bcalX_i) \V_k\V_k^{\top} \fraM_{k}(\bcalX_j)^{\top})^2\leq\frac{r_k^2}{d^{*2}}d_k,
\end{align*}
and $\left|C_{ij}'\right|\leq \frac{\mu^{m-1}r_k^-}{\dkm}$. Then by Bernstein's inequality Theorem~\ref{thm:Bernstein:valuetype}, we have
\begin{align*}
	\sum_{i\neq j} C_{ij}'\leq \frac{\mu^{m}r^*n^2}{d^*},
\end{align*}
holds with probability exceeding $1-\exp(-n^2/\dmax^2)$. Thus altogether we have the upper bound $\fro{\calP_{\TT}(\bcalG)}^2\leq C m\frac{\mu^{m}r^*n^2}{d^*}$ with probability exceeding $1-m\exp(-n^2/\dmax)$.

\paragraph{Analysis of $f(\calP_{\Omega_{j}^{(k)}}(\bcalT))-f(\calP_{\Omega_{j}^{(k)}}(\bcalT^*))$}
Notice that under event $\bcalE_1$, we have
\begin{align*}
	\EE f(\calP_{\Omega_{j}^{(k)}}(\bcalT))-\EE f(\calP_{\Omega_{j}^{(k)}}(\bcalT^*))&=\sum_{i=1}^{n} \EE  \left[\left|\inp{\bcalX_i}{\bcalT-\bcalT^*}-\xi_{i}-s_i\right|-\left|\xi_{i}+s_i\right|\right]\cdot 1_{\{s_i=0\}}\cdot 1_{\{\bcalX_i\in\Omega_{j}^{(k)}\}}\\
	&~~~+\sum_{i=1}^{n} \EE  \left[\left|\inp{\bcalX_i}{\bcalT-\bcalT^*}-\xi_{i}-s_i\right|-\left|\xi_{i}+s_i\right|\right]\cdot 1_{\{s_i\neq0\}}\cdot 1_{\{\bcalX_i\in\Omega_{j}^{(k)}\}}\\
	&\geq \frac{(1-\alpha)n}{d^*}\lone{\calP_{\Omega_{j}^{(k)}}(\bcalT-\bcalT^*)}-2\frac{n}{d_k}\gamma-\frac{\alpha n}{d_k}\linft{\calP_{\Omega_{j}^{(k)}}(\bcalT-\bcalT^*)}.
\end{align*}
Denote 
\begin{align*}
	\bcalE_j^{(k)}(t)&:= \left\{\left|f(\calP_{\Omega_{j}^{(k)}}(\bcalT))-f(\calP_{\Omega_{j}^{(k)}}(\bcalT^*))-\EE \left[f(\calP_{\Omega_{j}^{(k)}}(\bcalT))-f(\calP_{\Omega_{j}^{(k)}}(\bcalT^*))\right]\right|\leq t \right\}.
\end{align*}
And with similar proofs in Lemma~\ref{tecprop:concentration completion indepen}, we have $\PP(\bcalE_j^{(k)})\geq 1-\exp\left(-c\frac{t^2}{\frac{n}{d^*}\fro{\calP_{\Omega_{j}^{(k)}}(\bcalT-\bcalT^*)}^2+ t\linft{\calP_{\Omega_{j}^{(k)}}(\bcalT-\bcalT^*)}}\right)$. Then under event $\bcalE_j^{(k)}$, we have
\begin{align*}
	f(\calP_{\Omega_{j}^{(k)}}(\bcalT))- f(\calP_{\Omega_{j}^{(k)}}(\bcalT^*))&\geq \frac{n}{2d^*}\linft{\calP_{\Omega_{j}^{(k)}}(\bcalT-\bcalT^*)}^{-1}\cdot \fro{\calP_{\Omega_{j}^{(k)}}(\bcalT-\bcalT^*)}^2\\
	&~~~~~~~~~~~~~~~~~~~~~~~~~~-2\frac{n}{d_k}\gamma-\frac{\alpha n}{d_k}\linft{\calP_{\Omega_{j}^{(k)}}(\bcalT-\bcalT^*)}-t.
\end{align*}
\paragraph{Analysis of $\ltinf{\fraM_k(\calP_{\TT}(\bcalG))}$}
Denote $\bcalT:=\bcalC\cdot\llbracket\U_{1},\dots,\U_m\rrbracket$ and then we have
\begin{align*}
	&~~~~\fraM_k(\calP_{\TT}(\bcalG))\\
	&=\fraM_k(\bcalG)\left(\otimes_{i\neq k}\U_i\right)\fraM_k(\bcalC)^{\dagger}\fraM_k(\bcalC)\left(\otimes_{i\neq k}\U_i\right)^{\top}+\U_k\U_k^{\top}\fraM_k(\bcalG)\left(\otimes_{i\neq k}\U_i\right)\left(\I-\fraM_k(\bcalC)^{\dagger}\fraM_k(\bcalC)\right)\left(\otimes_{i\neq k}\U_i\right)^{\top}\\
	&~~~~+\sum_{i\neq k} \U_k\fraM_k(\bcalC\times_{j\neq i,k}\U_j\times\V_i),
\end{align*}
where $\V_i:=\left(\I_{d_i}-\U_i\U_i^{\top}\right)\fraM_k(\bcalG)\left(\otimes_{j\neq i}\U_j\right)\fraM_k(\bcalC)^{\dagger}$. Then with similar analyses in $A_1$, $A_2$ and $B_2$, we have
\begin{align*}
	\ltinf{\fraM_k(\calP_{\TT}(\bcalG))}^2&\leq m\ltinf{\U_k}^2\cdot n^2\frac{\mu^m r^*}{d^*}+\ltinf{\fraM_k(\bcalG)\left(\otimes_{i\neq k}\U_i\right)\fraM_k(\bcalC)^{\dagger}\fraM_k(\bcalC)\left(\otimes_{i\neq k}\U_i\right)^{\top}}^2\\
	&\leq mn^2C_1\frac{\mu^m r^*}{d^*}\cdot\frac{\mu r}{d_k}+C_2\frac{n^2}{d_kd^*}\\
	&\leq C(m+1)n^2\frac{\mu^m r^*}{d^*}\cdot\frac{\mu r}{d_k},
\end{align*}
holds with probability exceeding $1-d_k\exp(-n^2/d_k^2)$. Thus $\ltwo{\fraM_k(\calP_{\TT}(\bcalG))}^2\leq C(m+1)n^2\frac{\mu^m r^*}{d^*}\cdot\frac{\mu r}{d_k}$ holds for all $k=1.\dots,m$ with probability exceeding $1-\sum_{k=1}^{m}d_k\exp(-n^2/d_k^2)$.
\subsubsection{Phase Two Analysis}
\paragraph{Analysis of $f(\bcalT)-f(\bcalT^*)$}
Notice that
\begin{align*}
	\EE f(\bcalT)- \EE f(\bcalT^*)&=\sum_{i=1}^{n} \EE  \left[\left|\inp{\bcalX_i}{\bcalT-\bcalT^*}-\xi_{i}-s_i\right|-\left|\xi_{i}+s_i\right|\right]\cdot 1_{\{s_i=0\}}\\
	&~~~+\sum_{i=1}^{n} \EE  \left[\left|\inp{\bcalX_i}{\bcalT-\bcalT^*}-\xi_{i}-s_i\right|-\left|\xi_{i}+s_i\right|\right]\cdot 1_{\{s_i\neq0\}}\\
	&\geq \frac{(1-\alpha)n}{d^*} \EE\left[\lone{\bcalT-\bcalT^*-\bXi}-\lone{\bXi}\right] - \alpha n\linft{\bcalT-\bcalT^*}\\
	&\geq \frac{1}{b_0}\frac{(1-\alpha)n}{d^*}\fro{\bcalT-\bcalT^*}^2-\alpha n \linft{\bcalT-\bcalT^*}.
\end{align*}
Then under event $\bcalE$, we have
\begin{align*}
	f(\bcalT)- f(\bcalT^*)&\geq \frac{1}{b_0}\frac{(1-\alpha)n}{d^*}\fro{\bcalT-\bcalT^*}^2-\alpha n \linft{\bcalT-\bcalT^*}-t.
\end{align*}
\paragraph{Analysis of $\fro{\calP_{\TT}(\bcalG)}$}
Note that
\begin{align*}
	\fro{\calP_{\TT}(\bcalG)}^2&=\underbrace{\fro{\bcalG\times_1\U_1\U_1^{\top}\times_2\cdots\times_m\U_m\U_m^{\top}}^2}_{=B_1}\\
	&{~~~~~}+\sum_{k=1}^{m}\underbrace{\fro{\left(\I_{d_k}-\U_k\U_k^{\top}\right)\fraM_k(\bcalG)\left(\otimes_{i\neq k}\U_i\right)\fraM_k(\bcalC^*)^{\dag} \fraM_k(\bcalC^*)\left(\otimes_{i\neq k}\U_i\right)^{\top} }^2}_{=B_2}.
\end{align*}
Also, the sub-gradient has the expression of $\bcalG=\sum_{i=1}^n \text{sign}(\inp{\bcalX_i}{\bcalT-\bcalT^*}-\xi_i)\cdot\bcalX_i$, where $\text{sign}(0)$ takes arbitrary values in $[-1,1]$. First consider $B_1$ term,
\begin{align*}
	&B_1=\sum_{i=1}^{n}\text{trace}(\U_1^{\top} \fraM_1(\bcalX_i) \otimes_{k\neq 1}\U_k\U_k^{\top} \fraM_1(\bcalX_i)^{\top} \U_1)\\
	&+\sum_{i\neq j}\underbrace{\text{trace}(\U_1^{\top} \fraM_1(\bcalX_i) \otimes_{j\neq 1}\U_j\U_j^{\top} \fraM_1(\bcalX_j)^{\top} \U_1)\cdot \text{sign}(\inp{\bcalX_i}{\bcalT-\bcalT^*}-\xi_i-s_i)\cdot \text{sign}(\inp{\bcalX_j}{\bcalT-\bcalT^*}-\xi_j-s_j)}_{C_{ij}}.
\end{align*}
Notice that
\begin{align*}
	\EE \text{trace}(\U_1^{\top} \fraM_1(\bcalX_i) \otimes_{k\neq 1}\U_k\U_k^{\top} \fraM_1(\bcalX_i)^{\top} \U_1)&=\frac{1}{d^*}\fro{\U_1}^2\cdots\fro{\U_m}^2=\frac{r^*}{d^*}.
\end{align*}
\begin{align*}
	\EE \text{trace}(\U_1^{\top} \fraM_1(\bcalX_i) \otimes_{k\neq 1}\U_k\U_k^{\top} \fraM_1(\bcalX_i)^{\top} \U_1)^2 \leq\frac{1}{d^*}\cdot d^*\cdot\ltinf{\U_1}^4\cdots\ltinf{\U_m}^4=\frac{\mu^{2m}r^{*2}}{d^{*2}}.
\end{align*}
\begin{align*}
	\left| \text{trace}(\U_1^{\top} \fraM_1(\bcalX_i) \otimes_{k\neq 1}\U_k\U_k^{\top} \fraM_1(\bcalX_i)^{\top} \U_1)\right|\leq \ltinf{\U_1}^2\cdots \ltinf{\U_m}^2\leq \frac{\mu^m r^*}{d^*}.
\end{align*}
Thus by Bernstein's inequality Theorem~\ref{thm:Bernstein:valuetype}, we have
\begin{align*}
	\PP\left(\left|\sum_{i=1}^{n}\text{trace}(\U_1^{\top} \fraM_1(\bcalX_i) \otimes_{k\neq 1}\U_k\U_k^{\top} \fraM_1(\bcalX_i)^{\top} \U_1)-n\frac{r^*}{d^*}\right|\geq t\right)\leq 2\exp\left(-\frac{t^2}{\frac{n\mu^{2m}r^{*2}}{d^{*2}}+\frac{\mu^m r^*}{d^*} t}\right).
\end{align*}
Take $t=n\frac{r^*}{d^*}$ and then it leads to $$\sum_{i=1}^{n}\text{trace}(\U_1^{\top} \fraM_1(\bcalX_i) \otimes_{k\neq 1}\U_k\U_k^{\top} \fraM_1(\bcalX_i)^{\top} \U_1)\geq 2n\frac{r^*}{d^*}$$ holds with probability less than $2\exp(-n/\mu^{2m}r^{*2})$. Suppose $\bcalX_j^{i},\xi_j^{i}, s_j^{i}$ is an i.i.d. copy of $\bcalX_j, \xi_j, s_j$ respectively. Denote $C_{ij}':= \text{trace}(\U_1^{\top} \fraM_1(\bcalX_i) \otimes_{j\neq 1}\U_j\U_j^{\top} \fraM_1(\bcalX_j^i)^{\top} \U_1)\cdot \text{sign}(\inp{\bcalX_i}{\bcalT-\bcalT^*}-\xi_i-s_i)\cdot \text{sign}(\inp{\bcalX_j^i}{\bcalT-\bcalT^*}-\xi_j^i-s_j^i)$. Then by decoupling technique \citep{de2012decoupling}, we have
\begin{align*}
	\PP\left(\left|\sum_{i\neq j} C_{ij}\right|\geq t\right)\leq C\PP\left(\left|\sum_{i\neq j} C_{ij}'\right|\geq t\right).
\end{align*}
We have
\begin{align*}
	\EE C_{ij}' &= 2(1-\alpha)^2 \EE \text{trace}(\U_1^{\top} \fraM_1(\bcalX_i) \otimes_{j\neq 1}\U_j\U_j^{\top} \fraM_1(\bcalX_j)^{\top} \U_1) \\
	&~~~~~~~~~~~\times(H_{\xi}(\inp{\bcalX_i}{\bcalT-\bcalT^*})-H_{\xi}(0))(H_{\xi}(\inp{\bcalX_j}{\bcalT-\bcalT^*})-H_{\xi}(0))\\
	&+4\alpha(1-\alpha)  \EE \text{trace}(\U_1^{\top} \fraM_1(\bcalX_i) \otimes_{j\neq 1}\U_j\U_j^{\top} \fraM_1(\bcalX_j)^{\top} \U_1) \\
	&~~~~~~~~~~~\times(H_{\xi}(\inp{\bcalX_i}{\bcalT-\bcalT^*}-s_i)-H_{\xi}(0))(H_{\xi}(\inp{\bcalX_j}{\bcalT-\bcalT^*})-H_{\xi}(0))\\
	&+2\alpha^2 \EE \text{trace}(\U_1^{\top} \fraM_1(\bcalX_i) \otimes_{j\neq 1}\U_j\U_j^{\top} \fraM_1(\bcalX_j)^{\top} \U_1) \\
	&~~~~~~~~~~~\times(H_{\xi}(\inp{\bcalX_i}{\bcalT-\bcalT^*}-s_i)-H_{\xi}(0))(H_{\xi}(\inp{\bcalX_j}{\bcalT-\bcalT^*}-s_j)-H_{\xi}(0))\\
	&\leq 2\frac{\mu^{m} r^{*}}{d^{*}}\left(\EE H_{\xi}(\inp{\bcalX_i}{\bcalT-\bcalT^*})-H_{\xi}(0)\right)^2+4\alpha \frac{\mu^{m} r^{*}}{d^{*}} \EE\left[ H_{\xi}(\inp{\bcalX_i}{\bcalT-\bcalT^*})-H_{\xi}(0)\right]\\
	&~~~+2\alpha^2 \frac{\mu^{m} r^{*}}{d^{*}}\\
	&\leq 2\frac{\mu^{m} r^{*}}{d^{*}}\frac{\lone{\bcalT-\bcalT^*}^2}{ d^{*2}b_1^2}+4\alpha \frac{\mu^{m} r^{*}}{d^{*}}\frac{\lone{\bcalT-\bcalT^*}}{ d^{*}b_1}+2\alpha^2 \frac{\mu^{m} r^{*}}{d^{*}}\\
	&\leq 2\frac{\mu^m r^*}{d^{*2}b_1^2}\fro{\bcalT-\bcalT^*}^2+4\alpha \frac{\mu^{m} r^{*}}{d^{*}}\frac{\fro{\bcalT-\bcalT^*}}{ \sqrt{d^{*}}b_1}+2\alpha^2 \frac{\mu^{m} r^{*}}{d^{*}}\\
	&\leq 3\frac{\mu^m r^*}{d^{*2}b_1^2}\fro{\bcalT-\bcalT^*}^2,
\end{align*}
and
\begin{align*}
	\EE(C_{ij}')^2&=\EE \text{trace}(\U_1^{\top} \fraM_1(\bcalX_i) \otimes_{j\neq 1}\U_j\U_j^{\top} \fraM_1(\bcalX_j)^{\top} \U_1)^2\\
	&\leq \frac{\mu^{2m}r^{*2}}{d^{*2}},
\end{align*}
\begin{align*}
	\left|C_{ij}'\right|\leq \frac{\mu^{m}r^*}{d^*}.
\end{align*}
Thus by Bernstein Inequality Theorem~\ref{thm:Bernstein:valuetype}, we have
\begin{align*}
	\PP\left(\left|\sum_{i\neq j} C_{ij}'- \EE C_{ij}'\right|\geq t\right)\leq2\exp\left(-\frac{t^2}{\frac{n\mu^{2m}r^{*2}}{d^{*2}}+\frac{\mu^m r^*}{d^*} t}\right),
\end{align*}
which shows that with probability exceeding $1-\exp(-n)$,
\begin{align*}
	\left|\sum_{i\neq j} C_{ij}'\right|\leq n^2\frac{\mu^m r^*}{d^{*2}b_1^2}\fro{\bcalT-\bcalT^*}^2.
\end{align*}
Thus we have $|B_1|\leq n^2\frac{\mu^m r^*}{d^{*2}b_1^2}\fro{\bcalT-\bcalT^*}^2$. With similar analyses, we have $|B_2|\leq n^2\frac{\mu^m r^*}{d^{*2}b_1^2}\fro{\bcalT-\bcalT^*}^2$ holds with probability exceeding $1-\sum_{k=1}^{m}\exp(-d_k)$. Thus in total we have
\begin{align*}
	\fro{\calP_{\TT}(\bcalG)}^2\leq C_2(m+1)n^2\frac{\mu^m r^*}{d^{*2}b_1^2}\fro{\bcalT-\bcalT^*}^2.
\end{align*}
\paragraph{Analysis of $f(\calP_{\Omega_{j}^{(k)}}(\bcalT))-f(\calP_{\Omega_{j}^{(k)}}(\bcalT^*))$}
Notice that under event $\bcalE_1$, we have
\begin{align*}
	\EE f(\calP_{\Omega_{j}^{(k)}}(\bcalT))-\EE f(\calP_{\Omega_{j}^{(k)}}(\bcalT^*))&=\sum_{i=1}^{n} \EE  \left[\left|\inp{\bcalX_i}{\bcalT-\bcalT^*}-\xi_{i}-s_i\right|-\left|\xi_{i}+s_i\right|\right]\cdot 1_{\{s_i=0\}}\cdot 1_{\{\bcalX_i\in\Omega_{j}^{(k)}\}}\\
	&~~~+\sum_{i=1}^{n} \EE  \left[\left|\inp{\bcalX_i}{\bcalT-\bcalT^*}-\xi_{i}-s_i\right|-\left|\xi_{i}+s_i\right|\right]\cdot 1_{\{s_i\neq0\}}\cdot 1_{\{\bcalX_i\in\Omega_{j}^{(k)}\}}\\
	&\geq \frac{(1-\alpha)n}{d^*}\frac{1}{b_0}\fro{\calP_{\Omega_{j}^{(k)}}(\bcalT-\bcalT^*)}^2-\frac{\alpha n}{d_k}\linft{\bcalT-\bcalT^*}.
\end{align*}
Denote 
\begin{align*}
	\bcalE_j^{(k)}&:= \left\{\left|f(\calP_{\Omega_{j}^{(k)}}(\bcalT))-f(\calP_{\Omega_{j}^{(k)}}(\bcalT^*))-\EE \left[f(\calP_{\Omega_{j}^{(k)}}(\bcalT))-f(\calP_{\Omega_{j}^{(k)}}(\bcalT^*))\right]\right|\leq t_j^{(k)} \right\}.
\end{align*}
And with similar proofs in Lemma~\ref{teclem:norm-relation}, we have $\PP(\bcalE_j^{(k)})\geq 1-\exp\left(-c\frac{n\fro{\calP_{\Omega_{j}^{(k)}}(\bcalT-\bcalT^*)}^2}{d^* \linft{\bcalT-\bcalT^*}^2}\right)$. Then under event $\bcalE_j^{(k)}$, we have
\begin{align*}
	f(\calP_{\Omega_{j}^{(k)}}(\bcalT))- f(\calP_{\Omega_{j}^{(k)}}(\bcalT^*))\geq \frac{n}{2d^*}\frac{1}{b_0}\cdot \fro{\calP_{\Omega_{j}^{(k)}}(\bcalT-\bcalT^*)}^2-\frac{\alpha n}{d_k}\linft{\bcalT-\bcalT^*}.
\end{align*}

\paragraph*{Analysis of $\ltinf{\fraM_k(\calP_{\TT}(\bcalG))}$}
Denote $\bcalT:=\bcalC\cdot\llbracket\U_{1},\dots,\U_m\rrbracket$ and then we have
\begin{align*}
	&~~~~\fraM_k(\calP_{\TT}(\bcalG))\\
	&=\fraM_k(\bcalG)\left(\otimes_{i\neq k}\U_i\right)\fraM_k(\bcalC)^{\dagger}\fraM_k(\bcalC)\left(\otimes_{i\neq k}\U_i\right)^{\top}+\U_k\U_k^{\top}\fraM_k(\bcalG)\left(\otimes_{i\neq k}\U_i\right)\left(\I-\fraM_k(\bcalC)^{\dagger}\fraM_k(\bcalC)\right)\left(\otimes_{i\neq k}\U_i\right)^{\top}\\
	&~~~~+\sum_{i\neq k} \U_k\fraM_k(\bcalC\times_{j\neq i,k}\U_j\times\V_i),
\end{align*}
where $\V_i:=\left(\I_{d_i}-\U_i\U_i^{\top}\right)\fraM_k(\bcalG)\left(\otimes_{j\neq i}\U_j\right)\fraM_k(\bcalC)^{\dagger}$. Then with similar analyses in $B_1$ and $B_2$, we have
\begin{align*}
	\ltinf{\fraM_k(\calP_{\TT}(\bcalG))}^2&\leq m\ltinf{\U_k}^2\cdot C_1 n^2\frac{\mu^m r^*}{d^{*2} b_1^2}\fro{\bcalT-\bcalT^*}^2+\ltinf{\fraM_k(\bcalG)\left(\otimes_{i\neq k}\U_i\right)\fraM_k(\bcalC)^{\dagger}\fraM_k(\bcalC)\left(\otimes_{i\neq k}\U_i\right)^{\top}}^2\\
	&\leq Cmn^2\frac{\mu^m r^*}{d^{*2} b_1^2}\cdot\frac{\mu r}{d_k}\cdot\fro{\bcalT-\bcalT^*}^2+Cn^2\frac{\mu^m r^*}{d^{*2} b_1^2}\cdot\ltinf{\fraM_{k}(\bcalT-\bcalT^*)}^2,
\end{align*}
holds with probability exceeding $1-cd^{*-10}$ when $\ltinf{\fraM_{k}(\bcalT-\bcalT^*)}\geq C_0b_1\cdot\sqrt{\frac{n}{d^*}\log \dmax}$.
\begin{proposition}[Concentration in the setting of Completion and Independence]
	Suppose there are $n$ pairs of i.i.d. observation, $\{(Y_i,\X_i)\}_{i=1}^n$, satifying $Y_i=\inp{\X_i}{\bcalT^*}+\xi_{i}$. Suppose the loss function is given by
	$f(\bcalT):=\sum_{i=1}^{n}\left|Y_i-\inp{\X_i}{\bcalT}\right|$. Then for any fixed $\bcalT\in\RR^{d_1\times\cdots\times d_m}$ we have with probability exceeding $1-2\exp\left(-\frac{t^2}{\frac{n\fro{\bcalT-\bcalT^*}^2}{d^*}+t\linft{\bcalT-\bcalT^*}}\right)$,
	\begin{align*}
		\left|f(\bcalT)-f(\bcalT^*)-\EE\left[f(\bcalT)-f(\bcalT^*)\right]\right|\leq t
	\end{align*}
\label{tecprop:concentration completion indepen}
\end{proposition}
\begin{proof}
	The proof follows Bernstein's Inequality Theorem~\ref{thm:Bernstein:valuetype}. First note that for any $i=1,\dots,n$,
	\begin{align*}
		&~~~\EE\left[\left|Y_i-\inp{\X_i}{\bcalT}\right|-\left|Y_i-\inp{\X_i}{\bcalT^*}\right| -\EE\left[\left|Y_i-\inp{\X_i}{\bcalT}\right|-\left|Y_i-\inp{\X_i}{\bcalT^*}\right| \right] \right]^2\\
		&\leq \EE\left[\left|Y_i-\inp{\X_i}{\bcalT}\right|-\left|Y_i-\inp{\X_i}{\bcalT^*}\right| \right]^2\\
		&\leq\EE\left[\inp{\X_i}{\bcalT-\bcalT^*}\right]^2\\
		&=\frac{1}{d^*}\fro{\bcalT-\bcalT^*}^2.
	\end{align*}
At the same time, it has
\begin{align*}
	\left|\left|Y_i-\inp{\X_i}{\bcalT}\right|-\left|Y_i-\inp{\X_i}{\bcalT^*}\right| -\EE\left[\left|Y_i-\inp{\X_i}{\bcalT}\right|-\left|Y_i-\inp{\X_i}{\bcalT^*}\right| \right]\right|\leq2\linft{\bcalT-\bcalT^*}.
\end{align*}
Thus Theorem~\ref{thm:Bernstein:valuetype} leads to
\begin{align*}
	\left|f(\bcalT)-f(\bcalT^*)-\EE\left[f(\bcalT)-f(\bcalT^*)\right]\right|\geq t
\end{align*}
holds with probability bounded with $2\exp\left(-\frac{t^2}{\frac{n\fro{\bcalT-\bcalT^*}^2}{d^*}+t\linft{\bcalT-\bcalT^*}}\right)$.
\end{proof}
\begin{theorem}[Bernstein's Inequality]
	Let $X_1,\dots,X_n$ be independent zero-mean random variables. Suppose that $|X_i|\leq M$ almost surely, for all $i$. Then for all positive $t>0$,\begin{align*}
		\PP\left(\big| \sum_{i=1}^n X_i\big|\geq t\right)\leq2\exp\left(-\frac{\frac{1}{2}t^2}{\sum_{i=1}^{n}\EE X_i^2 + \frac{1}{3}Mt}\right).
	\end{align*}
	\label{thm:Bernstein:valuetype}
\end{theorem}

\section{Technical Lemma}
The following lemma connects $\lone{\cdot}$, $\linft{\cdot}$ norm and $\fro{\cdot}$ norm.
\begin{lemma}
	For any tensor $\bcalT\in\RR^{d_1\times\cdots\times d_m}$, its entrywise $\ell_1$, $\ell_{\infty}$ and Frobenius norm  have the following relationship:
	\begin{align*}
		\linft{\bcalT}\lone{\bcalT}\geq \fro{\bcalT}^2,\ \ \ \lone{\bcalT}\leq\sqrt{d_1\cdots d_m}\fro{\bcalT}.
	\end{align*}
	\label{teclem:norm-relation}
\end{lemma}
\begin{proof}
	Notice that we could obtain $\lone{\bcalT}\leq\sqrt{d_1\cdots d_m}\fro{\bcalT}$ by Cauchy-Schwarz inequality. Then we only need to discuss the first inequality. Note that Frobenius norm is defined to be $\fro{\bcalT}=\sup_{\bcalM:\,\fro{\bcalM}=1} \inp{\bcalT}{\bcalM}$ and suppose it achieves the supremum at $\bcalM_0$, which implies $$\fro{\bcalT}= \inp{\bcalT}{\bcalM_0}, \quad \fro{\bcalM_0}=1,\quad\text{sign}(\bcalT)=\text{sign}(\bcalM_0),\quad  \linft{\bcalT}/\fro{\bcalT}=\linft{\bcalM_0}.$$
	Hence, we have \begin{align*}
		\lone{\bcalT}=\inp{\bcalT}{\text{sign}(\bcalT)}=\frac{1}{\linft{\bcalM_0}} \inp{\bcalT}{\linft{\bcalM_0}\cdot\text{sign}(\bcalT)}\geq\frac{\fro{\bcalT}}{\linft{\bcalT}}\cdot \inp{\bcalT}{\bcalM_0}=\fro{\bcalT}^2/\linft{\bcalT}.
	\end{align*}
\end{proof}

The following lemma analyzes slice sum of the heavy-tailed noise term. Recall that $d^*=d_1\cdots d_m$ and $\dkm=d^*/d_k$, for each $k=1,\dots,m$. Also recall that $$\lone{\calP_{\Omega_{j}^{(k)}}(\bXi)}=\sum_{i_1=1}^{d_1}\cdots\sum_{i_{k-1}=1}^{d_{k-1}}\sum_{i_{k+1}=1}^{d_{k+1}}\cdots\sum_{i_m=1}^{d_m}\big|\xi_{i_1\cdots i_{k-1}ji_{k+1}\cdots i_m}\big|=:\sum_{i_k=j}\big|\xi_{i_1\cdots i_{k-1}ji_{k+1}\cdots i_m}\big|. $$
\begin{lemma}
	Suppose random tensor $\bXi=\left(\xi_{i_1\cdots i_m}\right)\in\RR^{d_1\times\cdots\times d_m}$ contains i.i.d. entries with finite $2+\eps$ moment, namely, $\EE|\xi_{i_1\cdots i_m}|^{2+\eps}<+\infty$. Then for each $k=1,\dots,m$, with probability exceeding $1-c_1\frac{d_k}{\dkm}\cdot \left(\dkm\right)^{-\min\{\eps,1\}}-c_2 \frac{d_k}{\dkm}\cdot \left(\dkm\right)^{-1}$, we have
	$$\lone{\calP_{\Omega_{j}^{(k)}}(\bXi)}\leq 3\left(\EE|\xi|^{2+\eps}\right)^{1/(2+\eps)}\cdot \dkm,\quad \text{ for all }j=1,\dots,d_k.$$
	\label{teclem:Contraction of Heavy Tailed Random Variables:slice}
\end{lemma}
\begin{proof}
	First consider the case when $\eps<1$. For convenience, denote $\varphi_{i_1\cdots i_m}:=\big| \xi_{i_1\cdots i_m}\big|$. Introduce $\xi$ i.i.d. with $\xi_{i_1\cdots i_m}$ and denote $\gamma:=\left(\EE|\xi|^{2+\eps}\right)^{1/(2+\eps)}$. For constant $s>0$, define the truncated variable $\bar{\varphi}_{i_1\cdots i_m}:=\big| \xi_{i_1\cdots i_m}\cdot 1_{\{|\xi_{i_1\cdots i_m}|\leq s\}}\big|$. And $\varphi,\bar{\varphi}$ are i.i.d. copy of $\varphi_{i_1\cdots i_m},\bar{\varphi}_{i_1\cdots i_m}$, respectively. Consider the probability of $ \varphi_{i_1\cdots i_m}\neq\bar{\varphi}_{i_1\cdots i_m}$,
	\begin{align*}
		\PP\left(\varphi_{i_1\cdots i_m}\neq\bar{\varphi}_{i_1\cdots i_m}\right)=\PP\left(|\xi_{i_1\cdots i_m}|>s\right)\leq\frac{\EE|\xi|^{2+\eps}}{s^{2+\eps}}.
	\end{align*}
	Hence, for the slice, we have
	\begin{align*}
		\PP\left(\sum_{i_k=j}\bar{\varphi}_{i_1\cdots i_{k-1}ji_{k+1}\cdots i_m}\neq \sum_{i_k=j}\varphi_{i_1\cdots i_{k-1}ji_{k+1}\cdots i_m}\right)&\leq \sum_{i_k=j}\PP\left(\bar{\varphi}_{i_1\cdots i_{k-1}ji_{k+1}\cdots i_m}\neq \varphi_{i_1\cdots i_{k-1}ji_{k+1}\cdots i_m}\right)\\
		&\leq \frac{d_1\cdots d_m}{d_k}\cdot\frac{\EE|\xi|^{2+\eps}}{s^{2+\eps}}=\dkm\cdot \frac{\EE|\xi|^{2+\eps}}{s^{2+\eps}}.
	\end{align*}
	Then consider $\bar{\varphi}$,
	\begin{align*}
		&\PP\left(\bigg|\sum_{i_k=j}\left[\bar{\varphi}_{i_1\cdots i_{k-1}ji_{k+1}\cdots i_m}-\EE\bar{\varphi}_{i_1\cdots i_{k-1}ji_{k+1}\cdots i_m}\right]\bigg|\geq s\right)\\
		&{~~~~~~~~~~~~~~~~~~~~~~}\leq \frac{\EE\bar{\varphi}^4\cdot d_1\cdots d_m/d_k+\left(\EE\bar{\varphi}^2\right)^2\cdot d_1^2\cdots d_m^2/d_k^2}{s^4}\\
		&{~~~~~~~~~~~~~~~~~~~~~~}\leq \frac{\EE|\xi|^{2+\eps}\cdot d_1\cdots d_m/d_k}{s^{2+\eps}}+\frac{\left(\EE\xi^2\right)^2\cdot d_1^2\cdots d_m^2/d_k^2}{s^4}= \frac{\EE|\xi|^{2+\eps}\cdot \dkm}{s^{2+\eps}}+\frac{\left(\EE\xi^2\right)^2\cdot (\dkm)^2}{s^4},
	\end{align*}
	where the first inequality is from Markov inequality and the second inequality uses $\EE\bar{\varphi}^4\leq s^{2-\eps}\EE|\xi|^{2+\eps}$. We take $s=\dkm\cdot \gamma$ and then by the above two equations we have
	\begin{align*}
		&{~~~~}\PP\left(\bigg|\sum_{i_k=j}\varphi_{i_1\cdots i_{k-1}ji_{k+1}\cdots i_m}-\EE\varphi_{i_1\cdots i_{k-1}ji_{k+1}\cdots i_m}\bigg|\geq 2\gamma\cdot \dkm\right)\\
		&\leq\PP\left(\sum_{i_k=j}\bar{\varphi}_{i_1\cdots i_{k-1}ji_{k+1}\cdots i_m}\neq \sum_{i_k=j}\varphi_{i_1\cdots i_{k-1}ji_{k+1}\cdots i_m}\right)\\
		&{~~~~~~~~~~~~~~~~~~~~~~~~~~~~}+\PP\left(\bigg|\sum_{i_k=j}\left[\bar{\varphi}_{i_1\cdots i_{k-1}ji_{k+1}\cdots i_m}-\EE\bar{\varphi}_{i_1\cdots i_{k-1}ji_{k+1}\cdots i_m}\right]\bigg|\geq \gamma\cdot \dkm\right)\\
		&\leq 2\left(\dkm\right)^{-(1+\eps)}+ \left(\dkm\right)^{-2},
	\end{align*}
	where we use Markov ineuqality, $\EE\varphi_{i_1\cdots i_{k-1}ji_{k+1}\cdots i_m}- \EE\bar{\varphi}_{i_1\cdots i_{k-1}ji_{k+1}\cdots i_m}\leq \gamma$ and $\EE\xi^2\leq\gamma^{2}$.
	Hence, take the union for all $j=1,\dots,d_m$ and then we have with probability exceeding $1-2\frac{d_k}{\dkm}\cdot \left(\dkm\right)^{-\eps}- \frac{d_k}{\dkm}\cdot \left(\dkm\right)^{-1}$, the following holds
	$$\lone{\calP_{\Omega_{j}^{(k)}}(\bXi)}\leq 3\gamma\cdot \dkm,\quad \text{ for all } j=1,\dots,d_k.$$
	Case of $\eps\geq1$ has similar proof where $\lone{\calP_{\Omega_{j}^{(k)}}(\bXi)}\leq 3\gamma\cdot \dkm$ holds for all $j=1,\dots,d_k$ with probability exceeding $1-3\frac{d_k}{\dkm}\cdot \left(\dkm\right)^{-1}$.
\end{proof}

\begin{lemma}
	Suppose tensors $\bcalT,\bcalT^*\in\MM_{\r,\mu}$ have same Tucker rank, with Tucker decomposition $\bcalT=\bcalC\times \U_{1}\times_2\cdots\times_m\U_{m}$ and $\bcalT^*=\bcalC^*\times \U_{1}^*\times_2\cdots\times_m\U_{m}^*$. Introduce matrices $\H_k:=\U_k^{\top}\U_k^*$, for each $k=1,\dots,m$. Then we have
	$$\linft{\bcalT-\bcalT^*}\leq\sqrt{\frac{\mu^m r_1\cdots r_m}{d_1\cdots d_m}}\fro{\bcalT-\bcalT^*}+\sum_{k=1}^m \sqrt{\frac{\mu^{m-1} r_1\cdots r_m/r_k}{d_1\cdots d_m/d_k}} \ltinf{\left(\U_k\H_k-\U_k^*\right)\fraM_k(\bcalC^*)}.$$
	\label{teclem:entrynorm-expansion}
\end{lemma}
\begin{proof}
	First consider difference between $\bcalT$ and $\bcalT^*$,
	\begin{align*}
		&{~~~~}\bcalT-\bcalT^*\\
		&=\bcalC\times \U_{1}\times_2\cdots\times_m\U_{m}-\bcalC^*\times \U_{1}^*\times_2\cdots\times_m\U_{m}^*\\
		&=\left(\bcalC-\bcalC^*\times_1\H_1\times_2\cdots\times_m\H_m\right)\times_1\U_1\times_2\cdots\times_m\U_{m}+\sum_{k=1}^{m} \bcalC^*\times_{i<k} \U_i^*\times_k (\U_k\H_k-\U^*)\times_{j>k}\U_j\H_j.
	\end{align*}
	Note that the first term has the expression \begin{align*}
		\bcalC-\bcalC^*\times_1\H_1\times_2\cdots\times_m\H_m&=\bcalC-\bcalT^*\times_{1}\U_{1}^{\top}\times_2\cdots\times_m\U_m^{\top}=\left(\bcalT-\bcalT^*\right)\times_{1}\U_{1}^{\top}\times_2\cdots\times_m\U_m^{\top},
	\end{align*}
	which shows
	$$\fro{\bcalC-\bcalC^*\times_1\H_1\times_2\cdots\times_m\H_m }\leq\fro{\bcalT-\bcalT^*},$$
	$$ \linft{\bcalT-\bcalT^*}\leq \sqrt{\frac{\mu^m r_1\cdots r_m}{d_1\cdots d_m}}\fro{\bcalT-\bcalT^*}+\sum_{k=1}^m \sqrt{\frac{\mu^{m-1} r_1\cdots r_m/r_k}{d_1\cdots d_m/d_k}} \ltinf{\left(\U_k\H_k-\U_k^*\right)\fraM_k(\bcalC^*)}$$
\end{proof}

\begin{lemma}
	Pseudo-Huber loss function $\rho(x)=\sqrt{x^2+\delta^2}$ maps $\RR$ to $\RR$. Denote derivative of $\rho(\cdot)$ as $\dot{\rho}(\cdot)$ and than we have $\dot{\rho}(\cdot)$ is Lipschitz continuous with $\delta^{-1}$, namely, $$\left|\dot{\rho}(x_1)-\dot{\rho}(x_2)\right|\leq \delta^{-1}\left|x_1-x_2\right|\quad\text{ for all } x_1,x_2\in\RR,$$ moreover, we have $$\left(\dot{\rho}(x_1)-\dot{\rho}(x_2) \right)^2\leq \delta^{-1}(x_1-x_2)(\dot{\rho}(x_1)-\dot{\rho}(x_2)), \quad\text{ for all } x_1,x_2\in\RR.$$
	\label{teclem:hubfunction}
\end{lemma}
\begin{proof}
	Notice that $\dot{\rho}(x)=\frac{x}{\sqrt{x^2+\delta^2}}$ and second derivative of $\rho(\cdot)$ is $\ddot{\rho}(x)=\frac{\delta^2}{\left(x^2+\delta^2\right)^{3/2}}$. $\ddot{\rho}$ is a bounded function $0\leq\ddot{\rho}(x)\leq\delta^{-1}$. Then for any $x_1,x_2$, we have $$\dot{\rho}(x_1)-\dot{\rho}(x_2)=\ddot{\rho}(\theta x_1+(1-\theta)x_2) (x_1-x_2),$$ where $\theta\in[0,1]$ is some constant. Hence, we have $\left|\dot{\rho}(x_1)-\dot{\rho}(x_2)\right|\leq \delta^{-1}\left|x_1-x_2\right|$. Then, by $\ddot{\rho}(x)>0$, we have $$\left(\dot{\rho}(x_1)-\dot{\rho}(x_2) \right)^2\leq \delta^{-1}(x_1-x_2)(\dot{\rho}(x_1)-\dot{\rho}(x_2)).$$
\end{proof}

\begin{lemma}[Lemma B.8 of \cite{cai2022generalized}]
	Let $\Omega$ be the $\alpha$-fraction set. Suppose $\bcalT^*\in\MM_{\r}$ is $\muT$-incoherent. Under the assumptions that $\bcalT\in\MM_{\r}$ is $\mu$-incoherent and $\left\|\bcalT_l- \bcalT^*\right\|_{\mathrm{F}} \leq \frac{\mins^*}{16 m}$, we have
	$$
	\left\|\mathcal{P}_{\Omega}\left(\bcalT-\bcalT^*\right)\right\|_{\mathrm{F}}^2 \leq C_m\alpha\max\{\muT,\mu\}^m r^* \left\|\bcalT-\bcalT^*\right\|_{\mathrm{F}}^2
	$$
	where $C_m=4(m+1)$.
	\label{teclem:alpha-bound}
\end{lemma}

\subsection{Empirical processes for tensor PCA}
\begin{lemma}
	Let $\bcalE=\left(\eps_{i_1\dots i_m}\right)\in\RR^{d_1\times\dots\times d_m}$ be a random tensor with i.i.d.Rademacher entries, namely, $\PP(\eps_{i_1\dots i_m}=1)=\PP(\eps_{i_1\dots i_m}=-1)=1/2$. Then there exists some $c>0$ such that for all $t>0$, \begin{align*}
		\PP\left(\sup_{\bcalM\in\MM_{\r},\fro{\bcalM}\leq 1}|\inp{\bcalE}{\bcalM}|\geq t\right)\leq 2\exp\left(-\frac{t^2}{2} +C\left(r_1\cdots r_m+\sum_{j=1}^{m}r_jd_j\right)\right).
	\end{align*}
	Specifically, it infers $$\EE\sup_{\bcalM\in\MM_{\r},\fro{\bcalM}\leq 1}|\inp{\bcalE}{\bcalM}|\leq C\sqrt{r_1\cdots r_m+\sum_{j=1}^{m}r_jd_j}.$$
	\label{teclem:rademacher}
\end{lemma}
\begin{proof}
	The proof follows $\eps$-net arguments. Suppose it achieves the supremum at $\bcalM_0\in\MM_{\r}$, $$\sup_{\bcalM\in\MM_{\r},\, \fro{\bcalM}\leq1}\inp{\bcalE}{\bcalM}=\inp{\bcalE}{\bcalM_0},$$ with $\fro{\bcalM_0}=1$.
	Then there exist core tensors $\bcalC_0\in\RR^{r_1\times \cdots\times r_m}$ and orthogonal matrices $ \U_1^{(0)}\in\OO_{d_1, r_1},\dots,\U_m^{(0)}\in\OO_{d_m, r_m}$ such that 
	$$\bcalM_0=\bcalC_0\times_{1} \U_1^{(0)}\times_2\cdots\times_m \U_m^{(0)},$$ 
	Notice that $\fro{\bcalC_0}=1$. Define $\FF_\r=\{\bcalC\in\RR^{r_1\times\dots\times r_m}: \fro{\bcalC}=1\}$ to be the set of tensors with unit Frobenius norm. Note that $\FF_\r$ has one $\eps/(m+1)$-net $\calN_{\eps/(m+1)}^{\FF_\r}$ of  cardinality $|\calN_{\eps/(m+1)}^{\FF_\r}|\leq (3(m+1)/\eps)^{r_1r_2\cdots r_m}$ with respect to the Frobenius norm. 
	
	Suppose $\calN_j$  is $\eps/(m+1)$-nets of orthogonal matrix sets $\OO_{d_k,r_k}$ with respect to $\fro{\cdot}$ norm. They have cardinalities $$|\calN_1|\leq(3(m+1)/\eps)^{d_1r_1},\dots,|\calN_m|\leq(3(m+1)/\eps)^{d_mr_m}.$$ See \cite{rauhut2017low,vershynin2018high} for more about $\eps$-nets. Furthermore, the net $$\calN:=\{\bcalM=\bcalD\times_{1} \V_1\times_2\cdots\times_m \V_m: \bcalD\in\calN_{\eps/(m+1)}^{\FF_\r},\,\V_1\in\calN_1,\dots,\V_m\in\calN_m\}$$ forms a net of $\MM_{\r}\bigcap \{\bcalT:\fro{\bcalT}\leq1\}$ with cardinality $|\calN|\leq(3(m+1)/\eps)^{r_1r_2\cdots r_m+\sum_{j=1}^{m}r_jd_j}$.
	
	Hence, tensor $\bcalM_0=\bcalC_0\times_{1} \U_1^{(0)}\times_2\cdots\times_m \U_m^{(0)}$ has close approximation in the nets. Exist $\bcalC\in\calN_{\eps/(m+1)}^{\FF_\r}, \,\U_1\in\calN_1,\dots, \U_m\in\calN_m$ such that 
	$$\fro{\bcalC_0-\bcalC}\leq \eps/(m+1),\quad \fro{\U_k^{(0)}-\U_k}\leq \eps/(m+1),\text{ for all }k=1,\dots,m.$$ 
	Denote the approximation in the nets as $\bcalT=\bcalC\times_{1} \U_1\times_2\cdots\times_m \U_m$. Note that $\bcalT$ belongs to $\calN$ and it has \begin{align*}
		\bcalM_0-\bcalT&=(\bcalC_0-\bcalC)\cdot\llbracket\U_1^{(0)},\dots,\U_m^{(0)}\rrbracket+\sum_{i=1}^{m}\bcalC\cdot\llbracket\U_1,\dots,\U_j^{(0)}-\U_j,\dots,\U_m^{(0)}\rrbracket,
	\end{align*}
	by which we have $\fro{\bcalM_0-\bcalT}\leq \eps$. Then come back to $\sup_{\bcalM\in\MM_{\r},\, \fro{\bcalM}\leq1}\inp{\bcalE}{\bcalM}$ and we have  \begin{align*}
		\sup_{\bcalM\in\MM_{\r},\, \fro{\bcalM}\leq1}\inp{\bcalE}{\bcalM}&=|\inp{\bcalE}{\bcalM_0}|\leq |\inp{\bcalE}{\bcalT}|+|\inp{\bcalE}{\bcalM_0-\bcalT}|\\
		&\leq \sup_{\bcalM\in\calN}|\inp{\bcalE}{\bcalM}|+\eps\sup_{\bcalM\in\MM_{\r},\, \fro{\bcalM}\leq1}\inp{\bcalE}{\bcalM},
	\end{align*}
	which leads to
	\begin{align}
		\sup_{\bcalM\in\MM_{\r},\, \fro{\bcalM}\leq1}\inp{\bcalE}{\bcalM}\leq\frac{1}{1-\eps}\sup_{\bcalM\in\calN}|\inp{\bcalE}{\bcalM}|.
		\label{eq1}
	\end{align}
	On the other hand, for any fixed $\bcalM\in\calN$, we have $$\inp{\bcalE}{\bcalM}=\sum_{i_1=1}^{d_1}\dots\sum_{i_m=1}^{d_m}\eps_{i_1\dots i_m}M_{i_1\dots i_m},$$
	Also, note that $$-|M_{i_1\dots i_m}|\leq\eps_{i_1\dots i_m}M_{i_1\dots i_m}\leq |M_{i_1\dots i_m}|.$$
	By Hoeffding's inequality, we have \begin{align*}
		\PP\left(|\inp{\bcalE}{\bcalM}|\geq t\right)\leq 2\exp\left(-\frac{t^2}{2}\right).
	\end{align*}
	Take the union over $\calN$ and it yields \begin{align*}
		\PP\left(\sup_{\bcalM\in\calN}\inp{\bcalE}{\bcalM}\geq t\right)\leq 2(3(m+1)/\eps)^{r_1r_2\cdots r_m+\sum_{j=1}^{m}r_jd_j}\exp\left(-\frac{t^2}{2}\right).
	\end{align*}
	The above equation could be simplified to
	\begin{align*}
		\PP\left(\sup_{\bcalM\in\MM_{\r},\, \fro{\bcalM}\leq1}\inp{\bcalE}{\bcalM}\geq t\right)\leq 2\exp\left(-\frac{t^2}{2} +C\left(r_1\cdots r_m+\sum_{j=1}^{m}r_jd_j\right)\right)
	\end{align*}
	Take $\eps=1/2$	and then with Equation~\eqref{eq1}, it verifies \begin{align*}
		\PP\left(\sup_{\bcalM\in\MM_{\r,\mu}^{(k)},\fro{\bcalM}\leq 1}|\inp{\bcalE}{\bcalM}|\geq 2t\right)\leq 2\exp\left(-\frac{t^2}{2} +C\left(r_1\cdots r_m+\sum_{j=1}^{m}r_jd_j\right)\right).
	\end{align*}
\end{proof}

\begin{lemma} Suppose $f(\cdot)$ is given by $f(\bcalT):=\sum_{i_1,\dots,i_m}\rho([\bcalT]_{i_1,\dots,i_m}-[\bcalY]_{i_1,\dots,i_m})$, where $\rho(\cdot)$ is Lipschitz $\tilde{L}$ continuous and $\bcalY=\bcalT^*+\bXi$ with independent entries in $\bXi$. Then there exist constants $C,C_1,C_2>0$ such that,
	\begin{align}
		\bigg|f(\bcalT+\Delta\bcalT)-f(\bcalT)-\EE[f(\bcalT+\Delta\bcalT) - f(\bcalT)]\bigg|\leq \tilde{L}\left(t+C\sqrt{r_1\cdots r_m+\sum_{j=1}^m r_jd_j}\right)\fro{\Delta\bcalT}
	\end{align}
	holds for all $\Delta\bcalT\in\MM_{\r}$ and $\bcalT\in\RR^{d_1\times\cdots\times d_m}$ with probability exceeding $1-\exp\left(-t^2/2\right)$.
	\label{teclem:empirical}
\end{lemma}
\begin{proof}
	For simplicity, we shall use $T_{i_1\cdots i_m}$ to represent the $(i_1,\dots,i_m)$ entry of tensor $\bcalT$. Denote $Z:=\sup_{\bcalT\in\MM_{\r}}\bigg|f(\bcalT+\Delta\bcalT)-f(\bcalT)-\EE[f(\bcalT+\Delta\bcalT) - f(\bcalT)]\bigg|\cdot \fro{\Delta\bcalT}^{-1}$. First consider $\EE Z$,
	\begin{align*}
		\EE Z&=\EE \sup_{\Delta\bcalT\in\MM_{\r},\bcalT}\bigg|f(\bcalT+\Delta\bcalT)-f(\bcalT)-\EE[f(\bcalT+\Delta\bcalT) - f(\bcalT)]\bigg|\cdot \fro{\Delta\bcalT}^{-1}\\
		&\leq 2\EE\sup_{\Delta\bcalT\in\MM_{\r},\bcalT}\bigg|\sum_{i_1=1}^{d_1}\dots\sum_{i_m=1}^{d_m} \eps_{i_1\dots i_m}(\rho(T_{i_1\cdots i_m}+\Delta T_{i_1\dots i_m}-Y_{i_1\dots i_m})-\rho(T_{i_1\dots i_m}-Y_{i_1\dots i_m}) )\bigg|\cdot \fro{\Delta\bcalT}^{-1}\\
		&\leq 4\tilde{L}\EE\sup_{\Delta\bcalT\in\MM_{\r}}\bigg|\sum_{i_1=1}^{d_1}\dots\sum_{i_m=1}^{d_m}  \eps_{i_1\dots i_m}\Delta T_{i_1\dots i_m}\bigg|\cdot \fro{\Delta\bcalT}^{-1}\\
		&\leq 4 \tilde{L}\EE\sup_{\bcalM\in\MM_{\r},\, \fro{\bcalM}\leq1} \inp{\bcalE}{\bcalM},
	\end{align*}
	where $\bcalE=(\eps_{i_1\dots i_m})$ is the $d_1\times\dots\times d_m$ random tensor with i.i.d. Rademacher entries. The second line is from Theorem~\ref{Symmetrization of Expectation}, the third line is from Theorem~\ref{Contraction Theorem}. Thus by  Lemma~\ref{teclem:rademacher} we finally get the upper bound of $\EE Z$. 
	$$\EE Z\leq C\tilde{L}\sqrt{r_1\cdots r_m+\sum_{j=1}^{m}r_jd_j}.$$
	Note that with Lipschitz continuity of the loss function $$ \bigg| \rho(T_{i_1\dots i_m}+\Delta T_{i_1\dots i_m}-Y_{i_1\cdots i_m})-\rho(T_{i_1\cdots i_m}-Y_{i_1\cdots i_m})\bigg|\cdot \fro{\Delta\bcalT}^{-1}\leq\tilde{L}|\Delta T_{i_1\dots i_m}|\cdot\fro{\Delta\bcalT}^{-1},$$
	and sum of the squared upper bound is $\sum_{i_1=1}^{d_1}\cdots\sum_{i_m=1}^{d_m} |\Delta T_{i_1\dots i_m}|^2\cdot\fro{\Delta\bcalT}^{-2}=1$.
	Then by Theorem~\ref{tecthm:Hoeffding empirical}, we have $$\PP\left(Z\geq t\tilde{L}+C\tilde{L}\sqrt{r_1r_2\dots r_m+\sum_{j=1}^{m} r_jd_j}\right)\leq \exp\left(-t^2/2\right).$$
	
\end{proof}

\begin{theorem}[Symmetrization of Expectations, \citep{van1996weak}]\label{Symmetrization of Expectation} 
	Consider $\mathbf{X}_{1},\mathbf{X}_{2},\cdots,\mathbf{X}_{n}$ independent matrices in $\chi$ and let $\mathcal{F}$ be a class of real-valued functions on $\chi$. Let $\tilde{\varepsilon}_{1},\cdots,\tilde{\varepsilon}_{n}$ be a Rademacher sequence independent of $\mathbf{X}_{1},\mathbf{X}_{2},\cdots,\mathbf{X}_{n}$, then
	\begin{equation}
		\mathbb{E}\big[\sup_{f\in\mathcal{F}}\big{|} \sum_{i=1}^{n} (f(\mathbf{X}_{i}) - \mathbb{E}f(\mathbf{X}_{i}))\big{|}\big]\leq 2\mathbb{E}\big[\sup_{f\in\mathcal{F}} \big{|} \sum_{i=1}^{n} \tilde{\varepsilon}_{i}f(\mathbf{X}_{i})\big{|}\big]
	\end{equation}
\end{theorem}

\begin{theorem}[Contraction Theorem, \citep{ludoux1991probability}]\label{Contraction Theorem}
	
	Consider the non-random elements $x_{1}, \ldots, x_{n}$ of $\chi$. Let $\mathcal{F}$ be a class of real-valued functions on $\chi$. Consider the Lipschitz continuous functions $\rho_{i}: \mathbb{R} \rightarrow \mathbb{R}$ with Lipschitz constant $L$, i.e.
	$$
	\left|\rho_{i}(\mu)-\rho_{i}(\tilde{\mu})\right| \leq L|\mu-\tilde{\mu}|, \text { for all } \mu, \tilde{\mu} \in \mathbb{R}
	$$
	Let $\tilde{\varepsilon}_{1}, \ldots, \tilde{\varepsilon}_{n}$ be a Rademacher sequence $.$ Then for any function $f^{*}: \chi \rightarrow \mathbb{R}$, we have
	
	\begin{equation}
		\mathbb{E}\left[\sup _{f \in \mathcal{F}}\left|\sum_{i=1}^{n} \tilde{\varepsilon}_{i}\left\{\rho_{i}\left(f\left(x_{i}\right)\right)-\rho_{i}\left(f^{*}\left(x_{i}\right)\right)\right\}\right|\right] \leq 2 \mathbb{E}\left[L\sup_{f \in \mathcal{F}} \mid \sum_{i=1}^{n} \tilde{\varepsilon}_{i}\left(f\left(x_{i}\right)-f^{*}\left(x_{i}\right)\right)\mid \right]
	\end{equation}
\end{theorem}

\begin{theorem}[Theorem 12.1 of \cite{boucheron2013concentration}]
	Assume that the sequences of vectors $\left(b_{i, s}\right)_{s \in \mathcal{T}}$ and $\left(a_{i, s}\right)_{s \in \mathcal{T}}$, $i=1, \ldots, n$ are such that $a_{i, s} \leq X_{i, s} \leq b_{i, s}$ holds for all $i=1, \ldots, n$ and $s \in \mathcal{T}$ with probability 1. Denote
	$$
	v=\sup _{s \in \mathcal{T}} \sum_{i=1}^n\left(b_{i, s}-a_{i, s}\right)^2 \quad \text { and } \quad V=\sum_{i=1}^n \sup _{s \in \mathcal{T}}\left(b_{i, s}-a_{i, s}\right)^2 .
	$$
	Then for all $\lambda \in \mathbb{R}$,
	$$
	\log \boldsymbol{E} e^{\lambda(Z-E Z)} \leq \frac{v \lambda^2}{2} \quad \text { and } \quad \log \boldsymbol{E} e^{\lambda(Z-E Z)} \leq \frac{V \lambda^2}{8}.
	$$
	\label{tecthm:Hoeffding empirical}
\end{theorem}

\subsection{Expectation of Loss Functions}
\begin{lemma}[Pseudo-Huber Loss]
	Suppose the noise assumption~\ref{assu:hub:noise} holds and $f(\cdot)$ is given in Equation~\eqref{eq:loss-pHuber}, then for all $\bcalT,\bcalM$ we have 
	$$\EE f(\bcalT)-\EE f(\bcalM)\leq \frac{1}{2\delta}\left|\fro{\bcalT-\bcalT^*}^2-\fro{\bcalM-\bcalT^*}^2\right|.$$
	Furthermore, if $\linft{\bcalT-\bcalT^*}\leq C_{m,\mu,r^*}(6\gamma+\delta)$ holds, we have $$\EE f(\bcalT)-\EE f(\bcalT^*)\geq\frac{1}{3b_0}\fro{\bcalT-\bcalT^*}^2.$$
	\label{teclem:pseudo:function expectation}
\end{lemma}
\begin{proof}
	Define $g(t):=\EE\sqrt{(t-\xi)^2+\delta^2}=\int_{-\infty}^{+\infty} \sqrt{(t-s)^2+\delta^2}\, dH_{\xi}(s)$. Note that \begin{align*}
		g'(t)=\int_{-\infty}^{+\infty} \frac{t-s}{\sqrt{(t-s)^2+\delta^2}}\, dH_{\xi}(s),\quad
		g''(t)=\int_{-\infty}^{+\infty} \frac{\delta^2}{\left((t-s)^2+\delta^2\right)^{3/2}}\, dH_{\xi}(s).
	\end{align*}
	According to density $h_{\xi}(\cdot)$ condition, we have $g'(0)=\int_{-\infty}^{+\infty} \frac{-s}{\sqrt{s^2+\delta^2}}\, dH_{\xi}(s)=0$. Then for arbitrary $t_1,t_2\in\RR$, we have
	\begin{align*}
		g(t_2)-g(t_1)&=\int_{t_1}^{t_2} \int_{-\infty}^{+\infty} \frac{t-s}{\sqrt{(t-s)^2+\delta^2}}\, dH_{\xi}(s)\, dt\\
		&=\int_{t_1}^{t_2} \int_{-\infty}^{+\infty} \frac{t}{\sqrt{(t-s)^2+\delta^2}}\, dH_{\xi}(s)\, dt\\
		&\leq \frac{1}{2\delta}\left|t_2^2-t_1^2\right|,
	\end{align*}
	where $\sqrt{(t-s)^2+\delta^2}\geq\delta$ is used. Besides, we have the Taylor expansion at $0$ using the second order derivative $g''(t)$,
	\begin{align*}
		g(t_0)-g(0)=\int_{0}^{t_0}\int_{-\infty}^{+\infty} \frac{t\delta^2}{\left((t-s)^2+\delta^2\right)^{3/2}}\, dH_{\xi}(s)\, dt=\int_{-\infty}^{+\infty}\int_{0}^{t_0} \frac{t\delta^2}{\left((t-s)^2+\delta^2\right)^{3/2}}\cdot h_{\xi}(s)\, dt\, ds
	\end{align*}
	When $|t_0|\leq C_{m,\mu,r^*}(6\gamma+\delta)$, with density lower bound in Assumption~\ref{assu:hub:noise}, we have
	\begin{align*}
		g(t_0)-g(0)&= \int_{0}^{t_0}\int_{-\infty}^{+\infty} \frac{t\delta^2}{\left((t-s)^2+\delta^2\right)^{3/2}}h_{\xi}(s)\, ds\, dt\\
		&\geq \int_{0}^{t_0}\int_{t-\delta}^{t+\delta} \frac{t\delta^2}{\left((t-s)^2+\delta^2\right)^{3/2}}h_{\xi}(s)\, ds\, dt \\
		&\geq  \frac{1}{3\delta \cdot b_0}\cdot \int_{0}^{t_0}\int_{t-\delta}^{t+\delta} t\, ds\, dt\\
		&\geq \frac{t_0^2}{3b_0},
	\end{align*}
	where the third line is because when $|s-t|\leq\delta$, $\frac{\delta^2}{\left((t-s)^2+\delta^2\right)^{3/2}}\geq \frac{1}{3\delta}$. Thus altogether, we get $$ \EE\sqrt{(t-\xi)^2+\delta^2}-\EE\sqrt{\xi^2+\delta^2}\geq\frac{t_0^2}{3b_0}, \quad  \EE\sqrt{(t_2-\xi)^2+\delta^2}-\EE\sqrt{(t_1-\xi)^2+\delta^2}\leq\frac{1}{2\delta}\left|t_2^2-t_1^2\right|,$$ Then come back to $\EE f(\bcalT)-\EE f(\bcalT^*)$,
	\begin{align*}
		\EE f(\bcalT)-\EE f(\bcalT^*)
		=\sum_{i_1=1}^{d_1}\cdots\sum_{i_m=1}^{d_m}\EE\left[ \sqrt{\left( [\bcalT]_{i_1\cdots i_m}-[\bcalT^*]_{i_1\cdots i_m}-[\bXi]_{i_1\cdots i_m}\right)^2+\delta^2} -\sqrt{\left([\bXi]_{i_1\cdots i_m} \right)^2+\delta^2}\right].
	\end{align*}
	Thus when $\linft{\bcalT-\bcalT^*}\leq C_{m,\mu,r^*}(6\gamma+\delta)$, we have $$\EE f(\bcalT)-\EE f(\bcalT^*)\geq\frac{1}{3b_0}\fro{\bcalT-\bcalT^*}^2.$$
	Similarly, we have
	$$\EE f(\bcalT)-\EE f(\bcalM)\leq\frac{1}{2\delta}\cdot\left|\fro{\bcalT-\bcalT^*}^2-\fro{\bcalM-\bcalT^*}^2\right|.$$
\end{proof}

\begin{lemma}[Absolute Loss]
	Suppose Assumption~\ref{assu:abs:noise} holds, then for all $\bcalT\in\RR^{d_1\times\cdots\times d_m}$ it has 
	$$\EE \lone{\bcalT-\bcalT^*-\bXi}-\EE \lone{\bXi}\leq\frac{1}{b_1}\fro{\bcalT-\bcalT^*}^2.$$
	Furthermore, if $\linft{\bcalT-\bcalT^*}\leq C_{m,\muT,r^*,\kappa}\gamma$, it has $$\EE \lone{\bcalT-\bcalT^*-\bXi}-\EE \lone{\bXi}\geq\frac{1}{b_0}\fro{\bcalT-\bcalT^*}^2.$$
	\label{teclem:abs:function expectation}
\end{lemma}
\begin{proof}
	Suppose $\xi$ satisfies distributions in Assumption~\ref{assu:abs:noise}. Then we have \begin{align*}
		\EE|t_0-\xi|=2\int_{s>t_0}(1-H_{\xi}(s))\, ds+t_0-\int_{-\infty}^{+\infty} s\, dH_{\xi}(s),
	\end{align*}
	which has more detailed calculationss in \cite{shen2023computationally,elsener2018robust}. When $t_0=0$, it becomes $\EE|\xi|=2\int_{s>0}(1-H_{\xi}(s))\, ds-\int_{-\infty}^{+\infty} s\, dH_{\xi}(s)$. Thus, with $H_{\xi}(0)=1/2$, we have $$\EE|t_0-\xi|-\EE|\xi|=2\int_{0}^{t_0}H_{\xi}(s)\, ds-t_0=2\int_{0}^{t_0}\int_{0}^{s}h_{\xi}(x)\, dx\, ds.$$
	Then by Assumption~\ref{assu:abs:noise}, we have $\EE|t_0-\xi|-\EE|\xi|\leq \frac{1}{b_1}t_0^2$. Then come back to $\EE \lone{\bcalT-\bcalT^*-\bXi}-\EE \lone{\bXi}$,
	\begin{align*}
	\EE\lone{\bcalT-\bcalT^*-\bXi}-\EE\lone{\bXi}
		&=\sum_{i_1=1}^{d_1}\cdots\sum_{i_m=1}^{d_m}\EE\left[ \bigg| [\bcalT]_{i_1\cdots i_m}-[\bcalT^*]_{i_1\cdots i_m}-[\bXi]_{i_1\cdots i_m}\bigg|-\bigg|[\bXi]_{i_1\cdots i_m}\bigg|\right]\\
		&\leq \sum_{i_1=1}^{d_1}\cdots\sum_{i_m=1}^{d_m}\frac{1}{b_1}\left([\bcalT]_{i_1\cdots i_m}-[\bcalT^*]_{i_1\cdots i_m} \right)^2\\
		&= b_1^{-1}\fro{\bcalT-\bcalT^*}^2.
	\end{align*}
	On the other hand, $h_{\xi}(x)\geq b_0^{-1}$ when $|x|\leq C_{m,\muT,r^*,\kappa}\gamma$. Thus, when $|t_0|\leq C_{m,\muT,r^*,\kappa}\gamma$, it has
	$$\EE|t_0-\xi|-\EE|\xi|=2\int_{0}^{t_0}\int_{0}^{s}h_{\xi}(x)\, dx\, ds\geq\frac{1}{b_0}t_0^2.$$
	Thus, when $\linft{\bcalT-\bcalT^*}\leq C_{m,\muT,r^*,\kappa}\gamma$, we have 
	\begin{align*}
	\EE\lone{\bcalT-\bcalT^*-\bXi}-\EE\lone{\bXi}&=\sum_{i_1=1}^{d_1}\cdots\sum_{i_m=1}^{d_m}\EE\left[ \bigg| [\bcalT]_{i_1\cdots i_m}-[\bcalT^*]_{i_1\cdots i_m}-[\bXi]_{i_1\cdots i_m}\bigg|-\bigg|[\bXi]_{i_1\cdots i_m}\bigg|\right]\\
		&\geq b_0^{-1}\fro{\bcalT-\bcalT^*}^2.
	\end{align*}
\end{proof}

\subsection{Perturbation Type Bound}
\begin{lemma}(Matrix Perturbation \cite{shen2022computationally})
	Suppose matrix $\M^{*} \in\mathbb{R}^{d_1\times d_2}$ has rank $r$ and has singular value decomposition  $\M^{*}=\mathbf{U}\boldsymbol{\Sigma}\mathbf{V}^{\top}$ where $\boldsymbol{\Sigma}=\operatorname{diag}\{\sigma_{1},\sigma_{2},\cdots,\sigma_{r}\}$ and $\sigma_{1}\geq\sigma_{2}\geq\cdots\geq\sigma_{r}> 0$. Then for any $\hat{\M}\in\mathbb{R}^{d\times d}$ satisfying $\Vert\hat{\M}-\M\Vert_{\mathrm{F}}<\sigma_{r}/4$, with $\hat{\mathbf{U}}_{r}\in\mathbb{R}^{d_1\times r}$ and $\hat{\mathbf{V}}_{r}\in\RR^{d_2\times r}$ the left and right singular vectors of $r$ largest singular values, we have
	\begin{align*}
		\Vert \hat{\mathbf{U}}_{r}\hat{\mathbf{U}}_{r}^{\top} -\mathbf{U}\mathbf{U}^{\top}\Vert&\leq \frac{4}{\sigma_{r}}\Vert \hat{\mathbf{M}}-\mathbf{M}\Vert,\quad \Vert \hat{\mathbf{V}}_{r}\hat{\mathbf{V}}_{r}^{\top} -\mathbf{V}\mathbf{V}^{\top}\Vert\leq \frac{4}{\sigma_{r}}\Vert \hat{\mathbf{M}}-\mathbf{M}\Vert,\\
		\Vert\operatorname{SVD}_{r}(\hat{\M})-\M^{*}\Vert&\leq\Vert\mathbf{\hat{\M}-\M^{*}}\Vert+20\frac{\Vert \hat{\M}-\M^{*}\Vert^2}{\sigma_{r}},\\
		\Vert \operatorname{SVD}_{r}(\hat{\M})-\M^{*}\Vert_{\mathrm{F}}&\leq\Vert\mathbf{\hat{\M}-\M^{*}}\Vert_{\mathrm{F}}+20\frac{\Vert \hat{\M}-\M^{*}\Vert\Vert \hat{\M}-\M^{*}\Vert_{\mathrm{F}}}{\sigma_{r}}.
	\end{align*}
	\label{teclem:perturbation:matrix}
\end{lemma}

\begin{lemma}
	Suppose $\M^*\in\RR^{d\times d}$ is a symmetric rank $r$ matrix, with singular value decomposition  $\M^*=\U^*\bSigma^*\U^{*\top}$, where $\bSigma^*=\textrm{diag}(\sigma_{1}^*,\cdots,\sigma_{r}^*)$, $\sigma_{1}^*\geq\cdots\geq \sigma_{r}^*>0$. Then for any symmetric matrix satisfying $\fro{\M-\M^*}\leq\sigma_{r}^*/4$ with rank $r$ singular vector decomposition $\textrm{SVD}_{r}(\M)=\U\bSigma\U^{\top}$. Denote $\H=\U^{\top}\U^*\in\RR^{r\times r}$. We have \begin{align*}
		\ltinf{\left(\U\H-\U^*\right)\bSigma^*}\leq\ltinf{\U_{\perp}^*\U_{\perp}^*\Z\U^*}+64\ltinf{\U^*}\frac{\fro{\Z}^2}{\sigma_{r}^*}+16\ltinf{\U_{\perp}^*\U_{\perp}^*\Z\U^*}\frac{\fro{\Z}}{\sigma_{r}^*}.
	\end{align*}
	\label{teclem:l2infpertb:symmetric}
\end{lemma}
\begin{proof}
	Note that $\left(\U\H-\U^*\right)\bSigma^*=\left(\U\U^{\top}-\U^*\U^{*\top}\right)\U^{*}\bSigma^*$. Denote $\mathbf{Z}=\mathbf{M}-\mathbf{M}^*$. Define $\mathbf{U}_{\perp}^*\in\mathbb{R}^{d\times (d-r)}$ such that $[\mathbf{U}^*, \mathbf{U}_{\perp}^*]\in\mathbb{R}^{d\times d}$ is orthonormal and then define the projector $$\mathfrak{P}^{\perp}:=\mathbf{U}_{\perp}^*\mathbf{U}_{\perp}^{*\top},\quad \mathfrak{P}^{-1}:=\mathbf{U}\boldsymbol{\Lambda}^{-1}\mathbf{U}^{\top}.$$
	Write $\mathfrak{P}^{-k}=\mathbf{U}^*\bSigma^{-k}\mathbf{U}^{*\top}$, for all $k\geq1$ and for convenience when $k=0$, we write $\mathfrak{P}^{0}=\mathfrak{P}^{-1}$. Define the $k$-th order perturbation \begin{equation*}
		\mathcal{S}_{\mathbf{M}^{*},k}(\mathbf{Z}):=\sum_{\mathbf{s}:s_{1}+\cdots+s_{k+1}=k} (-1)^{1+\tau(\mathbf{s})} \mathfrak{P}^{-s_{1}}\mathbf{Z}\mathfrak{P}^{-s_{2}}\cdots\mathfrak{P}^{-s_{k}}\mathbf{Z}\mathfrak{P}^{-s_{k+1}},
	\end{equation*} where $s_{1},\cdots,s_{k}$ are non-negative integers and $\tau(\mathbf{s})=\sum_{i=1}^{k}\mathbb{I}(s_{i}>0)$ is the number of positive indices in $\mathbf{s}$. Work \cite{xia2021normal} proves \begin{equation*}
		\mathbf{U}\mathbf{U}^{\top}-\mathbf{U}^*\mathbf{U}^{*\top}=\sum_{k\geq1}\mathcal{S}_{\mathbf{M}^{*},k}(\mathbf{Z}).
	\end{equation*}
	Then consider $\ltinf{\left(\U\U^{\top}-\U^*\U^{*\top}\right)\U^*\bSigma^*}$,
	\begin{align*}
		\left(\U\U^{\top}-\U^*\U^{*\top}\right)&\U^*\bSigma^*=\sum_{k\geq1}\mathcal{S}_{\mathbf{M}^{*},k}(\mathbf{Z})\U^*\bSigma^*\\
		&=\U_{\perp}^*\U_{\perp}^*\Z\U^*\U^{*\top}+\sum_{k\geq2}\sum_{\mathbf{s}:s_{1}+\cdots+s_{k+1}=k} (-1)^{1+\tau(\mathbf{s})} \mathfrak{P}^{-s_{1}}\mathbf{Z}\mathfrak{P}^{-s_{2}}\cdots\mathfrak{P}^{-s_{k}}\mathbf{Z}\mathfrak{P}^{-s_{k+1}}\U^*\bSigma^*
	\end{align*}
	Note that for $k\geq 2$, \begin{align*}
		&\ltinf{ \mathfrak{P}^{-s_{1}}\mathbf{Z}\mathfrak{P}^{-s_{2}}\cdots\mathfrak{P}^{-s_{k}}\mathbf{Z}\mathfrak{P}^{-s_{k+1}}\U^*\bSigma^* }\\
		&{~~~~~~~~~~~~~~~~~~}\leq \binom{2k-1}{k}\ltinf{\U^*}\frac{\fro{\Z}^k}{\sigma_r^{*k-1}}+\binom{2k-1}{k-1}\ltinf{\U_{\perp}^*\U_{\perp}^*\Z\U^*}\frac{\fro{\Z}^{k-1}}{\sigma_r^{*k-1}}\\
		&{~~~~~~~~~~~~~~~~~~}\leq \sigma_{r}^{*}\ltinf{\U^*}\left(\frac{4\fro{\Z}}{\sigma_r^*}\right)^{k}+\ltinf{\U_{\perp}^*\U_{\perp}^*\Z\U^*}\left(\frac{4\fro{\Z}}{\sigma_{r}^*}\right)^{k-1}
	\end{align*}
	Hence, \begin{align*}
		&\ltinf{\left(\U\U^{\top}-\U^*\U^{*\top}\right)\U^*\bSigma}\\
		&{~~~~~~~~}\leq \ltinf{\U_{\perp}^*\U_{\perp}^*\Z\U^*\U^{*\top}}+\sum_{k\geq2} \ltinf{ \mathfrak{P}^{-s_{1}}\mathbf{Z}\mathfrak{P}^{-s_{2}}\cdots\mathfrak{P}^{-s_{k}}\mathbf{Z}\mathfrak{P}^{-s_{k+1}}\U^*\bSigma^* }\\
		&{~~~~~~~~}\leq \ltinf{\U_{\perp}^*\U_{\perp}^*\Z\U^*}+64\ltinf{\U^*}\frac{\fro{\Z}^2}{\sigma_{r}^*}+16\ltinf{\U_{\perp}^*\U_{\perp}^*\Z\U^*}\frac{\fro{\Z}}{\sigma_{r}^*}.
	\end{align*}
\end{proof}

\begin{lemma}
	Suppose $\M^*\in\RR^{d\times d}$ is a rank $r$ matrix, with singular value decomposition  $\M^*=\U^*\bSigma^*\V^{*\top}$, where $\bSigma^*=\textrm{diag}(\sigma_{1}^*,\cdots,\sigma_{r}^*)$, $\sigma_{1}^*\geq\cdots\geq \sigma_{r}^*>0$. Then for any matrix satisfying $\fro{\M-\M^*}\leq\sigma_{r}^*/4$ with rank $r$ singular vector decomposition $\textrm{SVD}_{r}(\M)=\U\bSigma\V^{\top}$. Denote $\H_1=\U^{\top}\U^*\in\RR^{r\times r}$ and $\H_2=\V^{\top}\V^*\in\RR^{r\times r}$. We have \begin{align*}
		\ltinf{\left(\U\H_1-\U^*\right)\bSigma^*}&\leq\ltinf{\U_{\perp}^*\U_{\perp}^*\Z\V^*}+64\ltinf{\U^*}\frac{\fro{\Z}^2}{\sigma_{r}^*}+16\ltinf{\U_{\perp}^*\U_{\perp}^*\Z\V^*}\frac{\fro{\Z}}{\sigma_{r}^*},\\
		\ltinf{\left(\V\H_2-\V^*\right)\bSigma^*}&\leq\ltinf{\V_{\perp}^*\V_{\perp}^*\Z^{\top}\V^*}+64\ltinf{\V^*}\frac{\fro{\Z}^2}{\sigma_{r}^*}+16\ltinf{\V_{\perp}^*\V_{\perp}^*\Z^{\top}\U^*}\frac{\fro{\Z}}{\sigma_{r}^*}
	\end{align*}
	\label{teclem:l2infpertb:matrix}
\end{lemma}
\begin{proof}
	Apply Lemma~\ref{teclem:l2infpertb:symmetric} with symmetrization of $\M^*$ and $\M$:
	$$\mathbf{Y}^{*}:=\left(\begin{matrix}
		0&\mathbf{M}^{*}\\
		\mathbf{M}^{*\top}&0
	\end{matrix}\right),\,\mathbf{Y}:=\left(\begin{matrix}
		0&\mathbf{M}\\
		\mathbf{M}^{\top}&0
	\end{matrix}\right),$$
	and then we could get the desired result.
\end{proof}

\begin{lemma}[Type-\Romannum{1} Tensor Perturbation]
	Suppose tensor $\bcalT^*\in\RR^{d_1\times\dots\times d_m}$ has Tucker rank $\r=(r_1,\dots,r_m)$. Let $\bcalT^*=\bcalC^*\times_1 \U_{1}^{*}\times_2\dots\times_m\U_{m}^{*}$ be its Tucker decomposition with $\mins^*:=\min_{k=1,\dots,m}\sigma_{r_k}(\fraM_k(\bcalT^*))$. Then for any tensor $\bcalT\in\RR^{d_1\times\dots\times d_m}$ such that $\max_{k=1,\dots,m}\op{\fraM_k(\bcalT)-\fraM_{k}(\bcalT^*)}\leq \mins^*/8$ with $\text{HOSVD}(\bcalT)=\bcalC\cdot \llbracket \U_1,\cdots,\U_m\rrbracket$, then we have 
	\begin{align}
		\fro{\text{HOSVD}(\bcalT)-\bcalT^*}&\leq \fro{\bcalT-\bcalT^*}+32m\frac{\fro{\bcalT-\bcalT^*}^2}{\mins^*}.\label{eq:perturbation:3}
	\end{align}
	Also, for each order $k=1,\dots,m$, we have
	\begin{align}
		\op{\U_k\U_k^{\top}-\U_k^*\U_k^{*\top}}&\leq4\frac{\fro{\bcalT-\bcalT^*}}{\mins^*},\label{eq:perturbation:1}
	\end{align}
	and
	\begin{equation}
		\begin{split}
			\ltinf{\fraM_k(\text{HOSVD}(\bcalT)-\bcalT^*)}&\leq \ltinf{\fraM_k(\calP_{\TT^*}(\bcalT-\bcalT^*))}+32m\ltinf{\U_k^*}\frac{\fro{\bcalT-\bcalT^*}^2}{\mins^*}\\
			&{~~~~~~}+32m\ltinf{(\I_{d_k}-\U_k^*\U_k^{*\top})\fraM_k(\bcalT-\bcalT^*)}\frac{\fro{\bcalT-\bcalT^*}}{\mins^*},
		\end{split}
		\label{eq:perturbation:4}
	\end{equation}
	\begin{equation}
		\begin{split}
			&{~~~}\ltinf{\left(\U_k\H_k-\U_k^*\right)\fraM_k(\bcalC^*)}\\
			&\leq\ltinf{\U_{k\perp}\U_{k\perp}^{\top}\fraM_k(\bcalT-\bcalT^*)}+64\ltinf{\U_k^*}\frac{\fro{\bcalT-\bcalT^*}^2}{\mins^*}+16\ltinf{\U_{k\perp}\U_{k\perp}^{\top}\fraM_k(\bcalT-\bcalT^*)}\frac{\fro{\bcalT-\bcalT^*}}{\mins^*},
		\end{split}
		\label{eq:perturbation:2}
	\end{equation}
	where $\H_k:=\U_k^{\top}\U_k^*$.
	\label{teclem:perturbation:tensor}
\end{lemma}

\begin{proof}
	Equation~\eqref{eq:perturbation:1} could be obtained by using Lemma~\ref{teclem:perturbation:matrix} and Euqation~\eqref{eq:perturbation:3} is from proof of Lemma 13.2 in work \cite{cai2022generalized}. Then we focus on Equation~\eqref{eq:perturbation:2} and Equation~\eqref{eq:perturbation:4}. 
	
	Note that $\left(\U_k\H_k-\U_k^*\right)\fraM_k(\bcalC^*)=\left(\U_k\U_k^{\top}-\U_k^*\U_k^{*\top}\right)\U_k^*\fraM_k(\bcalC^*)$. Suppose $\fraM_k(\bcalT^*)$ has singular value decomposition $\fraM_k(\bcalT^*)=\U_k^*\bSigma_k\V_k^{*\top}$ and its Tucker decomposition matricization is $\fraM_k(\bcalT^*)=\U_k^*\fraM_k(\bcalC^*)\left(\otimes_{j\neq k} \U_j^*\right)^{\top} $. It implies $$\fraM_k(\bcalC^*)= \bSigma_k\V_k^{*\top}\left(\otimes_{j\neq k} \U_j^*\right).$$
	Hence, we have \begin{align*}
		&{~~~}\ltinf{\left(\U_k\H_k-\U_k^*\right)\fraM_k(\bcalC^*)}\\&\leq\ltinf{\left(\U_k\H_k-\U_k^*\right)\bSigma_k^*}\\
		&\leq\ltinf{\U_{k\perp}\U_{k\perp}^{\top}\fraM_k(\bcalT-\bcalT^*)}+64\ltinf{\U_k^*}\frac{\fro{\bcalT-\bcalT^*}^2}{\mins^*}+16\ltinf{\U_{k\perp}\U_{k\perp}^{\top}\fraM_k(\bcalT-\bcalT^*)}\frac{\fro{\bcalT-\bcalT^*}}{\mins^*},
	\end{align*}
	where the last line is from Lemma~\ref{teclem:l2infpertb:matrix}. Then consider $\ltinf{\fraM_k(\text{HOSVD}(\bcalT)-\bcalT^*)}$. Work~\cite{cai2022generalized} expands $\fraM_k(\text{HOSVD}(\bcalT)-\bcalT^*) $ and accordingly we have 
	\begin{align*}
		\ltinf{\fraM_k(\text{HOSVD}(\bcalT)-\bcalT^*)}&\leq \ltinf{\fraM_k(\calP_{\TT^*}(\bcalT-\bcalT^*))}+32m\ltinf{\U_k^*}\frac{\fro{\bcalT-\bcalT^*}^2}{\mins^*}\\
		&{~~~~~~~~~~~}+32m\ltinf{(\I_{d_k}-\U_k^*\U_k^{*\top})\fraM_k(\bcalT-\bcalT^*)}\frac{\fro{\bcalT-\bcalT^*}}{\mins^*}
	\end{align*}
	
\end{proof}

\begin{lemma}
	Suppose tensor $\bcalT^*\in\MM_{\r}$ has Tucker decomposition $\bcalT^*=\bcalC^*\cdot\llbracket\U_1^*,\dots,\U_m^*\rrbracket$. Let tensor $\bcalT\in\RR^{d_1\times\dots\times d_m}$ have $\text{HOSVD}(\bcalT)=\bcalC\cdot \llbracket \U_1,\cdots,\U_m\rrbracket$. Denote $\textsf{dist}(\U_k,\U_k^*):=\min_{\Q\in\OO_{r_k,r_k}}\op{\U_k\Q_k-\U_k^*}$ and $\Q_k=\arg\min_{\Q\in\OO_{r_k,r_k}}\op{\U_k\Q_k-\U_k^*}$. Then we have
	\begin{align*}
		\fro{\bcalC\cdot\llbracket\Q_1^\top,\dots,\Q_m^\top\rrbracket-\bcalC^*}\leq \sum_{k=1}^{m}\fro{\bcalT}\text{dist}(\U_k,\U_k^*)+\sqrt{r^*}\op{\bcalT-\bcalT^*}.
	\end{align*}
	\label{teclem:perturbation-coretensor}
\end{lemma}
\begin{proof}
	Note that
	\begin{align*}
		&{~~~~}\bcalC\cdot\llbracket\Q_1^\top,\dots,\Q_m^\top\rrbracket-\bcalC^*\\
		&=\bcalT\times_{1}\Q_1^\top\U_1^\top\times_2\cdots\times_m\Q_m^\top\U_m^\top-\bcalT^*\times_{1}\U_1^{*\top}\times_2\cdots\times_m\U_m^{*\top}\\
		&=\sum_{k=1}^{m}\bcalT\times_{i<k}\U_i^{*\top}\times_k\left(\U_k\Q_k-\U_k^{*}\right)^{\top}\times_{j>k}\Q_j^{\top}\U_j^{\top}-\left(\bcalT-\bcalT^*\right)\times_{1}\U_1^{*\top}\times_2\cdots\times_m\U_m^{*\top}.
	\end{align*}
	Then we have \begin{align*}
		\fro{\bcalC\cdot\llbracket\Q_1,\dots,\Q_m\rrbracket-\bcalC^*}\leq \sum_{k=1}^{m}\fro{\bcalT}\text{dist}(\U_k,\U_k^*)+\sqrt{r^*}\op{\bcalT-\bcalT^*}.
	\end{align*}
\end{proof}

\begin{lemma}
	Suppose $\bcalT,\bcalT^*\in\MM_{\r}$ with Tucker decomposition $\bcalT=\bcalC\cdot\llbracket\U_{1},\dots,\U_m\rrbracket$, $\bcalT^*=\bcalC^*\cdot\llbracket\U_1^*,\dots,\U_m^*\rrbracket$ and $\fro{\bcalT-\bcalT^*}\leq\mins^*/8$. Then we have 
	$$\fro{\calP_{\TT}^{\perp}(\bcalT^*)}\leq4m^2\mins^{*-1}\fro{\bcalT-\bcalT^*}^2.$$
	Furthermore, if $\bcalT,\bcalT^*$ are incoherent with parameter $\mu$, namely, $\ltinf{\U_k}\leq\sqrt{\frac{\mu r_k}{d_k}}$, $\ltinf{\U_k^*}\leq\sqrt{\frac{\mu r_k}{d_k}}$, then we have
	\begin{align*}
		\ltinf{\fraM_{k}\left(\calP_{\TT}^{\perp}(\bcalT^*) \right)}&\leq 4(m+1)\ltinf{\U_k\U_k^{\top}-\U_k^*\U_k^{*\top}}\fro{\bcalT-\bcalT^*}+4m^2\sqrt{\frac{\mu r_k}{d_k}}\mins^{*-1}\fro{\bcalT-\bcalT^*}^2\\
		&{~~~}+4(m-1)\ltinf{\bcalT-\bcalT^*}\cdot\mins^{*-1}\fro{\bcalT-\bcalT^*}.
	\end{align*}
	\label{teclem:rieman-orthogonal-project}
\end{lemma}
\begin{proof}
	Note that \begin{align*}
		\calP_{\TT}^{\perp}(\bcalT^*)&=(\calP_{\TT^*}-\calP_{\TT})(\bcalT^*)\\
		&=\bcalT^*\times_{k=1,\dots,m} \U_k^*\U_k^{*\top}-\bcalT^*\times_{k=1,\dots,m}\U_k\U_k^{\top}\\
		&{~~~~~~~}-\sum_{k=1}^m \textrm{tensor}\left(\left(\I_{d_k}-\U_k\U_k^{\top}\right)\fraM_k(\bcalT^*)(\otimes_{j\neq k}\U_j)\fraM_k(\bcalC)^{\dagger}\fraM_k(\bcalC) (\otimes_{j\neq k}\U_j)^{\top}\right)\\
		&=\sum_{k=1}^m \bcalT^*\times_k(\U_k^*\U_k^{*\top}-\U_k\U_k^{\top})\times_{i<k}\U_i\U_i^{\top}\\
		&{~~~~~~~}-\sum_{k=1}^m \textrm{tensor}\left(\left(\U_k^*\U_k^{*\top}-\U_k\U_k^{\top}\right)\fraM_k(\bcalT^*)(\otimes_{j\neq k}\U_j)\fraM_k(\bcalC)^{\dagger}\fraM_k(\bcalC) (\otimes_{j\neq k}\U_j)^{\top}\right)
	\end{align*}
	Also, we have 
	\begin{align*}
		&{~~~~}\textrm{tensor}\left(\left(\U_k^*\U_k^{*\top}-\U_k\U_k^{\top}\right)\fraM_k(\bcalT^*)(\otimes_{j\neq k}\U_j)\fraM_k(\bcalC)^{\dagger}\fraM_k(\bcalC) (\otimes_{j\neq k}\U_j)^{\top}\right)\\
		& =\textrm{tensor}\left(\left(\U_k^*\U_k^{*\top}-\U_k\U_k^{\top}\right)\fraM_k(\bcalT^*-\bcalT)(\otimes_{j\neq k}\U_j)\fraM_k(\bcalC)^{\dagger}\fraM_k(\bcalC) (\otimes_{j\neq k}\U_j)^{\top}\right)+\bcalT\times_k\left(\U_k^*\U_k^{*\top}-\U_k\U_k^{\top} \right)
	\end{align*}
	Notice that by Lemma~\ref{teclem:perturbation:tensor}, the first term has the upper bound
	\begin{align*}
		\fro{\textrm{tensor}\left(\left(\U_k^*\U_k^{*\top}-\U_k\U_k^{\top}\right)\fraM_k(\bcalT^*-\bcalT)(\otimes_{j\neq k}\U_j)\fraM_k(\bcalC)^{\dagger}\fraM_k(\bcalC) (\otimes_{j\neq k}\U_j)^{\top}\right)}\leq4\mins^{*-1}\fro{\bcalT-\bcalT^*}^2.
	\end{align*}
	Hence, we have
	\begin{align*}
		\fro{\calP_{\TT}^{\perp}(\bcalT^*)}&\leq 4m\mins^{*-1}\fro{\bcalT-\bcalT^*}^2+\sum_{k=1}^m\fro{\left(\bcalT^*-\bcalT\right)\times_k\left(\U_k^*\U_k^{*\top}-\U_k\U_k^{\top}\right) }\\
		&{~~~~~~}+\sum_{k=1}^m\fro{\bcalT^*\times_k(\U_k^*\U_k^{*\top}-\U_k\U_k^{\top})\times_{i<k}\U_i\U_i^{\top}-\bcalT^*\times_k\left(\U_k^*\U_k^{*\top}-\U_k\U_k^{\top}\right)}\\
		&\leq 4m^2\mins^{*-1}\fro{\bcalT-\bcalT^*}^2,
	\end{align*}
	which uses \begin{align*}
		&{~~~~}\fro{\bcalT^*\times_k(\U_k^*\U_k^{*\top}-\U_k\U_k^{\top})\times_{i<k}\U_i\U_i^{\top}-\bcalT^*\times_k\left(\U_k^*\U_k^{*\top}-\U_k\U_k^{\top}\right)}\\
		&=\fro{\sum_{j=1}^{k-1} \bcalT^*\times_{i<j} \U_i^*\U_i^{*\top} \times_j\left(\U_j\U_j-\U_j^*\U_j^{*\top}\right)\times_{i>j}\U_i\U_i^{\top}\times_k\left(\U_k^*\U_k^{*\top}-\U_k\U_k^{\top}\right)}\\
		&\leq (k-1)\mins^{*-1}\fro{\bcalT-\bcalT^*}^2.
	\end{align*}
	Simialy, we could get upper bound for $ \ltinf{\fraM_{k}\left(\calP_{\TT}^{\perp}(\bcalT^*)\right)}$.
\end{proof}

\begin{lemma}
	For any two matrices $\U,\U^*\in\OO_{d,r}$, denote $\H:=\U^{\top}\U^*$ and it has $$\ltinf{\U\U^{\top}-\U^*\U^{*\top}}\leq \ltinf{\U\H-\U^*}+\ltinf{\U}\cdot\fro{\U\U^{\top}-\U^*\U^{*\top}}.$$
	\label{teclem:ltinf-transformation}
\end{lemma}

\begin{proof}
	By triangle inequality and the inequality $\ltinf{\A\B}\leq\ltinf{\A}\op{\B}\leq\ltinf{\A}\fro{\B}$, we have
	\begin{align*}
		\ltinf{\U\U^{\top}-\U^*\U^{*\top}}&\leq\ltinf{\U\U^{\top}\U^{*}\U^{*\top}-\U^*\U^{*\top}}+\ltinf{\U\U^{\top}\U^{*}\U^{*\top}-\U\U^{\top}}\\
		&\leq\ltinf{\U\H-\U^*}+\ltinf{\U}\cdot\fro{\U\U^{\top}-\U^*\U^{*\top}}.
	\end{align*}
\end{proof}

\begin{lemma}
	Suppose $\bcalT,\bcalT^*\in\MM_{\r}$ with Tucker decomposition $\bcalT=\bcalC\cdot\llbracket\U_{1},\dots,\U_m\rrbracket$, $\bcalT^*=\bcalC^*\cdot\llbracket\U_1^*,\dots,\U_m^*\rrbracket$ and $\fro{\bcalT-\bcalT^*}\leq\mins^*/8$. If $\bcalT,\bcalT^*$ are incoherent with parameter $\mu$, then for tensor $\bcalG\in\RR^{d_1\times\cdots\times d_m}$ and any $k=1,\dots,m$, we have
	\begin{align*}
		\ltinf{\fraM_{k}\left(\calP_{\TT^*}^{\perp}\calP_{\TT}\bcalG\right)}&\leq2m^2\mins^{*-1}\sqrt{\frac{\mu r_k}{d_k}} \fro{\calP_{\TT}\bcalG}\fro{\bcalT-\bcalT^*}\\
		&{~~~~~~~}+(m+1)\ltinf{\U_k\H_k-\U_k^*}\fro{\calP_{\TT}\bcalG}+\mins^{*-1}\ltinf{\fraM_{k}(\calP_{\TT}\bcalG)}\fro{\bcalT-\bcalT^*},
	\end{align*}
	where $\H_k:=\U_k^{\top}\U_k^*$. Similarly, we have
	\begin{align*}
		\ltinf{\fraM_{k}\left((\calP_{\TT^*}-\calP_{\TT})\bcalG\right)}&\leq2m^2\mins^{*-1}\sqrt{\frac{\mu r_k}{d_k}} \fro{\bcalG}\fro{\bcalT-\bcalT^*}\\
		&{~~~~~~~}+(m+1)\ltinf{\U_k\H_k-\U_k^*}\fro{\bcalG}+\mins^{*-1}\ltinf{\fraM_{k}(\bcalG)}\fro{\bcalT-\bcalT^*},
	\end{align*}
	\label{teclem:remian-perturb-subgradient}
\end{lemma}

\begin{proof}
	Note that \begin{align*}
		\calP_{\TT^*}^{\perp}\calP_{\TT}\bcalG=\left(\I-\calP_{\TT^*}\right)\calP_{\TT}\bcalG=\left(\calP_{\TT}-\calP_{\TT^*}\right)\calP_{\TT}\bcalG.
	\end{align*}
	Denote $\bcalH:=\calP_{\TT}\bcalG$ and we need to bound $\ltinf{\fraM_{k}(\left(\calP_{\TT}-\calP_{\TT^*}\right)\bcalH)}$. 
	Notice that \begin{align*}
		\left(\calP_{\TT}-\calP_{\TT^*}\right)\bcalH&=\bcalH\times_{k=1,\dots,m} \U_k\U_k^{\top}-\bcalH\times_{k=1,\dots,m}\U_k^*\U_k^{*\top}\\
		&{~~~~~~~}+\sum_{j=1}^m \textrm{tensor}_{j}\left(\left(\I_{d_j}-\U_j\U_j^{\top}\right)\fraM_j(\bcalH)(\otimes_{i\neq j}\U_i)\fraM_j(\bcalC)^{\dagger}\fraM_j(\bcalC) (\otimes_{i\neq j}\U_i)^{\top}\right)\\
		&{~~~~~~~}+\sum_{j=1}^m \textrm{tensor}_{j}\left(\left(\I_{d_j}-\U_j^*\U_j^{*\top}\right)\fraM_j(\bcalH)(\otimes_{i\neq j}\U_i^*)\fraM_j(\bcalC^*)^{\dagger}\fraM_j(\bcalC^*) (\otimes_{i\neq j}\U_i^*)^{\top}\right),
	\end{align*}
	where $\textrm{tensor}_j(\cdot):\RR^{d_j\times d_j^-}\to \RR^{d_1\times\cdots \times d_m}$ is inverse of $j$-matricization. First consider $\bcalH\times_{k=1,\dots,m} \U_k\U_k^{\top}-\bcalH\times_{k=1,\dots,m}\U_k^*\U_k^{*\top} $. By Lemma~\ref{teclem:perturbation:tensor}, we have
	\begin{align*}
		&{~~~~}\ltinf{\fraM_{k}\left( \bcalH\times_{k=1,\dots,m} \U_k\U_k^{\top}-\bcalH\times_{k=1,\dots,m}\U_k^*\U_k^{*\top}\right)}\\
		&\leq (m-1)\sqrt{\frac{\mu r_k}{d_k}}\cdot \mins^{*-1}\fro{\bcalT-\bcalT^*}\cdot \fro{\bcalH}+\ltinf{\U_k\U_k^{\top}-\U_k^*\U_k^{*\top}}\fro{\bcalH}.
	\end{align*}
	Then consider $\textrm{tensor}_j\left(\fraM_j(\bcalH)\left((\otimes_{i\neq j}\U_i^*)\fraM_j(\bcalC^*)^{\dagger}\fraM_j(\bcalC^*) (\otimes_{i\neq j}\U_i^*)^{\top}-(\otimes_{i\neq j}\U_i)\fraM_j(\bcalC)^{\dagger}\fraM_j(\bcalC) (\otimes_{i\neq j}\U_i)^{\top} \right) \right)$, with $j\neq k$. Introduce orthogonal matrices $\Q_k:=\arg\min_{\Q\in\OO_{r_k,r_k}}\op{\U_k\Q_k-\U_k^*}$, for each $k=1,\dots,m$. Then we have
	\begin{align*}
		&{~~~~}\underbrace{\fraM_j(\bcalH)\left((\otimes_{i\neq j}\U_i)\fraM_j(\bcalC)^{\dagger}\fraM_j(\bcalC) (\otimes_{i\neq j}\U_i)^{\top}-(\otimes_{i\neq j}\U_i^*)\fraM_j(\bcalC^*)^{\dagger}\fraM_j(\bcalC^*) (\otimes_{i\neq j}\U_i^*)^{\top} \right)}_{C}\\
		&=\fraM_j(\bcalH)\left((\otimes_{i\neq j}\U_i\Q_i)\fraM_j(\tilde\bcalC)^{\dagger}\fraM_j(\tilde\bcalC) (\otimes_{i\neq j}\U_i\Q_i)^{\top}-(\otimes_{i\neq j}\U_i^*)\fraM_j(\bcalC^*)^{\dagger}\fraM_j(\bcalC^*) (\otimes_{i\neq j}\U_i^*)^{\top} \right)\\
		&=\underbrace{\sum_{l\neq j}\fraM_j(\bcalH)(\otimes_{i\neq j}\U_i^*)\fraM_j(\bcalC^*)^{\dagger}\fraM_j(\bcalC^*)(\otimes_{i>l,i\neq j} \U_i^{*})^{\top}\otimes(\U_l\Q_l-\U_l^*)^{\top}(\otimes_{i<l,i\neq j}\U_i\Q_i)^{\top}}_{C_1}\\
		&{~~~}+\underbrace{\fraM_j(\bcalH)\sum_{l\neq j}(\otimes_{i>l,i\neq j}\U_i\Q_i)\otimes(\U_l\Q_l-\U_l^*)(\otimes_{i<l,i\neq j}\U_i^*)\fraM_j(\bcalC^*)^{\dagger}\fraM_j(\bcalC^*)(\otimes_{i\neq j}\U_i\Q_i)^{\top}}_{C_2}\\
		&{~~~}+\underbrace{\fraM_j(\bcalH)(\otimes_{i\neq j}\U_i\Q_i)\left(\fraM_j(\tilde\bcalC)^{\dagger}\fraM_j(\tilde\bcalC)-\fraM_j(\bcalC^*)^{\dagger}\fraM_j(\bcalC^*) \right) (\otimes_{i\neq j}\U_i\Q_i)^{\top}}_{C_3}.
	\end{align*}
	Then, we could bound $$\ltinf{\fraM_k(\textrm{tensor}_j (C_3))}\leq\mins^{*-1}\ltinf{\U_k} \fro{\bcalH}\fro{\bcalT-\bcalT^*}\leq\mins^{*-1}\sqrt{\frac{\mu r_k}{d_k}} \fro{\bcalH}\fro{\bcalT-\bcalT^*},$$ where $j\neq k$. Similarly, we also have
	\begin{align*}
		\ltinf{\fraM_k(\textrm{tensor}_j (C_2))}&\leq (m-1)\mins^{*-1}\sqrt{\frac{\mu r_k}{d_k}} \fro{\bcalH}\fro{\bcalT-\bcalT^*},\\
		\ltinf{\fraM_k(\textrm{tensor}_j (C_1))}&\leq (m-2)\mins^{*-1}\sqrt{\frac{\mu r_k}{d_k}} \fro{\bcalH}\fro{\bcalT-\bcalT^*}+\ltinf{\U_k\H_k-\U_k^*}\fro{\bcalH}.
	\end{align*}
	Altogether, for $j\neq k$ we have the upper bound for $\ltinf{\fraM_k(\textrm{tensor}_j (C))} $, namely,
	$$ \ltinf{\fraM_k(\textrm{tensor}_j (C))}\leq 2(m-1)\mins^{*-1}\sqrt{\frac{\mu r_k}{d_k}} \fro{\bcalH}\fro{\bcalT-\bcalT^*}+\ltinf{\U_k\H_k-\U_k^*}\fro{\bcalH}.$$
	Similarly, we have
	\begin{align*}
		&{~~~}\left\Vert \fraM_{k}\left(\textrm{tensor}_{j}\left(\U_j\U_j^{\top}\fraM_j(\bcalH)(\otimes_{i\neq j}\U_i)\fraM_j(\bcalC)^{\dagger}\fraM_j(\bcalC) (\otimes_{i\neq j}\U_i)^{\top}\right)\right) \right.\\
		&{~~~~~}\left.-\fraM_{k}\left(\textrm{tensor}_j\left(\U_j^*\U_j^{*\top}\fraM_j(\bcalH)(\otimes_{i\neq j}\U_i^*)\fraM_j(\bcalC^*)^{\dagger}\fraM_j(\bcalC^*) (\otimes_{i\neq j}\U_i^*)^{\top}\right)\right)\right\Vert_{2,\infty}\\
		&\leq(2m-1)\mins^{*-1}\sqrt{\frac{\mu r_k}{d_k}} \fro{\bcalH}\fro{\bcalT-\bcalT^*}+\ltinf{\U_k\H_k-\U_k^*}\fro{\bcalH}.
	\end{align*}
	Then consider case of $j=k$ and similar to $j\neq k$ case, it arrives at
	\begin{align*}
		&{~~~}\left\Vert \left(\I_{d_k}-\U_k\U_k^{\top}\right)\fraM_k(\bcalH)(\otimes_{i\neq k}\U_i)\fraM_k(\bcalC)^{\dagger}\fraM_k(\bcalC) (\otimes_{i\neq k}\U_k)^{\top} \right.\\
		&{~~~~~}\left.-\left(\I_{d_k}-\U_k^*\U_k^{*\top}\right)\fraM_k(\bcalH)(\otimes_{i\neq k}\U_i^*)\fraM_k(\bcalC^*)^{\dagger}\fraM_k(\bcalC^*) (\otimes_{i\neq k}\U_i^*)^{\top}\right\Vert_{2,\infty}\\
		&\leq (2m+1)\mins^{*-1}\ltinf{\fraM_{k}(\bcalH)}\fro{\bcalT-\bcalT^*}+\ltinf{\U_k\U_k^{\top}-\U_k^*\U_k^{*\top}}\fro{\bcalH}+ \mins^{*-1}\sqrt{\frac{\mu r_k}{d_k}} \fro{\bcalH}\fro{\bcalT-\bcalT^*}.
	\end{align*}
	Besides, Lemma~\ref{teclem:ltinf-transformation} implies $$\ltinf{\U_k\U_k^{\top}-\U_k^*\U_k^{*\top}}\leq \ltinf{\U_k\H_k-\U_k^*}+\mins^{*-1}\sqrt{\frac{\mu r_k}{d_k}}\fro{\bcalT-\bcalT^*}. $$ In conclusion, we have\begin{align*}
		\ltinf{\fraM_{k}\left(\calP_{\TT}-\calP_{\TT^*}\right)\bcalH}&\leq 2m^2\mins^{*-1}\sqrt{\frac{\mu r_k}{d_k}} \fro{\bcalH}\fro{\bcalT-\bcalT^*}\\
		&{~~~~~~~}+(m+1)\ltinf{\U_k\H_k-\U_k^*}\fro{\bcalH}+\mins^{*-1}\ltinf{\fraM_{k}(\bcalH)}\fro{\bcalT-\bcalT^*},
	\end{align*}which proves the bound for $\ltinf{\fraM_{k}\left(\calP_{\TT^*}^{\perp}\calP_{\TT}\bcalG\right)}$. Upper bound of $\ltinf{\fraM_{k}((\calP_{\TT^*}-\calP_{\TT})\bcalG)}$ would be similar and hence we skip it.
\end{proof}
\begin{remark}
	Note that if we are only interested in the Frobenius norm of one slice of $\calP_{\TT^*}^{\perp}\calP_{\TT}\bcalG$, namely $\fro{\calP_{\Omega_{j}^{(k)}}\left(\calP_{\TT^*}^{\perp}\calP_{\TT}\bcalG \right)}$, it has the following bound \begin{align*}
		\fro{\calP_{\Omega_{j}^{(k)}}\left(\calP_{\TT^*}^{\perp}\calP_{\TT}\bcalG \right)}&\leq2m^2\mins^{*-1}\sqrt{\frac{\mu r_k}{d_k}} \fro{\calP_{\TT}\bcalG}\fro{\bcalT-\bcalT^*}\\
		&{~~~~~~~}+(m+1)\ltwo{\left(\U_k\H_k-\U_k^*\right)_{j,\cdot}}\fro{\calP_{\TT}\bcalG}+\mins^{*-1}\fro{\calP_{\Omega_{j}^{(k)}}(\calP_{\TT}\bcalG)}\fro{\bcalT-\bcalT^*}
	\end{align*}
\end{remark}
\begin{lemma}[Type-\Romannum{2} Tensor Perturbation]
	Suppose tensor $\bcalT^*\in\RR^{d_1\times\dots\times d_m}$ has Tucker rank $\r=(r_1,\dots,r_m)$. Let $\bcalT^*=\bcalC^*\cdot\llbracket\U_1^*,\dots,\U_m^*\rrbracket$ be its Tucker decomposition. Suppose tensor $\bcalT\in\RR^{d_1\times\dots\times d_m}$ has $\text{HOSVD}(\bcalT)=\bcalC\cdot \llbracket \U_1,\cdots,\U_m\rrbracket$, then
	for each order $k=1,\dots,m$ and each $j=1,\dots,d_k$, we have
	\begin{align*}
		&{~~~~}\ltwo{\left(\U_k\H_k-\U_k^*\right)_{j,\cdot}\fraM_k(\bcalC^*)}\\
		&\leq\ltwo{\calP_{\Omega_{j}^{(k)}}(\bcalT-\bcalT^*)}+\ltwo{\calP_{\Omega_{j}^{(k)}}(\bcalT)}\op{\W_k\fraM_k(\bcalC)^{\dag}\fraM_k(\bcalC)\W_k^{\top}- \W_k^*\fraM_k(\bcalC^*)^{\dag}\fraM_k(\bcalC^*)\W_k^{*\top} }\\
		&{~~~~~~~~~~~~~~~~~~~~~~~~~~~~~~~~~~~~~~~~~~~~~~~~~~~~~~}+\ltwo{\calP_{\Omega_{j}^{(k)}}(\bcalT)}\frac{\op{ \fraM_k(\bcalT-\bcalT^*)\W_k^*}}{\mins^*}.
	\end{align*}
	where $\H_k:=\U_k^{\top}\U_k^*$, $\W_k^*=\U_{m}^*\otimes\dots\otimes\U_{k+1}^*\otimes\U_{k-1}^*\otimes\dots\otimes\U_1^*$, $\W_k=\U_{m}\otimes\dots\otimes\U_{k+1}\otimes\U_{k-1}\otimes\dots\otimes\U_1$.
	\label{teclem:yuxinchen}
\end{lemma}

\begin{proof}
	The proof follows Lemma 4.6.4 of work \cite{chen2021spectral}.
	
	For simplicity, denote $\W_k^*=\U_{m}^*\otimes\dots\otimes\U_{k+1}^*\otimes\U_{k-1}^*\otimes\dots\otimes\U_1^*$, $\W_k=\U_{m}\otimes\dots\otimes\U_{k+1}\otimes\U_{k-1}\otimes\dots\otimes\U_1$ and then $\fraM_k(\bcalT^*), \fraM_{k}(\bcalT)$ has expression $\fraM_k(\bcalT^*)=\U_k^*\fraM_k(\bcalC^*)\W_{k}^{*\top}$, $\fraM_k(\bcalT)=\U_k\fraM_k(\bcalC)\W_{k}^{\top}$
	
	First, consider the term $\U_k\H_k\fraM_k(\bcalC^*) $,
	\begin{align*}
		\U_k\H_k\fraM_k(\bcalC^*)&=\U_k\fraM_{k}(\bcalC)\fraM_{k}(\bcalC)^{\dag}\U_k^{\top}\U_k^*\fraM_k(\bcalC^*)=\fraM_k(\bcalT)\W_k\fraM_k(\bcalC)^{\dag}\U_k^{\top}\U_k^*\fraM_k(\bcalC^*).
	\end{align*}
	Then consider $\U_k^{\top}\U_k^*\fraM_k(\bcalC^*)$,
	\begin{align*}
		\U_k^{\top}\U_k^*\fraM_k(\bcalC^*)&=\U_k^{\top}\fraM_k(\bcalT^*)\W_{k}^*\\
		&=\U_k^{\top}\fraM_k(\bcalT)\W_{k}^*
		-\U_k^{\top}\fraM_k(\bcalT-\bcalT^*)\W_{k}^*\\
		&=\fraM_k(\bcalC)\W_k^{\top}\W_{k}^*
		-\U_k^{\top}\fraM_k(\bcalT-\bcalT^*)\W_{k}^*.
	\end{align*}
	Combine the above two equations and then we have
	\begin{align*}
		&{~~~~}\U_k\H_k\fraM_k(\bcalC^*)\\
		&=\fraM_k(\bcalT)\W_k\fraM_k(\bcalC)^{\dag}\fraM_k(\bcalC)\W_k^{\top}\W_{k}^*-\fraM_k(\bcalT)\W_k\fraM_k(\bcalC)^{\dag}\U_k^{\top}\fraM_k(\bcalT-\bcalT^*)\W_{k}^*\\
		&=\fraM_k(\bcalT)\W_k^*\fraM_k(\bcalC^*)^{\dag}\fraM_k(\bcalC^*)\W_k^{*\top}\W_{k}^*-\fraM_k(\bcalT)\W_k\fraM_k(\bcalC)^{\dag}\U_k^{\top}\fraM_k(\bcalT-\bcalT^*)\W_{k}^*\\
		&{~~~~}+\fraM_k(\bcalT)\left(\W_k\fraM_k(\bcalC)^{\dag}\fraM_k(\bcalC)\W_k^{\top}- \W_k^*\fraM_k(\bcalC^*)^{\dag}\fraM_k(\bcalC^*)\W_k^{*\top}\right)\W_{k}^*
	\end{align*}
	Note that with $\fraM_k(\bcalT)\W_k^*\fraM_k(\bcalC^*)^{\dag}\fraM_k(\bcalC^*)\W_k^{*\top}\W_{k}^*=\U_k^*\fraM_k(\bcalC^*)+\fraM_k(\bcalT-\bcalT^*)\W_k^*\fraM_k(\bcalC^*)^{\dag}\fraM_k(\bcalC^*)\W_k^{*\top}\W_{k}^*$, we could get Equation~\eqref{eq:perturbation:2}. For each $j=1,\dots,d_k$, it has
	\begin{align*}
		&{~~~~}\ltwo{\left(\U_k\H_k-\U_k^*\right)_{j,\cdot}\fraM_k(\bcalC^*)}\\
		&\leq\ltwo{\calP_{\Omega_{j}^{(k)}}(\bcalT-\bcalT^*)}+\ltwo{\calP_{\Omega_{j}^{(k)}}(\bcalT)}\op{\W_k\fraM_k(\bcalC)^{\dag}\fraM_k(\bcalC)\W_k^{\top}- \W_k^*\fraM_k(\bcalC^*)^{\dag}\fraM_k(\bcalC^*)\W_k^{*\top} }\\
		&{~~~~~~~~~~~~~~~~~~~~~~~~~~~~~~~~~~~~~~~~~~~~~~~~~~~~~~}+\ltwo{\calP_{\Omega_{j}^{(k)}}(\bcalT)}\frac{\op{ \fraM_k(\bcalT-\bcalT^*)\W_k^*}}{\mins^*}.
	\end{align*}
\end{proof}

\begin{lemma}
	Suppose tensor $\bcalT^*\in\MM_{\r}$ has Tucker decomposition $\bcalT^*=\bcalC^*\cdot\llbracket\U_1^*,\dots,\U_m^*\rrbracket$. Let tensor $\bcalT\in\RR^{d_1\times\dots\times d_m}$ have $\text{HOSVD}(\bcalT)=\bcalC\cdot \llbracket \U_1,\cdots,\U_m\rrbracket$. Denote $\text{dist}(\U_k,\U_k^*):=\min_{\Q\in\OO_{r_k,r_k}}\op{\U_k\Q_k-\U_k^*}$ and $\Q_k=\arg\min_{\Q\in\OO_{r_k,r_k}}\op{\U_k\Q_k-\U_k^*}$. Then we have
	\begin{align*}
		&{~~~~}\op{(\otimes_{j\neq k}\U_j)\fraM_k(\bcalC)^{\dagger}\fraM_k(\bcalC)(\otimes_{j\neq i}\U_j)^{\top}-(\otimes_{j\neq k}\U_j^*)\fraM_k(\bcalC^*)^{\dagger}\fraM_k(\bcalC^*)(\otimes_{j\neq i}\U_j^*)^{\top} }\\
		&\leq 2\sum_{j\neq k}\text{dist}(\U_k,\U_k^*)+8\frac{\op{\bcalC\cdot\llbracket\Q_1^\top,\dots,\Q_m^\top\rrbracket-\bcalC^*}}{\mins^*}.
	\end{align*}
\label{teclem:coretensor1}
\end{lemma}

\begin{proof}
	For simplicity, denote $\tilde{\bcalC}:= \bcalC\cdot\llbracket\Q_1,\dots,\Q_m\rrbracket$. Then we have,
	\begin{align*}
		&{~~~~}\underbrace{\left((\otimes_{j\neq k}\U_j)\fraM_j(\bcalC)^{\dagger}\fraM_k(\bcalC) ( \otimes_{j\neq k}\U_j)^{\top}-(\otimes_{j\neq k}\U_j^*)\fraM_k(\bcalC^*)^{\dagger}\fraM_k(\bcalC^*) (\otimes_{j\neq k}\U_j)^{\top} \right)}_{C}\\
		&=\left((\otimes_{j\neq k}\U_j\Q_j)\fraM_k(\tilde\bcalC)^{\dagger}\fraM_j(\tilde\bcalC) (\otimes_{j\neq k}\U_j\Q_j)^{\top}-(\otimes_{j\neq k}\U_j^*)\fraM_k(\bcalC^*)^{\dagger}\fraM_k(\bcalC^*) (\otimes_{j\neq k}\U_j^*)^{\top} \right)\\
		&=\underbrace{\sum_{l\neq k}(\otimes_{j\neq k}\U_j^*)\fraM_k(\bcalC^*)^{\dagger}\fraM_k(\bcalC^*)(\otimes_{i>l,i\neq k} \U_i^{*})^{\top}\otimes(\U_l\Q_l-\U_l^*)^{\top}(\otimes_{i<l,i\neq k}\U_i\Q_i)^{\top}}_{C_1}\\
		&{~~~~}+\underbrace{\sum_{l\neq k}(\otimes_{i>l,i\neq j}\U_i\Q_i)\otimes(\U_l\Q_l-\U_l^*)(\otimes_{i<l,i\neq k}\U_i^*)\fraM_k(\bcalC^*)^{\dagger}\fraM_k(\bcalC^*)(\otimes_{j\neq k}\U_j\Q_j)^{\top}}_{C_2}\\
		&{~~~~}+\underbrace{(\otimes_{j\neq k}\U_j\Q_j)\left(\fraM_k(\tilde\bcalC)^{\dagger}\fraM_k(\tilde\bcalC)-\fraM_k(\bcalC^*)^{\dagger}\fraM_k(\bcalC^*) \right) (\otimes_{j\neq k}\U_j\Q_j)^{\top}}_{C_3}.
	\end{align*}
	Notice that \begin{align*}
		\op{C_1}\vee\op{C_2}\leq \sum_{j\neq k}\text{dist}(\U_k,\U_k^*).
	\end{align*}
	As for  term $C_3$, note that by Lemma~\ref{teclem:perturbation:matrix}, we have
	\begin{align*}
		\op{C_3}\leq 8\frac{\fro{\bcalC\cdot\llbracket\Q_1^\top,\dots,\Q_m^\top\rrbracket-\bcalC^*}}{\mins^*}.
	\end{align*}
\end{proof}

\end{document}